\documentclass[11pt]{article}
    
\usepackage{amsmath,amsfonts,amsthm,amscd,amssymb,graphicx}
\usepackage{mathabx}

\usepackage[utf8]{inputenc}

\usepackage{subfigure}
\usepackage{xcolor}

\numberwithin{equation}{section}

\usepackage{hyperref}

\newtheorem{theorem}{Theorem}[section]
\newtheorem{theo}{Theorem}[section]
\newtheorem{lemma}[theorem]{Lemma}
\newtheorem{lem}[theorem]{Lemma}
\newtheorem{proposition}[theorem]{Proposition}

\newtheorem{corollary}[theorem]{Corollary}

\newtheorem{remark}[theorem]{Remark}

\newtheorem{defi}[theorem]{Definition}

\newcommand{\RR}{\mathbb{R}}
\newcommand{\cO}{\mathcal{O}}

\newcommand{\cS}{\mathcal{S}}
\newcommand{\FcS}{\widehat{\mathcal{S}}}
\newcommand{\CC}{\mathbb{C}}

\newcommand{\cE}{\mathcal{E}}
\newcommand{\cH}{\mathcal{H}}
\newcommand{\ZZ}{{\mathbb Z}}
\newcommand{\FF}{{\widehat F}}

\newcommand{\les}{\lesssim}

\def\cI{\mathcal{I}}

\def\cL{\mathcal{L}}
\def\hv{\hat{v}}
\def\hw{\hat{w}}

\def\hV{\widehat{V}}
\def\hPsi{\widehat{\Psi}}
\def\cQ{\mathcal{Q}}

\def\Ff{\widehat{f}}

\def\FH{\widehat{H}}
\def\FcH{\widehat{{\mathcal{H}}}}

\def\Fg{\widehat{g}}

\def\Frho{\widehat{\rho}}
\def\Fphi{\widehat{\phi}}
\def\FE{\widehat{E}}

\def\FS{\widehat{S}}

\def\FG{\widehat{G}}

\def\FQ{\widehat{Q}}

\def\FB{\widehat{B}}
\def\Fb{\widehat{b}}

\def\FM{\widehat{M}}

\def\FJ{\widehat{J}}

\def\FcE{\widehat{\mathcal{E}}}
\def\FcQ{\widehat{\mathcal{Q}}}

\def\Tf{\widetilde{f}}

\def\Trho{\widetilde{\rho}}
\def\TE{\widetilde{E}}
\def\TS{\widetilde{S}}
\def\TG{\widetilde{G}}

\def\TX{\widetilde{X}}
\def\TY{\widetilde{Y}}
\def\TV{\widetilde{V}}

\def\TJ{\widetilde{J}}
\def\Tphi{\widetilde{\phi}}

\def\FS{\widehat{S}}
\def\bj{{\bf j}}

\begin{document}

\title{
Landau damping below survival threshold
}

\author{
Toan T. Nguyen\footnotemark[1]
}

\maketitle

\footnotetext[1]{Penn State University, Department of Mathematics, State College, PA 16802. Email: nguyen@math.psu.edu. The research is supported in part by the NSF under grants DMS-2054726 and DMS-2349981. The author would like to acknowledge the hospitality of the Institut des Hautes \'Etudes Scientifiques for a visit during which part of this research was
carried out.
}

\begin{abstract}

In this paper, we establish nonlinear Landau damping below survival threshold for collisionless charged particles following the meanfield Vlasov theory near general radial equilibria. In absence of collisions, the long-range Coulomb pair interaction between particles self-consistently gives rise to oscillations, known in the physical literature as plasma oscillations or Langmuir's oscillatory waves, that disperse in space like a Klein-Gordon's dispersive wave. As a matter of fact, there is a non-trivial survival threshold of wave numbers that characterizes the large time dynamics of a plasma: {\em phase mixing} above the threshold driven by the free transport dynamics and {\em plasma oscillations} below the threshold driven by the collective meanfield interaction. The former mechanism provides exponential damping, while the latter is much slower and dictated by Klein-Gordon's dispersion which gives decay of the electric field precisely at rate of order $t^{-3/2}$. Up to date, all the works in the mathematical literature on nonlinear Landau damping fall into the phase mixing regime, in which plasma oscillations were absent. The present work resolves the problem in the plasma oscillation regime. Our nonlinear analysis includes (1) establishing the existence and dispersion of Langmuir's waves, (2) decoupling oscillations from phase mixing in different time regimes, (3) detailing the oscillatory structure of particle trajectories in the phase space, (4) treating plasma echoes via a detailed analysis of particle-particle, particle-wave, and wave-wave interaction, and (5) designing a nonlinear iterative scheme in the physical space that captures both phase mixing and dispersion in low norms and allows growth in time in high norms. As a result, we establish nonlinear plasma oscillations and Landau damping below survival threshold for data with finite Sobolev regularity.

\end{abstract}

\tableofcontents


\section{Introduction}


One of the most central questions in plasma physics is whether a plasma, namely a collection of charged particles, in a non-equilibrium state will transition to turbulence or relax to neutrality in the large time. The question poses a great mathematical challenge due to the extremely rich underlying physics including phase mixing, Landau damping, oscillations, and trapped geodesics, among others. While {\em phase mixing} is purely a free transport phenomenon, {\em plasma oscillations} arise self-consistently due to the long-range pair interaction between the charged particles that generates oscillatory electric fields, also known as Langmuir's oscillatory waves in the physical literature \cite{Trivelpiece}. These Langmuir waves oscillate in time and disperse in space like a Klein-Gordon dispersive wave as recently confirmed in the mathematical literature \cite{HKNR3, BMM-lin, Toan, HKNR4} for a linearized system near fixed background equilibria. 

L. Landau in his original paper \cite{Landau} addresses the very question of whether such an oscillation is damped (i.e. energy exchange from waves to particles, or potential to kinetic energy), ever since known as {\em Landau damping}, for the linearized system. Landau damping can be explicitly computed, as done in \cite{Landau,Trivelpiece} near spatially homogenous states, and is sensitive to the decaying rate of the background electrons: the faster the background vanishes at its maximal velocity, the weaker Landau damping is. Furthermore, as was recently discovered in \cite{Toan,HKNR4}, there is a nontrivial {\em survival threshold} of wave numbers that completely characterizes the decay mechanism for the linearized electric field: 
\begin{itemize}
\item Phase mixing above the threshold


\item Plasma oscillations below the threshold

\end{itemize}
The classical Landau damping occurs due to the resonant interaction at the threshold between the two regimes. We note in particular that plasma oscillations are always present for plasmas in the whole space (in fact, also on a confined space as long as a sufficiently ``long-range" pair interaction is allowed). See also \cite{ChanjinToan, Chanjin} for a quantum mechanical counterpart. This dynamical picture remains open for the full nonlinear problem, and all the previous works in the mathematical literature on nonlinear Landau damping fall into the phase mixing regime, leaving plasma oscillations untouched. In this paper, we resolve this very issue, thus establishing {\em nonlinear Landau damping below survival threshold.} As a byproduct, we establish the existence and the survival of nonlinear plasma oscillations or  Langmuir's nonlinear oscillatory waves in the large time. 

\subsection{Phase mixing}\label{secPM}

Phase mixing is a damping mechanism due to shearing in the phase space $\RR^3_x\times\RR^3_v$ by the free-transport dynamics 
\begin{equation}\label{free}
\partial_t f + v \cdot \nabla_x f = 0
\end{equation}
leading to decay of macroscopic quantities such as the charged density 
$\rho(t,x) = \int_{\RR^3} f(t,x,v) \; dv.
$
Indeed, as particles are shearing along the particle trajectories $(x,v)\mapsto (x+vt,v)$, the contracting in volume $dv = t^{-3} dx_t$ leads to dispersive decay of the density. In addition, since the vector field $\nabla_v +t\nabla_x$ commutes with the transport dynamics, $x$-derivatives of the density gain extra decay of order $t^{-1}$ at a cost of $v$-derivatives. 

The notion of phase mixing may be more apparent when examining the transport dynamics of an elementary wave $e^{ik\cdot (x-vt)} g(v)=e^{ik\cdot x} e^{-ikt\cdot v}g(v)$ (or a moving wave packet). The resulting oscillation $e^{-ikt\cdot v}$ in $v$ leads to rapid decay of the density, thanks to velocity averaging, at cost of $v$-derivatives. The decay is exponentially fast in $|kt|$ if data are analytic and polynomially fast if data only have finite regularity. This exponential damping is at the heart of Mouhot-Villani's celebrated proof of nonlinear Landau damping and its subsequent extensions for analytic and Gevrey data \cite{MV,BMM-apde,GNR1,GNR2, Ionescu2}. As a mater of fact, phase mixing \eqref{free} is the sole damping mechanism in all the previous works on nonlinear Landau damping both on torus \cite{MV,BMM-apde,Young, GNR1,GNR2, Ionescu2, Bed1,CLN} as well as in the whole space \cite{BMM-cpam, HKNR2,TrinhLD,IPWW}. It is also the driving force in inviscid damping \cite{BM1, IonescuJia1,MasZ1}.  

\subsection{Landau damping}\label{secLD}

In the mathematical literature, Landau damping broadly concerns decay in the large time of the self-consistent electric field in a non-trivial non-equilibrium state. Specifically, charged particles are transported through the meanfield Vlasov equation:
\begin{equation}\label{VP}
\partial_t f + v \cdot \nabla_x f + E \cdot \nabla_v f = 0
\end{equation} 
for particle distribution function $f(t,x,v)$, in which the electric field $E=-\nabla_x \phi$ is self-consistent and generated by the charged density (or rather by the non-neutrality of plasmas in a non-equilibrium state) through the Poisson equation $-\Delta_x \phi = \rho -n_{\mathrm{ions}}$ for some nonnegative constant $n_{\mathrm{ions}}$ accounting for density of a uniform ions background. 

To identify damping mechanism \eqref{VP} near a spatially homogenous equilibrium state $\mu(v)$ at neutrality $\int \mu\; dv = n_{\mathrm{ions}}$, one may first linearize the quadratic interaction $E\cdot \nabla_vf \sim E\cdot \nabla_v\mu$, and study the linearization. This can in fact be solved completely mode by mode, leading to study the spacetime Fourier-Laplace symbol  
\begin{equation}\label{def-Dintro}
D(\lambda,k) = 1 - \frac{1}{|k|^2} \int_{\RR^3} \frac{ik \cdot \nabla_v\mu }{ \lambda + ik \cdot v}\; dv,
\end{equation}
for each wave number $k\in \RR^3$ and temporal frequency $\lambda \in \CC$. That is, $\frac{1}{D(\lambda,k)}$ is the resolvent of the linearized electric field. This symbol $D(\lambda,k)$, also known as the dielectric function, is a central function in plasma physics \cite{Trivelpiece}, collectively examining the ratio between the meanfield effect $E\cdot \nabla_v \mu$ to the free transport dynamics (whose resolvent is exactly $\frac{1}{\lambda+ik\cdot v}$). 
By construction, the zeros $\lambda(k)$ of $D(\lambda,k)=0$ are eigenvalues of the corresponding linearized system, and therefore, the behavior of $e^{\lambda(k)t+ik\cdot x}$ dictates that of the linearized electric field for each elementary mode $e^{ik\cdot x}$; for details, see Section \ref{sec-lineartheory}.  

It follows that general non-negative radial profiles $\mu(|v|)$ are {\em spectrally stable} \cite{Penrose, Trivelpiece, MV, Toan}, namely no unstable modes exist. Under a stronger Penrose's stability condition: $|D(\lambda,k)|\gtrsim 1$ for all $\Re \lambda \ge 0$, the equilibria are indeed {\em linearly stable}, and linearized electric field decays rapidly fast, dictated by phase mixing mechanism \cite{MV}. See also \cite{GNR1,HKNR2} where the linearized dynamics is decomposed into that of the free-transport, plus an exponentially localized term (for analytic equilibria). Namely, the free-transport dynamics dictates that of the Vlasov system in the large time. Remarkably, this holds nonlinearly as established in \cite{MV,BMM-apde,Young, GNR1,GNR2, Ionescu2} for confined plasmas on torus, in \cite{BMM-cpam, HKNR2,TrinhLD} for unconfined plasmas with screening in the whole space, and in \cite{Bed1,CLN} for weakly collisional plasmas. Phase mixing \eqref{free} is again the sole damping mechanism in all the aforementioned works. Specifically, 

\begin{itemize}

\item Imposing Penrose's stability condition allows to treat linear term $E\cdot \nabla_v\mu$ as a perturbation: namely, the Green function of the linearized problem is exactly that of the free-transport dynamics, plus an exponentially localized-in-time kernel; see \cite{GNR1,HKNR2}. 

\item Suppressing plasma echoes \cite{Gould, Bed2, GNR2} via exponential damping that is present thanks to analytic or Gevrey regularity, see \cite{MV,BMM-apde,Young, GNR1,GNR2, Ionescu2} for Landau damping on torus. See also \cite{BM1, IonescuJia1,MasZ1} for an analog in inviscid damping.  

\item Suppressing plasma echoes via enhanced dissipation that is turned on before echoes develop (i.e. up to the time scale that phase mixing remains dominant, leading to suppression threshold of Sobolev perturbations, see \cite{Bed1,CLN}).  
 
\item Screening\footnote{Namely, replacing the Poisson equation by $(1-\Delta)\phi = \rho-n_{\mathrm{electron}}$. The screened Vlasov-Poison system is used in the physical literature to describe the dynamics of ions, while electrons are assumed to have reached their Maxwell-Boltzmann equilibrium, see, e.g., \cite{BGNS}.} out the low frequency regime so that the Penrose's stability condition remains valid, and therefore phase mixing is again the driving force to control plasma echoes in the whole space, see \cite{BD,BMM-cpam,HKNR2,TrinhLD}. Thanks to dispersion of \eqref{free}, the nonlinear Landau damping results are established for data with a finite regularity (roughly, $C^1$).
   
\end{itemize} 

Finally, we mention a recent work \cite{IPWW} on Landau damping near a special equilibrium $\mu(v) = \langle v\rangle^{-4}$ on the whole space $\RR^3\times \RR^3$, for which Landau damping incidentally coincides phase mixing (i.e. retaining both exponential damping $e^{-|kt|}$ and extra decay for spatial derivatives for each perturbed wave packet $e^{ik\cdot x}$). For all other equilibria (i.e. as soon as $\mu(v)$ decays faster than $\langle v\rangle^{-4}$), such a phase mixing is no longer available and Landau damping occurs at an extremely slower rate of order $e^{-|k|^Nt}$, for equilibria that decay at order $\langle v\rangle^{-N-3}$, and at an exponentially slower rate of order $|k|^{-3}exp(e^{-1/|k|^2}t)$ for Gaussians as $|k|\ll1$, see \cite{GS-LD1,GS-LD2,Toan}.  It remains elusive whether such a slower damping rate plays a role in resolving the classical Landau damping at the nonlinear level. We stress that at the very low frequency $|k|\ll1$, Penrose's stability condition $|D(\lambda,k)|\gtrsim 1$ for $\Re \lambda \ge 0$ never holds for {\em any equilibria!}, and therefore $E\cdot \nabla_v\mu$ can no longer be seen as a perturbation of the free transport dynamics. Plasma oscillations arise due to the vanishing of $D(\lambda,k)=0$, which in turn drives the particle dynamics.

\subsection{Plasma oscillations}\label{secOsc}

At the low frequency, the electric field oscillates in time and disperses in space like a Klein-Gordon dispersive wave. Indeed, integrating \eqref{VP} in $v$, we obtain the local mass conservation $\partial_t \rho + \nabla_x \cdot \bj=0$ and momentum conservation $\partial_t \bj + \nabla_x \cdot M_2 -\rho E =0$, in which $\bj = \int vf\; dv$ and $M_2 = \int v\otimes v f\; dv$. Next, combining the conservation laws and using the Poisson equation for the electric potential, we obtain the following ``plasma oscillations'' equation 
\begin{equation}\label{plsmaosc}
\partial_t^2 E + \rho E =\nabla_x \Delta_x^{-1}\nabla_x\cdot \nabla_x\cdot M_2.
\end{equation}  
Namely, near equilibria $\rho \approx n_{\mathrm{ions}}$, the electric field oscillates in time at frequency $\pm \sqrt{n_{\mathrm{ions}}}$. Such an oscillation is always present, though not ``visible'' in the phase mixing regime due to exponential damping on torus or dispersion in the whole space (noting $\partial_x M_2 \sim t^{-4}$ in the phase mixing regime, see Sections \ref{secPM} and \ref{secLD}).    

However, $\partial_x M_2 $ may be larger than expected in the low frequency regime. Indeed, for sake of presentation, let us focus on the evolution of each elementary oscillatory electric pulse $E=ik e^{\lambda_0 t + ik\cdot x}$, where $\lambda_0 = \pm i \sqrt{n_{\mathrm{ions}}}$. From \eqref{VP}, we study its linearization near equilibria $\mu(v)$, yielding  
$$
\begin{aligned} 
\nabla_x \cdot M_2 
& = \nabla_x \cdot M^{\mathrm{free}}_2 
-\nabla_x \cdot\int v\otimes v \frac{E\cdot \nabla_v\mu}{\lambda_0+ik\cdot v}\; dv 
\end{aligned}$$
in which $M^{\mathrm{free}}_2$ is the second moment generated by the free transport dynamics, and therefore is of order $t^{-3}$ as expected from phase mixing. On the other hand, in the low frequency regime $|k|\ll1$, we may expand $\frac{1}{\lambda_0 + ik\cdot v} \approx \frac1{\lambda_0} - \frac{1}{\lambda_0^2} ik\cdot v$ in the above integration. The first term $\frac1{\lambda_0}$ contributes nothing into $\nabla_x \cdot M_2$ (provided that equilibria are radial), while the second leads to a contribution of $\frac53e_0\Delta_x E$ to the right hand side of \eqref{plsmaosc} for $e_0 = \frac{1}{n_{\mathrm{ions}}} \int|v|^2 \mu\; dv$. That is, to leading order, the plasma oscillation equation \eqref{plsmaosc} becomes  
 \begin{equation}\label{plsmaosc1}
\partial_t^2 E + n_{\mathrm{ions}} E =  \frac53e_0\Delta_x E + \nabla_x \cdot M^{\mathrm{free}}_2 .
\end{equation} 
Namely, instead of the Poisson equation $E = \nabla_x \Delta^{-1}_x \rho$, the electric field is now generated through the Klein-Gordon equation \eqref{plsmaosc1} at the very low frequency. As a result, the electric field no longer decays rapidly fast, dictated by phase mixing, but disperses like a Klein-Gordon oscillatory wave at a rate of order $t^{-3/2}$. Namely, unlike all the aforementioned works where the free-transport dynamics dominates, the Vlasov dynamics \eqref{VP} is now driven by the oscillatory electric field (i.e. through the meanfield interaction $E\cdot \nabla_v f$).  

These oscillations are classically known as Langmuir's waves in the physical literature \cite{Trivelpiece}. For the linearized system, plasma oscillations and dispersive behavior of the electric field have been justified recently \cite{HKNR3, BMM-lin, Toan, HKNR4} near fixed background equilibria. However, there are no available works in the literature that justify plasma oscillations at a nonlinear level. 


\subsection{Survival threshold}

Plasma oscillations \eqref{plsmaosc1} may in fact persist not only at the very low frequency, but also up to a non-trivial threshold of wavenumbers $\kappa_0$, depending on the maximal speed of particle velocities $\Upsilon = \sup\{|v|, \mu(v)\not =0\}$. Indeed, it was recently shown in \cite{Toan} that for radial equilibria $\mu(|v|)$, there is a survival threshold $\kappa_0$ 
defined by 
 \begin{equation}\label{def-introkappa0}
\kappa_0^2 = 4\pi \int_0^\Upsilon \frac{u^2\mu(u)}{\Upsilon^2-u^2} \;du
\end{equation}
so that  the followings hold.

\begin{itemize}

\item {\em Plasma oscillations:} for $0\le |k|\le \kappa_0$, there are exactly two pure imaginary solutions $\lambda_\pm(k) = \pm i \tau_*(k)$ of the dispersion relation $D(\lambda,k)=0$, which obey a Klein-Gordon type dispersion relation 
$$\tau_*(k) \approx \sqrt{1 + |k|^2}$$ 
(and in particular, $\tau_*(k)$ is a strict convex function in $|k|$). These oscillatory modes experience no Landau damping $\Re \lambda_\pm(k)=0$, but disperse in space, since the group velocity $\tau_*'(k)$ is strictly increasing in $|k|$. This dispersion leads to a $t^{-3/2}$ decay of the electric field in the physical space. These oscillations are known as Langmuir's waves in plasma physics \cite{Trivelpiece}. In addition, the phase velocity of these oscillatory waves $\nu_*(k) = \tau_*(k)/|k|$ is a decreasing function in $|k|$ with $\nu_*(0) = \infty$ and $\nu_*(\kappa_0) = \Upsilon$ (the maximal speed of particle velocities). 

\item {\em Landau damping at $|k|=\kappa_0$:} as $|k|$ increases past the critical wave number $\kappa_0$, the phase velocity of Langmuir's oscillatory waves enters the range of admissible particle velocities, namely $|\nu_*(k)| <\Upsilon$. That is, there are particles that move at the same propagation speed of the waves. This resonant interaction causes the dispersion functions $\lambda_\pm(k)$ to leave the imaginary axis, and thus the oscillatory modes get damped. Landau  \cite{Landau} managed to compute this damping rate $\Re\lambda_\pm(k)<0$. This leads to a transfer of energy from the electric energy (i.e. damping in the $L^2$ energy norm) to the
kinetic energy of these particles. This transfer of energy at the resonant velocity defines the classical notion of Landau damping. In the other words, Landau damping occurs due to the resonant interaction between particles and the oscillatory waves. 

\item {\em Phase mixing regime:} for $|k| > \kappa^+_0$, the Penrose stability condition holds, and therefore the electric field is exponentially damped and governed by the free transport dynamics as discussed in Section \ref{secLD}. 

\end{itemize}

In other words, the survival threshold characterizes the spectral gap in the linearized problem: positive spectral gap of order $\Re \lambda_\pm(k) \approx -|k|$ above the threshold, rapidly vanishing spectral gap of order $\Re \lambda_\pm(k)\approx -|k-\kappa_0|^N$ at the threshold (where $N$ depends on the vanishing rate of equilibria at its maximal speed, see \cite{Toan}), and no spectral gap $\Re \lambda_\pm(k)=0$ below the threshold. Note that the survival threshold $\kappa_0>0$ iff $\Upsilon<\infty$, which includes compactly supported equilibria or equilibria with relativistic velocities. The above linear picture, with a non-trivial survival threshold $\kappa_0$, also holds for the relativistic Vlasov-Maxwell as recently established in \cite{HKNR4}, and for a quantum model as in \cite{ChanjinToan}. Note that the Landau damping mechanism is much weaker than the dispersion of oscillations. 
The faster the profile $\mu(v)$ vanishes at the maximal particle speed, the weaker Landau damping is. In particular, as was shown in the pioneering works by Glassey and Schaeffer \cite{GS-LD1,GS-LD2} that the linearized electric field near Gaussian equilibria cannot in general decay faster than $1/ (\log t)^{13/2}$ in $L^2$ norm, while near polynomially decaying equilibria at rate $\langle v\rangle^{-\alpha}$, $\alpha>1$, the electric field cannot decay faster than $t^{-\frac{1}{2(\alpha-1)}}$. In addition, there is no Landau damping (i.e. no decay for $L^2$ norm of the electric field) near compactly supported equilibria. The main mechanism is therefore the dispersion of the electric field or plasma oscillations. 
As discussed, all the previous works on nonlinear Landau damping fall into the phase mixing regime driven by the free transport dynamics, see Section \ref{secLD}, and therefore plasma oscillations were excluded. In this paper, we provide a new framework to study plasma oscillations and Landau damping in the regime where the dynamics is driven by the electric field. 

\subsection{The Vlasov-Klein-Gordon system}

In this paper, we aim to investigate Landau damping below the survival threshold, namely the large time behavior of the Vlasov dynamics \eqref{VP} for charged particles driven by a self-consistent electric field that is generated through the plasma oscillation equation \eqref{plsmaosc1}. Specifically, in this paper, we consider the following relativistic Vlasov-Klein-Gordon system
\begin{equation}
\label{VKG1}
\partial_t f + \hv\cdot \nabla_x f + E\cdot \nabla_v f = 0 
\end{equation} 
\begin{equation}\label{VKG2}
E = -\nabla_x \phi, \qquad (\Box_{t,x}+m^2_0) \phi = -\rho + n_{\mathrm{ions}} 
\end{equation}
posed in the whole space $\RR^3_x\times \RR^3_v$, with $\Box_{t,x} = \partial_t^2 - \Delta_x$, for initial data 
\begin{equation}\label{data}f(0,x,v) = f_0(x,v), \qquad \phi(0,x) = \phi_0(x), \qquad \partial_t\phi(0,x) = \phi_1(x), \end{equation}
modeling the dynamics of an electron, where the density $\rho(t,x)$ is defined by 
$$\rho(t,x) = \int_{\RR^3} f(t,x,v)\; dv.$$
This system is relativistic with particle velocity $ \hv = v/\sqrt{1+|v|^2}$, and is nonlinear through the quadratic interaction $E\cdot \nabla_v f$, since $E$ depends linearly on $f$ through its density $\rho(t,x)$. In \eqref{VKG2}, $ n_{\mathrm{ions}}$ is a non-negative constant representing the uniform ions background.   

Not only is the exact system \eqref{VKG1}-\eqref{data} 
found in the study of Landau damping below the survival threshold, but may also be seen as a special case of the relativistic Vlasov-Maxwell system where the electromagnetic fields are self-consistently generated through the classical Maxwell equations. Like the case of Vlasov-Maxwell systems \cite{GlasseyStrauss, Klainerman, BGP}, the global smooth solutions to the Vlasov-Klein-Gordon system for general data are not known to exist, but global weak solutions or continuation criteria for classical solutions \cite{Kunzinger1, Kunzinger2}. In addition, the scattering theory and large time behavior of solutions near vacuum is well-understood \cite{GlasseyStrauss2, Big1, Big2, Wang} for the Vlasov-Maxwell systems, and only recently established \cite{ToanVKG} for the Vlasov-Klein-Gordon system. Not surprisingly, nonlinear Landau damping for both relativistic Vlasov-Klein-Gordon and Vlasov-Maxwell systems has been widely open due to the presence of survival threshold and plasma oscillations, while the linear damping theory for Vlasov-Maxwell systems\footnote{Interestingly, the exact Klein-Gordon's dispersive behavior remains valid for electric potentials below survival threshold and magnetic potentials for the full range of wave numbers as in \eqref{VKG2}, see \cite{HKNR4}.} has only been resolved recently in \cite{HKNR4}. 

In this paper, we prove nonlinear Landau damping for the relativistic Vlasov-Klein-Gordon system, the first such a nonlinear Landau damping result in presence of oscillations and outside of the phase mixing regime. The novelty of this work is to provide a new framework to study oscillations in the physical phase space via a detailed description of particles moving in an oscillatory field and of plasma oscillations generated by the collective interacting particles. The framework of dealing with the interaction between particles and oscillatory waves should also be found useful in several other contexts such as Vlasov-Maxwell, Vlasov-{N}ordstr\"{o}m, or Einstein-Vlasov systems. 

\subsection{Main results}

In this paper, we address precisely the question of asymptotic stability of  \eqref{VKG1}-\eqref{VKG2} near general stable equilibria $\mu(v)$ and prove the nonlinear Landau damping for data with finite Sobolev regularity in the whole space $\RR^3_x\times \RR^3_v$. Specifically, we consider nonnegative, sufficiently smooth, compactly supported, and radial equilibria of the form $\mu(v) =\varphi(\langle v\rangle)$ 
with $$
\int_{\RR^3} \varphi(\langle v\rangle) \; dv =n_{\mathrm{ions}}$$ 
where $\langle v\rangle = \sqrt{1+|v|^2}$. That is, at the equilibrium, the charged particles are electrically neutral, and the Klein-Gordon field is identically zero. In addition, we assume that the effective mass 
\begin{equation}\label{def-tau0}
\tau_0^2 = m_0^2 +  \int_{\RR^3}\varphi'(\langle v\rangle)\; dv 
\end{equation}
is well-defined and strictly positive. This positivity can be viewed as a spectral stability condition, without which there may be unstable growing modes to the linearization of  \eqref{VKG1}-\eqref{VKG2} around equilibria $\mu(v)$, see Section \ref{sec-spectral} below. 

Our main results read as follows. 

\begin{theo} \label{maintheorem} Fix $N_0\ge 14$. Let $\mu(v) =\varphi(\langle v\rangle)$ be nonnegative, sufficiently smooth, compactly supported, and radial equilibria so that $\int_{\RR^3} \varphi(\langle v\rangle) \; dv =n_{\mathrm{ions}}$ and $\tau_0^2>0$, with $\tau_0$ defined as in \eqref{def-tau0}. Let $\epsilon_0>0$ and $f_0(x,v), \phi_0(x), \phi_1(x)$ be initial data \eqref{data} so that 
\begin{equation}\label{data-phi}
\| f_0-\mu\|_{L^1_v W^{N_0+4,1}_x}
+ \| \nabla_x\phi_0\|_{W^{N_0+4,p}}+\|\nabla_x\phi_1\|_{W^{N_0+3,p}} \le \epsilon_0,
\end{equation}
for any $p\in [1,\infty]$. In addition, $f_0(x,v)$ is compactly supported in $v$.

Then, if $\epsilon_0$ is small enough, the nonlinear Landau damping holds: precisely, the solution $(f(t,x,v),E(t,x))$ to the system \eqref{VKG1}-\eqref{data} exists globally in time, and the nonlinear electric field can be decomposed into  
\begin{equation}\label{main-repE}
\begin{aligned}
E &= \sum_\pm E^{osc}_\pm(t,x) + E^r(t,x)
\end{aligned}\end{equation}
where $E^{osc}_\pm(t,x)$ is oscillatory like a Klein-Gordon wave, and there hold
\begin{equation}\label{maintheo-bdsE}
\begin{aligned}
\| E^{osc}_\pm(t)\|_{W^{1,p}_x} &\lesssim \epsilon_0 \langle t\rangle^{-3(1/2-1/p)} , \qquad p\in [2,\infty],
\\
\|  E^r(t)\|_{W^{1,p}_x}& \lesssim \epsilon_0 \langle t\rangle^{-3(1-\frac1p)}, \qquad p\in [1,\infty], 
\end{aligned}
\end{equation}
for all $t\ge 0$. 


\end{theo}

As a consequence, we also obtain the following scattering result. 

\begin{corollary}\label{cor-main} Under the same assumptions as in Theorem \ref{maintheorem}, the scattering theory holds and particles scatter in the large time: namely, for any initial state $(x,v)$, there exists a final velocity $V_\infty(x,v)$ in $C^1(\RR^3\times \RR^3)$ so that the particle trajectories\footnote{Throughout this paper, we in fact work with the backward characteristic, see Section \ref{sec-Char}.} $(X(t;x,v), V(t;x,v))$ satisfy 
\begin{equation}\label{scatterXV}
\| V(t;x,v) - V_\infty(x,v)\|_{L^\infty_{x,v}}  \lesssim \epsilon_0 \langle t\rangle^{-3/2}, \qquad \| X(t;x,v) - x + tV_\infty(x,v)\|_{L^\infty_{x,v}} \lesssim \epsilon \langle t\rangle^{-1/2}
.\end{equation}
In particular, there exists an $f_\infty(x,v)$ in $C^1(\RR^3\times \RR^3)$ so that 
\begin{equation}\label{scatterf}\|f(t,x+t\hv,v) - f_\infty(x, v)\|_{L^\infty_{x,v}} \lesssim \epsilon_0\langle t\rangle^{-1/2}, \end{equation}
for all $t\ge 0$.

\end{corollary}

The main results assert that Landau damping (or rather, the survival of plasma oscillations from the classical Landau damping) holds at the nonlinear level for data with finite Sobolev regularity. The assumptions on initial data are not optimal and may be relaxed, including the fact that both initial data and equilibria are compactly supported in $v$. Unlike all the previous works on Landau damping where phase mixing is the sole driving force, the Vlasov-Klein-Gordon system is primarily driven by the oscillatory electric field. 
In fact, in the proof, see Section \ref{sec-iterscheme}, we shall keep track the oscillatory structure of the electric field, namely 
\begin{equation}\label{true-Eosc}
\begin{aligned}
E^{osc}_\pm(t,x) &= G_\pm^{osc}(t,x) \star_{x} \nabla_x S_0(x)+ G_\pm^{osc} \star_{t,x} \nabla_x S(t,x) .
\end{aligned}
\end{equation}
Here, $G_\pm^{osc}(t,x)$ is the oscillatory Green function whose Fourier transform is equal to $e^{\pm i \nu_*(k)t} a_\pm(k)$, where the dispersion relation $\nu_*(k)$ behaves like a Klein-Gordon dispersion, namely $\nu_*(k) \sim \sqrt{1+|k|^2}$ for all $k\in \RR^3$, with $\nu_*(0) = \tau_0$ defined as in \eqref{def-tau0}, while $S(t,x)$ arises as a nonlinear source density in view of the Klein-Gordon equation $(\Box_{t,x}+m_0^2)E = \nabla_x \rho$. 
We stress however that the structure of Klein-Gordon oscillations is far from being obvious even at the linearized level, noting the collective effect from the density $\rho$ in \eqref{VKG2} which may disturb the Klein-Gordon structure for the electric field. 

\subsection{Difficulties and main ideas}
Let us now point out several new difficulties that we must face and provide the main ideas on how we shall resolve them. A sketch of the proof of the main results will also be outlined. 

\subsection*{Langmuir's waves.}
As particles are near non-trivial equilibria $\mu(v)$, the linearized Vlasov dynamics is now driven by the linear forcing term $E\cdot \nabla_v \mu$, leading to 
\begin{equation}\label{lin-den} \nabla_x \rho = \nabla_x \rho^{free} -\int_0^t \int_{\RR^3} \nabla_x [E\cdot \nabla_v\mu](s,x-\hv (t-s),v) \; dvds,\end{equation}
in which $\rho^{free}$ is the density generated by the free transport dynamics, and is of order $t^{-4}$ by phase mixing. The nonlocal integral term, denoted by $\mathcal{K} E$, is put back into the Klein-Gordon equation \eqref{VKG2}, yielding a closed equation for the linearized electric field 
\begin{equation}\label{KG-E}(\Box_{t,x}+m_0^2+ \mathcal{K})E = \nabla_x \rho^{free}.
\end{equation}
It is however highly non-trivial to ensure that this nonlocal integral term would not disturb the Klein-Gordon dispersive structure at all frequencies. This was resolved in \cite{Toan, HKNR4} for the Vlasov-Poisson and Vlasov-Maxwell systems, yielding the existence of Langmuir's oscillatory  waves which behave exactly like that of a Klein-Gordon wave, provided the spectral condition $\tau_0^2\ge 0$, see \eqref{def-tau0}. The inverse of the modified Klein-Gordon operator $\Box_{t,x}+m_0^2+ \mathcal{K}$ is thus solved through its resolvent, giving its temporal Green function of the form  
$$\sum_\pm G^{osc}_\pm+ G^r(t,x)$$
where $G^{osc}_\pm(t,x)$ is the Klein-Gordon dispersive component which decays at rate of order $t^{-3/2}$, while $G^r(t,x)$ is the regular component dictated by phase mixing. Remarkably, phase mixing component is of order $t^{-4}$ and thus well-separated from oscillations. 
This well-separated behavior may formally be seen from \eqref{KG-E}, which asserts that $E^r \sim \nabla_x\rho^{free}$, up to an extraction of oscillations (noting we only have $E^r \sim t^{-3}$ at the nonlinear level which is sharp in view of quadratic oscillations). This is in a great contrast to the case near vacuum where $E = \nabla_x \Delta_x^{-1}\rho^{free}$, which only decays at order $t^{-2}$, see \cite{BD}. See Section \ref{sec-lineartheory} where we develop this linear theory for the Vlasov-Klein-Gordon system near general radial equilibria.

\subsection*{Plasma echoes.} 
The first common issue in dealing with the nonlinear Vlasov dynamics is to control quadratic interaction $E\cdot \nabla_v f$, which is the main source for plasma echoes \cite{Gould, Bed2,GNR2}. Observe that the Klein-Gordon dispersion only yields decay at rate of order $t^{-3/2}$, while the free transport dynamics dictates $\partial_v = \cO(t)$, leading to $E\cdot \nabla_v f = \cO(t^{-1/2})$, which is far from being integrable in time, not to mention the apparent loss of derivatives in $v$. Therefore, we must bootstrap and make use of oscillations, see \eqref{true-Eosc}. We may attempt to resolve the issue, following the ``standard procedure'', by deriving decay for low norms and propagating high norms with some possible growth in time. However, this does not work for Vlasov equations, since at the top order of derivatives, the nonlinear interaction $E_\mathrm{low} \cdot \nabla_v f_\mathrm{high} = \mathcal{O}(t^{-1/2})$ and $E_\mathrm{high} \cdot \nabla_v f_\mathrm{low} = \mathcal{O}(1)$, leading to a growth in time of orders $\langle t\rangle^{1/2}$ and $\langle t\rangle$, respectively (noting $f$ does not decay and $E$ has no oscillation in high norms). 

To avoid the apparent loss of $v$-derivatives, we shall work with the Lagrangian coordinates, an approach that was developed in \cite{HKNR2} for the screened Vlasov-Poisson system, namely following the genuine particle trajectories, which are defined by  
\begin{equation}\label{odeXV} \dot X_{s,t} = \hV_{s,t}, \qquad \dot V_{s,t} = E(s,X_{s,t}), \end{equation}
 with initial positions $(x,v)$ at $s=t$. Therefore, at the nonlinear level, the Klein-Gordon equation for the electric field \eqref{KG-E} is now forced additionally by $\nabla_x S^\mu$, where 
\begin{equation}\label{def-Smuintro}
S^{\mu}(t,x) = \int_0^t \int_{\RR^3}  \Big[ E(s,x - (t-s)\hv )\cdot \nabla_v \mu(v) 
 -E(s,X_{s,t}(x,v))\cdot \nabla_v \mu(V_{s,t}(x,v)) \Big] \, dv  ds,
\end{equation}
which accounts for the feedback of collective particles to the waves. The integral $S^{\mu}(t,x)$ is viewed as a quadratic density source, which decays at best at order $t^{-3}$ in view of the free transport dispersion for densities and the quadratic oscillations for the oscillatory electric field. In addition, due to the presence of oscillations, we do not expect extra decay for spatial derivatives of $S^{\mu}(t,x)$ (i.e. phase mixing is no longer available). This is the content of Section \ref{sec-sourceest}. 

In view of the electric field decomposition $E = E^{osc}_\pm + E^r$, we may expect that the position and velocity of the particles obey a similar decomposition. Indeed, we shall prove that 
\begin{equation}\label{Xdep-intro} 
X_{s,t}(x,v) = x - t \hV^{mod}_{t,t}(x,v) + X^{osc}_{s,t}(x,v) + X^{tr}_{s,t}(x,v) + X^R_{s,t}(x,v)
\end{equation}
where $V^{mod}_{t,t}(x,v) = v + \cO(\epsilon t^{-3/2})$, $X_{s,t}^{osc} = \cO(s^{-3/2})$, and $X_{s,t}^{tr} = \cO(s^{-1})$, plus a faster remainder $X^R_{s,t}(x,v)$. See Proposition \ref{prop-charPsi}. Roughly speaking, $X^{osc}_{s,t}(x,v)$ collects Klein-Gordon dispersion, while $X^{tr}_{s,t}(x,v)$ is due to phase mixing component of the electric field $E^r \sim t^{-3}$ and is therefore sharp after double integration in time in view of \eqref{odeXV}. Therefore, plasma echoes may arise due to particle-particle, particle-wave, and wave-wave interaction, where particles satisfy phase mixing (or free-transport dispersion), while waves follow Klein-Gordon dispersion. Namely,  

\begin{itemize}

\item {\em Particle-particle interaction:}
This interaction arises due to the quadratic term  $E^r \cdot \nabla_v f^r$, namely the regular part of the electric field and the regular part of the Vlasov dynamics, leading to a density source $S^{\mu,r,tr}(t,x)$ with $E$ replaced by $E^r$, see \eqref{def-Smuintro}. Due to the presence of the transport component $X^{tr}_{s,t}(x,v) = \cO(s^{-1})$, this causes a log loss in deriving decay of $S^{\mu,r,r}(t,x)$. However, as this is the term due to transport dynamics, we expect that phase mixing remains to play a role in closing the analysis, as treated in the screening case \cite{HKNR2}. See Section \ref{sec-SEr}.      

\item {\em Particle-wave interaction:}
This is due to the quadratic term $E^{osc}_\pm \cdot \nabla_v f^r$ (or $E^r \cdot \nabla_v f^{osc}$), leading to corresponding density sources in view of \eqref{def-Smuintro}. In this case, we make use of the fact that waves oscillate, while particles are driven by phase mixing, thus gaining extra decay after an integration by part in time, see Section \ref{sec-SEosc}. In a recent work \cite{ToanVKG}, we develop a new framework to treat particle-wave interaction near vacuum, which plays a role in designing the nonlinear analysis near equilibria, see Section \ref{sec-iterscheme}.  

\item {\em Wave-wave interaction:}
This is due to the quadratic term $E^{osc}_\pm \cdot \nabla_v f^{osc}_\pm$, leading to corresponding density sources $S^{\mu, osc}_{\pm,\pm}$ and $S^{\mu, osc}_{\pm,\mp}$. Note that in this case, the integrand in \eqref{def-Smuintro} is at best of order $s^{-3}$, which would not be sufficient to derive decay of order $t^{-3}$ for the sources. In the non-resonant interaction $\{\pm, \pm\}$, we may integrate by parts in time, gaining an extra decay of order $s^{-1}$. However, in the resonant case $\{\pm, \mp\}$, this is a priori unclear. A crucial fact is that in the resonant case, this term must be in a derivative form (i.e. gaining a frequency factor), which can be used to gain extra decay from phase mixing, see Section \ref{sec-resHosc}. 

\end{itemize}

Formally speaking, these inductive bounds on source densities $S(t,x)$ rely on two arguments: for $s > t/2$ we use Klein-Gordon's dispersion of the electric field, while for $s < t/2$, we use free transport dispersion in $v$, in addition to an appropriate decay from the characteristic \eqref{Xdep-intro}. In the proof, we make use of the fact that oscillations do not lose decay after an integration in time, while phase mixing gains extra decay for derivatives in space and time. 

\subsection*{Particles in an oscillatory field.}

In treating plasma echoes, we rely on the precise description of positions and velocities of the particles, see \eqref{Xdep-intro}, which plays a crucial roles in describing the oscillatory structure in the nonlinear interaction, see Section \ref{prop-SR}. Note that the dispersive decay $E^{osc}_\pm(t) = \mathcal{O}(t^{-3/2})$ is far from being sufficient to locate positions via the characteristic \eqref{odeXV}. One of the key observations in this work to overcome such a lack of decay is that time integration of the oscillatory field along the particle trajectories is better than expected. Indeed, suppose that the electric field is of the form $E_\pm^{osc}(t,x) = e^{ik\cdot x + \lambda_\pm(k) t}$, with $\lambda_\pm(k) = \pm i\nu_*(k)$ (namely, a Klein-Gordon wave at frequency $k$). Setting $\omega^v_\pm(k) = \lambda_\pm(k) + ik\cdot \hv$,   
we compute 
\begin{equation}\label{keyint}
 \begin{aligned}
 \int_0^t E_\pm^{osc}(\tau, x+\hv \tau)\; d\tau 
 &=  \int_0^t e^{ik\cdot x + \omega^v_\pm(k)\tau}\; d\tau 
= \frac{1}{\omega^v_\pm(k)} (e^{ik\cdot x + \omega^v_\pm(k)t} - e^{ik\cdot x})
\\& = \frac{1}{\omega^v_\pm(k)} \Big(E^{osc}_\pm(t,x+\hv t) - E^{osc}_\pm(0,x)\Big).
 \end{aligned}
\end{equation}
Effectively, up to a shift of $x \mapsto x-\hv t$, this calculation shows that particle velocities are a superposition of a purely oscillatory component and a pure transport part. This very decomposition turns out to hold for the genuine nonlinear particle trajectories \eqref{odeXV}, where the electric field is nonlinear and of the form \eqref{main-repE}, see Proposition \ref{prop-charV}. A similar decomposition is also carried out for particle positions, namely \eqref{Xdep-intro}. See Proposition \ref{prop-charPsi}. The precise description of particle positions and velocities in the physical space is one of the novelties of this work.  

\subsection*{Decay of the electric field}

The next task is to derive the right expected decay of the oscillatory electric field, which is computed through the bootstrap ansatz \eqref{true-Eosc}. One quickly faces a serious obstruction due to the lack of decay of the source density $S(t,x)$: namely, we wish to obtain a $t^{-3/2}$ dispersive decay from the spacetime convolution 
\begin{equation}\label{convEGs}G_\pm^{osc} \star_{t,x} \nabla_x S(t,x) = \int_0^t G_\pm^{osc}(t-s) \star_{x} \nabla_x S(s,x) \; ds.\end{equation}
The source density $S(t,x)$ is of order $t^{-3}$ in $L^\infty$ and merely bounded in $L^1_x$. In addition, phase mixing is no longer available, and spatial derivatives of $S(t,x)$ do not gain any extra decay due to the presence of oscillations, which is one of the main issues in dealing with Klein-Gordon's type dispersion. The decay of $S(t,x)$ is therefore insufficient to derive that of $E^{osc}_\pm(t,x)$ through the above spacetime convolution \eqref{convEGs}. 

To overcome the lack of decay and to propagate the Klein-Gordon dispersion nonlinearly, we again perform the time integration at the level of nonlinear interaction, remarkably leading to 
\begin{equation}\label{keyint2}
\begin{aligned}
G^{osc}_\pm \star_{t,x} \nabla_xS(t,x)  &= \frac{1}{\omega^v_\pm(i\partial_x)}G^{osc}_\pm(t) \star_{x} \nabla_x S(0,x) - \frac{1}{\omega^v_\pm(i\partial_x)}G^{osc}_\pm(0) \star_{x} \nabla_x S(t,x)
\\&\quad + \frac{1}{\omega^v_\pm(i\partial_x)} G^{osc}_\pm \star_{t,x} \nabla_x [ES](t,x) .
\end{aligned}
\end{equation}
See Propositions \ref{prop-convGosc}-\ref{prop-convGoscS0} for the precise details. The presentation \eqref{keyint2} reveals the deep structure hidden in the nonlinear interaction. Namely, the nonlinear electric field is again a superposition of a pure oscillation (i.e. Klein-Gordon dispersion) and a pure transport (i.e. phase mixing), plus a higher-order nonlinearity. This decoupling of oscillations from phase mixing is the key to the nonlinear iteration. In addition, the structure of each particle-particle, particle-wave, and wave-wave interaction also plays an important role in deriving decay of the electric field, see Section \ref{sec-quadinteraction}.  
  
\subsection*{Loss of derivatives}  
Finally, it is well-known that the spacetime convolution \eqref{convEGs} experiences loss of derivatives (e.g., \cite{RS, HKNR4}), see Proposition \ref{prop-Greenphysical}. However, the number of derivative losses is finite, we shall only derive sharp decay for low norms, but allow growth in time for higher derivatives. Precisely, we will propagate phase mixing and oscillations up to the penultimate derivatives, remarkably allowing a growth in time not only at the top order, but also all but one derivatives of the characteristic (see Proposition \ref{prop-HDchar}). The nonlinear iterative scheme with a cascading decay in norms is devised in Section \ref{sec-bootstrap}. The cascade of decay does not come from the simple interpolation between low and high norms, but from the phase mixing estimates for the transport, see \eqref{bootstrap-decaydaS} and Proposition \ref{prop-bdS}.

%

\subsection{Outline of the paper}

The paper is outlined as follows. 

\begin{itemize}

\item Section \ref{sec-lineartheory} is devoted to the linear theory where a detailed spectral analysis, following \cite{GNR1, HKNR3, HKNR4,Toan}, is done to establish that the linearized electric field can be decomposed into a Klein-Gordon dispersive component of order $t^{-3/2}$, plus a regular part which decays at order $t^{-3}$. This regular part can in fact be improved to decay at $t^{-4}$ (upon refining the initial data), however $t^{-3}$ is the best we could expect for the nonlinear iteration due to the presence of quadratic oscillations. 

\item An integral formulation of the whole system via using characteristics, following the Lagrangian framework introduced in \cite{HKNR2,ToanVKG}: this gives an integrate
reformulation of Vlasov equations, only involving characteristics, leading to linearized Vlasov
equations with a quadratic source term. Using the precise description of the characteristics, we have a detailed
picture of this quadratic source term. This is presented in Section \ref{sec-nonlinearframework}. We introduce the nonlinear iterative scheme in Section \ref{sec-iterscheme}. 

\item A detailed description of the characteristics; both position and velocity of particles oscillate with
an amplitude $O(t^{-3/2})$. We describe these oscillations, together with the effect of the remaining electric field,
in order to get an accurate picture of the position and speed of the particles. This is done in Section \ref{sec-Char}.

\item Using the detailed structure of the source terms, and the accurate description of the electric field for linearized Vlasov system, we inductively prove bounds on the source term and the nonlinear electric field, involving the detailed analysis of particle-particle, particle-wave, and wave-wave interaction. This is done in Section \ref{sec-sourceest}.

\item Finally, Section \ref{sec-decayosc} is devoted to derive decay of the oscillatory electric field through the spacetime convolution of the oscillatory Green function against the nonlinear sources (see Section \ref{sec-nonlinearframework}). Note that the sources decay at best like the nonlinear density, which is at order $t^{-3}$ and therefore insufficient to close the iteration for the oscillatory field (e.g., comparing with solving quadratic Klein-Gordon equations in three dimension). We introduce a phase space resonant method to resolve this very issue. 

\end{itemize}


\section{Linear theory}\label{sec-lineartheory}



\subsection{Introduction}


In this section, we introduce the Fourier-Laplace approach to study the linearized Vlasov-Klein-Gordon problem around radial equilibria $\varphi(\langle v\rangle)$: namely,  
\begin{equation}
\label{V-lin}
\partial_t f + \hv\cdot \nabla_x f + E\cdot \nabla_v\mu= 0 
\end{equation} 
\begin{equation}\label{KG-lin}
E = -\nabla_x \phi, \qquad (\Box_{t,x} + m_0^2)\phi = -\rho,
\end{equation}
with initial data 
\begin{equation}\label{data-lin}
f(0,x,v) = f_0(x,v), \qquad \phi(0,x) = \phi_0(x), \qquad \partial_t \phi(0,x) = \phi_1(x).
\end{equation} 
The spectral analysis is classical, following Landau \cite{Landau}, who analyzed the mode stability for the linearized Vlasov-Poisson system. To derive decay estimates for the linearized solution, we follow closely the analysis developed in our recent works \cite{GNR1, HKNR3, HKNR4,Toan} for linearized Vlasov-Poisson and Vlasov-Maxwell systems. For the sake of completeness, we provide the linear analysis for the linearized Klein-Gordon system \eqref{V-lin}-\eqref{KG-lin}. 
Throughout the section, we denote by $\Ff_k(t,v), \Fphi_k(t)$ the Fourier transform in $x$ of $f(t,x,v), \phi(t,x)$, respectively, while $\Tf_k(\lambda), \Tphi_k(\lambda)$ denote their Fourier-Laplace transform in $t,x$.   

\subsection{Resolvent equation}

We first derive the resolvent equation for \eqref{V-lin}. Precisely, we obtain the following. 

\begin{lemma}\label{lem-resolvent} Let $\phi(t,x)$ be the electric potential of the linearized Vlasov-Klein-Gordon system \eqref{V-lin}, and $\Tphi_k(\lambda)$ be the Fourier-Laplace transform of $\phi(t,x)$. Then, for each $k\in \RR^3$ and $\lambda\in \CC$, there hold
\begin{equation}\label{resolvent}
\begin{aligned}
M(\lambda,k)\Tphi_k(\lambda) &= - \Trho^0_k(\lambda) + \lambda \Fphi_k^0+ \Fphi_k^1, 
\end{aligned}
\end{equation}
where the spacetime symbol $M(\lambda,k)$ is defined by 
\begin{equation}\label{def-Dlambda}
\begin{aligned}
M(\lambda,k) &:=  \lambda^2 + |k|^2 + m_0^2 + \int \frac{ ik\cdot \hv}{\lambda +  ik \cdot \hv} \varphi'(\langle v\rangle)\;dv,
\end{aligned}
\end{equation}
with $\Fphi_k^j$ denoting the Fourier transform of the initial data $\phi^j(x)$, $j=0,1$,  and 
$\Trho^0_k(\lambda):= \int \frac{\Ff_k^0(v)}{\lambda +  ik \cdot \hv} dv$ being the charge density generated by the free transport. 

\end{lemma}

\begin{proof} Indeed, noting $\nabla_v \mu = \hv \varphi'(\langle v\rangle)$ and taking the Laplace-Fourier transform of \eqref{V-lin} with respect to variables $(t,x)$, respectively, we obtain 
\begin{equation}\label{VKG-lambda1}
\begin{aligned}
(\lambda +  ik \cdot \hv ) \Tf_k &= \Ff_k^0 +ik \cdot \hv \varphi'(\langle v\rangle)\Tphi_k  
\end{aligned}\end{equation}
which gives 
$$  \Tf_k  = \frac{  ik \cdot \hv}{\lambda +  ik \cdot \hv} \varphi'(\langle v\rangle)\Tphi_k+ \frac{\Ff_k^0}{\lambda +  ik \cdot \hv} .$$
Integrating in $v$, we get $$
\begin{aligned} 
\Trho_k [f] &= \Big( \int \frac{ ik\cdot \hv}{\lambda +  ik \cdot \hv} \varphi'(\langle v\rangle) \;dv\Big)  \Tphi_k  + \int \frac{\Ff_k^0}{\lambda +  ik \cdot \hv} dv ,
\end{aligned}$$
On the other hand, the Klein-Gordon equation in the Fourier-Laplace variables reads
$$(\lambda^2 + |k|^2+m_0^2)\Tphi_k =  - \Trho_k[f] + \lambda \Fphi_k^0+ \Fphi_k^1$$
in which $\Fphi_k^j$ denotes the Fourier transform of the initial data $\phi^j(x)$, $j=0,1$. 
This gives \eqref{resolvent}, 
  \end{proof}

\begin{lemma}\label{lem-Dlambda} Let $M(\lambda,k)$ be defined as in \eqref{def-Dlambda}. Then, for each $k\in \RR^3$, we can write 
\begin{equation}\label{def-Dlambda1}
\begin{aligned}
M(\lambda,k) 
= \lambda^2 + |k|^2 + m_0^2 + \cH(i\lambda/|k|), \qquad \cH(z) 
= \int_{-1}^1 \frac{u \kappa(u)}{z-u} \;du.  
\end{aligned}
\end{equation}
for some even and non-negative function $\kappa(u)$. In particular, $M(\lambda,k)$ is analytic in $\Re\lambda>0$, and 
\begin{equation}\label{unibd-cH} |\partial_z^n\cH(i\lambda/|k|)| \lesssim 1,
\end{equation}
uniformly for $k\in \RR^3$, $\Re \lambda \ge 0$, and $0\le n<N_0$.  
\end{lemma}

\begin{proof} In view of definition \eqref{def-Dlambda}, we study the integral term $\int \frac{ ik\cdot \hv}{\lambda +  ik \cdot \hv} \varphi'(\langle v\rangle)\;dv$. For $k\not =0$, we introduce the change of variables $v \mapsto (u,w)$ defined by 
\begin{equation}\label{change-vw}
u := \frac{k\cdot v}{|k|\langle v\rangle} , \qquad w := v -   \frac{(k\cdot v)k}{|k|^2},
\end{equation}
with $u \in [-1,1]$ and $w\in k^\perp$, the hyperplane orthogonal to $k$ in $\RR^3$. Note that the Jacobian determinant is  
$$J_{u,w} = \langle v\rangle (1-u^2)^{-1}, $$
with  $\langle v\rangle  =\langle w\rangle / \sqrt{1-u^2}$. Therefore, we may write 
\begin{equation}\label{rew-Dlambda}
\begin{aligned}
\int \frac{ ik\cdot \hv}{\lambda +  ik \cdot \hv} \varphi'( \langle v\rangle)\;dv
&= \int_{-1}^1 \frac{ u }{-i\lambda/|k| +  u} \Big(\int_{w\in k^\perp}\varphi'\Big(\frac{\langle w\rangle}{\sqrt{1-u^2}}\Big) \frac{\langle w\rangle}{(1-u^2)^{3/2}} \; dw\Big) \;du.
\end{aligned}
\end{equation}
Setting 
\begin{equation}\label{def-kuw}
\begin{aligned}
\kappa(u) :&= -\int_{w\in k^\perp}\varphi'\Big(\frac{\langle w\rangle}{\sqrt{1-u^2}}\Big) \frac{\langle w\rangle}{(1-u^2)^{3/2}} \; dw ,
\end{aligned}
\end{equation}
we obtain \eqref{def-Dlambda1}, upon recalling \eqref{def-Dlambda}. Clearly, $\kappa(u)$ is even in $u$. It remains to check that $\kappa(u)$ is nonnegative. Indeed, we may parametrize the hyperplane $k^\perp $ via polar coordinates with radius $r = |w|$, and then introduce $s = \sqrt{1+r^2} / \sqrt{1-u^2}$, giving
$$
\begin{aligned}
\kappa(u) &= - \int_{\mathbb{R}^2}\varphi'\Big(\frac{\langle w\rangle}{\sqrt{1-u^2}}\Big) \frac{\langle w\rangle}{(1-u^2)^{3/2}} \; dw= - 2\pi \int_0^\infty \varphi'\Big(\frac{\sqrt{1+r^2}}{\sqrt{1-u^2}}\Big)\frac{\sqrt{1+r^2}}{(1-u^2)^{3/2}} \; r dr 
\\&= - 2\pi\int_{1/\sqrt{1-u^2}}^\infty \varphi'(s) s^2\; ds
\\
&= \frac{2\pi }{1-u^2}\mu\Big(\frac{1}{\sqrt{1-u^2}}\Big)  + 4\pi\int_{1/\sqrt{1-u^2}}^\infty \varphi(s) s\; ds
\end{aligned}
$$
which gives the non-negativity of $\kappa(u)$ for $u \in (-1,1)$, recalling $\varphi(\cdot)\ge0$. Since $\varphi(s)$ decays rapidly, $\kappa(u)$ tends to zero rapidly as $|u|\to 1$. The analyticity of $\cH(z)$ follows from the regularity and decay properties of $\kappa(u)$. 
\end{proof}

\subsection{Spectral stability}\label{sec-spectral}

In view of the resolvent equation \eqref{resolvent}, the solutions to the dispersion relation $M(\lambda,k)=0$ plays a crucial role in studying the large time dynamics of the linearized problem \eqref{V-lin}. In this section, we shall prove that there is no solution in the right half plane $\Re \lambda>0$. Namely, we obtain the following. 

\begin{proposition}\label{prop-nogrowth} Let $\mu(v) = \varphi(\langle v\rangle)$ be any non-negative radial equilibria in $\RR^3$, and define $\kappa_0$ as
\begin{equation}\label{def-kappa0}
\kappa_0^2 = - \int_{\RR^3}\varphi'(\langle v\rangle)\; dv.
\end{equation}
Assume that $\kappa_0^2\le m_0^2$. Then, the linearized system \eqref{V-lin} has no nontrivial mode solution of the form $e^{\lambda t + ik\cdot x} (\Fphi_k, \Ff_k(v))$ with $\Re \lambda >0$ for any nonzero pair $(\Fphi_k, \Ff_k(v))$. 
\end{proposition}

\begin{proof} It suffices to prove that for any $k\in \RR^3$, the spacetime symbol $M(\lambda,k)$ never vanishes for $\Re \lambda >0$. In view of \eqref{def-kuw} and \eqref{def-kappa0}, we note that $$\kappa_0^2 = \int_{-1}^1 \kappa(u)\;du,$$
which is well-defined, since $\kappa(u)\ge 0$.  
Recalling \eqref{def-Dlambda1}, we compute 
$$
\begin{aligned}
M(\lambda,k) 
&= \lambda^2+ |k|^2 +m_0^2 - \int_{-1}^1 \frac{u \kappa(u)}{u-i\lambda/|k|} \;du 
\\
&= \lambda^2+ |k|^2 +m_0^2 - \kappa_0^2 - \frac{i\lambda}{|k|}\int_{-1}^1 \frac{\kappa(u)}{u-i\lambda/|k|} \;du 
\\
&= \lambda^2+ |k|^2 +m_0^2 - \kappa_0^2 - \frac{i\lambda}{|k|}\int_{-1}^1 \frac{(u+i\bar\lambda/|k|)\kappa(u)}{|u-i\lambda/|k||^2} \;du
\\
&= \lambda^2+ |k|^2 +m_0^2 - \kappa_0^2 + \frac{|\lambda|^2}{|k|^2}\int_{-1}^1 \frac{\kappa(u)}{|u-i\lambda/|k||^2} \;du - \frac{i\lambda}{|k|}\int_{-1}^1 \frac{u\kappa(u)}{|u-i\lambda/|k||^2} \;du.
\end{aligned}$$
Now suppose that $M(\lambda,k) =0$ for some $\Re\lambda \neq 0$. The vanishing of the imaginary part gives 
$$
 \frac{1}{|k|}\int_{-1}^1 \frac{u\kappa(u)}{|u-i\lambda/|k||^2} \;du = 2\Im \lambda .
$$ 
Plugging this identity back into $M(\lambda,k)$ and noting $\lambda^2- 2i\lambda\Im \lambda = |\lambda|^2$, we get 
$$
\begin{aligned}
M(\lambda,k) 
&= |\lambda|^2+ |k|^2 +m_0^2 - \kappa_0^2 + \frac{|\lambda|^2}{|k|^2}\int_{-1}^1 \frac{\kappa(u)}{|u-i\lambda/|k||^2} \;du 
\end{aligned}$$
which never vanishes, since $\kappa(u)\ge 0$ and $m_0^2 \ge \kappa_0^2$. That is, $M(\lambda,k)$ never vanishes for $\Re \lambda \neq 0$, and therefore, the linearized system \eqref{V-lin} has no nontrivial mode solution as claimed.  
\end{proof}

\subsection{Langmuir oscillatory waves}
\label{sec-M}

Thanks to the stability result of Proposition~\ref{prop-nogrowth}, we are led to study the dispersion relation $M(\lambda,k)=0$ for $\lambda$ lying on the imaginary axis. The main result in this section is the existence of Langmuir's purely oscillatory waves. Namely, we obtain the following result.

\begin{theorem}\label{theo-LangmuirB} 
For each $k\in \RR^3$, there are exactly two zeros $\lambda_\pm = \pm i \nu_*(|k|)$ of the  dispersion relation 
$$ M(\lambda,k) = 0$$
that lie on the imaginary axis $\{\Re \lambda =0\}$. Moreover, we have $\nu_*(|k|)>|k|$, and $\nu_*$ is a smooth function satisfying the following Klein-Gordon type estimates:  
\begin{equation}\label{KG-behave1} 
c_0 \sqrt{1+|k|^2} \le \nu_*(|k|)\le C_0 \sqrt{1+|k|^2} ,
\end{equation}
\begin{equation}\label{KG-behave2}
c_0\frac{|k|}{ \sqrt{1+|k|^2}} \le \nu_*'(|k|) \le C_0 \frac{|k|}{ \sqrt{1+|k|^2}} , 
\end{equation}
\begin{equation}\label{KG-behave3}
 c_0 (1+|k|^2)^{-3/2}  \le \nu_*''(|k|)  \le C_0.
\end{equation}
uniformly with respect to $k\in \RR^3$, for some positive constants $c_0, C_0$. In addition, 
\begin{equation}\label{lowbd-Mstau} 
\begin{aligned}
|M(i\tilde\tau|k|,k) | 
\gtrsim \tau_0^2 + |\tilde\tau| +  |k|^2(1-\tilde\tau^2), \qquad \forall ~|\tilde\tau|\le 1,
\end{aligned}\end{equation}
uniformly in $k\in \RR^3$. 
\end{theorem}

\begin{proof} We follow the analysis done in \cite{HKNR4, Toan}. Indeed, in view of Proposition~\ref{prop-nogrowth}, we can restrict to the imaginary axis, that is we consider $\lambda = i\tau$ for $\tau \in \RR$. We distinguish between the two cases: $|\tau| < |k|$ and $|\tau| \ge |k|$.

\bigskip

\noindent {\bf Case 1: $|\tau|\ge |k|$.}
We first consider the case when $|\tau| \ge |k|$. 
Using the formulation \eqref{def-Dlambda1}, we may write 
$$
\begin{aligned}
M(i\tau,k) 
&= -\tau^2 + |k|^2 + m_0^2+  \int_{-1}^1 \frac{u \kappa(u)}{-\tau/|k| -u} \;du
\\
&= -\tau^2 + |k|^2 + m_0^2+  \int_{-1}^1 \frac{u^2 \kappa(u)}{\tau^2/|k|^2-u^2} \;du
\\
&= -\tau^2 + |k|^2 + m_0^2+ \frac{|k|^2}{\tau^2} \int_{-1}^1 \frac{u^2 \kappa(u)}{1-u^2|k|^2/\tau^2} \;du,
\end{aligned}  
$$
in which we have used the fact that $\kappa(u)$ is even in $u$. Namely, 
for each fixed $k\in \RR^3$, the dispersion relation is of the form 
\begin{equation}\label{rew-D}
M(i\tau,k) = \Psi(\tau^2),
\end{equation}
where, for convenience, we have set
\begin{equation}\label{def-pPsi} \Psi(x) := - x + |k|^2 + \psi(|k|^2/x),\qquad \psi(y): = m_0^2 + y\int_{-1}^1 \frac{u^2 \kappa(u)}{1-u^2y} \;du.
\end{equation}
Recall that $\kappa(u)$ decays rapidly to zero as $|u|\to 1$, and so $\psi(y)$ is well defined for $y\in [0,1]$. As a result, 
$ \Psi(x)$ is well-defined for $x\ge |k|^2$. In addition, for $y\in (0,1)$, we compute 
$$ \psi'(y) = \int_{-1}^1 \frac{u^2 \kappa(u)}{1-u^2y} \;du+ y\int_{-1}^1 \frac{u^4 \kappa(u)}{(1-u^2y)^2} \;du \ge 0.$$
In fact, since $\kappa(u)$ is nonnegative and not identically zero, $\psi'(y)>0$ and $\psi(y)$ is strictly increasing in $[0,1]$ with 
$$ \psi(0) = m_0^2, \qquad \psi(1) =m_0^2 + \int_{-1}^1 \frac{u^2 \kappa(u)}{1-u^2} \;du = m_1^2.$$ 
In addition, we compute 
\begin{equation}
\label{laderiveedePsi} \Psi'(x) = -1  - \frac{|k|^2}{x^2}\psi'(\frac{|k|^2}{x}) \le -1.
\end{equation}
That is, $\Psi(x)$ is strictly decreasing in $x \ge |k|^2$. Since $m_0^2 \le \psi(|k|^2/x)\le m_1^2$, the function $\Psi(x)$ is strictly negative for $x\gg1$, while at $x = |k|^2$, $\Psi(|k|^2) = m_1^2>0$. By the strict monotonicity of $\Psi(x)$, there is a unique solution $x_* > |k|^2$ of $\Psi(x) =0$, or equivalently,
\begin{equation}\label{eqs-xstar} x_* = |k|^2 + \psi(|k|^2/x_*) .\end{equation}
Note that $x_* = x_*(|k|)$ is a radial function in $k \in \RR^3$. 
In view of \eqref{rew-D}, this yields two solutions of $D(\lambda_\pm,k)=0$, with $\lambda_\pm = \pm i \sqrt{x_*(|k|)}$, in the region where $|\tau|> |k|$.  The smoothness of $x_*(|k|)$ follows from that of $\psi(\cdot)$. The stated dispersive properties of $\nu_*(|k|) =  \sqrt{x_*(|k|)}$ now follow from the similar lines as done in \cite[Theorem 2.18]{HKNR4} satisfying the same dispersion relation as in \eqref{eqs-xstar}. We therefore skip repeating the details.

\bigskip

\noindent {\bf Case 2: $|\tau| < |k|$.}
In this case, we shall take $\lambda = (\tilde\gamma + i\tilde\tau)|k|$ with $|\tilde \tau|<1$, and study the limit of $\tilde\gamma \to 0^+$. Using the formulation \eqref{def-Dlambda1} and recalling $\kappa_0^2 = \int_{-1}^1 \kappa(u)\;du$, we compute 
\begin{equation}\label{new-Msmall}
\begin{aligned}
M(\lambda,k) 
&=   |k|^2(1 + \tilde\gamma^2-\tilde\tau^2+2i \tilde \gamma \tilde \tau) + m_0^2  +\int_{-1}^1 \frac{u \kappa(u)}{i\tilde\gamma-\tilde\tau-u} \;du
\\
&=   |k|^2(1 + \tilde\gamma^2-\tilde\tau^2+2i \tilde \gamma \tilde \tau) + m_0^2  - \kappa_0^2 +(i\tilde\gamma-\tilde\tau)\int_{-1}^1 \frac{ \kappa(u)}{i\tilde\gamma-\tilde\tau-u} \;du.
\end{aligned}
\end{equation}
Therefore, by Plemelj's formula,  for $\tilde \gamma \to 0^+$ and $|\tilde\tau|<1$, we have 
\begin{equation}\label{comp-Mstau}
M(i\tilde\tau|k|,k) =  |k|^2(1-\tilde\tau^2) + m_0^2 - \kappa_0^2 + \tilde\tau  P.V. \int_{-1}^1 \frac{1}{\tilde\tau + u} \kappa(u) du  + i\pi  \tilde\tau \kappa(\tilde\tau) 
.\end{equation}
Since $\kappa(u) \not =0$ for $u\in (-1,1)$, the imaginary part of $M(i\tilde\tau|k|,k)$ never vanishes for $0<|\tilde\tau|<1$, while at $\tilde\tau=0$ the real part is equal to $|k|^2+ m_0^2 - \kappa_0^2$. On the other hand, at $\tilde\tau = \pm 1$, we have $ \kappa(\pm 1)=0$, and 
$$
D( \pm i|k|,k) = m_0^2 - \kappa_0^2+ \int_{-1}^1 \frac{\kappa(u)}{1 \pm u} \; du = m_0^2 - \kappa_0^2+ \int_{-1}^1 \frac{\kappa(u)}{1- u^2}  \; du , $$
recalling that $\kappa(u)$ is even in $u$. Recall that $m_0^2 \ge \kappa_0$ and $\kappa(u) \ge 0$. Therefore, we have 
\begin{equation}\label{lower-Mstau0} 
|M(i\tilde\tau|k|,k) | \ge \theta_0 |\tilde\tau|, \qquad \forall ~|\tilde\tau|\le 1,
\end{equation}
for some positive $\theta_0$, which in particular proves that there are no zeros of $M(i\tau,k) \not =0$ in this region when $|\tau| \le |k|$.
In fact, using the fact that the P.V. integral $ \int_{-1}^1 \frac{1}{\tilde\tau + u} \kappa(u) du$ is finite, we obtain from \eqref{comp-Mstau} that 
$$ 
|M(i\tilde\tau|k|,k) | \ge  |k|^2(1-\tilde\tau^2)  +  m_0^2 - \kappa_0^2 - C_0 |\tilde \tau|,
$$
for some constant $C_0$. Combining this with \eqref{lower-Mstau0}, we get  
$$
\begin{aligned}
|M(i\tilde\tau|k|,k) | &= (1- \delta) |M(i\tilde\tau|k|,k) | + \delta |M(i\tilde\tau|k|,k) | 
\\& 
\ge (1- \delta) \theta_0 |\tilde\tau| +  \delta |k|^2(1-\tilde\tau^2) + \delta( m_0^2 - \kappa_0^2)- C_0 \delta |\tilde \tau|.
\end{aligned}
$$
This yields \eqref{lowbd-Mstau} by taking $\delta$ sufficiently small and recalling $\tau_0^2 = m_0^2 - \kappa_0^2 $. We thus complete the proof of Theorem \ref{theo-LangmuirB}.
\end{proof}


\subsection{Green function}


In view of the resolvent equations \eqref{resolvent}, we introduce the Green function 
\begin{equation}\label{def-FGk}
\begin{aligned}
\FG_{k}(t) &=  \frac{1}{2\pi i}\int_{\{\Re \lambda = \gamma_0\}}e^{\lambda t} \frac{1}{M(\lambda,k)} \; d\lambda ,
\end{aligned}\end{equation}
which are well-defined as oscillatory integrals  for $\gamma_0 >0$, recalling from Lemma \ref{lem-Dlambda} that $M(\lambda,k)$ is holomorphic in $\Re \lambda>0$. The main goal of this section is to establish decay estimates for the Green function through the representation \eqref{def-FGk}. The main analysis follows from that in \cite{HKNR3, HKNR4, Toan} to compute the residual of $\TG_k(\lambda)$, plus a remainder that decay rapidly in $|kt|$. 

\begin{proposition}\label{prop-Green}  Let $\FG_k(t)$ be defined as in \eqref{def-FGk}, and let $\lambda_\pm = \pm i \nu_*(|k|)$
be the dispersion relation constructed in Theorem \ref{theo-LangmuirB}. Then, we can write 
\begin{equation}\label{decomp-FG} 
\FG_k(t) = \sum_\pm \FG^{osc}_{k,\pm}(t)  +   \FG^{r}_k(t) ,
\end{equation}
where 
\begin{equation}\label{def-Gosc}
\FG^{osc}_{k,\pm}(t) = e^{\pm i\nu_*(|k|) t} a_\pm(k),
\end{equation}
for some sufficiently smooth functions $a_\pm(k)$ that satisfy $|a_\pm(k)|\lesssim \langle k\rangle^{-1}$. In addition, the regular part satisfies
\begin{equation}\label{bound-Hrk}
| |k|^\alpha \partial_k^\alpha \FG^r_k(t)|  \lesssim |k| \langle k\rangle^{-2} \langle kt\rangle^{|\alpha|-N},
\end{equation}
uniformly for $k\in \RR^3$, $t \geq 0$, and $|\alpha|\ge 0$. 
\end{proposition}

\begin{proof} The proposition is a minor modification to that of Proposition 3.2 obtained in \cite{HKNR4}. Indeed, recalling \eqref{def-Dlambda}, we write  
$$
\begin{aligned}
M(\lambda,k) 
&=  \lambda^2 + |k|^2 + m_0^2 + \int \frac{ ik\cdot \hv}{\lambda +  ik \cdot \hv} \varphi'(\langle v\rangle)\;dv
\\
&=  \lambda^2 + |k|^2 + m_0^2 + \int_0^\infty e^{-\lambda t} \int e^{-ik\cdot \hv} ik\cdot \hv\varphi'(\langle v\rangle)\;dv dt.
\end{aligned}
$$
Namely, we can write 
\begin{equation}\label{rew-DLKt}
\begin{aligned}
M(\lambda,k) &=  \lambda^2 + |k|^2 + m_0^2 + \cL[N_k(t)](\lambda)
\end{aligned}
\end{equation}
where $\mathcal{L}[N_k(t)](\lambda)$ denotes the Laplace transform of $N_k(t)$, that is  $\mathcal{L}[N_k(t)](\lambda) =  \int_0^\infty e^{-\lambda t} N_k(t) dt$, upon setting 
\begin{equation}\label{def-N}
N_k(t) =  \int e^{-ikt \cdot \hat v } ik  \cdot \hv \varphi'(\langle v\rangle) \; dv.
\end{equation}
The dispersion relation $M(\lambda,k)$ is identical to that of the magnetic dispersion relation computed in \cite{HKNR4}, with the two zeros $\lambda_\pm = \pm i \nu_*(k)$ as constructed in Theorem \ref{theo-LangmuirB}, giving the results stated in the proposition (cf. \cite[Proposition 3.2]{HKNR4}). We note however that in the present setting, we obtain a better lower bound on $M(\lambda,k)$ in the small frequency regime than that in \cite{HKNR4}. Precisely, using the fact that $|N_k(t)|\lesssim |k|\langle kt\rangle^{-N}$, we first bound
$$
\begin{aligned}
|\partial_\lambda M(\lambda,k)| 
&\le 2| \lambda| + |\partial_\lambda \cL[N_k(t)](\lambda)|
\\
& \lesssim | \lambda| + \int_0^\infty t |k| \langle kt\rangle^{-N} \; dt
\lesssim | \lambda| + |k|^{-1},
\end{aligned}$$
uniformly for $\Re \lambda \ge0$. Therefore, together with the uniform lower bound from \eqref{lowbd-Mstau} for $\lambda = i\tilde \tau |k|$, we obtain 
$$
\begin{aligned}
\frac{|k| |\partial_\lambda M(\lambda,k)|}{|M(\lambda,k)|} 
&\lesssim \frac{|\tilde \tau||k|^2 + 1}{\tau_0^2 + |\tilde\tau| +  |k|^2(1-\tilde\tau^2)}, 
\end{aligned}$$
which is uniformly bounded for $|\tilde \tau|\le 1$ and $|k|\lesssim 1$, upon using that $\tau_0\not =0$. As a result, upon taking integration by parts in $\lambda$ repeatedly of the integral \eqref{def-FGk}, we obtain \eqref{bound-Hrk}. We refer the interested readers to the proof of \cite[Proposition 3.2]{HKNR4} for the remaining details.  
\end{proof}

\subsection{Green function in the physical space}

In this section, we provide bounds on the Green's  function $G(t,x)$ in the physical space defined by  
\begin{equation} \label{Greal}
\begin{aligned}
G(t,x) &= \int e^{i k \cdot x} \FG_k(t) \; dk 
\end{aligned}\end{equation}
where $\FG_k(t)$ is constructed as in Proposition \ref{prop-Green}. We shall prove the following proposition. 

\begin{proposition}\label{prop-Greenphysical} Let $G(t,x)$ be the Green function defined as in \eqref{Greal}. Then, we can write 
\begin{equation}\label{decomp-Greal} 
G(t,x) = \sum_\pm G^{osc}_{\pm}(t,x)  +   G^{r}(t,x) ,
\end{equation}
in which the oscillatory kernel satisfies 
\begin{equation}\label{est-HoscL2inf}
\begin{aligned}
\|G^{osc}_\pm\star_x \nabla_x f\|_{L^2_x} &\lesssim \| f\|_{L^{2}_x}, \qquad &\|G^{osc}_\pm\star_x \nabla_x f\|_{L^\infty_x}\lesssim \| f\|_{H^2_x},
\end{aligned}
\end{equation}
and for $p \in [2,\infty]$, 
\begin{equation}
\label{est-HoscLp}
\begin{aligned}
\|G^{osc}_\pm\star_x \nabla_x f\|_{L^p_x} &\lesssim \langle t\rangle^{-3\left(\frac 12-\frac1p\right)} \| f \|_{W_x^{1+\lfloor 3\left(1-\frac{2}{p}\right) \rfloor, p'}},
\end{aligned}
\end{equation}
with $\frac 1p+\frac 1{p'}=1$ and $\lfloor s \rfloor$ denoting the largest integer that is smaller than $s$. In addition, letting $\chi(k)$ be a smooth cutoff function whose support is contained in $\{|k|\le 1\}$, for any $n\ge 0$ and $p\in [1,\infty]$, there hold 
\begin{equation}\label{est-GtrLp} 
\begin{aligned}
\|\chi(i\partial_x) \partial_x^nG^r(t) \|_{L^p_x} &\lesssim \langle t\rangle^{-4+3/p - n} ,
\end{aligned}\end{equation}
and 
\begin{equation}\label{est-GLp} \| (1-\chi(i\partial_x)) \partial_x^n G^r (t) \|_{L^p_x} \lesssim \langle t\rangle^{-N_1},\end{equation}
for some large constant $N_1$. Moreover, for any $\delta>0$, 
\begin{equation}\label{conv-GtrLp}
\| \chi(i\partial_x) G^r (t)\star_x f\|_{L^p_x} \lesssim \langle t\rangle^{-1-\delta } \sup_{q\in \ZZ} 2^{-\delta q}\| P_qf\|_{L^p_x}
\end{equation}
for any $p\in [1,\infty]$, where $P_q$ is the Littlewood-Paley projection on the dyadic interval $[2^{q-1}, 2^{q+1}]$.

\end{proposition}

\begin{proof} By construction \eqref{decomp-FG}, we have 
\begin{equation} \label{Gsreal}
\begin{aligned}
G^{osc}_\pm(t,x) &= \int e^{\pm i\nu_*(|k|) t + i k \cdot x} a_\pm(k) dk,
\end{aligned}\end{equation}
where $\nu_*(|k|)$ satisfies the Klein-Gordon type dispersion as established in Theorem \ref{theo-LangmuirB}. The $L^2$ estimate is immediate, recalling that $\FG^{osc}_\pm(t,k) = e^{\pm i \nu_*(|k|) t} a_\pm(k)$, where $a_\pm(k)$ is sufficiently smooth and satisfying $a_\pm(k) \lesssim \langle k\rangle^{-1}$, while the $L^\infty$ estimate follows from the bound 
$$\| g\|_{L^\infty_x} \le \| \Fg\|_{L^1_k} \lesssim\| \langle k\rangle^{2}\Fg\|_{L^2_k} .$$ 
As for the $L^p$ estimates, we recall the following standard dispersive estimates for the Klein-Gordon type dispersion, see \cite{RS, HKNR4},  
$$\|G^{osc}_\pm\star_x \nabla_x f\|_{L^p_x} \lesssim \langle t\rangle^{-3\left(\frac 12-\frac1p\right)}\| a_\pm(i\partial_x)\nabla_x f\|_{B^{3\left(1-\frac{2}{p}\right)}_{p',2}}$$
with $\frac 1p+\frac 1{p'}=1$, with $\|\cdot \|_{B^{s}_{p,q}}$ denoting the standard Besov spaces, see Section \ref{sec-FM}. In particular, using \eqref{eq:fouriermultbesov}, for any $s\ge 0$, we have
$$
\| a(i\partial_x) \nabla_x  f \|_{B^{s}_{p',2}} \lesssim \|  f\|_{W^{1+ \lfloor s \rfloor, p'}},
$$
which yields \eqref{est-HoscLp}. 

Next, we bound the regular part. Let $\chi_0(k)$ be some smooth cut-off Fourier multiplier whose support is contained in $\{|k|\le 1\}$. Using \eqref{bound-Hrk}, we bound 
$$
\begin{aligned}
| \chi_0(i\partial_x) \partial_x^n G^r(t,x)|
& \lesssim \int_{\{|k|\le 1\}} |k|^{n} |\FG_k^r(t)| \; dk
 \lesssim \int_{\{|k|\le 1\}} |k|^{n+1}\langle |k|t \rangle^{-N} \; dk 
\lesssim \langle t\rangle^{-n-4}.
\end{aligned}$$
proving \eqref{est-GtrLp} for $p=\infty$. 
On the other hand, for the high frequency part, we bound $$
\begin{aligned}
| (1-\chi_0(i\partial_x)) \partial_x^n G^r(t,x)|
& \lesssim \int_{\{|k|\ge1\}} |k|^{n} |\FG_k^r(t)| \; dk
 \lesssim \int_{\{|k|\ge 1\}} |k|^{n-1}\langle |k|t \rangle^{-N} \; dk 
\lesssim \langle t\rangle^{- N} ,
\end{aligned}$$
proving \eqref{est-GLp} for $p=\infty$. As for $L^1_x$ estimates, we first note that for the high frequency part, we bound 
$$
\begin{aligned}
\| (1-\chi_0(i\partial_x)) \partial_x^n G^r(t)\|_{L^1_x}^2
& \lesssim \sum_{|\alpha|\le 2}\int_{\{|k|\ge 1\}} |k|^{2n} |\partial_k^\alpha \FG_k^r(t)|^2 \; dk
\\&  \lesssim  \sum_{|\alpha|\le 2} \int_{\{|k|\ge 1\}} |k|^{2n-2} \langle t\rangle^{2|\alpha|}\langle |k|t \rangle^{-2N} \; dk 
\lesssim \langle t\rangle^{-2(N-1)} .
\end{aligned}$$
It remains to treat the low frequency part. Using the standard Littlewood-Paley decomposition, we write 
$$
\begin{aligned}
\chi_0(i\partial_x) G^r(t,x) &= \sum_{q\le 0} P_q[ G^r(t,x) ],
 \end{aligned}$$
where 
$$
\begin{aligned}
P_q[G^r(t,x)] 
&=  \int_{\{2^{q-2}\le |k|\le 2^{q+2}\}} e^{ik \cdot x} \FG^r(t,k) \varphi(k/2^q)\; dk
\\
&= 2^{3q} \int_{\{\frac14\le |\tilde k|\le 4\}} e^{i \tilde k \cdot 2^q x}\FG^r(t,2^q \tilde k) \varphi(\tilde k)\; d\tilde k
\end{aligned}$$
Using the estimates \eqref{bound-Hrk} from Proposition \ref{prop-Green} for $|k|\le1$ (and so $q\le 0$), we obtain 
$$| \partial_{\tilde k}^\alpha [\FG^r(t,2^q \tilde k)]| \le | 2^{q|\alpha|}\partial_{k}^\alpha [\FG^r(t,2^q \tilde k)]|  \lesssim 2^q \langle 2^{q} t \rangle^{-N +|\alpha|} 
$$
for $\frac14 \le |\tilde k|\le 4$. Therefore, integrating by parts repeatedly in $\tilde k$ and taking $|\alpha| =4$ in the above estimate, we have 
$$
\begin{aligned}
|P_q[G^r(t,x)]| 
&
\lesssim  2^{q} \Big| 
 \int e^{i \tilde k \cdot 2^q x}\FG^r(t,2^q \tilde k) \varphi(\tilde k)\; d\tilde k
 \Big|
\\&
\lesssim  2^{q} 
 \int \langle 2^q x\rangle^{-4} |\partial_{\tilde k}^4 [\FG^r(t,2^q \tilde k) \varphi(\tilde k)]| \; d\tilde k
\\&
\lesssim  2^{2q} \langle 2^q x\rangle^{-4} \langle 2^{q} t \rangle^{-N +4}.
\end{aligned}$$
Taking the $L^1_x$ norm, we obtain 
$$
\begin{aligned}
 \| \chi_0(i\partial_x) \partial_x^n G^r(t)\|_{L^1_x} 
 &\lesssim \sum_{q\le 0} 2^{nq}\|P_q[G^r(t)]\|_{L^1_x} 
\lesssim \sum_{q\le 0} 2^{q(n+1)} \langle 2^{q} t \rangle^{-N + 4}\lesssim \langle t\rangle^{-n-1},
  \end{aligned}$$ 
which gives \eqref{est-GtrLp} for $p=1$. The $L^p$ estimates follow from an interpolation between the $L^1$ and $L^\infty$ estimates. 

Finally, we prove the last estimate \eqref{conv-GtrLp}. Let $\FG_1(t,k) = \chi_0(k) \FG^r_k(t)$. We again write 
$$
\begin{aligned}
G_1(t)\star_x f  &= \sum_{q\le 0} P_q[ G_1(t)\star_x f ]
= \sum_{q\le 0} 2^{\delta q}\widetilde P_q[G_1(t)]\star_x [ 2^{-\delta q}\widetilde P_qf ]
 \end{aligned}$$
in which we have written $P_q=\widetilde P^2_q $ for some Littlewood-Paley projection $\widetilde P_q$ so that $\| \widetilde P_q f\|_{L^p_x} \approx \| P_q f\|_{L^p_x}$ for any $p\ge 1$. Therefore, for any $p\in [1,\infty]$ and $\delta \ge 0$, using the above $L^1_x$ estimates on $G_1(t)$, we bound 
$$
\begin{aligned}
\|G_1(t)\star_x f\|_{L^p_x} 
& \lesssim \sum_{q\le 0} \|\widetilde P_q[G_1(t)]\|_{L^1_x} \|\widetilde P_qf \|_{L^p_x}
\\
& \lesssim  \sup_{q\in \ZZ} \Big(2^{-\delta q}\|\widetilde P_qf \|_{L^p_x} \Big)\sum_{q\le 0} 2^{\delta q}\|\widetilde P_q[ G_1(t)]\|_{L^1_x} 
\\
& \lesssim  \sup_{q\in \ZZ} \Big(2^{-\delta q}\|\widetilde P_qf \|_{L^p_x} \Big)\sum_{q\le 0} 2^{(1+\delta) q}
\langle 2^q t \rangle^{-N_0 +3}
\\&\lesssim \langle t\rangle^{-1-\delta}  \sup_{q\in \ZZ} 2^{-\delta q}\|\widetilde P_qf \|_{L^p_x} ,
 \end{aligned}$$
yielding \eqref{conv-GtrLp}. This completes the proof of the proposition. 
\end{proof}


\subsection{Electric potential representation}


In this section, we give a complete representation of the linearized electric potential $\phi$ in terms of the initial data. Precisely, we obtain the following.

\begin{proposition}\label{prop-linphi}
 Let $\phi(t,x)$ be the electric potential of the linearized Vlasov-Klein-Gordon system \eqref{V-lin}, and let $G(t,x)$ be the Green function constructed as in Proposition \ref{prop-Greenphysical}. 
Then, there holds
 \begin{equation}\label{rep-electricE}
\begin{aligned}
\phi &= \sum_\pm \phi^{osc}_\pm(t,x) + \phi^r(t,x) 
\end{aligned}\end{equation}
where 
 \begin{equation}\label{rep-phiosc}
\begin{aligned}
\phi^{osc}_\pm(t,x) &= G_{\pm}^{osc}(t) \star_{x} \Big(\lambda_\pm(i\partial_x)\phi_0(x) +  \phi_1(x)\Big) 
- G^{osc}_{\pm} \star_{t,x}  \rho^0(t,x)
\\
\phi^r(t,x) &= \partial_t  G^r(t) \star_x \phi_0(x) + G^r(t) \star_x\phi_1(x) - G^r\star_{t,x}\rho^0(t,x),
\end{aligned}\end{equation}
for $\rho^0(t,x) = \int f_0(x-\hv t, v)\; dv$ being the density generated by the free transport dynamics. 
\end{proposition}

\begin{proof} First, taking the Laplace inverse of the resolvent equation \eqref{resolvent}, we obtain 
\begin{equation}
\label{rep-phiA} 
\begin{aligned}
\Fphi_k(t) &=\partial_t  \FG_k(t) \Fphi_k^0 + \FG_k(t) \Fphi_k^1 - \FG_k\star_t\Frho^0_k(t),
\end{aligned}
\end{equation}
where $\FG_k(t)$ is the Green function constructed as in Proposition \ref{prop-Green}, and the free density $\Frho^0_k(t) $ is computed from its Laplace transform $\Trho^0_k(\lambda) $ given in \eqref{resolvent},  yielding
$$
\begin{aligned}
 \Frho^0_k(t) 
 & =  \frac{1}{2\pi i}\int_{\{\Re \lambda = \gamma_0\}}e^{\lambda t}  \Trho^0_k(\lambda)  \; d\lambda
=  \int_{\RR^3}  \Big( \frac{1}{2\pi i}\int_{\{\Re \lambda = \gamma_0\}} e^{\lambda t} \frac{1}{\lambda +  ik \cdot \hat v} \; d\lambda \Big) \Ff_k^0(v)\; dv.  
\end{aligned}
$$
As $\lambda \mapsto \frac{1}{\lambda + ik\cdot \hv}$ is meromorphic in $\lambda$, and therefore, by Cauchy's residue theorem, the integration in $\lambda$ is equal to $e^{-ikt\cdot \hv}$, which yields 
\begin{equation}\label{free-density}  \Frho^0_k(t) =\int e^{-ikt\cdot \hv}\Ff_k^0(v) \;dv.
\end{equation}
Hence, $\rho^0(t,x) = \int f_0(x-\hv t, v)\; dv$ as claimed. Finally, using the decomposition \eqref{decomp-FG} into \eqref{rep-phiA} and noting $\partial_t  \FG^{osc}_{k,\pm}(t) = \lambda_\pm(k)\FG^{osc}_{k,\pm}(t) $, we write 
$$
\begin{aligned}
\Fphi_k(t) &=\sum_\pm \FG^{osc}_{k,\pm}(t)  \Big(\lambda_\pm(k)\Fphi_k^0 +  \Fphi_k^1\Big) -  \sum_\pm \FG^{osc}_{k,\pm}(t) \star_t\Frho^0_k(t)
\\
&\qquad +\partial_t  \FG^r_k(t) \Fphi_k^0 + \FG^r_k(t) \Fphi_k^1 - \FG^r_k\star_t\Frho^0_k(t).
\end{aligned}
$$
The representation \eqref{rep-electricE} thus follows from taking the inverse Fourier transform of the above expression. 
\end{proof}

%

\section{Nonlinear framework}\label{sec-nonlinearframework}



\subsection{Characteristics}


This section is devoted to the reformulation of the nonlinear Vlasov-Klein-Gordon problem \eqref{VKG1}-\eqref{VKG2}, following the Lagrangian framework introduced in \cite{HKNR2, ToanVKG}. Indeed, we first write the Vlasov equation in the perturbed form 
\begin{equation} \label{VKG1-pert}
\partial_t f + \hv  \cdot \nabla_x f +  E \cdot \nabla_v f = - E \cdot \nabla_v \mu 
\end{equation}
with the initial data $f(0,x,v) = f_0(x,v)$, where the electric field $E$ solves the Klein-Gordon equation \begin{equation}\label{VKG2-pert}
E = -\nabla_x \phi, \qquad (\Box_{t,x}+m^2_0) \phi = -\rho.
\end{equation} 
To solve \eqref{VKG1-pert}, we introduce the characteristics $(X_{s,t}(x,v),V_{s,t}(x,v))$, which are defined by 
\begin{equation}\label{ode-char} 
\frac{d}{ds}X_{s,t} = \hV_{s,t}, \qquad \frac{d}{ds}V_{s,t} = E(s,X_{s,t})
 \end{equation}
with initial data 
$$
X_{t,t} (x,v)= x, \qquad V_{t,t}(x,v)= v,
$$ 
where we recall the notation for the relativistic velocity 
$\hv  = v/\sqrt{1+|v|^2}$. 
We will use the characteristics "backwards", namely for $s \le t$.
By definition, we have 
\begin{equation} \label{ode-char2}
\begin{aligned}
X_{s,t} (x,v) &= x - \int_s^t \hV_{\tau,t}(x,v) \; d\tau , 
\\
 V_{s,t}(x,v) &= v - \int_s^t E(\tau, X_{\tau,t}(x,v)) \; d\tau.
\end{aligned}
\end{equation}
Integrating \eqref{VKG1-pert} along characteristics, we obtain an explicit formula for the solution of \eqref{VKG1-pert}, namely 
\begin{equation}
\label{charf1}
f(t,x,v)= f_0(X_{0,t}(x,v) , V_{0,t}(x,v))  - \int_0^t E(s,X_{s,t}(x,v)) \cdot \nabla_v \mu(V_{s,t}(x,v))  \,  ds .
\end{equation}


\subsection{A closed nonlinear formulation}


We now derive a closed equation for the nonlinear electric field $E$. 
Precisely, we obtain the following. 

\begin{lemma}\label{lem-linear} Let $\phi$ be the electric potential of the nonlinear Vlasov-Klein-Gordon problem \eqref{VKG1-pert}-\eqref{VKG2-pert}. 
Then, there holds the following identity  
\begin{equation}\label{iter-rho}
\begin{aligned}
(\Box_{t,x} +m_0^2)\phi(t,x) + \int_0^t \int_{\RR^3}\nabla_x \phi(s,x -(t-s)\hv ) \cdot \nabla_v \mu(v) \, dv ds &= - S(t,x)
\end{aligned}\end{equation}
where the nonlinear source density $S(t,x)$ is defined by 
\begin{equation}\label{nonlinear-S}
\begin{aligned}
S(t,x) &= S^0(t,x)+  S^\mu(t,x)
\end{aligned}\end{equation}
where 
\begin{equation}\label{nonlinear-Smu}
\begin{aligned}
 S^0(t,x) &= \int_{\RR^3} f_0(X_{0,t}(x,v) , V_{0,t}(x,v)) \, dv 
  \\
S^{\mu}(t,x) &= \int_0^t \int_{\RR^3}  \Big[ E(s,x - (t-s)\hv )\cdot \nabla_v \mu(v) 
 -E(s,X_{s,t}(x,v))\cdot \nabla_v \mu(V_{s,t}(x,v)) \Big] \, dv  ds .
\end{aligned}\end{equation}
\end{lemma}
\begin{proof} Indeed, recalling that the density of $f$ is computed by 
$\rho(t,x)= \int_{\RR^3} f(t,x,v)\, dv$, and using \eqref{charf1}, 
we get  
\begin{equation}\label{nonlinear-rho} 
\begin{aligned}
\rho(t,x) +  \int_0^t \int_{\RR^3}  E(s,X_{s,t}(x,v)) \cdot \nabla_v \mu(V_{s,t}(x,v))  \, dv  ds &= S^0(t,x),
\end{aligned}
\end{equation}
where the source $S^0(t,x)$ is defined by 
$$ 
\begin{aligned}
S^0(t,x) = \int_{\RR^3} f_0(X_{0,t}(x,v) , V_{0,t}(x,v)) \, dv .
\end{aligned}
$$
Next, using the Klein-Gordon equation $(\Box_{t,x}+m_0^2)\phi = -\rho$ and $E = -\nabla_x \phi$, we obtain the formulation as claimed. 
\end{proof}


\subsection{Nonlinear electric potential}\label{sec-nonlinearE}


We now apply the linear theory to the closed electric potential equation \eqref{iter-rho}. Precisely, we obtain the following nonlinear version of Proposition \ref{prop-linphi}.

\begin{proposition}\label{prop-nonphi} Let $\phi(t,x)$ be the electric potential of the nonlinear Vlasov-Klein-Gordon system \eqref{VKG1-pert}-\eqref{VKG2-pert} with initial data \eqref{data}, and let $G(t,x)$ be the Green function constructed as in Proposition \ref{prop-Greenphysical}. 
Then, there holds
 \begin{equation}\label{rep-electricE}
\begin{aligned}
\phi &= \sum_\pm \phi^{osc}_\pm(t,x) + \phi^r(t,x) 
\end{aligned}\end{equation}
where 
 \begin{equation}\label{rep-phiosc}
\begin{aligned}
\phi^{osc}_\pm(t,x) &= G_{\pm}^{osc}(t) \star_{x} \Big(\lambda_\pm(i\partial_x)\phi_0(x) +  \phi_1(x)\Big) 
- G^{osc}_{\pm} \star_{t,x} S(t,x)
\\
\phi^r(t,x) &= \partial_t  G^r(t) \star_x \phi_0(x) + G^r(t) \star_x\phi_1(x) - G^r\star_{t,x}S(t,x),
\end{aligned}\end{equation}
for the nonlinear source density $S(t,x)$ defined as in \eqref{nonlinear-S}. 
\end{proposition}

\begin{proof}

Indeed, following the spectral analysis in Section \ref{sec-lineartheory}, we take the Laplace-Fourier transform of the equation \eqref{iter-rho}, yielding 
$$\Tphi_k =  \frac{1}{M(\lambda,k)} \Big[ - \TS_k(\lambda) + \lambda \Fphi_k^0+ \Fphi_k^1\Big]$$
in which $M(\lambda,k)$ is the spacetime symbol defined as in \eqref{def-Dlambda}, and $\TS_k(\lambda) $ denotes the Laplace-Fourier transform of the source $S(t,x)$. The proof of the proposition is thus similar to that of Proposition \ref{prop-linphi}, upon replacing the linear source density $\rho^0(t,x)$ to the nonlinear source $S(t,x)$. 
\end{proof}


\section{Nonlinear iterative scheme}\label{sec-iterscheme}


In view of the nonlinear formulation, we shall decompose the nonlinear electric field according to
\begin{equation}\label{boots-repE}
\begin{aligned}
E &= \sum_\pm E^{osc}_\pm(t,x) + E^r(t,x)
\end{aligned}\end{equation}
where $E^r$ denotes the regular part, while the oscillatory component is of the form 
\begin{equation}\label{bootstrap-Eosc} E^{osc}_\pm(t,x) =G_\pm^{osc}(t,x) \star_{x} F_0(x) + G_\pm^{osc}(t,x) \star_{t,x} F(t,x),
\end{equation}
in which $F_0(x)$ involves only initial data. Note that $F_0(x), F(t,x),$ and $E^r(t,x)$ are a gradient of some density, namely
\begin{equation}\label{grad-F} F_0(x) = \nabla_x S_0(x), \qquad F(t,x) = \nabla_x S(t,x), \qquad E^r(t,x) = \nabla_x \rho^r(t,x). 
\end{equation}

\subsection{Bootstrap assumptions}\label{sec-bootstrap}

We assume that the nonlinear electric field $E$ is of the form \eqref{boots-repE}-\eqref{bootstrap-Eosc}. Fix $N_0 \ge 14$ and a sufficiently small $\epsilon>0$. Our bootstrap assumptions are as follows:

\begin{itemize}

\item Decay assumptions: 

\begin{equation}\label{bootstrap-decayS}
\begin{aligned}
 \|S(t)\|_{L^p_x}   &\le \epsilon \langle t\rangle^{-3(1-1/p)} \log t, \qquad p\in [1,\infty],
\\
\sup_{q\in \ZZ} 2^{(1-\delta)q} \|P_q\cS(t)\|_{L^p_x}  + \|\partial_xS(t)\|_{L^p_x}   &\le \epsilon \langle t\rangle^{-3(1-1/p)} , \qquad p\in [1,\infty],
\\
\|E^r(t)\|_{L^p_x} + \|\partial_x E^r(t)\|_{L^p_x}  &\le \epsilon \langle t\rangle^{-3(1-1/p)} , \qquad p\in [1,\infty],
\\
\| E^{osc}_\pm(t)\|_{L^p_x} + \| \partial_x E^{osc}_\pm(t)\|_{L^p_x} 
&\le \epsilon \langle t\rangle^{-3(1/2-1/p)} , \qquad p\in [2,\infty] ,
\end{aligned}\end{equation}
for $\delta \in (0,1)$, where $P_q$ denotes the Littlewood-Paley projection on the dyadic interval $[2^{q-1}, 2^{q+1}]$, with $q\in \ZZ$. In addition,  
\begin{equation}\label{bootstrap-decayd2E}
\begin{aligned}
\|\partial_x^2 E^r(t)\|_{L^p_x}  &\le \epsilon \langle t\rangle^{-3(1-1/p)} ,\qquad p\in (1,\infty),
\\
 \| \partial^2_x E^{osc}_\pm(t)\|_{L^p_x} 
&\le \epsilon \langle t\rangle^{-3(1/2-1/p)}, \qquad p\in [2,\infty). 
\end{aligned}\end{equation}

\item Decay assumptions for higher derivatives\footnote{In fact, we obtain stronger decay estimates on $E^{osc}_\pm(t)$ than those that are stated in \eqref{bootstrap-decaydaS}. See Section \ref{sec-decayosc}}: for all $1\le |\alpha|\le N_0-1$, 
\begin{equation}\label{bootstrap-decaydaS}
\begin{aligned}
\|\partial_x^\alpha S(t)\|_{L^p_x} + \|\partial_x^\alpha E^r(t)\|_{L^p_x}   &\le \epsilon \langle t\rangle^{-3(1-1/p) + \delta_{\alpha,p}} , \qquad p\in [1,\infty], 
\\
\|\partial_x^\alpha E^{osc}_\pm(t)\|_{L^p_x}   &\le \epsilon \langle t\rangle^{-3(1/2-1/p) + \epsilon_{\alpha,p}} , \qquad p\in [2,\infty], 
 \end{aligned}\end{equation}
where  
\begin{equation}\label{def-dlp}
\delta_{\alpha,p} = \max\Big\{ \delta_\alpha ,  \frac32\delta_\alpha (1 - 1/p) \Big\}, \qquad \epsilon_{\alpha,p} = 3 \delta_\alpha(1/2 - 1/p),
\end{equation}
in which $\delta_\alpha = \frac{|\alpha|}{N_0-1} $. In addition, 
\begin{equation}\label{bootstrap-decaydaSLp}
\begin{aligned}
\|\partial_x^{\alpha+1} E^r(t)\|_{L^p_x}   &\le \epsilon \langle t\rangle^{-3(1-1/p) + \delta_{\alpha,p}} , \qquad p\in (1,\infty), 
\\
\|\partial_x^{\alpha+1}E^{osc}_\pm(t)\|_{L^p_x}   &\le \epsilon \langle t\rangle^{-3(1/2-1/p) + \epsilon_{\alpha,p}} , \qquad p\in [2,\infty). 
 \end{aligned}\end{equation}

\item Boundedness assumptions:
\begin{equation}\label{bootstrap-Hs}
\begin{aligned}
\| E^{osc}_\pm(t)\|_{H^{N_0}_x} &\le \epsilon , 
\\
\| E^{osc}_\pm(t)\|_{H^{N_0+1}_x} +  \| E^r(t)\|_{H^{N_0+1}_x} + \| S(t)\|_{H_x^{N_0}}
&\le \epsilon \langle t\rangle^{\delta_1},
\end{aligned}
\end{equation}
with $\delta_1 = \frac{1}{N_0-1}$. 
 \end{itemize}

Observe that the oscillatory field $E^{osc}_\pm(t,x)$ disperses in space like a Klein-Gordon wave at a rate of order $t^{-3/2}$, while the source density $S(t,x)$ disperses in space at a faster rate of order $t^{-3}$, dictated by the free transport dynamics. 
In addition, we observe that $S(t,x)$ and $E^r(t,x)$ behave like a quadratic oscillation, namely 
\begin{equation}
\label{quad-Eosc}
\begin{aligned}
\|\partial_x^\alpha [E^{osc}_\pm E^{osc}_\pm](t)\|_{L^p_x} 
&\le \epsilon \langle t\rangle^{-3(1-1/p) + \delta_{\alpha,p}} , \qquad p\in [1,\infty],
\end{aligned}
\end{equation}
for all $1\le |\alpha|\le N_0-1$. Indeed, noting $\delta_\alpha$ is linear in $|\alpha|$, we check  
$$
\begin{aligned} 
\|\partial_x^\alpha [E^{osc}_\pm E^{osc}_\pm](t)\|_{L^\infty_x}  
&\le \sum_{|\alpha_1| \le |\alpha|}  \|\partial_x^{\alpha_1} E^{osc}_\pm (t)\|_{L^\infty_x} \|\partial_x^{\alpha-\alpha_1}E^{osc}_\pm(t)\|_{L^\infty_x} 
\lesssim \epsilon \langle t\rangle^{-3+ \frac32 \delta_{\alpha}},
\\
\|\partial_x^\alpha [E^{osc}_\pm E^{osc}_\pm](t)\|_{L^1_x}  
&\le \sum_{|\alpha_1|\le |\alpha|}  \|\partial_x^{\alpha_1} E^{osc}_\pm (t)\|_{L^{2}_x} \|\partial_x^{\alpha-\alpha_1}E^{osc}_\pm(t)\|_{L^{2}_x}
\lesssim \epsilon,
\end{aligned}
$$
which prove \eqref{quad-Eosc} by interpolation.

Finally, to better keep track of the oscillations of the electric field, we write in Fourier space 
\begin{equation} \label{defiFB}
\FE^{osc}_\pm(t,k) = e^{\lambda_{\pm}(k) t} \FB_\pm(t,k)
\end{equation}
in which 
\begin{equation}\label{def-Bpm}
\FB_\pm(t,k) = a_\pm(k) \FF_0(k)+  \int_0^t e^{-\lambda_\pm(k)s} a_\pm(k) \FF(s,k)\; ds, 
\end{equation}
recalling that $\lambda_\pm(k) = \pm i \nu_*(|k|)$ and $a_\pm(k) = \cO(\langle k\rangle^{-1})$ as defined in \eqref{def-Gosc}. In particular, we note that 
\begin{equation}\label{dtBF}
e^{\lambda_\pm(k)t} \partial_t \FB_\pm(t,k)  = a_\pm(k) \FF(t,k) 
\end{equation}
which plays a role in the nonlinear analysis, since $F(t,x)$ decays faster, namely at quadratic order, 
when compared with that of $E^{osc}_\pm(t,x)$. 

\subsection{The bootstrap argument}

Applying the standard local-in-time existence theory,  the bootstrap estimates hold for $t \in [0,T]$ for some small $T>0$. Suppose that there is a finite time $T_*$ so that the bootstrap assumptions hold for all $t \in [0,T_*)$. It suffices to prove that they remain valid for $t=T_*$, and hence the solution is global in time and satisfies the stated bounds. 

Indeed, we write the nonlinear electric field $E$ in the form \eqref{boots-repE}-\eqref{grad-F}. We shall prove that there are universal constants $C_0, C_1$ so that for all $t \in [0,T_*]$, all the bootstrap assumptions listed in Section \ref{sec-bootstrap} hold with an improved constant $C_0 \epsilon_0 + C_1\epsilon^2$. Namely, the bootstrap bound on $S(t,x)$ in \eqref{bootstrap-decayS} is now replaced by 
\begin{equation}\label{bootstrap-goalrho} 
\begin{aligned}
\|S(t)\|_{L^p_x} &\le \Big(C_0 \epsilon_0 + C_1\epsilon^2\Big) \langle t\rangle^{-3(1-\frac1p)} \log t , \qquad p\in [1,\infty],
\\
\sup_{q\in \ZZ} 2^{(1-\delta)q} \|P_q\cS(t)\|_{L^p_x}  + \|\partial_xS(t)\|_{L^p_x}   &\le \Big(C_0 \epsilon_0 + C_1\epsilon^2\Big) \langle t\rangle^{-3(1-1/p)} , \qquad p\in [1,\infty].
\end{aligned}
\end{equation}
We will also prove similar bounds for all other bootstrap assumptions for $t<T_*$. Finally, choosing $\epsilon_0 \ll \epsilon \ll1$ so that 
$$ C_0 \epsilon_0 + C_1 \epsilon^2 < \epsilon. $$
Therefore, the bootstrap assumptions listed in Section \ref{sec-bootstrap} indeed hold for $t=T_*$. The main theorem thus follows. 

\subsection{Proof of Theorem \ref{maintheorem}}

In this section, we briefly give the proof of the main theorem of the paper, Theorem \ref{maintheorem}, provided the following two results establishing the decay of the nonlinear source density $S(t,x)$ and of the electric field. Precisely, in Section \ref{sec-sourceest}, we will prove the following results on the source density.

\begin{proposition} \label{prop-bdsS-goal}
Fix $N_0\ge 14$. Let $S$ be the nonlinear source terms defined as in \eqref{nonlinear-S}-\eqref{nonlinear-Smu}. Under the bootstrap assumptions listed in Section \ref{sec-bootstrap}, there hold
\begin{equation}\label{bdsS-goal}
\begin{aligned}
\| S (t)\|_{L^p_x}   &\lesssim (\epsilon_0+ \epsilon^2) \langle t\rangle^{-3(1-1/p)} \log t
\\
\sup_{q\in \ZZ} 2^{(1-\delta)q} \|P_qS(t)\|_{L^p_x} + \|\partial_x S(t)\|_{L^p_x} &\lesssim (\epsilon_0 + \epsilon^2) \langle t\rangle^{-3(1-1/p)} 
\end{aligned}\end{equation}
for any $1 \le p \le \infty$ and $\delta \in (0,1)$, where $P_q$ denotes the Littlewood-Paley projection on the dyadic interval $[2^{q-1}, 2^{q+1}]$. 
In addition, 
\begin{equation}\label{Hsbounds-goal}
\begin{aligned}
\|\partial_x^\alpha S(t)\|_{L^p_x}   \le \epsilon \langle t\rangle^{-3(1-1/p) + \delta_{\alpha,p}} 
, \qquad 
\|S(t)\|_{H^{\alpha_0}_x} &\lesssim \epsilon \langle t\rangle^{\delta_1},
\end{aligned}
\end{equation}
for all $|\alpha|\le N_0-1$, with $\delta_{\alpha,p} = \max\{ \delta_\alpha, \frac32\delta_\alpha (1-1/p)\}$, where $\delta_\alpha = \frac{|\alpha|}{N_0-1}$. 
\end{proposition}

Next, in Section \ref{sec-decayosc}, we will establish the decay estimates on the oscillatory electric field, namely the following results. 

\begin{proposition}\label{prop-decayEosc-goal} 
Fix $N_0\ge 14$ and $n_0 = \frac14 N_0$. Let $S$ be the nonlinear source term defined as in \eqref{nonlinear-S}-\eqref{nonlinear-Smu}.
Then, there hold
\begin{equation}\label{decayEosc-goal}
\begin{aligned}
\|\partial^\alpha_xG_\pm^{osc} \star_{t,x} \nabla_x S(t)\|_{L^\infty_x} 
&\lesssim (\epsilon_0 + \epsilon^2) \langle t\rangle^{-3/2} ,\qquad |\alpha|\le n_0,
\\
\|\partial^{\alpha_0-1}_xG_\pm^{osc} \star_{t,x} \nabla_x S(t)\|_{L^\infty_x} 
&\lesssim (\epsilon_0 + \epsilon^2) ,\qquad |\alpha_0| = N_0,
\end{aligned}\end{equation}
and 
\begin{equation}\label{bdEosc-goal}
\begin{aligned}
\|G_\pm^{osc} \star_{t,x} \nabla_x S(t)\|_{H^{\alpha_0}_x} &\lesssim (\epsilon_0 + \epsilon^2)
\\
\|G_\pm^{osc} \star_{t,x} \nabla_x S(t)\|_{H^{\alpha_0+1}_x} &\lesssim (\epsilon_0 + \epsilon^2) \langle t\rangle^{\delta_1} .
\end{aligned}\end{equation}
\end{proposition}

Let us now complete the proof of Theorem \ref{maintheorem}, using the above results. Indeed, using the linear theory as in Section \ref{sec-nonlinearE}, we write the nonlinear electric field $E$ in  the form \eqref{boots-repE}-\eqref{grad-F}, namely 
\begin{equation}\label{goal-repE}
\begin{aligned}
E &= \sum_\pm E^{osc}_\pm(t,x) + E^r(t,x)
\end{aligned}\end{equation}
where
\begin{equation}\label{goal-EoscEr} 
\begin{aligned}
E^{osc}_\pm(t,x) &=-G^{osc}_\pm(t) \star_x (\lambda_\pm(i\partial_x)\nabla_x\phi_0(x) +\nabla_x\phi_1(x)) + G_\pm^{osc}\star_{t,x} \nabla_x S(t,x),
\\E^r(t,x) &= -\partial_t  G^r(t) \star_x \nabla_x\phi_0(x) - G^r(t) \star_x\nabla_x\phi_1(x) + G^r\star_{t,x}\nabla_xS(t,x).
\end{aligned}
\end{equation}
Here, in the above expression, $\phi_j(x)$ denote the initial electric potentials, $G^{osc}_\pm(t,x), G^r(t,x)$ are the Green functions as in Proposition \ref{prop-Greenphysical}, and $S(t,x)$ is the nonlinear source density defined as in \eqref{nonlinear-S}. Note that by construction, both $E^{osc}_\pm$ and $E^r$ are apparently a gradient of some density functions, as desired in \eqref{grad-F}. 

We now check the bootstrap assumptions on $E(t,x)$ by verifying the bootstrap bounds for each term in the above expression. Indeed, for the oscillatory electric field, the last convolution term $G_\pm^{osc} \star_{t,x} \nabla_x S$ is treated in Proposition \ref{prop-decayEosc-goal}, while the terms involving the initial data are estimated using the dispersive estimates \eqref{est-HoscLp} and the assumptions on the initial data, see \eqref{data-phi}. 

It remains to bound $E^r(t,x)$. Recall from  Proposition \ref{prop-Greenphysical} that 
\begin{equation}\label{est-Grphysical} 
\|\partial_t^\alpha G^r(t)\|_{L^p_x} \lesssim \langle t\rangle^{-4+3/p-|\alpha|} ,
\end{equation}
for any $p\in [1,\infty]$ and for $|\alpha|\le 1$. 
Therefore, the bootstrap bounds on the convolution against initial data $\phi_j(x)$ follow directly, upon using the standard inequality $\| G^r(t)\star_{x} f\|_{L^p_x} \lesssim \|G^r(t)\|_{L^p_x}\| f\|_{L^1_x}$, and the boundedness of $\phi_j$ in $W^{N_0,1}_x$, see \eqref{data-phi}. On the other hand, using again \eqref{est-Grphysical}, we bound 
$$  
\begin{aligned}
\| G^r \star_{t,x} S\|_{L^p_x} 
&\le \int_0^t  \Big\| G^r(t-s) \star_{x} S(s)\Big\|_{L^p_x} \; ds
\\
&\lesssim
 \int_0^{t/2}  \| G^r(t-s)\|_{L^p_x}\|S(s)\|_{L^1_x} \; ds
 +  \int_{t/2}^t  \| G^r(t-s)\|_{L^1_x}\|S(s)\|_{L^p_x} \; ds
\\
&\lesssim(\epsilon_0+ \epsilon^2) 
 \int_0^{t/2} \langle t-s\rangle^{-4+3/p} \log s\; ds 
+(\epsilon_0+ \epsilon^2)   \int_{t/2}^t  \langle t-s\rangle^{-1}\langle s\rangle^{-3(1-1/p)} \log s\; ds
\\
&\lesssim
(\epsilon_0+ \epsilon^2)  \langle t\rangle^{-3+3/p} \log^2 t.
\end{aligned} $$
Similarly, we now bound the above convolution against $\nabla_x S$, for which the log loss is avoided. Indeed, we recall from \eqref{conv-GtrLp} that 
\begin{equation}\label{conv-GtrLp1}
\| \chi(i\partial_x) G^{r} (t)\star_x \nabla_x S\|_{L^p_x} \lesssim \langle t\rangle^{-1-\delta} \sup_{q\in \ZZ} 2^{(1-\delta) q}\| P_qS\|_{L^p_x},
\end{equation}
for any $p\in [1,\infty]$, and we compute 
$$  
\begin{aligned}
\|&\chi(i\partial_x) G^r \star_{t,x}  \nabla_xS\|_{L^p_x} 
\\&\lesssim
 \int_0^{t/2}  \|  \chi(i\partial_x) G^r(t-s)\|_{L^p_x}\| \nabla_xS(s)\|_{L^1_x} \; ds
+  \int_{t/2}^t  \langle t-s\rangle^{-1-\delta}  \sup_{q\in \ZZ} 2^{(1-\delta) q}\| P_q S(s)\|_{L^p_x} \; ds
\\
&\lesssim(\epsilon_0+ \epsilon^2) 
 \int_0^{t/2} \langle t-s\rangle^{-4+3/p} \; ds
 + (\epsilon_0+ \epsilon^2)  \int_{t/2}^t   \langle t-s\rangle^{-1-\delta} \langle s\rangle^{-3(1-1/p)} \; ds
\\
&\lesssim
(\epsilon_0+ \epsilon^2)  \langle t\rangle^{-3+3/p} ,
\end{aligned} $$
recalling $\delta>0$. The bounds for the high frequency part follow similarly, upon noting that $\|\chi(i\partial_x) G^r(t)\|_{L^p_x} 
$ decays rapidly fast in $t$. Finally, higher derivative estimates follow identically, therefore completing the proof of the bootstrap bounds on $E^r(t,x)$, provided Propositions \ref{prop-bdsS-goal}-\ref{prop-decayEosc-goal} to be proved in the remaining sections. This completes the proof of Theorem \ref{maintheorem}.

\begin{proof}[Proof of Corollary \ref{cor-main}]
Finally, we prove the scattering result. Indeed, recalling the characteristics \eqref{ode-char},  
we observe that 
$$\int_0^t E(s,X_{s,t}(x,v)) \cdot \nabla_v \mu(V_{s,t}(x,v))  \,  ds=  \int_0^t {d \over ds}V_{s,t}(x,v) \cdot \nabla_v \mu(V_{s,t}(x,v))  \,  ds
  = \mu(v) - \mu(V_{0, t}(x,v)).$$
Putting this into \eqref{charf1}, we obtain 
$$
f(t,x,v)= f_0(X_{0,t}(x,v) , V_{0,t}(x,v))  +  \mu(V_{0, t}(x,v)) -\mu(v) .
$$
In the next section, see Proposition \ref{prop-Dchar}, we shall prove that 
$$X_{0,t}(x,v) = x - t \hv+ Y_{0,t}(x,v), \qquad V_{0,t}(x,v) = v + W_{0,t}(x,v)$$
in which there are some $C^1$ functions $Y_{\infty}(x,v), W_{\infty}(x,v)$ so that $ \| Y_{0,t}-  Y_{\infty}\|_{L^\infty_{x,v}} \lesssim \epsilon_0 \langle t\rangle^{-1/2}$ and $\| W_{0,t}-  W_{\infty}\|_{L^\infty_{x,v}} \lesssim \epsilon_0 \langle t\rangle^{-3/2}$. As a result, we set 
$$
\begin{aligned}
f_{\infty}(x,v) &= f_{0}\left(x - t \hv+ Y_{\infty}(x,v), v+ W_{\infty}(x,v)\right) 
+ \mu\left(v + W_{\infty}(x,v)\right)-\mu(v ).
\end{aligned}
$$
Note that $f_\infty \in W^{1,\infty}$. Therefore, 
there holds
 \begin{equation}\label{scattering}
 \| f(t,x+t\hv, v) - f_\infty(x,v) \|_{L^\infty_{x,v}} \lesssim \epsilon_0 \langle t\rangle^{-1/2},
 \end{equation}
 for all $t \geq 0$. 
 \end{proof}


\section{Characteristics}\label{sec-Char}



\subsection{Introduction}


Let $X_{s,t}(x,v)$ and $V_{s,t}(x,v)$ be the nonlinear characteristics. As the electric field satisfies $\|E(t)\|_{L^\infty_x}  \lesssim \epsilon \langle t \rangle^{-3/2}$, it directly follows that 
\begin{equation}\label{quick-bdV}
 \| V_{s,t} - v\|_{L^\infty_{x,v}} \lesssim  \epsilon \langle s\rangle^{-1/2} ,
 \end{equation}
namely, $V_{s,t}(x,v)$ remains close to $v$. 
However, such a bound on the velocity is too weak to precisely locate the position of the moving particles $X_{s,t} (x,v)$.
As $E$ has on oscillatory component, the velocities and the trajectories of the particle "oscillate". 
We need to extract these oscillations in order to get better bounds on the velocities and a good localisation of the electrons. Throughout this section, we assume that the electric field is of the form
\begin{equation}\label{bootstrapE1}
E(t,x) = \sum_\pm E^{osc}_\pm(t,x)+ E^r(t,x)
\end{equation} 
with the oscillatory electric field 
$$E^{osc}_\pm(t,x) = G_\pm^{osc}(t,x) \star_{x} F_0(x)+G_\pm^{osc} \star_{t,x} F(t,x) $$
where $F_0, F$, and $E^r$ satisfy the bootstrap assumptions stated in Section \ref{sec-bootstrap}. In addition, in what follows, $X_{s,t}(x,v)$ and $V_{s,t}(x,v)$ are the nonlinear characteristic solving \eqref{ode-char}. 

\subsection{Integrated electric fields}

Let $\lambda_\pm = \pm i \nu_*(|k|)$ be the Langmuir oscillatory modes constructed as in Theorem \ref{theo-LangmuirB}. 
For $j\ge 1$, we set  
\begin{equation}\label{def-Eoscj}
E^{osc,j}_\pm(t,x,v) = \phi_{\pm,j}^v (i\partial_x)E^{osc}_{\pm} (t,x)
\end{equation}
where $ \phi_{\pm,j}^v(k)$ denote a family of Fourier multipliers defined by 
\begin{equation}\label{def-phipmj} 
\phi_{\pm,j}^v (k)= \frac{1}{(\lambda_\pm(k) + ik\cdot \hv)^{j}}.
\end{equation}
Then, we obtain the following lemma. 
\begin{lemma}\label{lem-PEosc}
For any compact subset $K \Subset \RR^3$ and $\delta>0$, there hold  
\begin{equation}\label{Lp-convphi}
\begin{aligned}
\| \sup_{v\in K}  \phi^v_{\pm,j}(i\partial_x)\partial_x^\alpha f \|_{L^p_x} &\lesssim \| f\|_{L^p_x} ,\qquad j>|\alpha|, 
\end{aligned}\end{equation}
for any $\alpha\ge 0$, $j\ge 1$, and $1\le p\le \infty$. 
\end{lemma}
\begin{proof} Recall from Theorem \ref{theo-LangmuirB} that $\lambda_\pm = \pm i \nu_*(|k|)$ with $\nu_*(|k|) > |k|$. Therefore, $\lambda_\pm(k) + ik\cdot \hv$ never vanishes, since $|\hv|<1$, and so $\phi_{\pm,j}^v (k)$ are well-defined. In particular, we may write 
\begin{equation}\label{expand-phi1} \frac{1}{\lambda_\pm(k) + ik \cdot \hv} = \mp i \nu_*(|k|)^{-1} \sum_{n\ge 0} (\pm 1)^n\nu_*(|k|)^{-n} (k\cdot \hat v)^n = \sum_{n\ge 0}  a_{\pm,n}(k) :: \hv^{\otimes n}\end{equation}
for $a_{\pm,n}(k) =  \mp i \nu_*(|k|)^{-1} (\pm 1)^n\nu_*(|k|)^{-n} k^{\otimes n}$, which are smooth Fourier multipliers and satisfy $|a_{\pm,n}(k)|\le \nu_*(|k|)^{-1}$, uniformly for all $n\ge 0$, since $\nu_*(|k|) >|k|$. In the above the notation $k^{\otimes n}::\hv^{\otimes n} = (k\cdot \hv)^n$ is used only for sake of presentation. The series is therefore absolutely converging for bounded $v$. The lemma thus follows from the results obtained in Section \ref{sec-FM}. 
\end{proof}

\begin{corollary}\label{cor-Ejosc}
Under the bootstrap assumptions on $E$, with $N_0 \ge 4$, there hold 
\begin{equation}\label{decay-Eoscj}
\begin{aligned}
\| \sup_{v\in K} \partial_x^\alpha \partial_v^\beta E^{osc,j}_\pm\|_{L^p_x} &\lesssim \epsilon \langle t\rangle^{-3(1/2-1/p)}, \qquad j \ge |\alpha|,
\end{aligned}
\end{equation}
for $p\in [2,\infty]$. 
\end{corollary}
\begin{proof}
The corollary follows from the definition \eqref{def-Eoscj},  \eqref{Lp-convphi}, and the bootstrap assumptions on the decay of $E$ and $\partial_xE$ from \eqref{bootstrap-decayS}. 
\end{proof}

\subsection{Velocity description}

We first study oscillations in the velocity. 

\begin{proposition}\label{prop-charV}
Let $X_{s,t}(x,v), V_{s,t}(x,v)$ be the nonlinear characteristic solving \eqref{ode-char}. Then, there holds
 \begin{equation}\label{decomp-Vst}
 \begin{aligned}
V_{s,t}(x,v) &=
v -V^{osc}_{t,t}(x,v) + V^{osc}_{s,t}(x,v)
+ V^{tr}_{s,t}(x,v)
 \end{aligned}
 \end{equation}
where 
$$
\begin{aligned}
  V^{osc}_{s,t}(x,v) &= \sum_\pm E^{osc,1}_{\pm}(s,X_{s,t}(x,v),V_{s,t}(x,v)) 
 \\
V^{tr}_{s,t}(x,v)& = -  \int_s^t  
Q^{tr}(\tau,X_{\tau,t}(x,v), V_{\tau,t}(x,v)) \; d\tau
 \end{aligned}$$
in which $Q^{tr}(t,x,v)$ is defined by 
\begin{equation}\label{def-Qtr}
\begin{aligned} Q^{tr} &=\sum_\pm [a_\pm(i\partial_x) \phi_{\pm,1}\star_x F]
-  \sum_\pm \nabla_v \hv  E\cdot  \nabla_x E^{osc,2}_\pm + E^r.
\end{aligned} 
\end{equation}
In addition, under the bootstrap assumptions on $E$ and $F$, there hold 
\begin{equation}\label{bounds-Vst}
\begin{aligned} 
\|V^{osc}_{s,t}\|_{L^\infty_{x,v}} \lesssim \epsilon \langle s\rangle^{-3/2} ,
\qquad   \|V^{tr}_{s,t}\|_{L^\infty_{x,v}} \lesssim \epsilon \langle s\rangle^{-2} , 
\qquad   \|\partial_sV^{tr}_{s,t}\|_{L^\infty_{x,v}} \lesssim \epsilon \langle s\rangle^{-3} , 
 \end{aligned}
\end{equation}
and 
\begin{equation}\label{decay-Vst}
\|V_{s,t}-v+V^{osc}_{t,t}\|_{L^\infty_{x,v}} \lesssim \epsilon \langle s\rangle^{-3/2} .
\end{equation}
 \end{proposition}

\begin{proof}
Recall that $E = \sum_\pm E^{osc}_\pm$, and therefore, from \eqref{ode-char2}, 
$$
\begin{aligned}V_{s,t}(x,v) =v   - \sum_\pm \int_s^t E^{osc}_\pm(\tau, X_{\tau,t}(x,v)) \; d\tau - \int_s^t E^r(\tau, X_{\tau,t}(x,v)) \; d\tau .
\end{aligned}$$
It suffices to study the integral of $E^{osc}_\pm$ along the characteristic. From \eqref{defiFB}, we write in Fourier space, $
\FE^{osc}_\pm(t,k) = e^{\lambda_{\pm}(k) t} \FB_\pm(t,k) .
$ Therefore, integrating by parts in $\tau$ and recalling that $\partial_\tau X_{\tau,t} = \hV_{\tau,t}$, we compute 
\begin{equation}\label{int-Eosc}
\begin{aligned}
 \int_s^t E^{osc}_{\pm}(\tau, X_{\tau,t}) \; d\tau
& = \int_s^t \int e^{\lambda_\pm(k) \tau + ik\cdot X_{\tau,t}}  \FB(\tau,k) \; dkd\tau 
\\
& = \int  \int_s^t\frac{d}{d\tau} (e^{\lambda_\pm(k) \tau + ik\cdot X_{\tau,t}})  \frac{\FB(\tau,k)}{\lambda_\pm(k) + ik \cdot \hV_{\tau,t}} \; d\tau dk 
\\
& = 
\int \Big[\frac{e^{\lambda_\pm(k) t + ik\cdot x}  \FB(t,k)}{\lambda_\pm(k) + ik \cdot \hv} -  \frac{e^{\lambda_\pm(k) s + ik\cdot X_{s,t}}\FB(s,k)}{\lambda_\pm(k) + ik \cdot \hV_{s,t}} \Big] \; dk\\
& \quad - \int_s^t \int e^{\lambda_\pm(k) \tau + ik\cdot X_{\tau,t}}  \Big[\frac{ \partial_\tau \FB(\tau,k)}{\lambda_\pm(k) + ik \cdot \hV_{\tau,t}} - \frac{ \FB(\tau,k) ik \cdot \partial_\tau \hV_{\tau,t}}{(\lambda_\pm(k) + ik \cdot \hV_{\tau,t})^2}\Big] \; dkd\tau .
\end{aligned}\end{equation}
Defining $\phi_{\pm,j}(x,v)$ as in \eqref{def-phipmj} and recalling from \eqref{defiFB} and \eqref{dtBF} that 
\begin{equation}\label{recall-FB}
 e^{\lambda_{\pm}(k) t} \FB_\pm(t,k) = \FE^{osc}_\pm(t,k) 
, \qquad e^{\lambda_\pm(k)t} \partial_t \FB_\pm(t,k)  = a_\pm(k) \FF(t,k), \end{equation}
and $\partial_\tau \hV_{\tau,t} =(\nabla_v \hv E)(\tau, X_{\tau,t},V_{\tau,t}),$
 we obtain \eqref{decomp-Vst}. It remain to prove the stated estimates. Indeed, 
the estimates on $V^{osc}_{s,t}$ follow from \eqref{decay-Eoscj} and the bootstrap assumption on $E^{osc}_\pm$. On the other hand, recalling that $F = \nabla_x S$ and using Lemma \ref{lem-PEosc}, we bound 
\begin{equation}\label{bdQtr}
\begin{aligned} 
 \| Q^{tr}(t)\|_{L^\infty_{x,v}} 
 &\lesssim \sup_v \sum_\pm \| a_\pm(i\partial_x) \phi_{\pm,1} \star_x \nabla_x S\|_{L^\infty_x}
+  \sum_\pm \|E\|_{L^\infty_x} \|\nabla_x E^{osc,2}_\pm \|_{L^\infty_{x,v}} + \| E^r (t)\|_{L^\infty_x}
\\
&\lesssim \|\nabla_x S(t)\|_{L^\infty_x}
+  \|E(t)\|^2_{L^\infty_x}  + \| E^r (t)\|_{L^\infty_x}
\\
&\lesssim \epsilon \langle t\rangle^{-3}
,\end{aligned} 
\end{equation}
upon using the bootstrap assumptions on $E$ and $S$. The estimates on $V^{tr}_{s,t}$ follow directly from that of $Q^{tr}(t)$. 
\end{proof}


\subsection{Characteristic description}


In this section, we study oscillations in the characteristic $X_{s,t}(x,v)$. In view of \eqref{ode-char2}, 
we write 
\begin{equation}\label{straightX}
X_{s,t}(x,v) = x - (t-s) \hPsi_{s,t}(x,v)
\end{equation}
where $\hPsi_{s,t}(x,v)$ is the velocity average defined by 
\begin{equation}\label{def-hatPsi}
\hPsi_{s,t} (x,v) = \frac{1}{t-s}\int_s^t \hV_{\tau,t}(x,v) \; d\tau
\end{equation}
in which we recall $\hv  = v/\sqrt{1+|v|^2}$. Note that $|\hPsi_{s,t} (x,v)|<1 $ and so $\Psi_{s,t} = \hPsi_{s,t}/\sqrt{1-|\hPsi_{s,t}|^2}$ is indeed well-defined. 
We first establish the following proposition which gives a precise description of oscillations in $\hPsi_{s,t}(x,v)$. 

\begin{proposition}\label{prop-charPsi} 
Let $X_{s,t}(x,v), V_{s,t}(x,v)$ be the nonlinear characteristic solving \eqref{ode-char}, and let $\hPsi_{s,t}(x,v)$ be defined as in \eqref{def-hatPsi}. Then, 
there hold 
\begin{equation}\label{decomp-Psist}
 \begin{aligned}
\Psi_{s,t}(x,v) &= v -  V^{osc}_{t,t}(x,v)+\Psi^{osc}_{s,t}(x,v)  + \Psi^{tr}_{s,t}(x,v) + \Psi^{R}_{s,t}(x,v)
\\
\hPsi_{s,t}(x,v) &= \hv - \hV^{osc}_{t,t}(x,v) + \hPsi^{osc}_{s,t}(x,v) + \hPsi^{tr}_{s,t}(x,v) + \Psi^{Q}_{s,t}(x,v)
 \end{aligned}
 \end{equation}
where  
$$
\begin{aligned}
\Psi^{osc}_{s,t}(x,v) &= \frac{1}{t-s}\sum_\pm \Big(E^{osc,2}_{\pm}(t,x,v) - E^{osc,2}_{\pm} (s,X_{s,t}(x,v),V_{s,t}(x,v)) \Big)
 \\
\Psi^{tr}_{s,t}(x,v)& = - \frac{1}{t-s} \int_s^t  
[ (\tau-s)  
Q^{tr} + Q^{tr,1}](\tau,X_{\tau,t}(x,v), V_{\tau,t}(x,v)) \; d\tau 
 \end{aligned}$$
and 
$$
\begin{aligned}
\Psi^Q_{s,t}(x,v) &=  \frac{1}{t-s}\int_s^t V_{\tau,t}^Q(x,v)\; d\tau, \qquad 
\\
\Psi^{R}_{s,t}(x,v)  &= \nabla_v \hv \Psi^Q_{s,t}(x,v)  + \cO(\langle v\rangle^5|\hPsi_{s,t}- \hv|^2) 
\end{aligned}$$
with $|V^Q_{s,t} (x,v)|\lesssim \langle v\rangle^{-2}|V_{s,t}-v|^2$, in which $Q^{tr}(t,x,v)$ is defined as in \eqref{def-Qtr} and 
\begin{equation}\label{def-Qtr1}
\begin{aligned} Q^{tr,1} &=\sum_\pm [a_\pm(i\partial_x) F\star_x \phi_{\pm,2}]
- 2 \sum_\pm \nabla_v \hv  E\cdot   \nabla_x E^{osc,3}_\pm .
\end{aligned} 
\end{equation}
Here in \eqref{decomp-Psist}, we abuse the notation to define $\hV^{osc}_{t,t} (x,v)=  \nabla_v \hv V_{t,t}^{osc}(x,v)$  and $\hPsi^{a}_{s,t} (x,v)=  \nabla_v \hv \Psi_{s,t}^a(x,v)$ for $a \in \{osc,tr\}$. 
In particular, there hold  
\begin{equation}\label{bounds-Wst}
\begin{aligned} \|  \Psi^{osc}_{s,t}\|_{L^\infty_{x,v}} \lesssim \epsilon \langle s\rangle^{-3/2} (t-s)^{-1}, 
\quad  \| \Psi^{tr}_{s,t}\|_{L^\infty_{x,v}} &\lesssim \epsilon \langle s\rangle^{-1} \langle t\rangle^{-1} , 
 \\
  \| \Psi^{Q}_{s,t}\|_{L^\infty_{x,v}} + \| \langle v\rangle^{-3}\Psi^{R}_{s,t}\|_{L^\infty_{x,v}} + \|\partial_s \Psi^{tr}_{s,t}\|_{L^\infty_{x,v}} &\lesssim \epsilon  \langle s\rangle^{-2}\langle t\rangle^{-1},
 \end{aligned}
\end{equation}
and 
\begin{equation}\label{decay-Xst}
\begin{aligned}
\| \hPsi_{s,t}-\hv + \hV^{osc}_{t,t}\|_{L^\infty_{x,v}} & \lesssim  \epsilon \langle s\rangle^{-1}(t-s)^{-1}.
\\
\| X_{s,t}-x + (t-s)(\hv - \hV^{osc}_{t,t})\|_{L^\infty_{x,v}} &\lesssim \epsilon \langle s\rangle^{-1}.
\end{aligned}
\end{equation}
\end{proposition}

\begin{proof} We first prove the expansion for $\hPsi_{s,t}(x,v)$. Note that $\hv  = v/\langle v\rangle$ is a regular and bounded function of $v$. In particular, $\nabla_v \hv  = \frac{1}{\langle v\rangle} (\mathbb{I} - \hv  \otimes \hv )$ and $|\nabla^2_v \hv|\lesssim \langle v\rangle^{-2}$. 
Hence, in view of \eqref{decomp-Vst}, we may write 
\begin{equation}\label{def-VQ}
\begin{aligned}
\hV_{s,t} (x,v)&=  \hv +\nabla_v \hv(V_{s,t} (x,v) - v) + V^Q_{\tau,t}(x,v)
\end{aligned}
\end{equation}
with a quadratic remainder $|V^Q_{\tau,t} (x,v)|\lesssim \langle v\rangle^{-2}|V_{s,t}-v|^2$. Therefore, by definition \eqref{def-hatPsi}, we compute 
$$
\begin{aligned}
\hPsi_{s,t}(x,v) &=  \hv 
+ \frac{1}{t-s}\int_s^t (\hV_{\tau,t}(x,v) - \hv ) \; d\tau
\\
&=\hv 
+ \frac{\nabla_v \hv }{t-s}\int_s^t (V_{\tau,t}(x,v) - v ) \; d\tau 
+ \frac{1}{t-s}\int_s^t V^Q_{\tau,t}(x,v) \; d\tau.
\end{aligned}
$$
Using Proposition \ref{prop-charV}, we compute the average of $V_{\tau,t}(x,v)$, 
$$
 \begin{aligned}
\frac{1}{t-s}\int_s^t V_{\tau,t}(x,v) &=
v -V^{osc}_{t,t}(x,v) + \frac{1}{t-s}\int_s^t [V^{osc}_{\tau,t}(x,v)
+ V^{tr}_{\tau,t}(x,v) ]\; d\tau
 \end{aligned}
$$
in which by construction, $  V^{osc}_{s,t} = \sum_\pm E^{osc,1}_{\pm} (s,X_{s,t},V_{s,t}) $. Therefore, similarly to what was done in the proof of Proposition \ref{prop-charV}, we have  
$$
\begin{aligned} \int_s^t V^{osc}_{\tau,t}(x,v) \;d\tau &= \int_s^t  \int \frac{e^{\lambda_\pm(k) \tau + ik\cdot X_{\tau,t}}\FB(\tau,k)}{\lambda_\pm(k) + ik \cdot \hV_{\tau,t}}\; dkd\tau
\\
& = \int \int_s^t \frac{d}{d\tau} (e^{\lambda_\pm(k) \tau + ik\cdot X_{\tau,t}})  \frac{\FB(\tau,k)}{(\lambda_\pm(k) + ik \cdot \hV_{\tau,t})^2} \; d\tau dk 
\\
& = 
\int \Big[\frac{e^{\lambda_\pm(k) t + ik\cdot x}  \FB(t,k)}{(\lambda_\pm(k) + ik \cdot \hv)^2} -  \frac{e^{\lambda_\pm(k) s + ik\cdot X_{s,t}}\FB(s,k)}{(\lambda_\pm(k) + ik \cdot \hV_{s,t})^2} \Big] \; dk\\
& \quad - \int_s^t \int e^{\lambda_\pm(k) \tau + ik\cdot X_{\tau,t}}  \Big[\frac{ \partial_\tau \FB(\tau,k)}{(\lambda_\pm(k) + ik \cdot \hV_{\tau,t})^2} - \frac{2 \FB(\tau,k) ik \cdot \partial_\tau \hV_{\tau,t}}{(\lambda_\pm(k) + ik \cdot \hV_{\tau,t})^3}\Big] \; dkd\tau .
\end{aligned}$$
The first two integrals equal to $E^{osc,2}_{\pm}(t,x,v) $ and $E^{osc,2}_{\pm} (s,X_{s,t}(x,v),V_{s,t}(x,v)) $, respectively, while the last integral is 
$$
\begin{aligned}  - \int_s^t  
Q^{tr,1}(\tau,X_{\tau,t}(x,v), V_{\tau,t}(x,v)) \; d\tau
\end{aligned}$$
having defined $Q^{tr,1}(t,x,v)$ as in \eqref{def-Qtr1}. In addition, by definition, we note 
$$
\begin{aligned}
\int_s^t V^{tr}_{\tau,t}\; d\tau
&= - \int_s^t\int_\tau^t  
Q^{tr}(\tau') \; d\tau' d\tau 
 = - \int_s^t(\tau-s)  
Q^{tr}(\tau) \; d\tau .
\end{aligned}$$
This gives the expansion for $\hPsi_{s,t}(x,v)$. The estimates on $Q^{tr}$ and $Q^{tr,1}$ follow from the bootstrap assumptions on $E$, which directly gives \eqref{bounds-Wst}.

Finally, we prove the expansion for $\Psi_{s,t}(x,v)$ from that of $\hPsi_{s,t}(x,v)$. Indeed, note that $1 - |\hv|^2 = \langle v\rangle^{-2}$, and observe that 
$$\hv \cdot \nabla_v \hv  = \frac{1}{\langle v\rangle} \hv \cdot (\mathbb{I} - \hv  \otimes \hv ) =  \frac{1}{\langle v\rangle} (1-|\hv|^2)\hv = \langle v\rangle^{-3} \hv.$$  
In addition, $V^Q_{s,t}(x,v)$ is of order $\cO(\epsilon^2\langle v\rangle^{-2})$. Therefore, 
$$
\begin{aligned}
 1 - |\hPsi_{s,t}(x,v)|^2 &= 1 - |\hv |^2 - 2 \hv \cdot (\hPsi_{s,t} - \hv) 
- |\hPsi_{s,t} - \hv|^2 \\ &= 1 - |\hv |^2 + \cO(\epsilon \langle v\rangle^{-2})
 \\ &= (1 - |\hv |^2) (1+ \cO(\epsilon)).
  \end{aligned}$$
In particular,  $1 - |\hPsi_{s,t}(x,v)|^2$ remains strictly positive for all $s,t,x,v$, and is close to $1 - |\hv |^2$. 
Therefore, $\Psi_{s,t}(x,v)$ is well-defined through the map $\hv \mapsto v = \hv / \sqrt{1-|\hv|^2}$, giving a similar expansion as that of $\hPsi_{s,t}(x,v)$. In particular, noting $\nabla_{\hv} v = \langle v\rangle (\mathbb{I} + v \otimes v)$ and $|\nabla^2_{\hv} v|\lesssim \langle v\rangle^5$, we have 
\begin{equation}\label{def-WR}
\Psi_{s,t} - v  = \nabla_{\hv} v (\hPsi_{s,t} - \hv) + \cO(\langle v\rangle^5|\hPsi_{s,t} - \hv|^2)
\end{equation}
in which the last term is put into the remainder $\Psi^R_{s,t}(x,v)$. This gives the expansion for $\Psi_{s,t}(x,v)$, upon noting that $\nabla_{\hv} v = (\nabla_v \hv)^{-1}$. 

Finally, in view of the definition, $Q^{tr,1}(t)$ satisfies the same estimates as done for $Q^{tr}(t)$ in \eqref{bdQtr}, yielding 
\begin{equation}\label{bdQtr1} 
\| Q^{tr}(t)\|_{L^\infty_{x,v}} + \| Q^{tr,1}(t)\|_{L^\infty_{x,v}} \lesssim \epsilon \langle t\rangle^{-3}.
\end{equation} 
The estimates in \eqref{bounds-Wst} thus follow at once from the definition and the above bounds. This yields the proposition.
\end{proof}


\subsection{Derivative bounds}\label{sec-dxchar}


In this section, we study derivatives of the characteristics. We prove the following proposition. 

\begin{proposition}\label{prop-Dchar} 
Let $X_{s,t}(x,v), V_{s,t}(x,v)$ be the nonlinear characteristic solving \eqref{ode-char}, and let $\hPsi_{s,t}(x,v)$ be defined as in \eqref{def-hatPsi}. Then,  for $0\le s\le t$, there hold 
\begin{equation}\label{Dxv-VWend}
 \| \partial_x \partial_v^\beta V^{osc}_{t,t} \|_{L^\infty_{x,v}}   \lesssim \epsilon \langle t\rangle^{-3/2} ,
\end{equation}
for any $\beta$, and
\begin{equation}\label{bounds-DVWst}
 \begin{aligned}
\| \partial_x (V_{s,t} - v + V^{osc}_{t,t} )\|_{L^\infty_{x,v}} &\lesssim \epsilon \langle s\rangle^{-3/2} ,
\\
 \| \partial_v (V_{s,t} - v + V^{osc}_{t,t} )\|_{L^\infty_{x,v}} &\lesssim \epsilon \langle s\rangle^{-3/2} (t-s) ,
 \\
\|\partial_x (\hPsi_{s,t} - \hv + \hV^{osc}_{t,t})\|_{L^\infty_{x,v}} 
&\lesssim  \epsilon \langle s\rangle^{-1}(t-s)^{-1},
 \\
\| \partial_v (\hPsi_{s,t} - \hv + \hV^{osc}_{t,t})\|_{L^\infty_{x,v}} 
&\lesssim  \epsilon \langle s\rangle^{-1}  .
 \end{aligned}
 \end{equation}
In particular,
\begin{equation}\label{Jacobian-dxdv}
\begin{aligned}
|\det(\nabla_xX_{s,t}(x,v)) - 1| &\lesssim \epsilon , \qquad |\det(\nabla_vX_{s,t}(x,v))|& \gtrsim \langle v\rangle^{-5} |t-s|^3,
\end{aligned}\end{equation}
uniformly for all $x,v$. 

\end{proposition}


\begin{proof} By recalling that $V^{osc}_{t,t}(x,v) = E^{osc,1}_{\pm}(t,x,v)$, the estimate \eqref{Dxv-VWend} thus reads off from \eqref{decay-Eoscj} with $j=1$. 
Next, in order to prove the bounds in \eqref{bounds-DVWst}, we shall first derive similar bounds on $\hPsi_{s,t}$. Indeed, recalling the expansion from Proposition \ref{prop-charPsi}, we compute
\begin{equation}\label{comp-dxPsi}
\begin{aligned}
\partial_x^\alpha \partial_v^\beta [ \hPsi_{s,t}(x,v) - \hv + \hV^{osc}_{t,t}(x,v)]  &= \partial_x^\alpha \partial_v^\beta \hPsi^{osc}_{s,t}(x,v)  +\partial_x^\alpha \partial_v^\beta \hPsi^{tr}_{s,t}(x,v)+\partial_x^\alpha \partial_v^\beta \Psi^Q_{s,t}(x,v) .
\end{aligned}
\end{equation}

\subsubsection*{Bounds on $x$-derivatives.}

We first prove the bounds on the spatial derivatives $\partial_x$. Note that the characteristics are nonlinear, and we shall need to close the estimates for derivatives. To this end, we introduce 
 \begin{equation}\label{def-zetadxV} 
 \begin{aligned}
 \zeta(s,t) &:= \sup_{0\le s\le \tau\le t} (t-\tau) \langle \tau\rangle \| \langle v\rangle\partial_x [\hPsi_{\tau,t}- \hv + \hV^{osc}_{t,t}] \|_{L^\infty_{x,v}}.
\end{aligned}\end{equation}
We first bound the derivatives of the characteristics $X_{s,t}, V_{s,t}$ in terms of $\zeta(s,t)$. Indeed, using \eqref{straightX}, \eqref{Dxv-VWend}, and the definition of $\zeta(s,t)$ in \eqref{def-zetadxV}, we note that
\begin{equation}\label{bd-dxXst}
\begin{aligned}
|\partial_x X_{s,t}(x,v)| &\le 1+ (t-s)|\partial_x \hPsi_{s,t}(x,v)| 
\lesssim 1+  \zeta(s,t).
  \end{aligned}\end{equation}
In order to bound $\partial_x V$, we use the expansion from Proposition \ref{prop-charV}, namely  
\begin{equation}\label{expVosc1}
\partial_x  V_{s,t}(x,v) = -\partial_x V^{osc}_{t,t}(x,v) + \partial_x V^{osc}_{s,t}(x,v)  +\partial_xV^{tr}_{s,t}(x,v).
\end{equation}
By definition, together with \eqref{decay-Eoscj} and \eqref{bd-dxXst}, 
we bound 
$$
\begin{aligned}
|\partial_xV^{osc}_{s,t}(x,v)| &\le \sum_\pm |\partial_x X_{s,t} \partial_x E^{osc,1}_{\pm}(s)|+ \sum_\pm |\partial_x V_{s,t} \partial_v E^{osc,1}_{\pm}(s)| 
\\&\lesssim \epsilon \langle s\rangle^{-3/2} (1+  \zeta(s,t)) + \epsilon \langle s\rangle^{-3/2} \| \partial_x V_{s,t}\|_{L^\infty_{x,v}} 
.  \end{aligned}$$
 As for $\partial_xV^{tr}_{s,t}$, recalling \eqref{def-Qtr} and using again \eqref{decay-Eoscj}, we first bound  
\begin{equation}\label{bddxQtr}
\begin{aligned} 
\|\partial_x Q^{tr}(t)\|_{L^\infty_{x,v}}&\lesssim \|\partial_x[a_\pm(i\partial_x) F(t)\star_x \phi_{\pm,1}]\|_{L^\infty_{x,v}}
+ \|\partial_x(E\cdot  \nabla_x E^{osc,2}_\pm(t)) \|_{L^\infty_{x,v}} + \| \partial_x E^r(t)\|_{L^\infty_x}
\\
&\lesssim \|F(t)\|_{L^\infty_{x,v}} + \|E\|_{L^\infty_x} \|\partial_x E\|_{L^\infty_x}+ \| \partial_x E^r(t)\|_{L^\infty_x}.
\end{aligned} 
\end{equation}
Similar estimates hold for $\partial_v Q^{tr}(t)$. The bootstrap assumptions on $E$ and $F$ thus yield 
\begin{equation}\label{Dxv-Qtr}
\begin{aligned} 
\|\partial_x Q^{tr}(t)\|_{L^\infty_{x,v}} &\lesssim \epsilon \langle t\rangle^{-3}, \qquad \|\partial_v Q^{tr}(t)\|_{L^\infty_{x,v}} \lesssim \epsilon \langle t\rangle^{-3}.
\end{aligned} 
\end{equation}
Therefore, we bound 
$$
\begin{aligned}
|\partial_xV^{tr}_{s,t}(x,v)| &\le  \int_s^t  \Big(
|\partial_x X_{\tau,t} \partial_xQ^{tr}(\tau)| + |\partial_xV_{\tau,t} \partial_v Q^{tr}(\tau)| \Big)\; d\tau .
\\
&\lesssim \epsilon\int_s^t  \Big( \langle \tau\rangle^{-3} (1+  \zeta(\tau,t)) 
+ \langle \tau\rangle^{-3}\|\partial_xV_{\tau,t}\|_{L^\infty_{x,v}} \Big)\; d\tau 
\\
&\lesssim \epsilon\langle s\rangle^{-2} (1+  \zeta(s,t)) 
+ \epsilon \langle s\rangle^{-2}\sup_{s\le \tau\le t}\|\partial_xV_{\tau,t}\|_{L^\infty_{x,v}} ,
  \end{aligned}$$
Putting these into \eqref{expVosc1}, we obtain 
$$
\begin{aligned}
|\partial_xV_{s,t}(x,v)| &\lesssim \epsilon \langle s\rangle^{-3/2} (1+  \zeta(s,t)) + \epsilon \langle s\rangle^{-3/2} \sup_{s\le \tau \le t}\| \partial_x V_{\tau,t}\|_{L^\infty_{x,v}} .
  \end{aligned}$$
Taking the supremum in $x,v,s$, and letting $\epsilon$ be sufficiently small, we arrive at 
\begin{equation}\label{est-dxVst}
\begin{aligned}
\|\partial_xV_{s,t}\|_{L^\infty_{x,v}} &\lesssim \epsilon \langle s\rangle^{-3/2} (1+  \zeta(s,t)).
  \end{aligned}
  \end{equation}

We are now ready to bound the derivatives of $\hPsi_{s,t}$. Precisely, using the expansions in \eqref{comp-dxPsi}, we shall prove that 
\begin{equation}\label{claim-dxVPsi}
|\langle v\rangle\partial_x \hPsi^{osc}_{s,t}(x,v)| +|\langle v\rangle\partial_x \hPsi^{tr}_{s,t}(x,v)|+|\langle v\rangle\partial_x \Psi^Q_{s,t}(x,v)|\lesssim \epsilon (t-s)^{-1} \langle s\rangle^{-1}(1+\zeta(s,t)),
\end{equation}
which would imply that 
$\zeta(s,t) \lesssim \epsilon,$ 
upon taking $\epsilon$ sufficiently small. By definition (see Proposition \ref{prop-charPsi}), together with \eqref{bd-dxXst} and \eqref{est-dxVst}, we bound
\begin{equation}\label{temp-Posc1}
\begin{aligned}
|\partial_x \hPsi^{osc}_{s,t}(x,v)| &\le \Big| \frac{\nabla_v \hv}{t-s}\sum_\pm \Big(\partial_x E^{osc,2}_{\pm}(t) - \partial_x X_{s,t}\cdot \nabla_x E^{osc,2}_{\pm} (s) -  \partial_x V_{s,t}\cdot \nabla_v E^{osc,2}_{\pm}(s) \Big)\Big|
\\&\lesssim \epsilon \langle v\rangle^{-1} (t-s)^{-1}\langle s\rangle^{-3/2} ( 1 + \zeta(s,t) ),
\end{aligned}\end{equation}
which proves the claim \eqref{claim-dxVPsi} for this term. Next, using again the estimates \eqref{bd-dxXst} and \eqref{est-dxVst}, we bound  
$$
\begin{aligned}
(t-s)|\langle v\rangle\partial_x \hPsi^{tr}_{s,t}(x,v)|& \le \int_s^t  
\Big( (\tau-s) (
|\partial_x Q^{tr}| 
+|\partial_v Q^{tr}|) + |\partial_xQ^{tr,1}|  + |\partial_vQ^{tr,1}| \Big) ( 1 + \zeta(\tau,t))\; d\tau.
\end{aligned}
$$
In view of \eqref{def-Qtr} and \eqref{def-Qtr1}, the quadratic terms $Q^{tr}(t)$ and $Q^{tr,1}(t)$ satisfy the same bounds as in \eqref{Dxv-Qtr}. 
Hence, we obtain 
\begin{equation}\label{temp-Posc2}
\begin{aligned}
|(t-s)\langle v\rangle\partial_x \hPsi^{tr}_{s,t}(x,v)|
&\lesssim \int_s^t \langle \tau\rangle^{-2}( 1 + \zeta(\tau,t))\; d\tau 
\\& \lesssim \epsilon \langle s\rangle^{-1} ( 1 + \zeta(s,t)),
\end{aligned}\end{equation}
proving the desired bound \eqref{claim-dxVPsi} for this term. 
On the other hand, recalling \eqref{def-VQ} and using \eqref{bounds-Vst} and \eqref{est-dxVst}, we bound 
$$ \begin{aligned} \langle v\rangle|\partial_x V^Q_{\tau,t}(x,v)| & \lesssim |V_{\tau,t}(x,v)-v| |\partial_x V_{\tau,t}(x,v)| + |V_{\tau,t}(x,v)-v|^2
\\& 
\lesssim \epsilon^2 \langle \tau\rangle^{-3 } (1 + \zeta(\tau,t)) .
\end{aligned}$$
Therefore, 
\begin{equation}\label{temp-Posc3}
\begin{aligned}
|(t-s)\langle v\rangle\partial_x \Psi^Q_{s,t}(x,v)|& \le \int_s^t \langle v\rangle|\partial_x V^Q_{\tau,t}(x,v)| \; d\tau 
\\& 
\lesssim \epsilon^2 \langle s\rangle^{-2} (1 + \zeta(s,t))
\end{aligned}
\end{equation}
giving the desired bound \eqref{claim-dxVPsi} on $\partial_x \Psi^Q_{s,t}(x,v)$.  This proves the claim \eqref{claim-dxVPsi} and therefore $\zeta(s,t) \lesssim \epsilon$ for all $0\le s\le t$. 
{ In particular, putting \eqref{temp-Posc1}, \eqref{temp-Posc2}, and \eqref{temp-Posc3} into the expansion \eqref{comp-dxPsi}, we obtain 
\begin{equation}\label{est-dxVosc}
\begin{aligned}
\| \langle v\rangle \partial^\alpha_x( \hPsi_{s,t}(x,v) - \hv + \hV^{osc}_{t,t}(x,v))\|_{L^\infty_{x,v}} 
\lesssim  \epsilon \langle s\rangle^{-3/2} (t-s)^{-1}+ \epsilon \langle s\rangle^{-1} (t-s)^{-1},
\end{aligned}
\end{equation}
for $|\alpha|\le 1$. Note that the bounds for $\alpha=0$ follow from the results obtained in Proposition \ref{prop-charPsi}. 

We now prove the first two estimates in \eqref{bounds-DVWst}. Using \eqref{def-WR}, we have 
$$
\partial_x (\Psi_{s,t} - v  +  V^{osc}_{t,t}) = \nabla_{\hv} v \partial_x (\hPsi_{s,t} - \hv +  \hV^{osc}_{t,t}) + \cO(\langle v\rangle^5|\hPsi_{s,t} - \hv| |\partial_x (\hPsi_{s,t} - \hv)|).
$$
The desired estimates on the first term follow from \eqref{est-dxVosc} and the fact that $|\nabla_{\hv} v| \le \langle v\rangle^3$. On the other hand, using again \eqref{est-dxVosc} and \eqref{Dxv-VWend}, we bound 
$$\| \langle v\rangle^2|\hPsi_{s,t} - \hv| |\partial_x (\hPsi_{s,t} - \hv)|\|_{L^\infty_{x,t}} \lesssim \Big[ \epsilon \langle s\rangle^{-3/2} (t-s)^{-1}+ \epsilon \langle s\rangle^{-1} \langle t\rangle^{-1}+ \epsilon \langle t\rangle^{-3/2}\Big]^2.$$
This proves the second estimate stated in \eqref{bounds-DVWst}. 

}

\subsubsection*{Bounds on $v$-derivatives.}

Next, we give bounds on the $v$-derivatives. Similarly as done above, we first introduce 
 \begin{equation}\label{def-zetadvV} 
 \begin{aligned}
 \zeta_1(s,t) &:= \sup_{0\le s\le \tau\le t} \langle \tau\rangle \| \langle v\rangle\partial_v [\hPsi_{\tau,t}- \hv + \hV^{osc}_{t,t}] \|_{L^\infty_{x,v}},
\end{aligned}\end{equation}
and proceed to bound the last two terms on the right of  \eqref{comp-dxPsi}. Using \eqref{straightX}, \eqref{Dxv-VWend}, and the definition of $\zeta_1(s,t)$, we note that
\begin{equation}\label{bd-dvXst}
\begin{aligned}
 |\partial_vX_{s,t}(x,v)| &\le (t-s)|\partial_v \hPsi_{s,t}(x,v)| 
 \lesssim (t-s)(1 + \zeta_1(s,t)).
  \end{aligned}\end{equation}
In order to bound $\partial_v V$, we again use the expansion from Proposition \ref{prop-charV}, namely  
\begin{equation}\label{expVosc2}
\partial_v  V_{s,t}(x,v) = 1-\partial_v V^{osc}_{t,t}(x,v) + \partial_v V^{osc}_{s,t}(x,v)  +\partial_vV^{tr}_{s,t}(x,v).
\end{equation}
By definition, together with \eqref{bd-dvXst}, 
we bound 
$$
\begin{aligned}
|\partial_vV^{osc}_{s,t}(x,v)| &\le \sum_\pm |\partial_v X_{s,t} \partial_x E^{osc,1}_{\pm}(s)|+ \sum_\pm |\partial_v V_{s,t} \partial_v E^{osc,1}_{\pm}(s)| 
\\&\lesssim \epsilon \langle s\rangle^{-3/2} (t-s)(1+  \zeta_1(s,t)) + \epsilon \langle s\rangle^{-3/2} \| \partial_v V_{s,t}\|_{L^\infty_{x,v}} 
.  \end{aligned}$$
 As for $\partial_vV^{tr}_{s,t}$, recalling \eqref{Dxv-Qtr}, we bound 
$$
\begin{aligned}
|\partial_vV^{tr}_{s,t}(x,v)| &\le  \int_s^t  \Big(
|\partial_v X_{\tau,t} \partial_xQ^{tr}(\tau)| + |\partial_vV_{\tau,t} \partial_v Q^{tr}(\tau)| \Big)\; d\tau .
\\
&\lesssim \epsilon\int_s^t  \Big( \langle \tau\rangle^{-3} (t-\tau)(1+  \zeta_1(\tau,t)) 
+ \langle \tau\rangle^{-3}\|\partial_vV_{\tau,t}\|_{L^\infty_{x,v}} \Big)\; d\tau 
\\
&\lesssim \epsilon\langle s\rangle^{-2} (t-s) (1+  \zeta_1(s,t)) 
+ \epsilon \langle s\rangle^{-2}\sup_{s\le \tau\le t}\|\partial_vV_{\tau,t}\|_{L^\infty_{x,v}} ,
  \end{aligned}$$
Putting these into \eqref{expVosc2}, we have obtained  
$$
\begin{aligned}
|\partial_vV_{s,t}(x,v)| &\lesssim 1 + \epsilon \langle s\rangle^{-3/2} (t-s)(1+  \zeta_1(s,t)) + \epsilon \langle s\rangle^{-3/2} \sup_{s\le \tau \le t}\| \partial_v V_{\tau,t}\|_{L^\infty_{x,v}} .
  \end{aligned}$$
Taking the supremum in $x,v,s$, and letting $\epsilon$ be sufficiently small, we arrive at 
\begin{equation}\label{est-dvVst}
\begin{aligned}
\|\partial_vV_{s,t}\|_{L^\infty_{x,v}} &\lesssim 1+\epsilon \langle s\rangle^{-3/2} (t-s) (1+  \zeta_1(s,t)).
  \end{aligned}
  \end{equation}

Next, by definition, together with \eqref{bd-dvXst} and \eqref{est-dvVst}, we bound
$$
\begin{aligned}
|\partial_v \hPsi^{osc}_{s,t}(x,v)| &\le 
\frac{|\partial_v\nabla_v \hv|}{t-s}\sum_\pm \Big| E^{osc,2}_{\pm}(t,x,v) - E^{osc,2}_{\pm} (s,X_{s,t}(x,v),V_{s,t}(x,v)) \Big|
\\&\quad +\frac{|\nabla_v \hv|}{t-s}\sum_\pm \Big| \partial_v E^{osc,2}_{\pm}(t) - \partial_v X_{s,t}\cdot \nabla_x E^{osc,2}_{\pm} (s) -  \partial_v V_{s,t}\cdot \nabla_v E^{osc,2}_{\pm}(s) \Big|
\\&\lesssim  \epsilon \langle v\rangle^{-1}\langle s\rangle^{-3/2} ( 1 + \zeta_1(s,t)),
\end{aligned}$$
in which we note that there is no singularity at $s=t$, at which the right hand side vanishes.  This proves the desired bounds for this term. Finally, using \eqref{bd-dvXst} and the fact that $|t-\tau|\le |t-s|$ for $0\le s\le \tau \le t$, we compute 
$$
\begin{aligned}
|\langle v\rangle\partial_v \hPsi^{tr}_{s,t}(x,v)|& \le \int_s^t  
\Big( (\tau-s) (
|\partial_x Q^{tr}| 
+|\partial_v Q^{tr}|) + |\partial_xQ^{tr,1}|  + |\partial_vQ^{tr,1}| \Big) ( 1 + \zeta_1(\tau,t))\; d\tau 
\\&\quad + \frac{\langle v\rangle|\partial_v\nabla_v\hv |}{t-s} \int_s^t  
[ (\tau-s)  
|Q^{tr}| + |Q^{tr,1}|] \; d\tau  .
\end{aligned}
$$
The first integral is bounded by $C_0\epsilon \langle s\rangle^{-1} ( 1 + \zeta_1(s,t))$ as done above, while the second integral is bounded by $C_0\epsilon \langle s\rangle^{-2}$. 
On the other hand, recalling \eqref{def-VQ} and using \eqref{bd-dvXst}, we bound 
$$
\begin{aligned} 
\langle v\rangle|\partial_v V^Q_{\tau,t}(x,v)| & \lesssim |V_{\tau,t}(x,v)-v| |\partial_v V_{\tau,t}(x,v)| + |V_{\tau,t}(x,v)-v|^2
\\
&  \lesssim  \epsilon \langle \tau\rangle^{-3/2}  + \epsilon^2 \langle \tau\rangle^{-2} (t-\tau) (1 + \zeta_1(\tau,t)) ,
\end{aligned}
$$
which gives 
$$ |\langle v\rangle\partial_v \Psi^Q_{s,t}(x,v)| \lesssim  \frac{1}{t-s} \int_s^t \langle v\rangle|\partial_v V^Q_{\tau,t}(x,v)| \; d\tau \lesssim  \epsilon \langle s\rangle^{-1} (1 + \zeta_1(s,t)) .$$
{
This completes the proof of the bounds on $\partial_v \hPsi_{s,t}(x,v)$, namely  
\begin{equation}\label{est-dvVosc}
\| \langle v\rangle\partial_v (\hPsi_{s,t} - \hv + \hV^{osc}_{t,t})\|_{L^\infty_{x,v}} 
\lesssim  \epsilon \langle s\rangle^{-1} ,\end{equation}
upon taking $\epsilon$ sufficiently small. To prove the last estimate in \eqref{bounds-DVWst}, we use \eqref{def-WR} and write
$$
\begin{aligned}
\partial_v (\Psi_{s,t} - v  +  V^{osc}_{t,t}) &=  (\partial_v\nabla_{\hv} v) (\hPsi_{s,t} - \hv +  \hV^{osc}_{t,t})  + \nabla_{\hv} v \partial_v (\hPsi_{s,t} - \hv +  \hV^{osc}_{t,t}) 
\\&\quad  + \cO(\langle v\rangle^4|\hPsi_{s,t} - \hv|^2) + \cO(\langle v\rangle^5|\hPsi_{s,t} - \hv| |\partial_v (\hPsi_{s,t} - \hv)|).
\end{aligned}
$$
Using \eqref{est-dxVosc} and \eqref{est-dvVosc}, we thus obtain the last estimate stated in \eqref{bounds-DVWst}. Finally, as for the third estimate on $\partial_v V_{s,t}$ stated in \eqref{bounds-DVWst}, we simply take the derivative of \eqref{decomp-Vst} and using \eqref{bd-dvXst}, with $\zeta_1(s,t) \lesssim \epsilon$. The bounds \eqref{bounds-DVWst} thus follow. 
}

Finally, the estimates \eqref{Jacobian-dxdv} follow directly from \eqref{straightX} and the bounds \eqref{Dxv-VWend}-\eqref{bounds-DVWst}. This completes the proof of the proposition. 
\end{proof}

\subsection{Jacobian description}

In the nonlinear analysis, the geometry of the characteristic plays an important role. Specifically, to make use of dispersion of the transport equation, we need to analyze oscillations in the Jacobian determinant of the map $v\mapsto X_{s,t}(x,v)$, which is the content of the following proposition. 

\begin{proposition}\label{prop-detPsi} 
Let $\hPsi_{s,t}(x,v)$ be the averaging velocity as in \eqref{def-hatPsi}. Then, there hold
 \begin{equation}\label{decomp-detPsi}
 \begin{aligned}
\det (\nabla_v \Psi_{s,t}(x,v) ) &= 1 + \Psi^{Tosc,1}_{s,t} + \Psi^{osc,1}_{s,t}+ \Psi^{tr,1}_{s,t}+ \Psi^{Q,1}_{s,t}(x,v)+ \Psi^{R,1}_{s,t}
 \end{aligned}
 \end{equation}
where
$$ 
\begin{aligned}
\Psi^{Tosc,1}_{s,t}(x,v) &= - \sum_\pm  \Big[ \nabla_v \cdot E^{osc,1}_{\pm} +  \frac{1}{t-s} \nabla_v \cdot E^{osc,2}_{\pm}\Big] (t,x,v) ,
\\
\Psi^{osc,1}_{s,t}(x,v) &= -  \sum_\pm  \Big[ \nabla_v \cdot E^{osc,1}_{\pm} + \frac{1}{t-s} \nabla_v \cdot E^{osc,2}_{\pm}\Big]  (s,X_{s,t}(x,v),V_{s,t}(x,v)) ,
\\
\Psi^{tr,1}_{s,t}(x,v) 
& = \int_s^t  
\frac{(t-\tau) (\tau-s)}{t-s}   \nabla_v\hv\cdot 
\nabla_xQ^{tr}(\tau,X_{\tau,t}(x,v), V_{\tau,t}(x,v)) \; d\tau ,
\\
\Psi^{Q,1}_{s,t}(x,v) &= c_0\Psi^{osc,1}_{s,t}(x,v) \Psi^{tr,1}_{s,t}(x,v),
\end{aligned}
$$
for some constant $c_0$. In addition, there hold 
\begin{equation}
\label{est-detJ}
\begin{aligned}
\| \Psi^{Tosc,1}_{s,t}\|_{L^\infty_{x,v}}  \lesssim \epsilon \langle t\rangle^{-3/2}, \qquad \| \Psi^{osc,1}_{s,t}\|_{L^\infty_{x,v}}  &\lesssim \epsilon \langle s\rangle^{-3/2}, \qquad 
\\
\| \Psi^{Q,1}_{s,t}\|_{L^\infty_{x,v}}  \lesssim  \epsilon\langle s\rangle^{-5/2} , \qquad \| \Psi^{tr,1}_{s,t}\|_{L^\infty_{x,v}} & \lesssim \epsilon \langle s\rangle^{-1},
\\
 \| \langle v\rangle^{-3}\Psi^{R,1}_{s,t}\|_{L^\infty_{x,v}} + \| \partial_s\Psi^{tr,1}_{s,t}\|_{L^\infty_{x,v}} & \lesssim \epsilon \langle s\rangle^{-2}.
\end{aligned}
\end{equation}

\end{proposition}

The expansion \eqref{decomp-detPsi} extracts the oscillation of order $s^{-3/2}$ and the transport component $\Psi^{tr,1}_{s,t}$ which has slowest decay at order $s^{-1}$. Note importantly that unlike the oscillation component, the $s$-derivatives of the transport part gain extra decay of order $s^{-1}$. Next, $\Psi^{Q,1}_{s,t}(x,v)$ collects one specific quadratic term, which decays at order $s^{-5/2}$, whose $s$-derivatives do not gain extra decay. The remainder decays at order $s^{-2}$ or faster, but importantly, the $s$-derivative decays at order $s^{-3}$. The structure of $\det (\nabla_v \Psi_{s,t})$ dictates that of the nonlinear interaction, see Section \ref{sec-NLstructure}.

\begin{proof}[Proof of Proposition \ref{prop-detPsi}]
We first note that 
\begin{equation}\label{det-1st-exp} 
\det (\nabla_v \Psi_{s,t}) - 1 = \nabla_v \cdot (\Psi_{s,t}(x,v) - v) + \cO([\partial_v(\Psi_{s,t}(x,v) - v)]^2) .
\end{equation}
Using Proposition \ref{prop-Dchar}, the last term is bounded by $\epsilon^2\langle s\rangle^{-5/2}$, which contributes into the remainder term $\Psi^{R,1}_{s,t}$. On the other hand, recalling Proposition \ref{prop-charPsi},
we compute 
\begin{equation}\label{decomp-navPsist}
\begin{aligned}
\partial_v [ \Psi_{s,t} - v]  = -\partial_v V^{osc}_{t,t} + \partial_v \Psi^{osc}_{s,t}  +\partial_v \Psi^{tr}_{s,t}+\partial_v \Psi^R_{s,t}
\end{aligned}
 \end{equation}
in which we recall that $\partial_vV^{osc}_{t,t} = \sum_\pm \partial_v E^{osc,1}_{\pm}(t,x,v) $, which contributes into $\Psi^{Tosc,1}_{s,t}$. Next, by definition, we compute 
\begin{equation}\label{cal-dvWosc}
\partial_v\Psi^{osc}_{s,t}(x,v) = \frac{1}{t-s}\sum_\pm \Big(\partial_vE^{osc,2}_{\pm}(t) - \partial_x E^{osc,2}_{\pm}(s) \partial_vX_{s,t} - \partial_v E^{osc,2}(s) \partial_v V_{s,t}\Big).\end{equation}
The first term contributes into $\Psi^{Tosc,1}_{s,t}$. On the other hand, using \eqref{bounds-DVWst}, we note $\partial_v V_{s,t} = 1 + \cO(\epsilon\langle s\rangle^{-3/2}(t-s))$, and so the last term in the above is equal to 
$$ - \frac{1}{t-s}\sum_\pm \partial_v E^{osc,2}(s,X_{s,t},V_{s,t})  +  \cO(\epsilon^2 \langle s\rangle^{-3})$$
which contributes into $\Psi^{osc,1}_{s,t}(x,v) $ and $\Psi^{R,1}_{s,t}(x,v) $, respectively. Let us now study the second term in \eqref{cal-dvWosc}. We note that 
\begin{equation}\label{cal-dvX}
\partial_v X_{s,t}(x,v) = -(t-s) \partial_v \hv - (t-s) \partial_v (\hPsi_{s,t}(x,v) - \hv).
\end{equation}
The first term in the above expression leads to the first term in the definition of $\Psi^{osc,1}_{s,t}(x,v)$, upon noting that $\partial_v \hv \cdot \nabla_x E^{osc,2}_{\pm}(s) = - \partial_v E^{osc,1}_{\pm}(s) $. Finally, plugging the last term into the second term in \eqref{cal-dvWosc} gives 
$$
\begin{aligned}
\partial_x E^{osc,2}_{\pm}(s, X_{s,t}, V_{s,t})\partial_v (\hPsi_{s,t}(x,v) - \hv)& = \partial_x E^{osc,2}_{\pm}(s, X_{s,t}, V_{s,t}) \partial_v \hv \Psi_{s,t}^{tr,1}(x,v) + \cO(\epsilon^2 \langle s\rangle^{-3})
\\
&=-  \partial_v E^{osc,1}_{\pm}(s, X_{s,t}, V_{s,t}) \Psi_{s,t}^{tr,1}(x,v) + \cO(\epsilon^2 \langle s\rangle^{-3})
\end{aligned}$$
whose first term contributes into $ \Psi^{Q,1}_{s,t}(x,v)$, while the last term is put into the remainder $\Psi^{R,1}_{s,t}$. 

Next, we focus on $\partial_v \Psi^{tr}_{s,t}$ in \eqref{decomp-navPsist}. In view of the definition of $\Psi^{tr}_{s,t}(x,v)$ from Proposition \ref{prop-charPsi}, we compute 
$$
\begin{aligned}
\partial_v \Psi^{tr}_{s,t}(x,v) 
& = - \frac{1}{t-s} \int_s^t  
[ (\tau-s)  
\partial_xQ^{tr} \partial_v X_{\tau,t}  +\partial_xQ^{tr,1} \partial_v X_{\tau,t}  ](\tau,X_{\tau,t}(x,v), V_{\tau,t}(x,v)) \; d\tau 
\\
& \quad - \frac{1}{t-s} \int_s^t  
[ (\tau-s)  
\partial_vQ^{tr} \partial_v V_{\tau,t} + 
\partial_vQ^{tr,1} \partial_v V_{\tau,t} ](\tau,X_{\tau,t}(x,v), V_{\tau,t}(x,v)) \; d\tau .
\end{aligned}
$$
We first recall from \eqref{bounds-DVWst} that $\|\partial_v V_{s,t}\|_{L^\infty_{x,v}} \lesssim 1 + \epsilon \langle s\rangle^{-3/2} (t-s)$ and from \eqref{Dxv-Qtr} that the derivatives of $Q^{tr}(t)$ and $Q^{tr,1}(t)$ satisfy 
\begin{equation}\label{Dxv-reQtr}
\begin{aligned} 
\|\partial_x Q^{tr,j}(t)\|_{L^\infty_{x,v}} &\lesssim \epsilon \langle t\rangle^{-3}, \qquad \|\partial_v Q^{tr,j}(t)\|_{L^\infty_{x,v}} \lesssim \epsilon \langle t\rangle^{-3}.
\end{aligned} 
\end{equation}
Therefore, the above integrals involving $\partial_v V_{\tau,t}$ are clearly bounded by $C_0 \epsilon \langle s\rangle^{-2}$. Importantly, their $s$-derivatives are bounded by bounded by $C_0 \epsilon \langle s\rangle^{-3}$, and therefore these integrals can be put into the remainder $\Psi^{R,1}_{s,t}$. 
As for the first integral involving $\partial_v X_{\tau,t}$, we write 
$$ 
\partial_v X_{s,t}(x,v) = -(t-s) \partial_v \hv - (t-s) \partial_v (\hPsi_{s,t}(x,v) - \hv)$$
in which the second term is bounded by $C_0 \epsilon (t-s)\langle s\rangle^{-1}$ in view of \eqref{bounds-DVWst}, leading to a term of order $\epsilon \langle s\rangle^{-5/2}$ in $\partial_v \Psi^{tr}_{s,t}(x,v) $ and thus can be put into $\Psi^{R,1}_{s,t}$. That is, if we define $\Psi^{tr,1}$ as in the proposition, 
we have 
$$ \| \partial_v \Psi^{tr}_{s,t} - \Psi^{tr,1}_{s,t}\|_{L^\infty_{x,v}} \lesssim \epsilon \langle s\rangle^{-2}, \qquad \| \Psi^{tr,1}_{s,t}\|_{L^\infty_{x,v}} \lesssim \epsilon \langle s\rangle^{-1}.$$ 

Finally, we study the quadratic term in \eqref{det-1st-exp}, which is of order $\epsilon^2\langle s\rangle^{-5/2}$. In view of the expansion \eqref{decomp-navPsist}, we note that we may put all the terms into the remainder, except the leading term in the product  
$$ \partial_v \Psi^{osc}_{s,t}  \times \partial_v \Psi^{tr}_{s,t}$$ 
which may not gain extra decay after taking $s$-derivatives. This defines $\Psi^Q_{s,t}$. 
\end{proof}

\subsection{$L^2$ bounds on characteristics}

In this section, we bound the characteristic in $L^p$ spaces, which are useful in the nonlinear analysis.

\begin{proposition}\label{prop-Lpchar} 
For $|\alpha|\le 1$, there hold
\begin{equation}\label{Lpbounds-VWst}
 \begin{aligned}
\sup_v\|\partial_x^\alpha (V_{s,t} - v)\|_{L^2_x} \lesssim \epsilon
, \qquad &\sup_x  \| \langle v\rangle^{-5/2}\partial_x^\alpha(V_{s,t} - v)\|_{L^2_v} \lesssim \epsilon \langle t\rangle^{-3/2} ,
\\
\sup_v\| \langle v\rangle^{-3/2} \partial_x^\alpha (\Psi_{s,t} - v)\|_{L^2_x}
\lesssim  \epsilon ,
\qquad &\sup_x\| \langle v\rangle^{-4} \partial_x^\alpha(\Psi_{s,t} - v)\|_{L^2_v}
\lesssim  \epsilon \langle t\rangle^{-3/2} .
 \end{aligned}
 \end{equation}
In addition, 
\begin{equation}\label{Lpbounds-detPsi}
\begin{aligned}
\sup_v\| \langle v\rangle^{-3/2}(\det (\nabla_v \Psi_{s,t}) - 1)^2 \|_{L^2_x}  & \lesssim \epsilon \langle s\rangle^{-1/2 } , 
\\
 \sup_x \| \langle v\rangle^{-4}(\det (\nabla_v \Psi_{s,t} ) -1)^2 \|_{L^2_v}  & \lesssim \epsilon \langle t\rangle^{-3/2}\langle s\rangle^{-1/2}.
 \end{aligned}
 \end{equation}
\end{proposition}

\begin{proof} We first check the bounds in the case when $|\alpha|=0$ and $|\beta|=0$. Recall from Proposition \ref{prop-charV} that 
$$
\begin{aligned}
  V^{osc}_{s,t}(x,v) &= \sum_\pm E^{osc,1}_{\pm}(s,X_{s,t}(x,v),V_{s,t}(x,v)) ,
 \\
V^{tr}_{s,t}(x,v)& = -  \int_s^t  
Q^{tr}(\tau,X_{\tau,t}(x,v), V_{\tau,t}(x,v)) \; d\tau .
 \end{aligned}$$
We shall prove that 
\begin{equation}\label{Lpbounds-Vosc}
\begin{aligned}
\sup_v\| \partial_x^\beta V^{osc}_{s,t} \|_{L^2_x}    \lesssim \epsilon , \qquad  &\sup_x \| \langle v\rangle^{-5/2}\partial_x^\beta V^{osc}_{s,t} \|_{L^2_v}   \lesssim \epsilon \langle t\rangle^{-3/2}
\\
\sup_v\| \partial_x^\beta V^{tr}_{s,t} \|_{L^2_x}   \lesssim \epsilon \langle s\rangle^{-1/2}  , \qquad  &\sup_x \| \langle v\rangle^{-5/2}  \partial_x^\beta V^{tr}_{s,t} \|_{L^2_v}   \lesssim \epsilon \langle t\rangle^{-3/2} \langle s\rangle^{-1/2},
\end{aligned}
\end{equation}
for $|\beta|\le 1$. The estimates \eqref{Lpbounds-VWst} on $V_{s,t}$ follow at once from the above bounds and the expansions established in Proposition \ref{prop-charV} and Proposition \ref{prop-charPsi}. To prove \eqref{Lpbounds-Vosc}, 
using \eqref{decay-Eoscj} and recalling \eqref{def-Qtr}, we first note that 
\begin{equation}\label{LpbdQtr}\|E^{osc,1}_\pm(t)\|_{L^p_xL^\infty_v} \lesssim \epsilon \langle t\rangle^{-3(1/2-1/p)}, \qquad 
\|Q^{tr}(t)\|_{L^p_xL^\infty_v} \lesssim \epsilon \langle t\rangle^{-3(1-1/p)} .\end{equation}
The estimates \eqref{Lpbounds-Vosc} in $L^2_x$ for $V_{s,t}(x,v)$ thus follow, upon introducing the change of variables $y = X_{s,t}(x,v)$ and using \eqref{Jacobian-dxdv} to deduce that $|\det \nabla_x X_{s,t}(x,v)|\ge 1/2$ for all $x,v$. Next, for the $L^2_v$ estimates, we first consider the case where $s\ge t/2$, for which we bound 
$$ \| \langle v\rangle^{-5/2}V^{osc}_{s,t}(x,\cdot)\|_{L^2_v} \lesssim \| V^{osc}_{s,t}\|_{L^\infty_{x,v}} \lesssim \epsilon \langle t\rangle^{-3/2},$$
using $s\ge t/2$. On the other hand, for $0\le s\le t/2$, introducing the change of variables $y = X_{s,t}(x,v)$, we first bound 
$$
\begin{aligned} 
\| \langle v\rangle^{-5/2}V^{osc}_{s,t}\|_{L^2_v}^2 
&\le \int |E^{osc,1}_{\pm}(s,X_{s,t}(x,v),V_{s,t}(x,v))|^2 \langle v\rangle^{-5} \; dv
\\
&\le \int |E^{osc,1}_{\pm}(s,y,V_{s,t}(x,v))|^2 \langle v\rangle^{-5} \; \frac{dy}{|\det (\nabla_v X_{s,t}(x,v))|}
\end{aligned}$$
in which $v = \widetilde X_{s,t}(x,y)$, the inverse map of $v\mapsto y = X_{s,t}(x,v)$. Using \eqref{Jacobian-dxdv} and \eqref{decay-Eoscj}, we thus obtain 
$$
\begin{aligned}
\| \langle v\rangle^{-5/2}V^{osc}_{s,t}\|_{L^2_v}^2 
&\lesssim (t-s)^{-3}  \int_{\RR^3} |E^{osc,1}_{\pm}(s,y,V_{s,t}(x,v))|^2\; dy
\lesssim t^{-3} ,
\end{aligned}
$$
since $s\in (0,t/2)$. Similarly, we bound 
$$
\begin{aligned}
\| \langle v\rangle^{-5/2}V^{tr}_{s,t}\|_{L^2_v}
&\lesssim  \int_s^t  \| 
Q^{tr}(\tau,X_{\tau,t}(x,v), V_{\tau,t}(x,v)) \langle v\rangle^{-5/2}\|_{L^2_v} \; d\tau 
\\
&\lesssim  \int_s^{t/2} \langle t\rangle^{-3/2} \| 
Q^{tr}(\tau)\|_{L^2_y L^\infty_v} \; d\tau +  \int_{t/2}^t \| 
Q^{tr}(\tau)\|_{L^\infty_{x,v}} \; d\tau 
\end{aligned}$$
in which we introduced the change of variables $y = X_{\tau,t}(x,v)$ in the first time integral. Now using \eqref{LpbdQtr} for $p=2$ and $p=\infty$, we obtain  \eqref{Lpbounds-Vosc} as claimed. The $x$-derivative estimates follow similarly. 

Next, the estimates on $\Psi_{s,t}(x,v)$ follow similarly as done above, using the representation \eqref{decomp-Psist}, the bounds \eqref{Lpbounds-Vosc} on $V_{t,t}^{osc}$, and the following direct bounds 
\begin{equation}\label{Lpbounds-Wosc}
\begin{aligned}
\sup_v\| \partial_x^\beta \Psi^{osc}_{s,t} \|_{L^2_x}    \lesssim \epsilon (t-s)^{-1}, \qquad  &\sup_x \| \langle v\rangle^{-5/2}\partial_x^\beta \Psi^{osc}_{s,t} \|_{L^2_v}   \lesssim \epsilon \langle t\rangle^{-3/2}(t-s)^{-1}
\\
\sup_v\|\partial_x^\beta \Psi^{tr}_{s,t} \|_{L^2_x}   \lesssim \epsilon \langle t\rangle^{-1/2} , \qquad  &\sup_x \| \langle v\rangle^{-5/2} \partial_x^\beta \Psi^{tr}_{s,t} \|_{L^2_v}   \lesssim \epsilon \langle t\rangle^{-2}
\\
\sup_v\|\langle v\rangle^{-3/2}\partial_x^\beta \Psi^{R}_{s,t} \|_{L^2_x}   \lesssim \epsilon \langle t\rangle^{-1}\langle s\rangle^{-1/2} , \qquad  &\sup_x \| \langle v\rangle^{-4} \partial_x^\beta \Psi^{R}_{s,t} \|_{L^2_v}   \lesssim \epsilon \langle t\rangle^{-5/2}\langle s\rangle^{-1/2},
\end{aligned}
\end{equation}
for $|\beta|\le 1$. Similarly, we also obtain the bounds for $\det (\nabla_v \Psi_{s,t})$, upon using the representation \eqref{decomp-detPsi}. The proposition thus follows. 
\end{proof}


\subsection{Decay for higher derivatives}\label{sec-Hdxchar}


In this section, we establish the decay and boundedness of higher derivatives of the characteristics. We prove the following proposition. 

\begin{proposition}\label{prop-HDchar} 
Fix $N_0 \ge 4$. Let $X_{s,t}(x,v), V_{s,t}(x,v)$ be the nonlinear characteristic solving \eqref{ode-char}. Then, for $0\le s\le t$ and for $|\alpha_1|+|\alpha_2|=|\alpha|$ with $2\le |\alpha|\le N_0-1$, there hold 
\begin{equation}\label{decayXV}
 \begin{aligned}
\| \partial^{\alpha_1}_x\partial^{\alpha_2}_v(V_{s,t}-v)\|_{L^\infty_{x,v}}  
&\lesssim \epsilon \langle s\rangle^{-\frac32 +\frac12\delta_\alpha} \langle t\rangle^{\delta_\alpha+|\alpha_2|}, 
\\
\| \partial^{\alpha_1}_x\partial^{\alpha_2}_v(X_{s,t}-x+\hv (t-s))\|_{L^\infty_{x,v}}  
&\lesssim \epsilon \langle s\rangle^{-\frac12 +\frac12\delta_\alpha} \langle t\rangle^{\delta_\alpha+|\alpha_2|
},
 \end{aligned}
 \end{equation}
with $\delta_\alpha = \frac{|\alpha|}{N_0-1}$. 

\end{proposition}

We shall prove Proposition \ref{prop-HDchar} by continuous induction, starting with the bounds on $x$-derivatives, namely $|\alpha_2|=0$. Indeed, we introduce the iterative bootstrap function
\begin{equation}\label{lownorm}
 \begin{aligned}
\zeta(s,t) &= \sup_{s\le \tau \le t} \sum_{|\alpha| < N_0}\Big[\langle s\rangle^{\frac32-\frac12\delta_\alpha} \langle t\rangle^{-\delta_\alpha}\| \partial^\alpha_x (V_{s,t}-v+V^{osc}_{t,t})\|_{L^\infty_{x,v}} 
\\&\qquad + \langle s\rangle^{1-\frac12\delta_\alpha} (t-s) \langle t\rangle^{-\delta_\alpha} \| \partial^\alpha_x(\hPsi_{s,t}-\hv+V^{osc}_{t,t})\|_{L^\infty_{x,v}} \Big]
 \end{aligned}
\end{equation}
with $\delta_\alpha = \frac{|\alpha|}{N_0-1}$. Specifically, we shall prove that 
\begin{equation}\label{claimXV1}
\zeta(s,t) \le C_0\epsilon + C_0\epsilon \zeta(s,t)
\end{equation}
for some universal constant $C_0$, which would then yield $\zeta(s,t) \le 2C_0 \epsilon$, upon taking $\epsilon$ sufficiently small. The boundedness and decay of the characteristics then follow from that of $\zeta(s,t)$. 

\subsubsection*{Consequences of the bootstrap function.}
We first recall that $V^{osc}_{t,t}(x,v)= E^{osc,1}_{\pm}(t,x,v)$, and therefore by the bootstrap assumption on $E$ from Section \ref{sec-bootstrap}, see \eqref{bootstrap-decaydaS} and \eqref{bootstrap-Hs}, 
\begin{equation}\label{bdVosct} \sup_v\| V^{osc}_{t,t}\|_{H^{N_0+1}_{x}} \lesssim \| E^{osc}_\pm(t)\|_{H^{N_0}} \lesssim \epsilon
\end{equation}
and 
\begin{equation}\label{daVosct}
\| \partial^\alpha_x V^{osc}_{t,t}\|_{L^\infty_{x,v}} + \| \partial^\alpha_x E^{osc,1}_{\pm}(t)\|_{L^\infty_{x,v}}  \lesssim \epsilon \langle t\rangle^{-\frac32 + \frac32\delta_{\alpha}}, 
\end{equation}
for $|\alpha|\le N_0-1$, where $\delta_{\alpha} = \frac{|\alpha|}{N_0-1}$. Next, in view of \eqref{daVosct} and the bootstrap funciton \eqref{lownorm}, as long as $\zeta(s,t)$ remains finite, we have
$$
 \begin{aligned}
\| \partial^\alpha_x ( V_{s,t}-v)\|_{L^\infty_{x,v}} &\lesssim \epsilon \langle t\rangle^{-\frac32 + \frac32\delta_{\alpha}} + \langle s\rangle^{-\frac32 +\frac12\delta_\alpha} \langle t\rangle^{\delta_\alpha}\zeta(s,t)
\\
 \| \partial^\alpha_x(\hPsi_{s,t}-\hv)\|_{L^\infty_{x,v}}   &\lesssim   \epsilon \langle t\rangle^{-\frac32 + \frac32\delta_{\alpha}} +  \langle s\rangle^{- 1+\frac12 \delta_\alpha}(t-s)^{-1} \langle t\rangle^{\delta_\alpha}\zeta(s,t). \end{aligned}
$$
Therefore, for $0\le s\le t$, we have
\begin{equation}\label{inductV}
\begin{aligned}
\| \partial^\alpha_x(V_{s,t}-v)\|_{L^\infty_{x,v}}  
&\lesssim \langle s\rangle^{-\frac32 +\frac12\delta_\alpha} \langle t\rangle^{\delta_\alpha}(\epsilon+\zeta(s,t)),
\end{aligned}\end{equation}
for all $ |\alpha|\le N_0-1$. Similarly, in view of \eqref{straightX}, \eqref{daVosct}, and \eqref{lownorm}, we have 
\begin{equation}\label{inductX}
\begin{aligned}
\| \partial^\alpha_x(X_{s,t}-x+\hv (t-s))\|_{L^\infty_{x,v}}  
&\le (t-s) \|\partial_x^\alpha \hPsi_{s,t}\|_{L^\infty_{x,v}}
\\& \lesssim   \epsilon \langle t\rangle^{-\frac12+\frac32\delta_{\alpha}} + 
 \langle s\rangle^{-1+\frac12\delta_\alpha}\langle t\rangle^{\delta_\alpha}\zeta(s,t)
\\
&\lesssim \langle s\rangle^{-\frac12+\frac12\delta_\alpha } \langle t\rangle^{\delta_\alpha}(\epsilon+\zeta(s,t)),
\end{aligned}\end{equation}
for all $2\le |\alpha|\le N_0-1$. The estimates \eqref{inductV}-\eqref{inductX} yield the desired bounds \eqref{decayXV} for $|\alpha_2|=0$, provided that $\zeta(s,t)$ remains finite. 

We now use the expansion from Proposition \ref{prop-charV} and Proposition \ref{prop-charPsi} to bound $\zeta(s,t)$, which we compute
\begin{equation}\label{expPVst}
\begin{aligned}
\partial^\alpha_x  \Big(V_{s,t}(x,v)- v + V^{osc}_{t,t}(x,v)\Big) &=  \partial_x^\alpha V^{osc}_{s,t}(x,v)  +\partial^\alpha_xV^{tr}_{s,t}(x,v),
\\
\partial_x^\alpha \Big(\hPsi_{s,t}(x,v) - \hv + \hV^{osc}_{t,t}(x,v)\Big )& = \partial_x^\alpha \hPsi^{osc}_{s,t}(x,v)  +\partial_x^\alpha \hPsi^{tr}_{s,t}(x,v)+\partial_x^\alpha \Psi^Q_{s,t}(x,v) .
\end{aligned}
\end{equation}
We shall estimate each term in \eqref{expPVst}. To proceed, we first recall 
the Fa\`a di Bruno's formula for composite functions (see, e.g., \cite{BCD}), namely 
\begin{equation}\label{Faa} \partial_x^n f(u) = \sum C_{k_j,n} \partial_x^k f (u) \prod_{j=1}^n (\partial_x^j u)^{k_j}  
\end{equation}
where $k_j \ge 0$, $\sum k_j = k$, and the summation is over all the partitions $\{ k_j\}_{j=1}^n$ of $n$ so that $\sum_j jk_j = n$. The extension to the multidimensional case is similar. 

\subsubsection*{Bounds on $\partial_x^\alpha V^{osc}_{s,t}(x,v)$.}
Recalling $ V^{osc}_{s,t} (x,v)= E^{osc,1}_{\pm}(s,X_{s,t}(x,v),V_{s,t}(x,v))$, we shall prove that 
\begin{equation}\label{est-daVosc}
\begin{aligned}  
\|\partial_x^\alpha V^{osc}_{s,t}\|_{L^\infty_{x,v}} 
&\lesssim 
\epsilon\langle s\rangle^{- \frac32+\frac12\delta_\alpha}\langle t\rangle^{\delta_\alpha} (1+\zeta(s,t)).
\end{aligned}
\end{equation}
Indeed, using \eqref{Faa}, we compute 
\begin{equation}\label{daVosci}
\begin{aligned}  
\partial_x^\alpha V^{osc}_{s,t}= \sum_{1\le |\mu|\le |\alpha|} C_{\beta,\mu}  
\partial_x^\mu E^{osc,1}_{\pm} (s) \prod_{1\le |\beta|\le |\alpha|} (\partial_x^\beta X_{s,t})^{k_\beta} + \mathrm{l.o.t.}
\end{aligned}\end{equation}
where $\sum k_\beta = |\mu|$ and $\sum |\beta|k_\beta = |\alpha|$. Here in the above expression, $\mathrm{l.o.t.}$ denotes the lower order terms, namely all those terms that involve one or more spatial derivatives of $V_{s,t}$. These terms decay at the same or better rate. Therefore, recalling \eqref{daVosct}, we bound 
\begin{equation}\label{daVosci1}
\begin{aligned}  
\Big \|\sum_{1\le |\mu|\le |\alpha|} C_{\beta,\mu}  
\partial_x^\mu E^{osc,1}_{\pm} (s) \prod_{1\le |\beta|\le |\alpha|} (\partial_x^\beta X_{s,t})^{k_\beta}\Big \|_{L^\infty_{x,v}}
&\lesssim \epsilon \sum_{1\le |\mu|\le |\alpha|}\langle s\rangle^{-\frac32 +\frac32\delta_{\mu}}  \prod_{1\le |\beta|\le |\alpha|} \|\partial_x^\beta X_{s,t}\|_{L^\infty_{x,v}}^{k_\beta}.
  \end{aligned}
  \end{equation}
Note that $\partial^\beta_xX_{s,t}$ is uniformly bounded for $|\beta|=1$, see Proposition \ref{prop-Dchar}. 
Therefore, using \eqref{inductX}, we bound 
$$
\begin{aligned}  
\prod_{2\le |\beta|\le |\alpha|} \|\partial_x^\beta X_{s,t}\|_{L^\infty_{x,v}}^{k_\beta} 
&\lesssim  
 \prod_{2\le |\beta|\le |\alpha|} \langle s\rangle^{-\frac{k_\beta}2 +\frac12 \delta_\beta k_\beta}\langle t\rangle^{\delta_\beta k_\beta} (\epsilon+\zeta(s,t))^{k_\beta}
\\
&\lesssim  
\langle s\rangle^{-\frac12\sum^*_\beta k_\beta +\frac12 \sum^*_\beta \delta_\beta k_\beta}\langle t\rangle^{\sum^*_\beta \delta_\beta k_\beta} (\epsilon+\zeta(s,t))^{\sum^*_\beta k_\beta}
\end{aligned}
$$
in which the summation $\sum_\beta^*$ is taken over $2\le |\beta|\le |\alpha|$ (noting $k_\beta \ge 0$). Now, using the fact that  $\sum_\beta k_\beta = |\mu|$ and $\sum |\beta|k_\beta = |\alpha|$, we have
$$k_1 + \sum^*_\beta k_\beta = |\mu|, \qquad k_1 + \sum^*_\beta |\beta| k_\beta = |\alpha|,$$
in which $k_1$ denotes the number of appearance of the first derivatives $\partial_xX_{s,t}$. This yields 
$$
\begin{aligned}  
\langle s\rangle^{\frac32\delta_{\mu}}  \prod_{1\le |\beta|\le |\alpha|} \|\partial_x^\beta X_{s,t}\|_{L^\infty_{x,v}}^{k_\beta}
&\lesssim  
\langle s\rangle^{ \frac32\delta_{\mu}} \Big[ 1 + 
\langle s\rangle^{-\frac12\sum^*_\beta k_\beta +\frac12\sum^*_\beta \delta_\beta k_\beta }\langle t\rangle^{\sum^*_\beta \delta_\beta k_\beta} (\epsilon+\zeta(s,t))\Big]  
\\
&\lesssim 
\langle s\rangle^{ \frac32\delta_{\alpha}}  +  
\langle s\rangle^{- \frac12 (|\mu|-k_1)+\frac32 \delta_\mu+\frac12(\delta_\alpha - \delta_{k_1})}\langle t\rangle^{\delta_\alpha - \delta_{k_1}} (\epsilon+\zeta(s,t))
\\
&\lesssim 
\langle s\rangle^{ \frac12\delta_{\alpha}} \langle t\rangle^{\delta_{\alpha}}  +  
\langle s\rangle^{ - \frac12 (|\mu|-k_1)+ \frac32(\delta_\mu - \delta_{k_1})}
\langle s\rangle^{ \frac12\delta_{\alpha}} \langle t\rangle^{\delta_\alpha} (\epsilon+\zeta(s,t)) .
\end{aligned}
$$
Note that $\frac32(\delta_\mu - \delta_{k_1}) = \frac32\frac{|\mu|-k_1}{N_0-1} \le \frac12 (|\mu|-k_1)$, provided that $N_0\ge 4$. Therefore, the above yields 
\begin{equation}\label{supprodX} 
\begin{aligned}  
\langle s\rangle^{\frac32\delta_{\mu}}  \prod_{1\le |\beta|\le |\alpha|} \|\partial_x^\beta X_{s,t}\|_{L^\infty_{x,v}}^{k_\beta}
&\lesssim 
\langle s\rangle^{\frac12\delta_\alpha} \langle t\rangle^{\delta_\alpha}   (1+\zeta(s,t)),
\end{aligned}
\end{equation}
for $|\mu|\le |\alpha|$, in which  $\sum_\beta k_\beta = |\mu|$ and $\sum |\beta|k_\beta = |\alpha|$. 
Putting these estimates into \eqref{daVosci}-\eqref{daVosci1}, 
we obtain \eqref{est-daVosc}.

\subsubsection*{Bounds on $\partial_x^\alpha V^{tr}_{s,t}(x,v)$.}

Next, we bound on $ \partial_x^\alpha V^{tr}_{s,t}(x,v)$, namely we shall prove 
\begin{equation}\label{est-daVtr}
\begin{aligned}
\| \partial_x^\alpha V^{tr}_{s,t} \|_{L^\infty_{x,v}} 
\lesssim \epsilon \langle s\rangle^{-2 +\frac12\delta_\alpha}\langle t\rangle^{\delta_\alpha} (1+\zeta(s,t)).
\end{aligned}
\end{equation}
By definition, we recall 
\begin{equation}\label{recalldaVtr}
\partial_x^\alpha V^{tr}_{s,t}(x,v) = -  \int_s^t  
\partial_x^\alpha[Q^{tr}(\tau,X_{\tau,t}(x,v), V_{\tau,t}(x,v))] \; d\tau.
\end{equation}
As before, since $\partial_x^\beta V_{s,t}$ (and $\partial^\gamma_v Q^{tr}$) satisfy better estimates than those for $\partial_x^\beta X_{s,t}$ (and $\partial^\gamma_x Q^{tr}$, respectively), we shall therefore focus terms that only involve derivatives of $X_{s,t}$. Namely, we bound 
\begin{equation}\label{compItr}
I^{tr,\alpha}_{s,t} = \sum_{1\le |\mu|\le |\alpha|} C_{\beta,\mu}  
\int_s^t \partial_x^\mu Q^{tr}(\tau) \prod_{1\le |\beta|\le |\alpha|} (\partial_x^\beta X_{\tau,t}(x,v))^{k_\beta} \; d\tau.
\end{equation}
By definition \eqref{def-Qtr}, we compute 
\begin{equation}\label{comdaQtr}
\begin{aligned} 
\partial_x^\mu Q^{tr}&=\sum_\pm \partial_x^\mu [a_\pm(i\partial_x) F\star_x \phi_{\pm,1}]
-  \sum_\pm \nabla_v \hv  \partial_x^\mu(E\cdot  \nabla_x E^{osc,2}_\pm) + \partial_x^\mu  E^r.
\end{aligned} 
\end{equation}
Using the bootstrap assumptions on $E$ and $S$, we bound
$$ \|  a_\pm(i\partial_x) \partial_x^\mu F(t)\star_x \phi_{\pm,1}\|_{L^\infty_{x,v}} \lesssim \| \partial_x^{\mu} S(t)\|_{L^\infty_{x,v}} \lesssim \epsilon \langle t\rangle^{-3 + \frac32\delta_\mu}$$
and 
$$
\begin{aligned} 
\| \partial_x^\mu(E\cdot  \nabla_x E^{osc,2}_\pm(t))\|_{L^\infty_{x,v}} 
& \lesssim 
\sum_{|\mu_1|+|\mu_2|\le |\mu|}\| \partial_x^{\mu_1}E(t)\|_{L^\infty_x}\|\partial_x^{\mu_2}\nabla_x E^{osc,2}_\pm(t)\|_{L^\infty_{x,v}}
\lesssim 
 \epsilon^2 \langle t\rangle^{-3 +\frac32\delta_{\mu}},  \end{aligned} 
$$
upon noting $\delta_{\mu_1}+\delta_{\mu_2} \le \delta_\mu$. 
Therefore, together with \eqref{bootstrap-decaydaS}, we have 
\begin{equation}\label{dmuQtr}
\| \partial_x^\mu Q^{tr}(t)\|_{L^\infty_{x,v}}  \lesssim 
 \epsilon \langle t\rangle^{-3 +\frac32\delta_{\mu}} 
\end{equation}
and 
\begin{equation}\label{est-Itrst}
\begin{aligned}
\|I^{tr,\alpha}_{s,t} \|_{L^\infty_{x,v}} 
&\lesssim 
\sum_{1\le |\mu|\le |\alpha|}  
\int_s^t \| \partial_x^\mu Q^{tr}(\tau) \|_{L^\infty_{x,v}}\prod_{1\le |\beta|\le |\alpha|} \| \partial_x^\beta X_{\tau,t}\|_{L^\infty_{x,v}}^{k_\beta} \; d\tau
\\&\lesssim \epsilon\sum_{1\le |\mu|\le |\alpha|}  
\int_s^t  \langle \tau\rangle^{-3 +\frac32\delta_{\mu}} \prod_{1\le |\beta|\le |\alpha|} \| \partial_x^\beta X_{\tau,t}\|_{L^\infty_{x,v}}^{k_\beta} \; d\tau
.\end{aligned}
\end{equation}
This yields \eqref{est-daVtr}, upon using \eqref{supprodX} and integrating the above integral in time. 

\subsubsection*{Bounds on $\partial_x^\alpha \hPsi^{osc}_{s,t}(x,v)$.}

Next, we prove bounds on $\partial_x^\alpha \hPsi^{osc}_{s,t}(x,v)$, which we recall
\begin{equation}\label{comPosc}
\begin{aligned}
(t-s)\partial_x^\alpha\hPsi^{osc}_{s,t}(x,v) &=  \nabla_v \hv\sum_\pm \partial_x^\alpha\Big(E^{osc,2}_{\pm}(t,x,v) - E^{osc,2}_{\pm} (s,X_{s,t}(x,v),V_{s,t}(x,v)) \Big).
 \end{aligned}\end{equation}
Since $E^{osc,2}(s)$ satisfies the same (or better) estimates than those for $ V^{osc}_{s,t} = E^{osc,1}_{\pm}(s)$, following the proof of \eqref{est-daVosc}, we thus obtain   
\begin{equation}\label{est-daPosc}
\begin{aligned}  
(t-s)\|\partial_x^\alpha\hPsi^{osc}_{s,t}\|_{L^\infty_{x,v}} 
 \lesssim 
 \epsilon\langle s\rangle^{- \frac32+\frac12\delta_\alpha}\langle t\rangle^{\delta_\alpha} (1+\zeta(s,t)),
 \end{aligned}
 \end{equation}
 for $|\alpha|\le N_0-1$. 
Note that these estimates are better than what's stated for $\partial_x^\alpha\hPsi_{s,t}(x,v)$. 

\subsubsection*{Bounds on $\partial_x^\alpha \hPsi^{tr}_{s,t}(x,v)$.}

Next, we give bounds on $\partial_x^\alpha\hPsi^{tr}_{s,t}(x,v)$, which we recall 
\begin{equation}\label{recalldaPtr}
\begin{aligned}
(t-s)\partial_x^\alpha\hPsi^{tr}_{s,t}(x,v)& = - \nabla_v \hv\int_s^t  
\partial_x^\alpha[ (\tau-s)  
Q^{tr} + Q^{tr,1}](\tau,X_{\tau,t}(x,v), V_{\tau,t}(x,v)) \; d\tau .
 \end{aligned}\end{equation}
We shall prove 
\begin{equation}\label{est-daPtr}
\begin{aligned}  
(t-s)\|\partial_x^\alpha\hPsi^{tr}_{s,t}\|_{L^\infty_{x,v}} 
 \lesssim 
 \epsilon\langle s\rangle^{- 1+\frac12\delta_\alpha}\langle t\rangle^{\delta_\alpha} (1+\zeta(s,t)),
 \end{aligned}
 \end{equation}
for $|\alpha|\le N_0-1$. Again, as $Q^{tr,1}(s)$ defined as in \eqref{def-Qtr1} is of the same form as that of $Q^{tr}(s)$ defined as in \eqref{def-Qtr}, using \eqref{supprodX} again, we obtain 
\begin{equation}\label{bdPtr}
\begin{aligned}
\|(t-s)\partial_x^\alpha\hPsi^{tr}_{s,t} \|_{L^\infty_{x,v}} 
&\lesssim \epsilon\sum_{1\le |\mu|\le |\alpha|}  
\int_s^t (\tau-s) \langle \tau\rangle^{-3 +\frac32\delta_{\mu}} \prod_{1\le |\beta|\le |\alpha|} \| \partial_x^\beta X_{\tau,t}\|_{L^\infty_{x,v}}^{k_\beta} \; d\tau
\\
&\lesssim \epsilon\sum_{1\le |\mu|\le |\alpha|}  
\int_s^t \langle \tau\rangle^{-2 +\frac12 \delta_\alpha} \langle t\rangle^{\delta_\alpha} (1+\zeta(s,t))\; d\tau
\\& \lesssim 
 \epsilon\langle s\rangle^{- 1+\frac12\delta_\alpha}\langle t\rangle^{\delta_\alpha} (1+\zeta(s,t))
,
\end{aligned}
\end{equation}
which proves \eqref{est-daPtr}. 

\subsubsection*{Bounds on $\partial_x^\alpha \Psi^Q_{s,t}(x,v)$.}
Finally, we bound $\partial_x^\alpha \Psi^Q_{s,t}(x,v)$, which are computed by 
$$
\begin{aligned}
(t-s)\partial_x^\alpha \Psi^Q_{s,t}(x,v) &= \int_s^t \partial_x^\alpha V_{\tau,t}^Q(x,v)\; d\tau
\end{aligned}$$
where $V^Q_{s,t} (x,v) = \mathcal{O}(|V_{s,t}-v|^2)$. Using the Leibniz's formula for derivatives and the inductive bounds \eqref{inductV}, we compute 
$$
\begin{aligned}
|(t-s)\partial_x^\alpha \Psi^Q_{s,t}(x,v)| 
& \lesssim \sum_{|\beta|+|\mu| \le |\alpha|} \int_s^t |\partial_x^\beta (V_{s,t}-v)||\partial_x^\mu (V_{s,t}-v)| \; d\tau
\\
& \lesssim \epsilon \sum_{|\beta|+|\mu| \le |\alpha|} \int_s^t \langle \tau\rangle^{-3+\frac12 \delta_\beta+\frac12 \delta_\mu} \langle t\rangle^{\delta_\beta + \delta_\mu}(\epsilon+\zeta(\tau,t))\; d\tau.
\end{aligned}$$
Since $|\beta|+|\mu|\le |\alpha|$, we have $\delta_\beta + \delta_\mu \le \delta_\alpha$. Therefore, integrating in time the above integral yields 
\begin{equation}\label{est-daPQ}
\begin{aligned}
\|(t-s)\partial_x^\alpha \Psi^Q_{s,t}\|_{L^\infty_{x,v}} 
& \lesssim \epsilon \langle s\rangle^{-2+\frac12\delta_\alpha} \langle t\rangle^{\delta_\alpha}(\epsilon+\zeta(s,t)).
\end{aligned}
\end{equation}

\subsubsection*{Bounds on $x$-derivatives.}

Combining all the estimates \eqref{est-daVosc}, \eqref{est-daVtr}, \eqref{est-daPosc}, \eqref{est-daPtr}, and \eqref{est-daPQ} into the expansions in \eqref{expPVst}, we complete the proof of the claim \eqref{claimXV1}, and hence the decay and boundedness in $L^\infty_{x,v}$ of $x$-derivatives of the characteristics, completing the proof of \eqref{decayXV} for $|\alpha_2|=0$.

\subsubsection*{Bounds on $v$-derivatives.}

Finally, we check the stated bounds \eqref{decayXV} that include $v$-derivatives of the characteristic. Indeed, the proof follows similarly as done above, upon observing that $\partial_v $ and $t\partial_x$ derivatives are of the same order. 

%
%


\subsection{Boundedness for top derivatives}\label{sec-Htopdxchar}


In this section, we establish the boundedness of the top derivatives of the characteristics. Specifically, we prove the following proposition. 

\begin{proposition}\label{prop-HDBchar} 
Fix $N_0 \ge 4$. Let $X_{s,t}(x,v), V_{s,t}(x,v)$ be the nonlinear characteristic solving \eqref{ode-char}. Then, for $0\le s\le t$ and $for |\alpha_0|=N_0$, there hold 
\begin{equation}\label{HsXV}
 \begin{aligned}
\sup_v \| \partial_x^{\alpha_0} X_{s,t}(x,v)\|_{L^2_x}  + \sup_v \| \partial_x^{\alpha_0} V_{s,t}(x,v)\|_{L^2_x}   &\lesssim \epsilon  \langle t\rangle^{\frac32 +\delta_{1}} ,
 \end{aligned}
 \end{equation}
with $\delta_1 = \frac{1}{N_0-1}$. 
\end{proposition}
\begin{proof} In this section, we bound the top derivatives of $X_{s,t}$ and $V_{s,t}$. Indeed, for $ |\alpha_0|=N_0\ge 4$, using again the expansion \eqref{expPVst}, we compute 
\begin{equation}\label{expXVtopda}
\begin{aligned}
\partial^{\alpha_0}_x V_{s,t}(x,v) &= -\partial_x^{\alpha_0} V^{osc}_{t,t}(x,v)+  \partial_x^{\alpha_0} V^{osc}_{s,t}(x,v)  +\partial^{\alpha_0}_xV^{tr}_{s,t}(x,v),
\\
\partial_x^{\alpha_0} \hPsi_{s,t}(x,v)& = -\partial_x^{\alpha_0} \hV^{osc}_{t,t}(x,v)+ \partial_x^{\alpha_0} \hPsi^{osc}_{s,t}(x,v)  +\partial_x^{\alpha_0} \hPsi^{tr}_{s,t}(x,v)+\partial_x^{\alpha_0} \Psi^Q_{s,t}(x,v) .
\end{aligned}
\end{equation}
The first term on the right hand side is already treated in \eqref{bdVosct}, namely 
\begin{equation}\label{L2daVosct}
\| \partial^{\alpha_0}_x V^{osc}_{t,t}\|_{L^2_{x}} \lesssim \| \partial^{\alpha_0}_x E^{osc,1}_{\pm}(t)\|_{L^2_{x}} \lesssim \|E^{osc}_{\pm}(t)\|_{H^{N_0-1}_{x}}  \lesssim \epsilon .
\end{equation}
We shall now check the remaining terms. Indeed, recalling \eqref{daVosci}, we compute 
$$
\begin{aligned}  
\partial_x^{\alpha_0}V^{osc}_{s,t}= C_{\alpha}  
\partial_xE^{osc,1}_{\pm} (s) (\partial_x^{\alpha_0} X_{s,t})+ \sum_{1\le |\mu|\le N_0} C_{\beta,\mu}  
\partial_x^\mu E^{osc,1}_{\pm} (s) \prod_{1\le |\beta|\le N_0-1} (\partial_x^\beta X_{s,t})^{k_\beta} + \mathrm{l.o.t.}
\end{aligned}
$$
in which $\mathrm{l.o.t.}$ denotes terms that involve at least one derivative of $V_{s,t}$, which satisfy better estimates than those that involve $X_{s,t}$, see \eqref{decayXV} and \eqref{HsXV}. Therefore, taking $L^2_x$ norm both sides, we obtain 
$$ \| \partial_xE^{osc,1}_{\pm} (s) (\partial_x^{\alpha_0} X_{s,t})\|_{L^2_x} \le \| \partial_xE^{osc,1}_{\pm} (s)\|_{L^\infty_{x,v}} \| \partial_x^{\alpha_0} X_{s,t}\|_{L^2_x} \lesssim \epsilon \langle s\rangle^{-\frac32 +\delta_1}\| \partial_x^{\alpha_0} X_{s,t}\|_{L^2_x} $$
with $\delta_{1} = \frac{1}{N_0-1}$. Next, we bound 
$$
\begin{aligned}
\Big\|\partial_x^\mu E^{osc,1}_{\pm} (s) \prod_{1\le |\beta|\le N_0-1} (\partial_x^\beta X_{s,t})^{k_\beta}
\Big \|_{L^2_x} & \le \| 
\partial_x^\mu E^{osc,1}_{\pm} (s)\|_{L^2_x} \prod_{1\le |\beta|\le N_0-1} \| \partial_x^\beta X_{s,t}\|_{L^\infty_{x,v}}^{k_\beta},
\end{aligned}$$
in which we note that $\sum_\beta k_\beta = |\mu|$ and $\sum |\beta|k_\beta = |\alpha_0|=N_0$. We bound $\| 
\partial_x^\mu E^{osc,1}_{\pm} (s)\|_{L^2_x}  \lesssim \epsilon$ by the bootstrap assumption \eqref{bootstrap-Hs}. On the other hand, using \eqref{decayXV} and noting $\delta_\beta \le 1$ for $2\le |\beta|\le N_0-1$, we have
$$
 \begin{aligned}
\| \partial^{\beta}_xX_{s,t}\|_{L^\infty_{x,v}}  
&\lesssim \epsilon  \langle t\rangle^{\delta_\beta}.
 \end{aligned}
$$
In addition, recall from Proposition \ref{prop-Dchar} that $\| \partial_xX_{s,t}\|_{L^\infty_{x,v}}  
\lesssim 1$. Therefore, $$
\begin{aligned}  
\prod_{1\le |\beta|\le |\alpha|} \|\partial_x^\beta X_{s,t}\|_{L^\infty_{x,v}}^{k_\beta}
&\lesssim  
1+
\prod_{2\le |\beta|\le |\alpha|} \|\partial_x^\beta X_{s,t}\|_{L^\infty_{x,v}}^{k_\beta}\lesssim 1 + \langle t\rangle^{\sum \delta_\beta k_\beta}
\lesssim \langle t\rangle^{\delta_{\alpha_0}}
\end{aligned}
$$
upon recalling $\sum |\beta|k_\beta = |\alpha_0|$ and $\delta_\beta$ is linear in $|\beta|$. This yields 
\begin{equation}\label{est-topVosc}
\sup_v\|\partial_x^{\alpha_0}V^{osc}_{s,t}\|_{L^2_x} \lesssim \epsilon\langle t\rangle^{\delta_{\alpha_0}} +  \epsilon \langle s\rangle^{-\frac32 +\delta_1}\sup_v\| \partial_x^{\alpha_0} X_{s,t}\|_{L^2_x} 
\end{equation}
Similarly, recalling \eqref{comPosc} and the fact that $E^{osc,2}(s)$ satisfies the same (or better) estimates than those for $E^{osc,1}_{\pm}(s)$, we obtain 
\begin{equation}\label{est-topPosc}
\sup_v\|(t-s)\partial_x^{\alpha_0}\hPsi^{osc}_{s,t}\|_{L^2_x} \lesssim \epsilon
\langle t\rangle^{\delta_{\alpha_0}} +  \epsilon \langle s\rangle^{-\frac32 + \delta_{1}}\sup_v\| \partial_x^{\alpha_0} X_{s,t}\|_{L^2_x} .
\end{equation}
Next, we bound $\partial_x^{\alpha_0} V^{tr}_{s,t}(x,v)$ and $\partial_x^{\alpha_0}\hPsi^{tr}_{s,t}(x,v)$. We shall prove that 
 \begin{equation}\label{est-topVPtr}
\begin{aligned}
\sup_v\|\partial_x^{\alpha_0}V^{tr}_{s,t}\|_{L^2_x} &\lesssim \epsilon
\langle t\rangle^{\delta_{\alpha_0}} + \epsilon \langle s\rangle^{-2 +\delta_{1}} \sup_{s\le \tau \le t}\sup_v\| \partial_x^{\alpha_0} X_{\tau,t}\|_{L^2_x}
\\
\sup_v\|(t-s)\partial_x^{\alpha_0}\hPsi^{tr}_{s,t}\|_{L^2_x} &\lesssim  \epsilon  \langle t\rangle^{\frac12 +\delta_{\alpha_0}}+ \epsilon \langle s\rangle^{-1 +\delta_{1}} \sup_{s\le \tau \le t}\sup_v\| \partial_x^{\alpha_0} X_{\tau,t}\|_{L^2_x}.
\end{aligned}
\end{equation}
In view of \eqref{recalldaVtr} and \eqref{recalldaPtr}, the latter is worse due to the double integrals in time. For this reason, we shall focus only on the following worst term, namely 
$$
\begin{aligned}
J_{s,t}& = \int_s^t  (\tau-s)
\Big[  C_\alpha 
\partial_xQ^{tr} (\tau) (\partial_x^{\alpha_0} X_{\tau,t}) + \sum_{2\le |\mu|\le N_0} C_{\beta,\mu}  
\partial_x^\mu Q^{tr}(\tau) \prod_{1\le |\beta|\le N_0-1} (\partial_x^\beta X_{\tau,t})^{k_\beta}\Big]\; d\tau.
 \end{aligned}$$
Recall from \eqref{dmuQtr} that $ \| \partial_x^\mu Q^{tr}(t)\|_{L^\infty_{x,v}} \lesssim \epsilon \langle t\rangle^{-3 +\frac32\delta_{\mu}} $. In addition, for $2\le |\mu|\le N_0$, we compute 
$$
\begin{aligned}
\|  a_\pm(i\partial_x) \partial_x^\mu F(t)\star_x \phi_{\pm,1}\|_{L^2_{x}} &
\lesssim \| \partial_x^{\mu-1} S(t)\|_{L^2_{x}} \lesssim \epsilon \langle t\rangle^{-\frac32 + \delta_{\mu-1}}
 \end{aligned} 
$$
On the other hand, we have 
$$
\begin{aligned}
\| \partial_x^\mu(E\cdot  \nabla_x E^{osc,2}_\pm(t))\|_{L^2_{x}} 
& \lesssim 
\| E(t)\|_{L^\infty_x}\|E(t)\|_{H^{\mu}_x} \lesssim \epsilon^2 \langle t\rangle^{-\frac32}.
 \end{aligned} 
$$
That is,  recalling \eqref{comdaQtr}, we have 
\begin{equation}\label{bdsupvQtr}
\sup_v \| \partial_x^\mu Q^{tr}(t)\|_{L^2_{x}} \lesssim \epsilon \langle t\rangle^{-\frac32 +\delta_{\mu-1}} .
\end{equation}
Therefore, we compute 
$$
\begin{aligned}
\|J_{s,t}\|_{L^2_x}
& \lesssim \int_s^t  (\tau-s)
\Big[ 
\|\partial_xQ^{tr} (\tau) \|_{L^\infty_x}\| \partial_x^{\alpha_0} X_{\tau,t}\|_{L^2_x} + \sum_{2\le |\mu|\le N_0}  
\|\partial_x^\mu Q^{tr}(\tau)\|_{L^2_x} \prod_{1\le |\beta|\le N_0-1} \|\partial_x^\beta X_{\tau,t}\|_{L^\infty_x}^{k_\beta}\Big]\; d\tau
 \end{aligned}$$
in which we bound 
$$
\begin{aligned}
\int_s^t  (\tau-s)
\|\partial_xQ^{tr} (\tau) \|_{L^\infty_x}\| \partial_x^{\alpha_0} X_{\tau,t}\|_{L^2_x} 
& \lesssim \epsilon \int_s^t  (\tau-s)
\langle \tau\rangle^{-3 +\delta_{1}} \| \partial_x^{\alpha_0} X_{\tau,t}\|_{L^2_x} \; d\tau 
\\
& \lesssim \epsilon \langle s\rangle^{-1 +\delta_{1}} \sup_{s\le \tau \le t}\sup_v\| \partial_x^{\alpha_0} X_{\tau,t}\|_{L^2_x}.
 \end{aligned}$$
On the other hand, for $2\le |\mu|\le N_0$, using \eqref{bdsupvQtr}, we bound 
$$
\begin{aligned}
\int_s^t & (\tau-s)
\|\partial_x^\mu Q^{tr}(\tau)\|_{L^2_x} \prod_{1\le |\beta|\le N_0-1} \|\partial_x^\beta X_{\tau,t}\|_{L^\infty_x}^{k_\beta}
\; d\tau
\lesssim \epsilon \int_s^t  \langle \tau\rangle^{-\frac12 +\delta_{\mu-1}}  \prod_{1\le |\beta|\le N_0-1} \|\partial_x^\beta X_{\tau,t}\|_{L^\infty_x}^{k_\beta}
\; d\tau
. 
 \end{aligned}$$
To prove \eqref{est-topVPtr}, it suffices to prove 
 \begin{equation}\label{claimtopVPtr}
  \langle \tau\rangle^{\delta_{\mu-1}}  \prod_{1\le |\beta|\le N_0-1} \|\partial_x^\beta X_{\tau,t}\|_{L^\infty_x}^{k_\beta} \lesssim \epsilon  \langle t\rangle^{\delta_{\alpha_0}},
  \end{equation}
 for $2\le |\mu|\le N_0$.  
Indeed, in the case when $|\mu|=N_0$, then $\delta_{\mu-1} = 1$ and the product $\prod\|\partial_x^\beta X_{\tau,t}\|_{L^\infty_x}^{k_\beta}  = (\partial_xX_{\tau,t})^{N_0}$, which remains bounded. The stated bound \eqref{claimtopVPtr} follows, since $|\mu|-1 \le |\alpha_0|$ and $\delta_\alpha$ is linear in $|\alpha|$. 
In the case when $|\mu|\le N_0-1$, we let $k_1 < |\mu|$ be the number of the first derivatives appearing in the above product and recall that $\sum_\beta k_\beta = |\mu|$ and $\sum |\beta|k_\beta = |\alpha_0|=N_0$. 
Using \eqref{decayXV}, we bound 
$$
\begin{aligned}  
\langle s\rangle^{\delta_{\mu-1}}  \prod_{1\le |\beta|\le N_0-1} \|\partial_x^\beta X_{s,t}\|_{L^\infty_{x,v}}^{k_\beta}
&\lesssim  \langle s\rangle^{\delta_{\mu-1}}  \prod_{2\le |\beta|\le N_0-1} \|\partial_x^\beta X_{s,t}\|_{L^\infty_{x,v}}^{k_\beta}
\\
&\lesssim \epsilon
\langle s\rangle^{\delta_{\mu-1}-\frac12\sum^*_\beta k_\beta +\frac12\sum^*_\beta \delta_\beta k_\beta }\langle t\rangle^{\sum^*_\beta \delta_\beta k_\beta}
\\
&\lesssim \epsilon
\langle s\rangle^{-\frac12 (|\mu|-k_1)+\delta_{\mu-1}+\frac12(\delta_{\alpha_0} - \delta_{k_1} )}\langle t\rangle^{\delta_{\alpha_0} - \delta_{k_1}} 
\\
&\lesssim \epsilon
\langle s\rangle^{-\frac12 (|\mu|-k_1-1)+(\delta_{\mu-1} -\delta_{k_1}) + \frac12(\delta_1-\delta_{k_1})}\langle t\rangle^{\delta_{\alpha_0}} 
,\end{aligned}
$$
in which we have used $\delta_{\alpha_0} = 1 + \delta_1$ (since $|\alpha_0|=N_0$). Note that $\delta_{\mu-1} -\delta_{k_1} = \frac{|\mu|-k_1-1}{N_0-1} \le \frac14 (|\mu|-k_1-1)$, since $N_0\ge 5$. This yields 
$$
\begin{aligned}  
\langle s\rangle^{\delta_{\mu-1}}  \prod_{1\le |\beta|\le N_0-1} \|\partial_x^\beta X_{s,t}\|_{L^\infty_{x,v}}^{k_\beta}
&\lesssim \epsilon
\langle s\rangle^{-\frac14 (|\mu|-k_1-1)+ \frac12(\delta_1-\delta_{k_1})}\langle t\rangle^{\delta_{\alpha_0}} 
.\end{aligned}
$$
In the case when $|\mu|=k_1+1$, we must have $k_1\ge 1$, since $|\mu|\ge 2$, and so $\delta_1 \le \delta_{k_1}$. On the other hand, when $|\mu| \ge k_1+2$, we have $\delta_1 =\frac{1}{N_0-1} \le \frac12 \le \frac12(|\mu|-k_1-1)$, since $N_0\ge 3$. The claim \eqref{claimtopVPtr} thus follows. Thsi ends the proof of \eqref{est-topVPtr}. 

Next, we bound $\partial_x^\alpha \Psi^Q_{s,t}(x,v)$. Indeed, by definition (see Proposition \ref{prop-charPsi}) and the decay estimates \eqref{decay-Vst}, we compute 
$$
\begin{aligned}
\|(t-s)\Psi^Q_{s,t} \|_{H^{\alpha_0}_x} 
& \lesssim  \int_s^t \|V_{\tau,t}-v\|_{L^\infty_x}\| V_{\tau,t}-v \|_{H^{\alpha_0}_x} \; d\tau
\\
& \lesssim \epsilon\int_s^t \langle \tau\rangle^{-3/2} \| V_{\tau,t}-v \|_{H^{\alpha_0}_x}\; d\tau
\\
& \lesssim \epsilon\langle s\rangle^{-1/2} 
\sup_{s\le \tau \le t}\sup_v\| V_{\tau,t}-v \|_{H^{\alpha_0}_x}.
\end{aligned}$$
Combining \eqref{est-topVosc}, \eqref{est-topPosc}, \eqref{est-topVPtr}, and the above estimates into the expansion \eqref{expXVtopda}, and recalling that $\partial_x^{\alpha_0} X_{s,t} = (t-s)\partial_x^{\alpha_0}\hPsi_{s,t}$, we obtain 
$$
\begin{aligned} 
\|\partial_x^{\alpha_0} V_{s,t}\|_{L^2_x} 
& \lesssim \epsilon
\langle t\rangle^{\delta_{\alpha_0}} +  \epsilon \langle s\rangle^{-\frac32+ \delta_{1}}\sup_v\| \partial_x^{\alpha_0} X_{s,t}\|_{L^2_x} 
\\
\|\partial_x^{\alpha_0} X_{s,t}\|_{L^2_x} &\lesssim \epsilon  \langle t\rangle^{\frac12 +\delta_{\alpha_0}}+ \epsilon \langle s\rangle^{-1 +\delta_{1}} \sup_{s\le \tau \le t}\sup_v\| \partial_x^{\alpha_0} X_{\tau,t}\|_{L^2_x} + \epsilon\langle s\rangle^{-1/2} 
\sup_{s\le \tau \le t}\sup_v\| V_{\tau,t}-v \|_{H^{\alpha_0}_x}.
\end{aligned}$$
Adding up the two inequalities and taking $\epsilon$ sufficiently small to absorb the top derivatives to the left, we obtain \eqref{HsXV}, and thus complete the proof of Proposition \ref{prop-HDBchar}. 
\end{proof}



\section{Density decay estimates}\label{sec-sourceest}


In this section, we establish bounds on the nonlinear sources, namely proving Proposition \ref{prop-bdsS-goal}, which we recall below.

\begin{proposition} \label{prop-bdS} Let $S$ be the nonlinear source terms defined as in \eqref{nonlinear-S}. Under the bootstrap assumptions listed in Section \ref{sec-bootstrap}, there hold
\begin{equation}\label{Lpbounds-cS}
\begin{aligned}
\| S (t)\|_{L^p_x}   &\lesssim (\epsilon_0+ \epsilon^2) \langle t\rangle^{-3(1-1/p)} \log t
\\
\sup_{q\in \ZZ} 2^{(1-\delta)q} \|P_qS(t)\|_{L^p_x} + \|\partial_x S(t)\|_{L^p_x} &\lesssim (\epsilon_0 + \epsilon^2) \langle t\rangle^{-3(1-1/p)} 
\end{aligned}\end{equation}
for any $1 \le p \le \infty$ and $\delta \in (0,1)$, where $P_q$ denotes the Littlewood-Paley projection on the dyadic interval $[2^{q-1}, 2^{q+1}]$. 
In addition, 
\begin{equation}\label{Hsbounds-cS0}
\begin{aligned}
\|\partial_x^\alpha S(t)\|_{L^p_x}   \le \epsilon \langle t\rangle^{-3(1-1/p) + \delta_{\alpha,p}} 
, \qquad 
\|S(t)\|_{H^{\alpha_0}_x} &\lesssim \epsilon \langle t\rangle^{\delta_1},
\end{aligned}
\end{equation}
for all $|\alpha|\le N_0-1$, with $\delta_{\alpha,p} = \max\{ \delta_\alpha, \frac32\delta_\alpha (1-1/p)\}$, where $\delta_\alpha = \frac{|\alpha|}{N_0-1}$. 
\end{proposition}

\subsection{Structure of nonlinear interaction}\label{sec-NLstructure}

In this section, we study the precise structure of nonlinear interaction. Specifically, recalling from \eqref{nonlinear-S} that the source density involves a contribution from initial data, plus a nonlinear reaction term defined by 
\begin{equation}\label{Smu-recall}
\begin{aligned}
S^{\mu}(t,x) &= \int_0^t \int_{\RR^3}  \Big[ E(s,x - (t-s)\hv )\cdot \nabla_v \mu(v) 
 -E(s,X_{s,t}(x,v))\cdot \nabla_v \mu(V_{s,t}(x,v)) \Big] \, dv  ds .
\end{aligned}\end{equation}
In view of the nonlinear characteristic description established in the previous section, see Proposition \ref{prop-charPsi}, 
the integrand decays at best at the rate of order $s^{-3/2}s^{-1}$, which is far from the bootstrap decay at order $t^{-3}$, not to mention the time integration. For this reason, we need to exploit the nonlinear structure of $S^\mu(t,x)$ in order to improve decay. 

Precisely, we obtain the following proposition. 

\begin{proposition}\label{prop-SR}
Let $S^\mu(t,x)$ be defined as in \eqref{Smu-recall}. Then, we may write 
\begin{equation}\label{key-SR}
\begin{aligned}
S^\mu(t,x) & =  \int_0^t \int_{\RR^3} E(s,x - (t-s)\hw ) \cdot H_{s,t}(x,w) \; dw ds
\end{aligned}\end{equation}
in which the nonlinear interaction kernel $H_{s,t}(x,w)$ can be expressed as 
\begin{equation}\label{decomp-Hst}
H_{s,t}(x,w)
 =  T_{s,t}^{osc}(x,w) + H_{s,t}^{osc}(x,w) + H_{s,t}^{tr}(x,w) + H^Q_{s,t}(x,w)+ H_{s,t}^R(x,w) 
 \end{equation}
where 
\begin{equation}\label{def-THosctr}
\begin{aligned}
 T^{osc,\pm}_{s,t}(x,w) &= -  \nabla_w\mu \nabla_w \cdot E^{osc,1}_{\pm}(t,x,w) +  \frac{1}{t-s} \Big(  \nabla_w^2\mu E^{osc,2}_{\pm} -  \nabla_w\mu \nabla_w \cdot E^{osc,2}_{\pm}\Big) (t,x,w), 
\\
H_{s,t}^{osc,\pm}(x,w) & = - \Big( \nabla_w^2 \mu E^{osc,1}_{\pm} + \nabla_w \mu \nabla_w \cdot E^{osc,1}_{\pm}\Big)(s,x - (t-s)\hw,w) 
\\&\quad - \frac{1}{t-s}  \Big(  \nabla_w^2 \mu  E^{osc,2}_{\pm} + \nabla_w \mu \nabla_w \cdot E^{osc,2}_{\pm}\Big) (s,x-(t-s)\hw,w) ,
\\
H^{tr}_{s,t}(x,w)  &= \nabla_w \mu  \int_s^t  
\frac{(t-\tau) (\tau-s)}{t-s}  
\nabla_w\hw\cdot \nabla_xQ^{tr}(\tau,x - (t-\tau)\hw, w) \; d\tau
\end{aligned}
\end{equation}
with a quadratic term $H^Q_{s,t}(x,w)$ of the form 
\begin{equation}\label{def-THosctr1}
\begin{aligned}
H^{Q,\pm}_{s,t}(x,w) = \TE_\pm^{osc}(s,x - (t-s)\hw, w) H^{tr}_{s,t}(x,w) 
\end{aligned}
\end{equation}
for some oscillatory component $\TE_\pm^{osc}(s,x,w)$ that is a linear combination of $E_\pm^{osc,j}(s,x,w)$. Furthermore, the remainder $H^R_{s,t}(x,w)$ satisfies
$$ H^R_{s,t}(x,w) = H^{R,1}_{s,t}(x,w) + H^{R,2}_{s,t}(x,w)$$
where for $|\alpha|\le 1$,
$$\| \langle w\rangle^N\partial_x^\alpha H^{R,1}_{s,t}\|_{L^\infty_{x,w}} \lesssim \epsilon \langle s\rangle^{-2+\delta_1} , \qquad \| \langle w\rangle^N\partial_x^\alpha H^{R,2}_{s,t}\|_{L^\infty_{x,w}} \lesssim \epsilon \langle s\rangle^{-3+\delta_1}.$$
In addition, the following remainder bounds
\begin{equation}\label{bd-keyHR}
\begin{aligned} 
 \sup_w \| \langle w\rangle^N\cH^R_{s,t}\|_{L^2_{x}} \lesssim \epsilon \langle s\rangle^{-3/2+\delta_1} , \qquad  &\sup_x \| \langle w\rangle^N \cH^R_{s,t}\|_{L^2_w} \lesssim \epsilon \langle t\rangle^{-3/2}\langle s\rangle^{-3/2+\delta_1} ,
 \end{aligned}\end{equation}
 hold for the functions $\partial_s^2 \partial_x^\alpha H^{tr}_{s,t}$, $\partial_s \partial_x^\alpha H^{R,1}_{s,t}$, and $ \partial_x^\alpha H^{R,2}_{s,t}$ for $|\alpha|\le 1$.

 \end{proposition}

Let us comment on the proposition. The interaction kernel $H_{s,t}(x,w)$ describes the deformation of the nonlinear characteristics from the free transport dynamics. The expansion \eqref{decomp-Hst} resembles that of the Jacobian determinant $\det(\nabla_v \Psi_{s,t})$ in \eqref{decomp-detPsi}, extracting the oscillation of order $s^{-3/2}$ and the transport component $H^{tr}_{s,t}(x,w)$ which has slowest decay at order $s^{-1}$, whose $s$-derivatives gain extra decay. We also single out the quadratic term $H^Q_{s,t}(x,t)$, which decays at order $s^{-5/2}$, whose $s$-derivatives do not gain extra decay. What is important in the proposition is that the remainder or its $s$-derivative is of order $s^{-3}$. 
In addition, under the bootstrap assumptions that for $|\alpha|\le 1$, we have    
\begin{equation}\label{bd-keyH}
\begin{aligned}
\| \langle w\rangle^N \partial_x^\alpha T^{osc}_{s,t}\|_{L^\infty_{x,w}} \lesssim \epsilon \langle t\rangle^{-3/2} , \qquad & \| \langle w\rangle^N\partial_x^\alpha H^{osc}_{s,t}\|_{L^\infty_{x,w}} \lesssim \epsilon \langle s\rangle^{-3/2} ,
\\
 \| \langle w\rangle^N\partial_x^\alpha H^{tr}_{s,t}\|_{L^\infty_{x,w}} \lesssim \epsilon \langle s\rangle^{-1} , \qquad & \| \langle w\rangle^N\partial_x^\alpha H^Q_{s,t}\|_{L^\infty_{x,w}} \lesssim \epsilon \langle s\rangle^{-5/2} .
\end{aligned}\end{equation}

\begin{proof}[Proof of Proposition \ref{prop-SR}] 
In view of \eqref{straightX}, we may introduce the change of variables $w = \Psi_{s,t}(x,v)$ in the second integral defining $S^\mu$ in \eqref{Smu-recall}, leading to \eqref{key-SR}, where we have set   
\begin{equation}\label{def-Hst}
H_{s,t}(x,w) = \nabla_v \mu(w) - \frac{\nabla_v \mu(V_{s,t}(x,v))}{\det(\nabla_v\Psi_{s,t}(x,v))}
\end{equation}
for each $(x,w) \in \RR^3\times \RR^3$. In the above, $v = \widetilde \Psi_{s,t}(x,w)$, the inverse map of $v\mapsto w = \Psi_{s,t}(x,v)$. Note that since $\|\det\nabla_v\Psi_{s,t} - 1\|_{L^\infty_{x,v}} \lesssim \epsilon \langle s\rangle^{-1}$ by Proposition \ref{prop-detPsi}, the map is a diffeomorphism from $\RR^3$ into itself, and the inverse map $v = \widetilde \Psi_{s,t}(x,w)$ is well-defined. It remains to study $H_{s,t}(x,w)$, which can be 
expanded via the standard Taylor's expansion
\begin{equation}\label{recall-deFGst}
\begin{aligned}
H_{s,t}(x,w) 
&  =   \nabla_w \mu(w)  - \nabla_v \mu(V_{s,t}) + \nabla_v \mu(V_{s,t})  \frac{\det(\nabla_v\Psi_{s,t}) - 1}{\det(\nabla_v\Psi_{s,t})} 
\\
&  =   \nabla^2_w\mu(w) ( w -V_{s,t}) + \nabla_w \mu(w) (\det(\nabla_v\Psi_{s,t}) - 1)  
\\&\quad + \nabla_w \mu(w) (\det(\nabla_v\Psi_{s,t}) - 1)^2  
 + \nabla^2_w\mu(w) ( w -V_{s,t}) (\det(\nabla_v\Psi_{s,t}) - 1) 
 \\&\quad + H^{R,0}_{s,t} (x,w).
\end{aligned}
\end{equation}
We stress that on the right, $V_{s,t} = V_{s,t}(x,v)$ and $\Psi_{s,t} = \Psi_{s,t}(x,v)$, with $v = \widetilde \Psi_{s,t}(x,w)$, the inverse map of $w = \Psi_{s,t}(x,v)$. The remainder $H^{R,0}_{s,t}(x,w)$ collects the quadratic or higher order terms in $V_{s,t}-w$ and cubic or higher order terms in $\det(\nabla_v\Psi_{s,t}) - 1$. Recall that $\| V_{s,t} - v\|_{L^\infty_{x,v}} \lesssim \epsilon \langle s\rangle^{-3/2}$, $\| \Psi_{s,t} - v\|_{L^\infty_{x,v}} \lesssim \epsilon \langle s\rangle^{-3/2}$, and $\| \det(\nabla_v\Psi_{s,t})-1\|_{L^\infty_{x,v}}\lesssim \epsilon \langle s\rangle^{-1}$.  It thus follows by construction that 
$$ \| \langle w\rangle^N H^{R,0}_{s,t}\|_{L^\infty_{x,w}} \lesssim \epsilon^2 \langle s\rangle^{-3} $$
which can therefore be put into the remainder $H^R_{s,t}(x,w)$. Let us bound each of the remaining terms in the above expression.

\subsubsection*{Linear term: $\nabla^2_w\mu(w) ( w -V_{s,t}) $.}

Let us start with the term  $\nabla^2_w\mu(w) ( w -V_{s,t}) $ in $H_{s,t}(x,w) $. Recalling the expansion \eqref{decomp-Vst} on $V_{s,t}(x,v)-v$ and the expansion \eqref{decomp-Psist} on $\Psi_{s,t}(x,v)-v$, 
we have 
\begin{equation}\label{recall-expPsi1}
\begin{aligned}
w - V_{s,t}(x,v)  &= \Psi^{osc}_{s,t}(x,v)  - V^{osc}_{s,t}(x,v) + \Psi^{tr}_{s,t}(x,v) - V^{tr}_{s,t}(x,v)+ \Psi^{R}_{s,t}(x,v).
\end{aligned}
\end{equation}
Recalling the bounds from \eqref{bounds-Vst} and \eqref{bounds-Wst}, the last three terms can be put into the remainder $H_{s,t}^R(x,w)$ with the stated remainder bounds.  As for the oscillation component, we recall 
$$  V^{osc}_{s,t}(x,v) = \sum_\pm E^{osc,1}_{\pm}(s,X_{s,t}(x,v),V_{s,t}(x,v)) .
$$
Since $w= \Psi_{s,t}(x,v)$, we have $X_{s,t}(x,v) = x - (t-s)\hw$. On the other hand, using \eqref{recall-expPsi1}, we note that $\| w - V_{s,t}\|_{L^\infty_{x,v}} \lesssim \epsilon \langle s\rangle^{-3/2}$. Therefore, in the above expression, we may replace $V_{s,t}(x,v)$ by $w$,
leaving a remainder of the form 
$$E^{osc,1}_{\pm}(s,X_{s,t}(x,v),V_{s,t}(x,v))  - E^{osc,1}_{\pm}(s,x - (t-s)\hw,w) 
$$
which is bounded by $\| E^{osc,1}_{\pm}(s)\|_{L^\infty_{x,v}} \| V_{s,t} - w\|_{L^\infty_{x,v}} \lesssim \epsilon^2 \langle s\rangle^{-3}$, which obeys the remainder bounds for $H^R_{s,t}$. The oscillation in $\Psi^{osc}_{s,t}(x,v) $ is treated similarly.

\subsubsection*{Linear term: $ \nabla_w \mu(w) (\det(\nabla_v\Psi_{s,t}) - 1)  $.}

Next, we consider the term $ \nabla_w \mu(w) (\det(\nabla_v\Psi_{s,t}) - 1)  $ in $H_{s,t}(x,w)$. Recalling \eqref{decomp-detPsi}, we write 
\begin{equation}\label{recall-expdetPsi}
 \det (\nabla_v \Psi_{s,t}(x,v) ) - 1 = \Big( \Psi^{Tosc,1}_{s,t} + \Psi^{osc,1}_{s,t}+ \Psi^{tr,1}_{s,t}+ \Psi^{Q,1}_{s,t}+ \Psi^{R,1}_{s,t}\Big)(x,v), 
\end{equation}
in which $v = \widetilde \Psi_{s,t}(x,w)$. Recalling the bounds obtained in Proposition \ref{prop-detPsi}, $ \Psi^{R,1}_{s,t}(x,v)$ can be put into the remainder $H_{s,t}^R(x,w)$, obeying the stated respective bounds. Let us treat the remaining terms in the expansion.  
First, we put $\Psi^{Tosc,1}_{s,t}(x,w)$ into $T_{s,t}^{osc}(x,w)$, leaving an error
$$ \Psi^{Tosc,1}_{s,t}(x,v) - \Psi^{Tosc,1}_{s,t}(x,w)$$
which can be put into the remainder $H_{s,t}^R(x,w)$, since $\| w - v\|_{L^\infty_{x,v}} \lesssim \epsilon \langle s\rangle^{-3/2}$, recalling $w = \Psi_{s,t}(x,v)$. Next, the oscillation term $\Psi^{osc,1}_{s,t}(x,v)$ is treated similarly as done above for $ V^{osc}_{s,t}(x,v) $. We now treat the term $\Psi^{tr,1}_{s,t}(x,v) $, which is of order $s^{-1}$ and cannot be put into the remainder. 
We recall from  Proposition \ref{prop-detPsi} that 
$$ 
\begin{aligned}
\Psi^{tr,1}_{s,t}(x,v) 
& = \int_s^t  
\frac{(t-\tau) (\tau-s)}{t-s}  \nabla_v\hv \cdot
\nabla_xQ^{tr}(\tau,X_{\tau,t}(x,v), V_{\tau,t}(x,v)) \; d\tau ,
\end{aligned}
$$
which defines $H_{s,t}^{tr}(x,w) $, upon replacing $X_{\tau,t}(x,v)$ and $V_{\tau,t}(x,v))$ by the straighten characteristics. Indeed, 
using \eqref{straightX}, \eqref{decay-Xst}, and the definition $w=\Psi_{s,t}(x,v)$, we bound  
\begin{equation}\label{approxXVst}
\begin{aligned}
|X_{\tau,t}(x,v)  - x + (t-\tau)\hw  |&= (t-\tau) |\hPsi_{s,t}(x,v) - \hPsi_{\tau,t}(x,v) |
\\
&\lesssim \epsilon \langle s\rangle^{-1} .
\end{aligned}\end{equation}
Using this, we can replace $Q^{tr}(\tau, X_{\tau,t}(x,v),V_{\tau,t}(x,v))$ in the integral for $\Psi^{tr,1}_{s,t}(x,v) $ by $Q^{tr}(\tau, x-(t-\tau)\hw ,w)$, leaving an error 
$$
\begin{aligned} 
&
\int_s^t  
\frac{(t-\tau) (\tau-s)}{t-s}  
\Big[ \nabla_xQ^{tr}(\tau,X_{\tau,t}(x,v), V_{\tau,t}(x,v)) - 
\nabla_xQ^{tr}(\tau,x - (t-\tau)\hw, w)\Big]  \; d\tau 
\\
& \lesssim \int_s^t  
\langle \tau\rangle 
\Big[\| \partial_x^2Q^{tr}(\tau)\|_{L^\infty_{x,v}} |X_{\tau,t}(x,v) - x + (t-\tau)\hw|  + \| \partial_x\partial_vQ^{tr}(\tau)\|_{L^\infty_{x,v}} |V_{\tau,t}(x,v)-w|\Big]  \; d\tau 
\end{aligned}$$
Using \eqref{approxXVst} and the fact that $\| \partial_x^\alpha \partial_v^\beta Q^{tr}(\tau)\|_{L^\infty_{x,v}}\lesssim \epsilon \langle \tau\rangle^{-3+\delta_1}$ for $|\alpha|+|\beta|\le 2$, the above integral is bounded by 
$$
\begin{aligned} 
\epsilon^2\int_s^t  
\langle \tau\rangle^{-2+\delta_1} \langle s\rangle^{-1}  \; d\tau \lesssim \epsilon^2 \langle s\rangle^{-2+\delta_1},
\end{aligned}$$
which obeys the stated remainder bounds. It is also direct that the $s$-derivative of the integral gains an extra decay of order $s^{-1}$. Finally, we study the term $\Psi_{s,t}^Q(x,v)$ in the expansion of $\det (\nabla_v \Psi_{s,t}(x,v) ) - 1$, which we recall
$$\Psi^{Q,1}_{s,t}(x,v) = c_0\Psi^{osc,1}_{s,t}(x,v) \Psi^{tr,1}_{s,t}(x,v).$$
This contributes into $H^Q_{s,t}(x,w)$, upon replacing $v$ by $w$ similarly as done above, leaving an admissible remainder. Note that the $s$-derivative of $H^Q_{s,t}(x,w)$ may not gain any extra decay.

\subsubsection*{Quadratic term: $\nabla_w \mu(w) (\det(\nabla_v\Psi_{s,t}) - 1)^2 $.}
Next, we study the quadratic term $\nabla_w \mu(w) (\det(\nabla_v\Psi_{s,t}) - 1)^2 $ in $H_{s,t}(x,w)$. Similarly as done above, using the expansion \eqref{recall-expdetPsi}, we write 
$$ ( \det (\nabla_v \Psi_{s,t}) - 1)^2 = - 2 \Psi^{osc,1}_{s,t} (x,v) \Psi^{tr,1}_{s,t} (x,v)+ H^{R,1}_{s,t}(x,w),
$$ 
in which the first term contributes into $H^Q_{s,t}(x,w)$, while the second term $H^{R,1}_{s,t}(x,w)$ is of order $s^{-2}$ and its $s$-derivative is of order $s^{-3}$. 

\subsubsection*{Quadratic term: $\nabla^2_w\mu(w) ( w -V_{s,t}) (\det(\nabla_v\Psi_{s,t}) - 1) $.}

Similarly to the previous case, using \eqref{recall-expPsi1} and \eqref{recall-expdetPsi}, we write 
$$( w -V_{s,t}) (\det(\nabla_v\Psi_{s,t}) - 1)  = \Big( \Psi^{osc}_{s,t}(x,v)  - V^{osc}_{s,t}(x,v) \Big)\Psi^{tr,1}_{s,t} (x,v)+ H^{R,2}_{s,t}(x,w) $$
in which the first term again contributes into $H^Q_{s,t}(x,w)$, while the second term $H^{R,2}_{s,t}(x,w)$ is already of order $s^{-3}$ and thus can be put into the remainder $H^R_{s,t}(x,w)$. 

\subsubsection*{$L^2$ estimates.}

Finally, the $L^2_x$ and $L^2_v$ estimates on the remainder follow directly from those established in \eqref{Lpbounds-Vosc}, \eqref{Lpbounds-Wosc}, and \eqref{Lpbounds-detPsi}. Note that $L^2_x$ bounds costs a decay of order $s^{-3/2}$ from their corresponding $L^\infty_x$ bounds, while $L^2_v$ estimates replace decay of order $s^{-3/2}$ into decay of order $t^{-3/2}$. 
\end{proof}

\subsection{Some useful lemmas}


In view of \eqref{key-SR}, we shall first prove some useful lemmas that we shall use repeatedly in the proof to bound the nonlinear sources.

\begin{lemma}\label{lem-auxSr} 
Suppose that for $1\le p \le \infty$,
\begin{equation}\label{condauxSr}
\| \cE^r(t)\|_{L^p_x}\lesssim  \langle t\rangle^{-3(1-1/p)} , 
\qquad \int_0^t \|\langle v\rangle^5 \cH_{s,t}\|_{L^\infty_{x,v}}  \; ds \lesssim  1.
\end{equation}
Then, for $1\le p\le \infty$, there hold 
\begin{equation}\label{est-auxSr} 
\Big\| \int_0^t \int_{\RR^3} \cE^r(s,x - (t-s)\hv ) \cdot \cH_{s,t}(x,v) \; dv ds\Big\|_{L^p_x} \lesssim \langle t\rangle^{-3(1-1/p)} .
\end{equation}
\end{lemma}

\begin{proof} First, it is direct to bound 
$$
\begin{aligned}
\Big\|  \int_0^t \int 
\cE^r(s,x - (t-s)\hv ) \cdot \cH_{s,t}(x,v) \; dv ds\Big\|_{L^p_x} 
&\lesssim  \int_0^t \int 
\| \cE^r(s)\|_{L^p_x} \langle v\rangle^{-4} \|\langle v\rangle^4 \cH_{s,t}\|_{L^\infty_{x,v}} \; dv ds 
\\
&\lesssim  \int_0^t \langle s\rangle^{-3 +3/p} \|\langle v\rangle^4 \cH_{s,t}\|_{L^\infty_{x,v}} \; ds,
 \end{aligned}$$
for any $p\ge 1$, which in particular proves \eqref{est-auxSr} for $p=1$, using \eqref{condauxSr}. This also gives the desired estimate for the integral over $s\in (t/2,t)$, since $s\ge t/2$. It remains to prove \eqref{est-auxSr} for the integral over $s\in (0,t/2)$ in $L^\infty_x$. In this case, we shall make use of the dispersion of the transport to bound the integral. Indeed, introduce the change of variables $y = x - (t-s)\hv $ and note that   
$v_{s,t}(x,y) = {\hv_{s,t}(x,y)}/{\sqrt{1-|\hv_{s,t}(x,y)|^2}}$ with $\hv_{s,t}(x,y)= (x-y)/{(t-s)}$. 
Observe that the Jacobian satisfies 
$$
|\det (\nabla_v\hv_{s,t}(x,y))|^{-1} \le C_0 \langle v_{s,t}(x,y)\rangle^5.
$$ 
Therefore, we bound  
$$
\begin{aligned}
\Big| 
&\int_0^{t/2}\int 
\cE^r(s,x - (t-s)\hv ) \cdot \cH_{s,t}(x,v)\; dv ds\Big|
\\&\le \int_0^{t/2} \int |
\cE^r(s,y) \cdot \cH_{s,t}(x,v_{s,t}(x,y))|\; \frac{dy ds}{(t-s)^3\det (\nabla_v\hv_{s,t}(x,y))}
\\
&\lesssim  \int_0^{t/2} (t-s)^{-3}
\|\cE^r(s)\|_{L^1_y} \|\langle v\rangle^5\cH_{s,t}\|_{L^\infty_{x,v}} \; ds 
\lesssim  t^{-3}\int_0^{t/2} 
\|\langle v\rangle^5 \cH_{s,t}\|_{L^\infty_{x,v}} \; ds, 
 \end{aligned}$$
which proves \eqref{est-auxSr} for $p=\infty$, noting $(t-s)\ge t/2$, since $s\le t/2$. The lemma follows. 
\end{proof}

\begin{lemma}\label{lem-auxSosc} 
Suppose that for $2\le p\le \infty$,
$$ 
\| \cE^{osc}(t)\|_{L^p_x}\lesssim  \langle t\rangle^{-3(1/2-1/p)} ,
$$
$$\int_0^t \sup_v\| \langle v\rangle^2\cH_{s,t} \|_{L^2_x} \; ds \lesssim 1, \qquad  \int_0^t \sup_x\| \langle v\rangle^{5/2}\cH_{s,t} \|_{L^2_v} \; ds \lesssim \langle t\rangle^{-3/2} .$$
Then, for $1\le p\le \infty$, there hold 
\begin{equation}\label{est-auxSosc} 
\Big\| \int_0^t \int_{\RR^3} \cE^{osc}(s,x - (t-s)\hv ) \cdot \cH_{s,t}(x,v) \; dv ds\Big\|_{L^p_x} \lesssim \langle t\rangle^{-3(1-1/p)} 
.
\end{equation}
\end{lemma}

\begin{proof} Unlike in the previous lemma, note that the ``oscillatory'' field $\cE^{osc}(t,x)$ decays at a slower rate, and in particular does not decay in $L^2_x$. For this reason, we need to make use of dispersive from the kernel $\cH_{s,t}$. 
We first note that 
\begin{equation}\label{dispEr} 
\int 
|\cE^{osc}(s,x - (t-s)\hv )|^2 \langle v\rangle^{-5} \; dv \lesssim \langle t\rangle^{-3}
\end{equation}
for any $0\le s\le t$. Indeed, the bound is clear for $s\ge t/2$ using the $L^\infty$ bound on $E^{osc}(s,x)$. On the other hand, for $s\le t/2$, we introduce the change of the variable $y = x-(t-s)\hv$ as done in the previous lemma, giving again the decay of order $t^{-3}$. Therefore,  
we can bound 
$$
\begin{aligned}
\Big| 
&\int_0^t \int 
\cE^{osc}(s,x - (t-s)\hv ) \cdot \cH_{s,t}(x,v) \; dv ds\Big |
\\
&\lesssim \int_0^t \Big(\int 
|\cE^{osc}(s,x - (t-s)\hv )|^2 \langle v\rangle^{-5} \; dv \Big)^{1/2} \Big( \int |\cH_{s,t}(x,v)|^2 \langle v\rangle^5 \; dv\Big)^{1/2} ds
\\
&\lesssim \langle t\rangle^{-3/2}\int_0^t \sup_x\| \langle v\rangle^{5/2}\cH_{s,t} \|_{L^2_v} ds,
 \end{aligned}$$
 which gives the $L^\infty_x$ estimate, upon using the assumption on $H_{s,t}$. 
Similarly, we bound 
$$
\begin{aligned}
\Big\| 
\int_0^t \int 
\cE^{osc}(s,x - (t-s)\hv ) \cdot \cH_{s,t}(x,v) \; dv ds\Big \|_{L^1_x}
&\lesssim \int_0^t\int 
\|\cE^{osc}(s)\|_{L^2_x} \| \cH_{s,t}(\cdot, v)\|_{L^2_x}\; dv ds
\\
&\lesssim \int_0^t \sup_v\| \langle v\rangle^2\cH_{s,t} \|_{L^2_x} ds,
 \end{aligned}$$
 which is finite by the assumption on $H_{s,t}$. The $L^p$ estimates follow from the interpolation between $L^1$ and $L^\infty$ bounds. 
\end{proof}

\subsection{Contribution from the initial data}\label{sec-S0}


In this section, we prove the bounds on the source term $ S^0(t,x)$ defined as in \eqref{nonlinear-S}.

\begin{proposition}\label{prop-bdS0} There hold
\begin{equation}\label{Lpbounds-cS0}
\begin{aligned}
\sup_{q\in \ZZ} 2^{(1-\delta)q} \|P_q S^0(t)\|_{L^p_x} + \sum_{|\alpha|\le 1}\|\partial^\alpha_x S^0(t)\|_{L^p_x} &\lesssim \epsilon_0  \langle t\rangle^{-3(1-1/p)} 
\end{aligned}\end{equation}
for any $1 \le p \le \infty$ and $\delta \in (0,1)$. In addition, 
\begin{equation}\label{Hsbounds-cS0}
\begin{aligned}
\|\partial_x^\alpha S^0(t)\|_{L^p_x}   \le \epsilon \langle t\rangle^{-3(1-1/p) + \delta_\alpha} 
, \qquad 
\|S^0(t)\|_{H^{\alpha_0}_x} &\lesssim \epsilon \langle t\rangle^{\delta_1},
\end{aligned}
\end{equation}
for all $|\alpha|\le N_0-1$, with $\delta_\alpha = \frac{|\alpha|}{N_0-1}$. 
\end{proposition}

\begin{proof}Indeed, recalling the definition of $S^0(t,x)$ from \eqref{nonlinear-S} and introducing the change of variables $y = X_{0,t}(x,v)$, we bound 
$$
\begin{aligned}
S^0(t,x) & = \int_{\RR^3} f_0(X_{0,t}(x,v) , V_{0,t}(x,v)) \, dv
= \int_{\RR^3} f_0(y, V_{0,t}(x,v)) \frac{dy}{ \det (\nabla_v X_{0,t}(x,v))}
\end{aligned}
$$
in which $v = \widetilde X_{0,t}(x,y)$, the inverse map of $v\mapsto y = X_{0,t}(x,v)$. Using \eqref{Jacobian-dxdv}, we thus obtain 
$$
\begin{aligned}
|S^0(t,x)| &\lesssim t^{-3}  \int_{\RR^3} |f_0(y, V_{0,t}(x,v_{0,t}(x,y))) |  \langle v\rangle^5\; dy
 \lesssim t^{-3}\int_{\RR^3} \sup_v \langle v\rangle^5 |f_0(y, v)| dy
\end{aligned}
$$
in which we noted that $\langle v\rangle \le 2 \langle V_{0,t}(x,v)\rangle$, using $\| V_{0,t} - v\|_{L^\infty_{x,v}} \lesssim \epsilon$. This gives the $L^\infty$ estimate. On the other hand, since $(x,v)\mapsto (X_{0,t},V_{0,t})$ is a volume-preserving map, we have 
$$
\begin{aligned}
\|S^0(t)\|_{L^1_x} &\le \iint |f_0(X_{0,t}(x,v) , V_{0,t}(x,v))| \; dxdv = \| f_0\|_{L^1_{x,v}} .
\end{aligned}
$$
The $L^p$ bounds in \eqref{Lpbounds-cS0} follow from the interpolation between $L^1$ and $L^\infty$ spaces. Finally, we compute  
$$
\begin{aligned}
\partial_x S^0(t,x) &= \int_{\RR^3} \partial_xX_{0,t} \cdot \nabla_x f_0(X_{0,t}(x,v) , V_{0,t}(x,v)) \, dv
\\
&\quad + \int_{\RR^3}  \partial_xV_{0,t} \cdot \nabla_v f_0(X_{0,t}(x,v) , V_{0,t}(x,v)) \, dv .
\end{aligned} 
$$  
Using Proposition \ref{prop-Dchar}, we have the boundedness of both 
$ \partial_x X_{0,t}$ and $\partial_xV_{0,t} $. 
Therefore, the estimates for $\partial_x S^0(t,x)$ follow similarly as done above for $S^0(t,x)$. Finally, as for the estimate on the Littlewood-Paley projection, we bound, using the classical Bernstein inequalities 
\begin{equation}\label{LP-interpolate}
\begin{aligned}
2^{(1-\delta)q} \|P_q f\|_{L^p_x}  
= \|P_qf\|^\delta_{L^p_x} \| 2^qP_qf\|^{1-\delta}_{L^p_x}
\lesssim \|f\|^\delta_{L^p_x} \|\partial_xf\|^{1-\delta}_{L^p_x},
\end{aligned}\end{equation}
for any $q\in \ZZ$. This gives the desired estimate on $\|P_qS^0(t)\|_{L^p_x}$, completing the proof of \eqref{Lpbounds-cS0}. 

Finally, we prove the decay and boundedness for high derivatives of $S(t)$. Using the Faa di Bruno's formula
for derivatives of a composite function, see \eqref{daVosci}, we compute 
\begin{equation}\label{daS}
\begin{aligned}  
\partial_x^\alpha S^0(t,x) & = \sum_{1\le |\mu|\le |\alpha|} C_{\beta,\mu}  \int_{\RR^3} 
\partial_x^\mu f_0(X_{0,t}, V_{0,t})\prod_{1\le |\beta|\le |\alpha|} (\partial_x^\beta X_{0,t})^{k_\beta} \; dv + \mathrm{l.o.t.}
\end{aligned}\end{equation}
in which $\mathrm{l.o.t.}$ denotes terms involving derivatives of $\partial_x^\beta V_{0,t}$, which are better than those for $\partial_x^\beta X_{0,t}$. Consider the case when $|\alpha|\le N_0-1$. In this case, we may use the decay and boundedness of $\partial_x^\beta X_{0,t}$ from Proposition \ref{prop-HDchar} to bound 
$$ 
\begin{aligned}
\int_{\RR^3} 
\partial_x^\mu f_0(X_{0,t}, V_{0,t})\prod_{1\le |\beta|\le |\alpha|} (\partial_x^\beta X_{0,t})^{k_\beta} \; dv
&\lesssim 
\int_{\RR^3} 
|\partial_x^\mu f_0(X_{0,t}, V_{0,t})| \prod_{2\le |\beta| \le |\alpha|} \langle t\rangle^{\delta_\beta}
\; dv
\\&\lesssim \langle t\rangle^{\delta_\alpha}
\int_{\RR^3} 
|\partial_x^\mu f_0(X_{0,t}(x,v), V_{0,t})| \; dv
\end{aligned}
$$
in which we have used $\sum_\beta \delta_\beta k_\beta = \delta_\alpha$. Now following the same proof as done above, we obtain the decay of $\int_{\RR^3} 
|\partial_x^\mu f_0(X_{0,t}(x,v), V_{0,t})| \; dv$
in $L^p_x$ as desired. Finally, for $|\alpha| = N_0$, we write 
$$ 
\begin{aligned}
&\sum_{1\le |\mu|\le |\alpha|} \int_{\RR^3} 
\partial_x^\mu f_0(X_{0,t}, V_{0,t})\prod_{1\le |\beta|\le N_0} (\partial_x^\beta X_{0,t})^{k_\beta} \; dv 
\\&= \int_{\RR^3} 
\partial_xf_0(X_{0,t}, V_{0,t})(\partial_x^{\alpha_0} X_{0,t}) \; dv 
+\sum_{2\le |\mu|\le |\alpha|}  \int_{\RR^3} 
\partial_x^\mu f_0(X_{0,t}, V_{0,t})\prod_{1\le |\beta|\le N_0-1} (\partial_x^\beta X_{0,t})^{k_\beta} \; dv.
\end{aligned}
$$
The last integral term involves $\partial_x^\beta X_{0,t}$ with $|\beta|\le N_0-1$, and is therefore treated as done above, using Proposition \ref{prop-HDchar}. As for the first term, we bound  
$$ 
\begin{aligned}
\int_{\RR^3} 
\partial_xf_0(X_{0,t}, V_{0,t})(\partial_x^{\alpha_0} X_{0,t}) \; dv
&\lesssim \Big(
\int_{\RR^3} 
|\partial_xf_0(X_{0,t}, V_{0,t})| \; dv\Big)^{1/2} \Big( \int_{\RR^3} 
\partial_xf_0(X_{0,t}, V_{0,t})|\partial_x^{\alpha_0} X_{0,t}|^2 \; dv\Big)^{1/2}
\\
&\lesssim \epsilon_0 \langle t\rangle^{-3/2}\Big( \int_{\RR^3} 
\partial_xf_0(X_{0,t}, V_{0,t})|\partial_x^{\alpha_0} X_{0,t}|^2 \; dv\Big)^{1/2}.\end{aligned}
$$
Therefore, using Proposition \ref{prop-HDBchar} and the fact that $\|\langle v\rangle^4\partial_xf_0\|_{L^\infty_{x,v}} \lesssim 1$, we bound 
$$ 
\begin{aligned}
\Big\|
\int_{\RR^3} 
\partial_xf_0(X_{0,t}, V_{0,t})(\partial_x^{\alpha_0} X_{0,t}) \; dv\Big\|^2_{L^2_x}
&\lesssim \epsilon^2_0 \langle t\rangle^{-3}\iint_{\RR^3\times \RR^3} 
\partial_xf_0(X_{0,t}, V_{0,t})|\partial_x^{\alpha_0} X_{0,t}|^2 \; dxdv
\\&\lesssim \epsilon^2_0 \langle t\rangle^{-3}\int_{\RR^3} 
\langle v\rangle^{-4} \sup_v\|\partial_x^{\alpha_0} X_{0,t}\|_{L^2_x}^2 dv
\lesssim \epsilon^2_0 \langle t\rangle^{2\delta_1},
\end{aligned}
$$
proving \eqref{Hsbounds-cS0}. This completes the proof of the proposition. 
\end{proof}

 
 \subsection{Nonlinear interaction with $E^r$}\label{sec-SEr}


In this section, we treat the nonlinear interaction with $E^r$, namely the following source term
\begin{equation}\label{def-Sr}
\begin{aligned}
S^r(t,x) & =  \int_0^t \int_{\RR^3} E^r(s,x - (t-s)\hw ) \cdot H_{s,t}(x,w) \; dw ds
\end{aligned}\end{equation}
where $H_{s,t}(x,w) $ is defined as in Proposition \ref{prop-SR}. We shall prove the following. 

\begin{proposition}\label{prop-decayEHR} There hold
$$
\begin{aligned}
 \|S^{r} (t)\|_{L^p_x} & \lesssim  \epsilon^2 \langle  t\rangle^{-3(1-1/p)} \log t
\\
\sup_{q\in \ZZ} 2^{(1-\delta)q} \|P_qS^{r}(t)\|_{L^p_x}  + \|\partial_x S^{r}(t)\|_{L^p_x} & \lesssim \epsilon^2 \langle t\rangle^{-3(1-1/p)} 
\end{aligned}$$
for any $1 \le p \le \infty$ and $\delta \in (0,1)$, where $P_q$ denotes the Littlewood-Paley projection on the dyadic interval $[2^{q-1}, 2^{q+1}]$. 
\end{proposition}

\begin{proof} From Proposition \ref{prop-SR}, we write 
$$H_{s,t}(x,w) = H^{tr}_{s,t}(x,w) + H^{R,1}_{s,t}(x,w)$$
where the remainder $H^{R,1}_{s,t}(x,w)$ satisfies $\| \langle v\rangle^N \partial_x^\alpha H^{R,1}_{s,t}\|_{L^\infty_{x,w}} \lesssim \epsilon \langle s\rangle^{-3/2}$ for $|\alpha|\le 1$. Let $S^{r,tr}$ and $S^{r,R}$ be the corresponding source in \eqref{def-Sr} with kernel $H^{tr}_{s,t}$ and $H^{R,1}_{s,t}$, respectively. 
Since $R_{s,t}$ decays faster than $\langle s\rangle^{-1}$, Lemma \ref{lem-auxSr} can be applied for $S^{r,R}$, giving the stated bound without the log factor. The $x$-derivative estimates, and hence the estimates on the Littlewood-Paley projection, follow similarly, upon using  
\eqref{LP-interpolate}. 

It remains to bound the source term $S^{r,tr}(t,x)$ with the interaction kernel $H^{tr}_{s,t}(x,w)$. Recall that $\| \langle v\rangle^N H^{tr}_{s,t}\|_{L^\infty_{x,w} } \lesssim \epsilon \langle s\rangle^{-1}$, and therefore, a similar argument in the proof of Lemma \ref{lem-auxSr} yields 
$$
\begin{aligned}
\| S^{r,tr}(t)\|_{L^p_x} 
&\lesssim  \epsilon^2 \langle t\rangle^{-3(1-1/p)} \int_0^t \langle s\rangle^{-1} \; ds,
 \end{aligned}$$
which also gives the stated bound for $S^{r,tr}$ with a $\log t$ factor. 

Next, we turn to bound $\partial_xS^{r,tr}(t,x)$. The main point is to remove the $\log t$ factor, which is only present in the above analysis when $s\in (0,t/2)$. Let $S^{r,tr}_1(t,x)$ be the corresponding integral in $S^{r,tr}(t,x)$ over $s\in (0,t/2)$. 
We shall bound 
\begin{equation} \label{S1rtr}
\partial_x S_1^{r,tr}(t,x)  = \frac{d}{dx}\int_0^{t/2} \int_{\RR^3} E^r(s,x - (t-s)\hv) \cdot H^{tr}_{s,t}(x,v) \; dv ds,
\end{equation}
where 
$$
H^{tr}_{s,t}(x,v)  =a(v)\int_s^t  
\frac{(t-\tau) (\tau-s)}{t-s}  
\partial_xQ^{tr}(\tau,x - (t-\tau)\hv, v) \; d\tau, 
$$
for $a(v) =  \nabla_v \mu \nabla_v\hv$. We let 
$$ y: = x - (t-s)\hv $$
and denote the inverse map, which determines $v$ from $x, y, t, s$, by $v_{s,t}(x,y)$. We have 
$$
H^{tr}_{s,t}(x, v_{s,t}(x,y)) = a(v_{s,t}(x,y))\int_s^t  \frac{(\tau-s)(t-\tau)}{t-s} \partial_xQ^{tr}
\Bigl(\tau,  \frac{\tau-s}{t-s}x + \frac{t-\tau}{t-s}y, v_{s,t}(x,y) \Bigr) \; d\tau.
$$
Taking the $x$-derivative, we have 
$$
\begin{aligned}
&\frac{d}{dx}H^{tr}_{s,t}(x, v_{s,t}(x,y)) 
\\&= a(v_{s,t}(x,y)) \int_s^t  \frac{(\tau-s)^2(t-\tau)}{(t-s)^2} \partial^2_xQ^{tr}
\Bigl(\tau,  \frac{\tau-s}{t-s}x + \frac{t-\tau}{t-s}y, v_{s,t}(x,y) \Bigr) \; d\tau
\\&\quad +a'(v_{s,t}(x,y)) \partial_x v_{s,t}(x,y) \int_s^t  \frac{(\tau-s)(t-\tau)}{t-s} \partial_xQ^{tr}
\Bigl(\tau,  \frac{\tau-s}{t-s}x + \frac{t-\tau}{t-s}y, v_{s,t}(x,y) \Bigr) \; d\tau
\\&\quad +a(v_{s,t}(x,y)) \partial_x v_{s,t}(x,y) \int_s^t  \frac{(\tau-s)(t-\tau)}{t-s} \partial_vQ^{tr}
\Bigl(\tau,  \frac{\tau-s}{t-s}x + \frac{t-\tau}{t-s}y, v_{s,t}(x,y) \Bigr) \; d\tau .
\end{aligned}$$
For convenience, recalling $0\le s\le t/2$, we write 
\begin{equation}\label{dxHtrst}
\frac{d}{dx}H^{tr}_{s,t}(x, v_{s,t}(x,y))  = H^{tr,1}_{s,t}(x, y)  + H^{tr,R}_{s,t}(x, y) ,
\end{equation}
where 
$$
\begin{aligned}
H^{tr,1}_{s,t}(x, y) 
&= a(v_{s,t}(x,y)) \int_s^{t/2}  \frac{(\tau-s)^2(t-\tau)}{(t-s)^2} \partial^2_xQ^{tr}
\Bigl(\tau,  \frac{\tau-s}{t-s}x + \frac{t-\tau}{t-s}y, v_{s,t}(x,y)  \Bigr) \; d\tau
\end{aligned}$$
and $H^{tr,R}_{s,t}(x, y) $ collects the remaining integrals. Note that $v_{s,t}(x,y) = \hv_{s,t}(x,y) / \sqrt{1-|\hv_{s,t}(x,y)|^2}$, where $\hv_{s,t}(x,y) = \frac{x-y}{t-s}$. This yields $| \partial_x v_{s,t}(x,y) | \le (t-s)^{-1} \langle v_{s,t}(x,y)\rangle^3$, which gives an extra decay of order $t^{-1}$, since $s\le t/2$. Therefore, the last two integrals in $\frac{d}{dx}H^{tr}_{s,t}(x, v_{s,t}(x,y))$ are bounded by $\epsilon\langle v_{s,t}\rangle^{-N+3} \langle t\rangle^{-1}\langle s\rangle^{-1}$. On the other hand, the first integral over $\tau\in (t/2,t)$ is also bounded by $\epsilon\langle v_{s,t}\rangle^{-N} \langle t\rangle^{-1}$. This gives
\begin{equation}\label{dxHtrR}
\| \langle v\rangle^{N-3} H^{tr,R}_{s,t}\|_{L^\infty_{x,v}} \lesssim \epsilon \langle t\rangle^{-1} .
\end{equation}
On the other hand, for each $x$, we bound 
$$
\begin{aligned}
\Big\| \langle v_{s,t}(x,y) \rangle^{ N} H^{tr,1}_{s,t}(x, y) \Big\|_{L^p_y} 
&\lesssim  \langle t\rangle^{-1} \int_s^{t/2} \tau^2 \| \sup_v \partial^2_xQ^{tr}
\Bigl(\tau,  \frac{\tau-s}{t-s}x + \frac{t-\tau}{t-s}y,v \Bigr)\|_{L^p_y} \; d\tau
\\
&\lesssim  \langle t\rangle^{-1} \int_s^{t/2} \tau^2 \| \sup_v\partial^2_xQ^{tr}(\tau,\cdot,v)\|_{L^p_z} \; d\tau
\end{aligned}$$
in which we have introduced a change of variable $z = \frac{\tau-s}{t-s}x + \frac{t-\tau}{t-s}y$. Since both $s,\tau\in [0,t/2]$, we have
$$ \frac12 \le \frac{t-\tau}{t-s} \le 2$$
and so the Jacobian of the change of variable from $y$ to $z$ is bounded above and away from zero. Now, for $1<p<\infty$, we obtain 
\begin{equation}\label{dxHtr1}
\begin{aligned}
\sup_x\Big\| \langle v_{s,t}(x,y) \rangle^{ N} H^{tr,1}_{s,t}(x, y) \Big\|_{L^p_y} 
&\lesssim  \langle t\rangle^{-1} \int_s^{t/2} \tau^2 \|\sup_v\partial^2_xQ^{tr}(\tau,\cdot,v)\|_{L^p_z} \; d\tau
\\
&\lesssim \epsilon \langle t\rangle^{-1} \int_s^{t/2} \langle \tau\rangle^{-1 + 3/p}
\; d\tau
\lesssim \epsilon \langle t\rangle^{-1+3/p} .
\end{aligned}\end{equation}
Here in the above, we have used $\|\sup_v\partial^2_xQ^{tr}(\tau,\cdot,v)\|_{L^p_z} \lesssim  \langle \tau\rangle^{-3 + 3/p}$ for $p<\infty$, without a loss of order $\langle \tau \rangle^{\delta_1}$. Indeed, by definition \eqref{def-Qtr}, for $1<p<\infty$, we compute 
\begin{equation}\label{estd2Qtr}
\begin{aligned} 
\|\partial_x^2 Q^{tr}(t)\|_{L^p_xL^\infty_v}
&\lesssim \sum_\pm \|\partial_x^2 [a_\pm(i\partial_x) \phi_{\pm,1}^v(i\partial_x)F]\|_{L^p_xL^\infty_v}
+ \|\partial_x^2(E\cdot  \nabla_x E^{osc,2}_\pm)\|_{L^p_xL^\infty_v} + \|\partial_x^2 E^r(t)\|_{L^p_x}
\\
&\lesssim \sum_\pm \| F(t)\|_{L^p_x}
+ \|E(t)\|_{W^{2,2p}_x}^2+ \|E^r(t)\|_{W^{2,p}_x}
\\
&\lesssim \epsilon \langle t\rangle^{-3 + 3/p}
\end{aligned} 
\end{equation}
as desired, upon recalling the bootstrap assumptions \eqref{bootstrap-decayd2E} on $\partial_x^2 E$. This yields \eqref{dxHtr1}.

We are now ready to bound $S_1^{r,tr}(t,x)  $ defined as in \eqref{S1rtr}. First, we perform the change of variable $w \to y$ as above in (\ref{S1rtr}). This leads to 
$$
S_1^{r,tr}(t,x)  = \int_0^{t/2} (t-s)^{-3}\int_{\RR^3} E^r(s,y) H^{tr}_{s,t}(x, v_{s,t}(x,y)) J_{v,\hv }(v_{s,t}(x,y))\; dy ds
$$
with Jacobian determinant $J_{v,\hv }$ satisfying 
\begin{equation}\label{bounds-dxJac} 
\begin{aligned}
|J_{v,\hv } (v_{s,t}(x,y)| &\lesssim \langle v_{s,t}(x,y)\rangle^3, 
\\
|\frac{d}{dx}J_{v,\hv }(v_{s,t}(x,y)| &\lesssim (t-s)^{-1} \langle v_{s,t}(x,y)\rangle^6.
\end{aligned}\end{equation}
Therefore, 
$$
\begin{aligned}
\partial_x S_1^{r,tr}(t,x)  
&= \int_0^{t/2} (t-s)^{-3}\int_{\RR^3} E^r(s,y) H^{tr}_{s,t}(x, v_{s,t}(x,y)) \frac{d}{dx}[J_{v,\hv }(v_{s,t}(x,y))]\; dy ds
\\
&\quad + \int_0^{t/2} (t-s)^{-3}\int_{\RR^3} E^r(s,y) \frac{d}{dx}[H^{tr}_{s,t}(x, v_{s,t}(x,y))] J_{v,\hv }(v_{s,t}(x,y))\; dy ds.
\end{aligned}$$
Using the bootstrap assumption on $E^r$ and \eqref{bounds-dxJac}, and noting $(t-s)^{-1} \le 2t^{-1}$ since $s\in [0,t/2]$, the first integral term is estimated by 
$$
\begin{aligned}
&\int_0^{t /2}\int_{\RR^3} (t-s)^{-4} |E^r(s,y)|
| \langle v_{s,t}(x,y)\rangle^6 H^{tr}_{s,t}(x,v_{s,t}(x,y))|\; dy ds 
\\&\lesssim \langle t\rangle^{-4} \int_0^{t /2} \|E^r(s)\|_{L^1_y} 
\| \langle v\rangle^6 H^{tr}_{s,t}\|_{L^\infty_{x,v}}
ds 
\\&\lesssim \epsilon^2 \langle t\rangle^{-4} \int_0^{t /2} \langle s\rangle^{-1}\; ds
\lesssim \epsilon^2 \langle t\rangle^{-3} .
\end{aligned}$$
As for the second integral term in $\partial_x S_1^{r,tr}(t,x)  $, using \eqref{dxHtrst} and \eqref{dxHtrR}, we 
first bound 
$$
\begin{aligned}
\Big|&\int_0^{t/2} (t-s)^{-3}\int_{\RR^3} E^r(s,y)  H^{tr,R}_{s,t}(x,y)J_{v,\hv }(v_{s,t}(x,y))\; dy ds\Big|
\\&\lesssim \langle t\rangle^{-3} \int_0^{t /2}\int_{\RR^3} |E^r(s,y)|
\sup_{x,y}\Big | \langle v_{s,t}(x,y)\rangle^3 |H^{tr,R}_{s,t}(x,y) \Big|\; dy ds 
\\&\lesssim \epsilon \langle t\rangle^{-4} \int_0^{t /2} \|E^r(s)\|_{L^1_y} 
ds 
\lesssim \epsilon^2 \langle t\rangle^{-3}, 
\end{aligned}$$
as claimed. On the other hand, using \eqref{dxHtr1} with $p>3$, we bound the integral term involving $H^{tr,1}_{s,t}(x,y)$ as 
\begin{equation}\label{logbd-dxSrH1}
\begin{aligned}
\Big|&\int_0^{t/2} (t-s)^{-3}\int_{\RR^3} E^r(s,y)  H^{tr,1}_{s,t}(x,y)J_{v,\hv }(v_{s,t}(x,y))\; dy ds\Big|
\\&\lesssim \langle t\rangle^{-3} \int_0^{t /2}\int_{\RR^3} |E^r(s,y)| \Big| \langle v_{s,t}(x,y)\rangle^3 |H^{tr,1}_{s,t}(x,y) \Big|\; dy ds 
\\&\lesssim \langle t\rangle^{-3} \int_0^{t /2} \|E^r(s)\|_{L^{p'}} \Big \| \langle v_{s,t}(x,y)\rangle^3 |H^{tr,1}_{s,t}(x,y) \Big\|_{L^p_y}\;  ds 
\\&\lesssim \epsilon^2 \langle t\rangle^{-4 + 3/p} \int_0^{t /2} \langle s\rangle^{-3 +3/p'} 
ds 
\\&\lesssim \epsilon^2 \langle t\rangle^{-4 + 3/p} \langle t\rangle^{-2 +3/p'}   
\lesssim \epsilon^2 \langle t\rangle^{-3},  
\end{aligned}\end{equation}
provided $p'<3/2$ (or equivalently $p>3$), where $1/p+ 1/p' = 1$.  Combining, this proves $\| \partial_x S_1^{r,tr}(t)\|_{L^\infty_x} \lesssim \epsilon^2 \langle t\rangle^{-3}$ as claimed.

Next, we check the estimates in $L^1_x$. We first focus on the integral term involving $H^{tr,1}_{s,t} $, which reads 
$$
\begin{aligned}
S_1^{r,tr,1}(t,x) = \int_0^{t/2} \int_{\RR^3} E^r(s,x - (t-s)\hv ) H^{tr,1}_{s,t}(x,v) \; dvds.
\end{aligned}$$
Similarly as done above, using \eqref{estd2Qtr}, we first obtain the following pointwise bound, for $1<p<\infty$,   
$$
\begin{aligned}\| H^{tr,1}_{s,t}(\cdot,v)\|_{L^p_x} 
&\le |a(v)|\int_s^{t/2}  \frac{(\tau-s)^2(t-\tau)}{(t-s)^2} \|\sup_v\partial_x^2Q^{tr}(\tau,\cdot,v)\|_{L^p_x} \; d\tau
\\
&\lesssim \epsilon\langle v\rangle^{-N} \langle t\rangle^{-1} \int_s^{t/2} \langle \tau\rangle^{-1+3/p} \; d\tau
\lesssim \epsilon \langle v\rangle^{-N} \langle t\rangle^{-1 + 3/p}.
\end{aligned} 
$$
Therefore, for any pair $(p,p')$ with $1/p+1/p'=1$, we bound 
\begin{equation}\label{logbd-dxSrH1L1}
\begin{aligned}
\|S_1^{r,tr,1}(t)\|_{L^1_x} 
&\le \int_0^{t/2} \int_{\RR^3} \|E^r(s)\|_{L^{p'}_x} \|H^{tr,1}_{s,t}(\cdot,v)\|_{L^p_x} \; dvds
\\
&\lesssim \epsilon^2 \langle t\rangle^{-1 + 3/p}
 \int_0^{t/2} \langle s\rangle^{-3+3/p'} \int_{\RR^3}  \langle v\rangle^{-N}  \; dvds
\\
&\lesssim \epsilon^2 \langle t\rangle^{-1 + 3/p}
\langle t\rangle^{-2+3/p'} 
\lesssim \epsilon^2,
\end{aligned}\end{equation}
provided that $p>3$ (or equivalently, $p'<3/2$). As for the integral involving $H^{tr,R}_{s,t}$, we simply use $L^\infty_x$ estimates from \eqref{dxHtrR}, which avoids the log factor, since the kernel decays at order $t^{-1}$, instead of $s^{-1}$. This proves the $L^1$ estimate for the derivative term. 
The $L^p$ estimates follow from the standard interpolation.

Finally, we prove the desired estimates on $\|P_qS_1^{r,tr}(t)\|_{L^p_x} $. By definition, see \eqref{S1rtr}, we write 
$$
S_1^{r,tr}(t,x)  = \int_0^{t/2} J^{r,tr}(s,t,x)ds 
, \qquad  J^{r,tr}(s,t,x) = \int_{\RR^3} E^r(s,x - (t-s)\hv ) H^{tr}_{s,t}(x,v) \; dv.$$ 
The previous bounds (e.g., see \eqref{logbd-dxSrH1} and \eqref{logbd-dxSrH1L1}) yield that for any $p\ge 1$,  
$$
\begin{aligned}
 \|  J^{r,tr}(s,t)\|_{L^p_x} &\lesssim \epsilon^2 \langle t\rangle^{-3+3/p}\langle s\rangle^{-1}
 \\
  \| \partial_xJ^{r,tr}(s,t)\|_{L^p_x}  &\lesssim \epsilon^2 \langle t\rangle^{-4+3/p + 3/\tilde p}\langle s\rangle^{-3+ 3/\tilde p'}
  \end{aligned}$$ 
for any $\tilde p>3$, where $1/\tilde p+ 1/\tilde p' = 1$. Now, using the interpolation inequality \eqref{LP-interpolate}, for any $q\in \ZZ$, we obtain 
$$
\begin{aligned}
2^{(1-\delta)q} \|P_q J^{r,tr}(s,t) \|_{L^p_x}  
&\lesssim \| J^{r,tr}(s,t) \|^\delta_{L^p_x} \|\partial_x J^{r,tr}(s,t)\|^{1-\delta}_{L^p_x}
\\
&\lesssim \epsilon^2 \langle t\rangle^{-3+3/p} \langle t\rangle^{-(1-\delta)(1-3/\tilde p)}\langle s\rangle^{-\delta - (1-\delta)(3-3/\tilde p')} ,
\end{aligned}
$$
for any $\delta>0$. Integrating in time, we get 
$$
\begin{aligned}
2^{(1-\delta)q} \|P_q S_1^{r,tr}(t) \|_{L^p_x}   & \le \int_0^{t/2} 2^{(1-\delta)q} \|P_q J^{r,tr}(s,t) \|_{L^p_x}  \; ds
\\
&\lesssim \epsilon^2 \langle t\rangle^{-3+3/p} \langle t\rangle^{-(1-\delta)(1-3/\tilde p)} \int_0^{t/2} \langle s\rangle^{-\delta - (1-\delta)(3-3/\tilde p')} \; ds
\\
&\lesssim \epsilon^2 \langle t\rangle^{-3+3/p} 
\end{aligned}
$$
provided that $\tilde p>3$ and $\delta \in (0,1)$. 
This ends the proof of the proposition.
\end{proof}
 

\subsection{Quadratic oscillation $E^{osc} H^{osc}$}\label{sec-resHosc}


In this section, we study the effect of quadratic oscillations, namely the effect of the interaction of the oscillations of the electric field
with the oscillations induced by them on $\nabla_v f$. 
Precisely, we will bound the interaction integrals
\begin{equation}\label{0def-Rosc}
\begin{aligned}
\cS_{\pm,\pm}^{osc} (t,x)&=\int_0^t \int_{\RR^3}  E^{osc}_\pm(s,x - (t-s)\hv ,v)\cdot H_{s,t}^{osc,\pm}(x,v) \; dv ds 
\\
\cS_{\pm,\mp}^{osc} (t,x) &=  \int_0^t \int_{\RR^3} E^{osc}_{\pm} (s,x - (t-s)\hv ,v)\cdot H_{s,t}^{osc,\mp}(x,v)\; dv ds
\end{aligned}
\end{equation}
for each pair $(\pm,\pm)$ and $(\pm,\mp)$, where the interaction kernels $H_{s,t}^{osc,\pm}(x,v)$ are defined in Proposition \ref{prop-SR}. 
The terms $\cS_{\pm,\pm}^{osc}$ describe $(+,+)$ and $(-,-)$ interactions whereas the terms
$\cS_{\pm,\mp}^{osc}$ describe $(+,-)$ and $(-,+)$ interactions. Note that $(+,+)$ and $(-,-,)$ interactions leads to oscillations
with double frequencies and thus to small integrals. Interactions $(+,-)$ and $(-,+)$ lead to non oscillatory terms
which must be carefully studied. Lemma \ref{0lem-Qosc} show that there is an algebraic cancellation in these interactions, leading
to a contribution smaller than expected.

In this section we prove the following proposition. 

\begin{proposition}\label{prop-Soscpm}
There hold
$$ 
\begin{aligned}
\|\cS_{\pm,\pm}^{osc} (t)\|_{L^p_x} + \|\cS_{\pm,\mp}^{osc} (t) \|_{L^p_x} &\lesssim \epsilon^2 \langle  t\rangle^{-3(1-1/p)} \log t 
\\
\|\partial_x \cS_{\pm,\pm}^{osc} (t)\|_{L^p_x} + \|\partial_x \cS_{\pm,\mp}^{osc} (t) \|_{L^p_x} 
&\lesssim  \epsilon^2\langle  t\rangle^{-3(1-1/p)} 
\\
\sup_{q\in \ZZ} 2^{(1-\delta)q} ( \|P_q\cS_{\pm,\pm}^{osc} (t)\|_{L^p_x}   + 
\|P_q\cS_{\pm,\mp}^{osc} (t)\|_{L^p_x}) 
&\lesssim  \epsilon^2\langle  t\rangle^{-3(1-1/p)} 
\end{aligned}
$$
for any $1 \le p \le \infty$ and $\delta \in (0,1)$, where $P_q$ denotes the Littlewood-Paley projection on the dyadic interval $[2^{q-1}, 2^{q+1}]$.
\end{proposition}



\subsubsection{Description of resonances}


We first study the structure of the quadratic nonlinearity. Precisely, we define
\begin{equation}\label{def-Qosc}
\begin{aligned} 
Q_{osc}^\pm(s,t,x,v) 
&=\sum_\pm E^{osc}_\pm(s,x - (t-s)\hv ,v)\cdot H_{s,t}^{osc,\pm}(x,v)
\\
Q_{osc}^{+,-}(s,t,x,v) 
&=\sum_\pm E^{osc}_\pm(s,x - (t-s)\hv ,v)\cdot H_{s,t}^{osc,\mp}(x,v).
\end{aligned}
\end{equation}
The kernels may be expressed as convolutions or as Fourier multipliers. We choose to see them as Fourier multipliers.
We will prove the following lemma. 

\begin{lemma}\label{0lem-Qosc}
There are smooth symbols $M_1(k,\ell,v)$, $M_2(k,\ell,v)$ satisfying 
\begin{equation}\label{0bds-m12} 
\begin{aligned}
|M_1(k,\ell,v)| &\lesssim |\partial_v \mu(v)|+ |\partial^2_v \mu(v)|,
\\
 |M_2(k,\ell,v)| &\lesssim C_{k,\ell}\Big(|k|+ |\lambda_+(\ell) +\lambda_-(k-\ell)|\Big) |\partial_v \mu(v)|
\end{aligned}
\end{equation} 
with $ C_{k,\ell} = \frac{1}{|\omega_-(k-\ell,v)|^2} + \frac{1}{|\omega_-(k-\ell,v)\omega_+(\ell,v)|}$, so that 
$$
\begin{aligned} 
e^{ik(t-s)\cdot \hv  }\FQ_{osc}^\pm(s,t,k,v) 
&=  \sum_\pm \int  \FE^{osc}_{\pm}(s,\ell)  \FE^{osc}_{\pm}(s,k-\ell) M_1(k,\ell,v)\; d\ell 
\\
e^{ik(t-s)\cdot \hv  }\FQ_{osc}^{+,-}(s,t,k,v) 
&= \int \FE^{osc}_+(s,\ell)  \FE^{osc}_-(s,k-\ell) M_2(k, \ell,v)\; d\ell  .
\end{aligned}$$

\end{lemma}

The bound on $M_2$ translates into a crucial cancellation: oscillations with frequency $+1$ and oscillations with frequency $-1$
could interact to give a non oscillatory term. It turns out that there is an algebraic cancellation in this interaction, which leads
to an interaction symbol $M_2$ which is smaller than expected, namely of order $|k|$ and not of order $\cO(1)$, for small $|k|$. This leads
to a gain of a factor $(t-s)^{-1}$ due to oscillation $e^{ik(t-s)\cdot \hv}$ in $v$. Note that this cancellation holds for each term $H_{s,t}^{osc,\pm}(x,v)$ in \eqref{def-THosctr}. 
  
\begin{proof} 
Here and in what follows, we note that $\FE^{osc}_{\pm,j}$ are components of the vector $\FE^{osc}_\pm$. 
Recall the definition \eqref{def-THosctr}, we write 
\begin{equation}\label{redef-Hosc1}
\begin{aligned}
H_{s,t}^{osc,\pm}(x,v)  
& =  \sum_{j}\partial_{v_j} \nabla_v \mu(v) \sum_\pm [E^{osc}_{\pm,j} \star_{x} \phi_{\pm,1}(v)](s,x - (t-s)\hv )
\\ &\quad +  \nabla_v\mu(v) \sum_{j} \sum_\pm [E^{osc}_{\pm,j} \star_{x} \partial_{v_j}\hv  \cdot \nabla_x \phi_{\pm,1}(v)](s,x - (t-s)\hv ) .
\end{aligned}
\end{equation}
Recall that the Fourier transform of $\phi_{\pm,1}(x,v)$ is $1/\omega_\pm(k,v)$. Therefore, in Fourier, we have  
\begin{equation} \label{FGosc}
\begin{aligned}
\FH_{s,t}^{osc,\pm}(k,v) &= e^{-ik(t-s)\cdot \hv }  \sum_\pm  \FE^{osc}_\pm(s,k) \cdot M_\pm(k,v)
\end{aligned}\end{equation}
where $M_\pm(k,v) = (M_{\pm,i,j}(k,v))$ is an $3\times 3$ matrix defined by 
\begin{equation} \label{defiMpm}
M_{\pm,i,j}(k,v) = \frac{1}{\omega_\pm(k,v)} \Big[ \partial_{v_i v_j}^2\mu(v) +(ik \cdot  \partial_{v_i}\hv ) \partial_{v_j} \mu\Big]
\end{equation}
Therefore, by definition, for $(+,+)$ and $(-,-)$ interactions we have
$$
\begin{aligned} 
e^{ik(t-s)\cdot \hv  }\FQ^\pm_{osc}(s,t,k,v) 
&= \sum_\pm \sum_{i,j}\int \FE^{osc}_{\pm,j}(s,\ell)  
\FE^{osc}_{\pm,i}(s,k-\ell) M_{\pm,i,j}(k-\ell, v)\; d\ell 
\end{aligned}$$
which gives the desired expression for this term, defining
$$
M_1(k,l,v) = M_{\pm,i,j}(k-\ell, v)
$$
which satisfies (\ref{0bds-m12}), since $\omega_\pm$ and its derivatives are bounded from below. 
We next compute $+/-$ and $-/+$ interactions, which gives 
$$
\begin{aligned} 
e^{ik(t-s)\cdot \hv  }
\FQ^{+,-}_{osc}(s,t,k,v) 
& =\sum_{i,j} \int \FE^{osc}_{+,j}(s,\ell)  \FE^{osc}_{-,i}(s,k-\ell)  M_{-,i,j}(k-\ell, v)\; d\ell 
\\
&\quad +\sum_{i,j} \int \FE^{osc}_{-,j}(s,\ell)  \FE^{osc}_{+,i}(s,k-\ell)  M_{+,i,j}(k-\ell, v)\; d\ell .
\end{aligned}$$
In the second integral term, we exchange $i\leftrightarrow j$ and $\ell \leftrightarrow k-\ell$, 
leading to 
$$
\begin{aligned} 
e^{ik(t-s)\cdot \hv  }\FQ^{+,-}_{osc}(s,t,k,v) 
& = 
\sum_{i,j} \int  \FE^{osc}_{+,j}(s,\ell)  \FE^{osc}_{-,i}(s,k-\ell) 
\\&\quad \times \Big [ M_{-,i,j}(k-\ell, v) +  M_{+,j,i}(\ell, v) \Big] \; d\ell .
\end{aligned}$$
We recall that the matrix $M_{\pm}(k,v)$ is defined by (\ref{defiMpm}). We aim to prove that the integrand is of order $\cO(|k|)$
and not merely $\cO(1)$ for small $k$. 
Indeed, we compute 
$$
\begin{aligned}
 M_{-,i,j}(k-\ell, v) +  M_{+,j,i}(\ell, v) &= \partial^2_{v_i v_j} \mu(v)\Big[ \frac{1}{\omega_-(k-\ell,v) } 
 +  \frac{1}{\omega_+(\ell,v) } \Big] 
 \\&\quad +\frac{\partial_{v_i} \mu(v) i(k-\ell)\cdot  \partial_{v_j}\hv }{\omega_-(k-\ell,v)^2}  
 + \frac{\partial_{v_j} \mu(v) i\ell \cdot \partial_{v_i} \hv }{\omega_+(\ell,v)^2}  .
 \end{aligned}$$
Recall that $\omega_\pm(k,v) = \lambda_\pm(k) + ik\cdot \hv $, which in particular never vanishes. 
Then,
$$
\frac{1}{\omega_-(k-\ell,v) } +  \frac{1}{\omega_+(\ell,v) } = \frac{\lambda_+(\ell) + \lambda_-(k-\ell) 
+ i k \cdot \hv }{\omega_-(k-\ell,v) \omega_+(\ell,v)} \le \frac{|\lambda_+(\ell) + \lambda_-(k-\ell) |
+ |k|}{|\omega_-(k-\ell,v) \omega_+(\ell,v)|}. 
$$
On the other hand, as 
$$\partial_{v_i} \hv _j = \frac{1}{\langle v\rangle} (\delta_{ij} - \hv _i \hv _j)$$
and 
$$
ik\cdot \partial_{v_i}\hv  
=  \frac{i}{\langle v\rangle}\sum_j k_j (\delta_{ij} - \hv _i \hv _j)= \frac{i}{\langle v\rangle} ( k_i - \hv _i k \cdot \hv ),
$$
we compute
$$
\begin{aligned}
&\frac{\partial_{v_j} \omega_-(k-\ell, v) \partial_{v_i} \mu(v) }{\omega_-(k-\ell,v)^2} 
+  \frac{\partial_{v_i} \omega_+(\ell, v) \partial_{v_j} \mu(v)}{\omega_+(\ell,v)^2}  
\\& 
=
\frac{i \partial_{v_i} \mu(v) ( (k-\ell)_j - \hv _j (k-\ell) \cdot \hv ) }{\langle v\rangle \omega_-(k-\ell,v)^2} 
+  \frac{i  \partial_{v_j} \mu(v) ( \ell_i - \hv _i \ell \cdot \hv )}{\langle v\rangle\omega_+(\ell,v)^2}  
\\& 
=
\frac{i \partial_{v_i} \mu(v) (k-\ell)_j }{\langle v\rangle \omega_-(k-\ell,v)^2} 
+  \frac{i  \partial_{v_j} \mu(v)  \ell_i }{\langle v\rangle\omega_+(\ell,v)^2}  
 -
\frac{i \partial_{v_i} \mu(v)  \hv _j (k-\ell) \cdot \hv  }{\langle v\rangle \omega_-(k-\ell,v)^2} 
- \frac{i  \partial_{v_j} \mu(v) \hv _i \ell \cdot \hv }{\langle v\rangle\omega_+(\ell,v)^2}  .
\end{aligned}$$
Since $\mu(v)$ is radial, the last two terms are 
$$
\frac{i \partial_{v_i} \mu(v)  \hv _j (k-\ell) \cdot \hv  }{\langle v\rangle \omega_-(k-\ell,v)^2} 
+ \frac{i  \partial_{v_j} \mu(v) \hv _i \ell \cdot \hv }{\langle v\rangle\omega_+(\ell,v)^2} 
= i \hv _j  \partial_{v_i}\mu(v) \frac{(k-\ell)\cdot \hv  \omega_+(\ell,v)^2 + \ell \cdot \hv  \omega_-(k-\ell,v)^2}{ \langle v\rangle\omega_-(k-\ell,v)^2\omega_+(\ell,v)^2}.
$$
Now 
$$\begin{aligned}
(k-\ell)\omega_+(\ell,v)^2 &+ \ell \omega_-(k-\ell,v)^2 = k\omega_+(\ell,v)^2+ \ell (\omega_-(k-\ell,v)^2-
\omega_+(\ell,v)^2)\notag
\\&= k\omega_+(\ell,v)^2 +\ell (\omega_-(k-\ell,v)-
\omega_+(\ell,v))(\omega_-(k-\ell,v)+
\omega_+(\ell,v))\notag
\\&= k\omega_+(\ell,v)^2 +\ell (\omega_-(k-\ell,v)-
\omega_+(\ell,v))(\lambda_-(k-\ell,v)+
\lambda_+(\ell,v) + ik\cdot \hv)
.
\end{aligned}
$$
This yields 
\begin{equation}\label{cmp-opminus}
\Big| \frac{i \partial_{v_i} \mu(v)  \hv _j (k-\ell) \cdot \hv  }{\langle v\rangle \omega_-(k-\ell,v)^2} 
+ \frac{i  \partial_{v_j} \mu(v) \hv _i \ell \cdot \hv }{\langle v\rangle\omega_+(\ell,v)^2} 
\Big| \lesssim C_{k,\ell}\Big(|\lambda_+(\ell) + \lambda_-(k-\ell) |
+ |k|\Big),
\end{equation}
in which 
$$ C_{k,\ell} = \frac{1}{|\omega_-(k-\ell,v)|^2} + \frac{1}{|\omega_-(k-\ell,v)\omega_+(\ell,v)|}$$
Next, to control the contribution of the other two terms 
$$
\frac{i \partial_{v_i} \mu(v) (k-\ell)_j }{\langle v\rangle \omega_-(k-\ell,v)^2} 
+  \frac{i  \partial_{v_j} \mu(v)  \ell_i }{\langle v\rangle\omega_+(\ell,v)^2}  
$$
we need to insert them in the integral 
$$
I_E = \sum_{i,j} \int \FE^{osc}_{+,j}(s,\ell)  \FE^{osc}_{-,i}(s,k-\ell) 
\Big [ \frac{i \partial_{v_i} \mu(v) (k-\ell)_j }{\langle v\rangle \omega_-(k-\ell,v)^2} 
+  \frac{i  \partial_{v_j} \mu(v)  \ell_i }{\langle v\rangle\omega_+(\ell,v)^2}   \Big] d\ell .
$$
Recall that we can write $ \FE^{osc}_{\pm}(t,k) = k \Fb_\pm(t,k)$ (we are exploiting here electric field is parallel to its propagation frequency)
for some scalar function $\Fb_\pm(t,k)$ (which may be singular $k$ at $k=0$). Therefore, 
$$
\begin{aligned}
I_E 
&=  \int \Fb_{+}(s,\ell)  \Fb_{-}(s,k-\ell)\sum_{i,j} \ell_j(k-\ell)_i
\Big [ \frac{i \partial_{v_i} \mu(v) (k-\ell)_j }{\langle v\rangle \omega_-(k-\ell,v)^2} 
+  \frac{i  \partial_{v_j} \mu(v)  \ell_i }{\langle v\rangle\omega_+(\ell,v)^2}   \Big] d\ell 
\\
&= i \int \Fb_{+}(s,\ell)  \Fb_{-}(s,k-\ell) 
\Big [ \frac{ \ell \cdot (k-\ell) (k-\ell)\cdot \nabla_v \mu}{\langle v\rangle \omega_-(k-\ell,v)^2} 
+  \frac{ \ell \cdot (k-\ell) \ell \cdot \nabla_v \mu}{\langle v\rangle\omega_+(\ell,v)^2}   \Big] d\ell
\\
&= i \int 
\FE^{osc}_{+}(s,\ell) \cdot \FE^{osc}_{-}(s,k-\ell) 
\frac{(k-\ell)\cdot \nabla_v \mu \omega_+(\ell,v)^2 
+  \ell \cdot \nabla_v \mu\omega_-(k-\ell,v)^2}{\langle v\rangle \omega_-(k-\ell,v)^2\omega_+(\ell,v)^2}   \; d\ell .
\end{aligned}$$
Using again \eqref{cmp-opminus}, we thus have 
$$I_E = \int \FE^{osc}_{+}(s,\ell) \cdot \FE^{osc}_{-}(s,k-\ell) M_2(k,\ell,v) \; d\ell $$ 
for some symbol $M_2(k,\ell,v)$ that satisfies the stated estimates. 
\end{proof}


\subsubsection{Bounds on $\cS_{\pm,\pm}^{osc} $}


By definition from \eqref{0def-Rosc} and \eqref{def-Qosc}, we compute in Fourier 
$$
\begin{aligned}
\FcS_{\pm,\pm}^{osc}(t,k) 
&= \int_0^t\int \FQ^\pm_{osc}(s,k, v) \;dvds 
\\
&= \sum_\pm \int_0^t  \iint  e^{-ik \cdot \hv (t-s)}  \FE^{osc}_\pm(s,\ell) \FE^{osc}_\pm(s,k-\ell) M_1(k,\ell,v)\; d\ell dv ds
\end{aligned}
$$ 
in which $|M_1(k,\ell,v)| \lesssim |\partial_v \mu(v)|+ |\partial^2_v \mu(v)|$ from Lemma \ref{0lem-Qosc}.
Note that only $M_1$ depends on $v$. We thus introduce  
\begin{equation}\label{0def-mkl0}
m_\pm(k, \ell, t) := \int e^{-ik \cdot \hv  t}  M_1(k,\ell,v)\; dv.
\end{equation}
By construction in  Lemma \ref{0lem-Qosc}, the symbol $M_1(k,\ell,v)$ is both sufficiently smooth and decaying sufficiently fast in $v$. 
Therefore, upon integrating by parts repeatedly in $v$, 
the symbol $ m_\pm(k,l,t) $ decays polynomially fast in $\langle kt\rangle$. Precisely, for any $j\ge 0$,
\begin{equation}\label{0bd-mkl0}
\partial_t^j m_\pm(k, \ell, t) = \int (-i k\cdot \hv )^je^{-ik \cdot \hv  t}  M_\pm(k-\ell,v)\; dv  
 \lesssim \langle t\rangle^{-j} \langle kt\rangle^{-N}
\end{equation}
for some large $N$. 
That is, each time-derivative gains an extra decay of order $t^{-1}$.  

Next, to bound $\FcS_{\pm,\pm}^{osc}(t,k)$, we need to exploit the oscillation in $E^{osc}_\pm(t,x)$. We thus write 
$$\FE^{osc}_\pm(t,k) = e^{\lambda_{\pm}(k) t} \FB_\pm(t,k)$$
for some vector $\FB_\pm(t,k)$ defined as in \eqref{def-Bpm}. We obtain 
\begin{equation}\label{0eqs-MFGat}
\begin{aligned}
\FcS_{\pm,\pm}^{osc}(t,k) 
&= 
\sum_\pm\int_0^t \int e^{(\lambda_\pm(\ell) + \lambda_\pm(k-\ell))s} \FB_\pm(s,\ell) \FB_\pm(s,k-\ell)m_\pm(k,\ell, t-s)\; d\ell ds.
\end{aligned}
\end{equation}
We now perform a non stationary phase argument on (\ref{0eqs-MFGat}).
Let us define the phase function by 
$$ 
W_\pm(k,\ell) = \lambda_\pm(\ell) + \lambda_\pm(k-\ell).
$$
As $\lambda_\pm(k)=\pm i\nu_*(k)$, with $\nu_*(k) \sim \langle k\rangle$, we note that the phase function satisfies 
$$
| W_\pm(k,\ell)| \gtrsim \langle \ell\rangle + \langle k-\ell\rangle
$$ 
for each $+$ and $-$, and in particular, the phase function never vanishes. This is natural since we are dealing with
interactions of time frequency plus-minus one with itself, leading to a time frequency plus-minus two. 
We can thus integrate by parts in $s$. 
We get (omitting the summation over $\pm$)
$$
\begin{aligned}
\FcS_{\pm,\pm}^{osc}(t,k) 
&=
\int \frac{e^{{ W_\pm(k,\ell)}t}}{ W_\pm(k,\ell)} \FB_\pm(t,\ell) \FB_\pm(t,k-\ell)m_\pm(k,\ell, 0)\; d\ell 
\\& \quad - \int \frac{1}{ W_\pm(k,\ell)} \FB_\pm(0,\ell) \FB_\pm(0,k-\ell)m_\pm(k,\ell, t)\; d\ell 
 \\&\quad - \int_0^t \int \frac{e^{{ W_\pm(k,\ell)}s}}{ W_\pm(k,\ell)}\partial_s \FB_\pm(s,\ell) \FB_\pm(s,k-\ell)m_\pm(k,\ell, t-s)\; d\ell ds
\\
 &\quad -\int_0^t \int \frac{e^{{ W_\pm(k,\ell)}s}}{ W_\pm(k,\ell)} \FB_\pm(s,\ell) \partial_s \FB_\pm(s,k-\ell)m_\pm(k,\ell, t-s)\; d\ell ds
 \\& \quad + \int_0^t \int \frac{e^{{ W_\pm(k,\ell)}s}}{ W_\pm(k,\ell)} \FB_\pm(s,\ell) \FB_\pm(s,k-\ell) \partial_sm_\pm(k,\ell, t-s)\; d\ell ds
\end{aligned}
$$
Note that $\ell$ and $k-\ell$ are symmetric, so the third and fourth terms are identical. 
Recall also from \eqref{def-Bpm} that $\FB_\pm(0,k) = a_\pm(k)\FF_0(k)$. 
Next, writing $e^{\lambda_\pm(k) t} \FB_\pm(t,k) = \FE^{osc}_\pm(t,k)$ and using \eqref{dtBF}, 
 $$
 e^{\lambda_\pm(k)t} \partial_t \FB_\pm(t,k)  =  a_\pm(k) \FF(t,k),
$$
we can write 
\begin{equation}\label{decomp-SoscFM}
\FcS_{\pm,\pm}^{osc}(t,k) = \FM_1(t,k) + \FM_2(t,k) + \FM_3(t,k) + \FM_4(t,k) +  \FM_5(t,k) 
\end{equation}
where 
$$
\begin{aligned}
 \FM_1(t,k) &=
 \int \frac{1}{ W_\pm(k,\ell)} \FE^{osc}_\pm(t,\ell)\FE^{osc}_\pm(t,k-\ell) m_\pm(k,\ell,0)\; d\ell 
\\
 \FM_2(t,k) &=
-  \int \frac{1}{ W_\pm(k,\ell)} a_\pm(\ell) a_\pm(k-\ell) \FF_0(\ell) \FF_0(k-\ell) m_\pm(k,\ell,t)\; d\ell 
 \\
 \FM_3(t,k) &=-
 \int_0^t \int \frac{1}{ W_\pm(k,\ell)} a_\pm(\ell) \FF(s,\ell) \FE^{osc}_\pm(s,k-\ell)m_\pm(k,\ell, t-s)\; d\ell ds
\\
 \FM_4(t,k) &=-
 \int_0^t \int \frac{1}{ W_\pm(k,\ell)}  a_\pm(k-\ell) \FE^{osc}_\pm(\ell) \FF(s,k-\ell) m_\pm(k,\ell, t-s)\; d\ell ds
\\
 \FM_5(t,k) & = \int_0^t \int \frac{1}{ W_\pm(k,\ell)}  \FE^{osc}_\pm(s,\ell) \FE^{osc}_\pm(s,k-\ell) \partial_sm_\pm(k,\ell, t-s)\; d\ell ds
\end{aligned}
$$
We now estimate each term. Using Lemma \ref{lem-bilinear}, we have 
$$ 
\|M_1(t)\|_{L^p_x} \lesssim \|E^{osc}_\pm(t)\|_{L^{2p}_x}^2 \lesssim \epsilon^2 \langle t\rangle^{-3(1-1/p)}.
$$
As for $M_2$, we use \eqref{0bd-mkl0} to obtain 
$$ 
\|M_2(t)\|_{L^\infty_x} \lesssim \|\FM_2(t)\|_{L^1_k} \lesssim \| \langle kt\rangle^{-N} \|_{L^1_k} \| \FF_0 \|^2_{L^2_k} 
\lesssim \langle t\rangle^{-3} \| F_0 \|_{L^2_x}^2 \lesssim \epsilon_0\langle t\rangle^{-3} 
$$
by the assumption on initial data, while using Lemma \ref{lem-bilinear} we bound $
\|M_2(t)\|_{L^1_x} \lesssim \| \FF_0 \|^2_{L^2_x} 
\lesssim \epsilon_0. 
$

Next, we turn to $M_3$ and $M_4$ which are similar.
We recall the bootstrap assumptions (see Section \ref{sec-bootstrap}) 
\begin{equation}\label{0bd-dtV}
\begin{aligned}
 \| F(t)\|_{L^p_x}  &\lesssim \epsilon \langle t\rangle^{-3(1-1/p)} .
 \end{aligned}\end{equation}
We split the time integral in $M_3$ in two parts: between $0$ and $t/2$ on one side, and between $t/2$ and $t$ on the
other side. Let $M_3^1$ and $M_3^2$ be the corresponding integrals. We start with $M_3^2$. Indeed, using Coifman-Meyer type bounds, we get 
$$
\begin{aligned}
\| M_3^2(t)\|_{L^p_x} 
& \lesssim  \int_{t/2}^t \| F(s)\|_{L^{p}_x} \|E^{osc}_\pm(s)\|_{L^{\infty}_x} \; ds
 \lesssim \epsilon^2 \int_{t/2}^t \langle s\rangle^{-3 + 3/p} \langle s\rangle^{-3/2}  \; ds
 \lesssim \epsilon^2 \langle t\rangle^{-3 + 3/p} ,
\end{aligned}
$$
as claimed for any $p\ge 1$. Note that a similar estimate as done above also gives the desired bound on $M_3^1(t)$ in $L^1_x$. We now focus to bound $M_3^1$ in $L^\infty$, for which we use the dispersion. 
More precisely we have
$$
M_3^1(t,x) = \int_0^{t/2} \iiint {a_\pm(\ell) \over  W_\pm(k,\ell)} \FF(s,\ell) \FE^{osc}_\pm(s,k-\ell)
e^{i k. x - i k . \hv  (t-s)} M_1(k,l,v) \; d\ell dk  dv ds.
$$
Note that $M_3^1(t,x)$ depends on $x$ through $x - \hv  (t-s)$. To underline this dependency and take advantage
of the associated dispersion, we introduce
$$
H(s,y,v) = \iint {a_\pm(\ell) \over  W_\pm(k,\ell)} \FF(s,\ell) \FE^{osc}_\pm(s,k-\ell)
e^{i k. y}  M_1(k,l,v) \; d\ell dk .
$$
Note that $H(s,y,v)$ depends on $v$ through $M_1$ only, which decays faster than $\langle v \rangle^{-3}$.
We write 
$$
\langle v\rangle^3 H(s,y,v)=\int e^{i\eta\cdot v}\hat{\underline H}(s,y,\eta) d\eta,
$$
\begin{align}
\hat{\underline H}(s,y,\eta) &=\iiint {a_\pm(\ell) \over  W_\pm(k,\ell)} \FF(s,\ell) \FE^{osc}_\pm(s,k-\ell)
e^{i k. y -i\eta\cdot v}  \langle v\rangle^3 M_1(k,l,v) \; d\ell dk d v\notag\\&=
\iint {a_\pm(\ell) \over  W_\pm(k,\ell)} e^{i k. y}\FF(s,\ell) \FE^{osc}_\pm(s,k-\ell)
\underline M_1(k,l,\eta) \; d\ell dk \label{eq:CM}
\end{align}
with
$$
\underline M_1(k,l,\eta) =\int e^{-i\eta\cdot v}  \langle v\rangle^3 M_1(k,l,v) dv
$$
Since $\langle v\rangle^3 M_1(k,l,v)$ is smooth and rapidly decaying in $v$, we have
$$
|\underline M_1(k,l,\eta)|\les \langle \eta\rangle^{-N}
$$
with the same estimates also holding true for the $\partial_k$ and $\partial_\ell$ derivatives. We now apply 
a Coifman-Meyer type estimate of Lemma \ref{lem-bilinear} to \eqref{eq:CM} with $\eta$ treated as a
parameter:
$$
\|\hat{\underline H}(s,y,\eta)\|_{L^1_y}\les \langle \eta\rangle^{-N} 
\| F(s)\|_{L^2_x}\| E^{osc}_\pm(s)\|_{L^2_x}
$$
From this we can conclude 
$$
\| \sup_v \langle v \rangle^3 H \|_{L^1_x}\les \| F(s)\|_{L^2_x}\| E^{osc}_\pm(s)\|_{L^2_x}
$$
 This estimate leads to dispersion in $x$, that is 
$$
\begin{aligned}
\| M_3^1(t)\|_{L^\infty_x}  
& \lesssim \int_0^{t/2} (t-s)^{-3} \| F(s) E^{osc}_\pm(s)\|_{L^1_x} \; ds
 \lesssim \epsilon^2\int_0^{t/2} s^{-3/2}(t-s)^{-3} \; ds
 \lesssim \epsilon^2 t^{-3} . 
\end{aligned}
$$
Such kind of dispersion estimates will be used several times in this paper. Note that $M_4$ is similar.

Finally, we bound 
$$
 \FM_5(t,k) = \int_0^t \int \frac{\partial_sm_\pm(k,\ell, t-s)}{ W_\pm(k,\ell)} \FE^{osc}_\pm(s,\ell) \FE^{osc}_\pm(s,k-\ell) \; d\ell ds ,
 $$
which we split it again into $M_5^1$ and $M_5^2$ that corresponds to the integral over $(0,t/2)$ and $(t/2,t)$, respectively. 
Using Coifman-Meyer type bounds and the estimate \eqref{0bd-mkl0}, we get 
$$
\begin{aligned}
\| M_5^2(t)\|_{L^p_x} 
& \lesssim  \int_{t/2}^t (t-s)^{-1}\|E^{osc}_\pm(s)\|^2_{L^{2p}_x} \; ds
 \lesssim \epsilon^2 \int_{t/2}^t s^{-3+3/p} (t-s)^{-1} \; ds
 \lesssim \epsilon^2 t^{-3+3/p} \log t,
\end{aligned}
$$
for any $p\ge 1$. To get rid of the log loss, in view of \eqref{0def-mkl0}, we note that 
$$ |k||\partial_sm_\pm(k,\ell, t-s)| \lesssim \langle t-s\rangle^{-2},$$
which yields the desired estimate on $ \partial_x M_5^2(t)$. Next, we bound $M_5^1$. Note that a similar estimate as done above also gives the desired $L^1$ bound for $M_5^1$. It remains to give the $L^\infty$ estimates, 
which we shall dispersion of the transport equation. We write 
$$
 \FM^1_5(t,x)  = \int_0^{t/2} \int \frac{1}{ W_\pm(k,\ell)}  \FE^{osc}_\pm(s,\ell) \FE^{osc}_\pm(s,k-\ell) e^{i k. x - i k . \hv  (t-s)} k . \hv  M_1(k,\ell, v)\; d\ell dv ds
$$
With the exception of the factor $k . \hv $, which appears from the $\partial_s$-derivative of $m_\pm$, this
is very similar to $\FM^1_3(t,x)$. We observe that 
$$
k . \hv =\frac {i\langle v\rangle^3}{t-s}\hv .\nabla_v e^{- i k . \hv  (t-s)}
$$
We can now integrate by parts in $v$ and obtain
$$
\FM^1_5(t,x)  = \int_0^{t/2} \int \frac1{t-s}\frac{1}{ W_\pm(k,\ell)}  \FE^{osc}_\pm(s,\ell) \FE^{osc}_\pm(s,k-\ell) e^{i k. x - i k . \hv  (t-s)}  \tilde M_1(k,\ell, v)\; d\ell dv ds
$$
with the symbol $\tilde M_1(k,\ell, v)$ similar to  $M_1(k,\ell, v)$. We can repeat verbatim the arguments 
applied to $\FM^1_3(t,x)$ noting the presence of an additional factor of $1/(t-s)$. This gives 
 $$
\begin{aligned}
\| M^1_5(t)\|_{L^\infty_x} 
& \lesssim \int_0^{t/2} (t-s)^{-4} \| E^{osc}_\pm(s)^2\|_{L^1_x} \; ds
 \lesssim \epsilon^2 \int_0^{t/2} t^{-4}\; ds
 \lesssim \epsilon^2 t^{-3} .
\end{aligned}
$$
The estimates for derivatives follow similarly. 

Finally, we prove the desired estimates on the Littlewood-Paley projection. Indeed, for terms $M_j$ that experience no log loss, we simply use the the interpolation inequality \eqref{LP-interpolate}. As for $M_5$, following the argument as done in the previous section, we write 
$$ M_5(t,x) = \int_0^t J_5(s,t,x) \; ds$$
and bound 
\begin{equation}\label{LPbounds-M5}
\begin{aligned}
2^{(1-\delta)q} \|P_q M_5(t) \|_{L^p_x}   & \le \int_0^t 2^{(1-\delta)q} \|P_q J_5(s,t) \|_{L^p_x}  \; ds
 \lesssim \int_0^t \| J_5(s,t) \|^\delta_{L^p_x} \|\partial_x J_5(s,t)\|^{1-\delta}_{L^p_x} \; ds
\\& \lesssim \int_0^{t/2} (t-s)^{-5 +3/p+ \delta} \| E^{osc}_\pm(s)^2\|_{L^1_x} \; ds
+ \int_{t/2}^t  \langle t-s \rangle^{-2+\delta}\|E^{osc}_\pm(s)\|^2_{L^{2p}_x} \; ds
\\
& \lesssim \epsilon^2 \int_0^{t/2} t^{-5+3/p+\delta}\; ds
 + \epsilon^2\int_{t/2}^t  \langle t-s \rangle^{-2+\delta}s^{-3+3/p} \; ds
 \lesssim \epsilon^2 t^{-3+3/p}
\end{aligned}
\end{equation}
for $\delta \in (0,1)$. This completes the proof of Proposition \ref{prop-Soscpm} for $\cS_{\pm,\pm}^{osc} $.


\subsubsection{Bounds on $\cS_{\pm,\mp}^{osc} $}


Next, we prove Proposition \ref{prop-Soscpm} for $\cS_{\pm,\mp}^{osc} $. This involves $(+,-)$ and $(-,+)$ resonant interactions, leading to non oscillatory terms. 
By definition from \eqref{0def-Rosc} and \eqref{def-Qosc}, we compute in Fourier 
$$
\begin{aligned}
\FcS_{\pm,\mp}^{osc}(t,k) 
&= \sum_\pm \int_0^t\int e^{-ik \cdot \hv (t-s)} \FQ^{+,-}_{osc}(s,k, v) \;dvds 
\\
&= \int_0^t  \iint  e^{-ik \cdot \hv (t-s)} \FE^{osc}_+(s,\ell) \FE^{osc}_-(s,k-\ell) M_2(k,\ell,v)\; d\ell dv ds 
\end{aligned}
$$ 
in which we note that 
$$
M_2(k,\ell,v) \lesssim C_{k,\ell}( |k| + |\lambda_+(\ell)+\lambda_-(k-\ell)|)|\partial_v \mu|
$$ 
from Lemma \ref{0lem-Qosc} with $ C_{k,\ell} = \frac{1}{|\omega_-(k-\ell,v)|^2} + \frac{1}{|\omega_-(k-\ell,v)\omega_+(\ell,v)|}$.  So, compared to $M_1(k,\ell,v)$, $M_2(k,\ell,v)$ has 
either an extra $k$ factor or $\lambda_+(\ell)+\lambda_-(k-\ell)$. 

To deal with $k$-portion, we reproduce the treatment of $M_5^1$, again noting that
$$
k -(k.\hv ) \hv =\frac {i\langle v\rangle}{t-s}\partial_v e^{- i k . \hv  (t-s)},\quad 
k.\hv =\frac {i\langle v\rangle^3}{t-s}\partial_v e^{- i k . \hv  (t-s)}
$$
As in the case of $M_5^1$, the $k$ factor allows one to integrate by parts in $v$ which in turn gains an 
extra $(t-s)^{-1}$ factor. Precisely, 
$$
\begin{aligned}
\| \cS_{\pm,\mp}^{1,osc}(t,x) \|_{L^\infty_x} 
& \lesssim \int_0^{t/2} (t-s)^{-4} \| E^{osc}_+(s)E^{osc}_-(s)\|_{L^1_x} \; ds
 + \int_{t/2}^t  \langle t-s \rangle^{-1} \|E^{osc}_\pm(s)\|^2_{L^\infty_x} \; ds
\\
& \lesssim \epsilon^2 t^{-4}\int_0^{t/2}  \; ds
 + \epsilon^2 t^{-3}\int_{t/2}^t   \langle t-s \rangle^{-1}  \; ds
\\& \lesssim\epsilon^2  t^{-3} \log t.
\end{aligned}
$$
The $L^p$ estimates are treated similarly as done in the previous section. On the other hand, to deal with the contribution from $\lambda_+(\ell)+\lambda_-(k-\ell)$, we go back 
to the analog of \eqref{0eqs-MFGat}
\begin{equation}\label{0eqs-MFGat'}
\begin{aligned}
\FcS_{\pm,\mp}^{2,osc}(t,k) 
&= 
\sum_\pm\int_0^t \int e^{(\lambda_+(\ell) + \lambda_-(k-\ell))s} \FB_+(s,\ell) \FB_-(s,k-\ell)m_2(k,\ell, t-s)\; d\ell ds.
\end{aligned}
\end{equation}
$$
m_2(k,\ell, t)=\int e^{-ik.\hv  t} M_2(k,\ell,v) dv
$$
Define the  phase function by 
$$ 
W(k,\ell) = \lambda_+(\ell) + \lambda_-(k-\ell).
$$
The phase function no longer has a good bound from below but we can still integrate by parts in $s$ using 
that 
$$
\frac{m_2(k,\ell, t)}{W(k,\ell)}=\tilde m_2(k,\ell, t),
$$
where now, as in the treatment of $\FcS_{\pm,\pm}$, the symbol $\tilde m_2(k,\ell, t)$ has similar properties 
to $m_\pm(k,\ell, t)$. 
We get 
$$
\begin{aligned}
\FcS_{\pm,\mp}^{2,osc}(t,k) 
&=
\int {e^{{ W(k,\ell)}t}} \FB_+(t,\ell) \FB_-(t,k-\ell)\tilde m_2(k,\ell, 0)\; d\ell 
\\& \quad - \int \FB_+(0,\ell) \FB_-(0,k-\ell)\tilde m_2(k,\ell, t)\; d\ell 
 \\&\quad - \int_0^t \int{e^{{ W_\pm(k,\ell)}s}}\partial_s \FB_+(s,\ell) \FB_-(s,k-\ell)\tilde m_2(k,\ell, t-s)\; d\ell ds
\\
 &\quad -\int_0^t \int {e^{{ W_\pm(k,\ell)}s}}\FB_+(s,\ell) \partial_s \FB_-(s,k-\ell)\tilde m_2(k,\ell, t-s)\; d\ell ds
 \\& \quad + \int_0^t \int {e^{{ W_\pm(k,\ell)}s}} \FB_+(s,\ell) \FB_-(s,k-\ell) \partial_s\tilde m_2(k,\ell, t-s)\; d\ell ds
\end{aligned}
$$
We can now repeat the arguments of the previous section to obtain 
$$
\| \cS_{\pm,\mp}^{2,osc}(t,x) \|_{L^\infty_x} 
 \lesssim\epsilon^2  t^{-3} \log t
$$
Finally, similar to what as done in the previous section, an extra factor of $|k|$ gains a decay of order $(t-s)^{-1}$. 
This removes the log loss in the previous bound, leading to 
$$
\begin{aligned}
\| \partial_x \cS_{\pm,\mp}^{osc}(t,x) \|_{L^\infty_x} 
& \lesssim \epsilon^2 t^{-3} .
\end{aligned}
$$
The estimates on the Littlewood-Paley projection follows identically as done for $M_5$ in \eqref{LPbounds-M5}.

 
 \subsection{Nonlinear interaction with $E^{osc}$}\label{sec-SEosc}


In this section, we treat the nonlinear interaction with $E^{osc}$, namely the following source term
\begin{equation}\label{def-Sosc}
\begin{aligned}
S^{osc}(t,x) & = \sum_\pm \int_0^t \int_{\RR^3} E_\pm^{osc}(s,x - (t-s)\hw ) H_{s,t}(x,w) \; dw ds,
\end{aligned}
\end{equation}
where $H_{s,t}(x,w) $ is defined as in Proposition \ref{prop-SR}. Precisely, we will prove the following proposition, which together with Propositions \ref{prop-bdS0} and \ref{prop-decayEHR} will complete the proof of Proposition \ref{prop-bdS}.

\begin{proposition}\label{prop-decaySoscQ} There hold
$$ 
\begin{aligned}
\|\cS^{osc} (t)\|_{L^p_x}  &\lesssim \epsilon^2 \langle  t\rangle^{-3(1-1/p)} \log t 
\\
\sup_{q\in \ZZ} 2^{(1-\delta)q}  \|P_q \cS^{osc}(t)\|_{L^p_x}   
+\|\partial_x \cS^{osc} (t)\|_{L^p_x} &\lesssim \epsilon^2 \langle  t\rangle^{-3(1-1/p)} 
\end{aligned}
$$
for any $1 \le p \le \infty$ and $\delta \in (0,1)$, where $P_q$ denotes the Littlewood-Paley projection on the dyadic interval $[2^{q-1}, 2^{q+1}]$.

\end{proposition}

\begin{proof} Recall from  Proposition \ref{prop-SR} that $H_{s,t}(x,w)$ can be decomposed into 
$$
H_{s,t}(x,w)
 =  T_{s,t}^{osc}(x,w) + H_{s,t}^{osc}(x,w) + H_{s,t}^{tr}(x,w) + H^Q_{s,t}(x,w)+ H_{s,t}^R(x,w) .
$$
The interaction $H_{s,t}^{osc}(x,w) $ is already treated in the previous section, see Proposition \ref{prop-Soscpm}, while for the quadratic interaction $H^Q_{s,t}(x,w) = \TE_\pm^{osc}(s,x - (t-s)\hw, w) H^{tr}_{s,t}(x,w) $, the source integral in \eqref{def-Sosc} reads 
$$\int_0^t \int_{\RR^3} [ E_\pm^{osc}\TE_\pm^{osc}] (s,x - (t-s)\hw ,w) H^{tr}_{s,t}(x,w)\; dw ds. 
$$
Since the quadratic $[E_\pm^{osc}\TE_\pm^{osc}](s,x,w)$ satisfies the same bounds as those for $E^r$, the source integral is treated similarly as done for $S^r$, the interaction between $E^r$ and $H^{tr}_{s,t}$, see Proposition \ref{prop-decayEHR}.   

We now treat the remaining interaction kernels in $H_{s,t}(x,w)$, namely 
\begin{equation}\label{def-newH}
\cH_{s,t}(x,w) :=  T_{s,t}^{osc}(x,w) + H_{s,t}^{tr}(x,w) + H_{s,t}^{R,1}(x,w) + H_{s,t}^{R,2}(x,w) ,
\end{equation}
in which we recall that 
\begin{equation}\label{recall-bTHR}
\begin{aligned}
\| \langle w\rangle^N  \partial_s^\alpha H^{tr}_{s,t}\|_{L^\infty} \lesssim &~\epsilon \langle s\rangle^{-1-|\alpha|},\qquad \forall |\alpha|\le 2,
\\
\|\langle w\rangle^N  T^{osc}_{s,t}\|_{L^\infty}\lesssim \epsilon \langle t\rangle^{-3/2}, 
 \qquad 
&\|\langle w\rangle^N  H^R_{s,t}\|_{L^\infty}\lesssim \epsilon \langle t\rangle^{-2}.
\end{aligned}\end{equation}
Despite their different bounds, we shall however treat these kernels simultaneously. The common feature to be exploited in $\cH_{s,t}$ is that $s$-derivatives of $\cH_{s,t}$ gain extra decay. Precisely, the remainder bounds
\begin{equation}\label{recall-keyHR}
\begin{aligned} 
 \sup_w \| \langle w\rangle^N \cH^R_{s,t}\|_{L^2_{x}} & \lesssim \epsilon \langle s\rangle^{-3/2} ,
 \\ \sup_x \| \langle w\rangle^N\cH^R_{s,t}\|_{L^2_w} &\lesssim \epsilon \langle t\rangle^{-3/2}\langle s\rangle^{-3/2} ,
 \end{aligned}\end{equation}
hold for all the terms $\partial_s^2 \partial_x^\alpha H^{tr}_{s,t}$, $\partial_s \partial_x^\alpha H^{R,1}_{s,t}$, and $\partial_x^\alpha H^{R,2}_{s,t}$, for $|\alpha|\le 1$. In addition, we also have 
\begin{equation}\label{recall-bdTosc}
\begin{aligned} 
 \sup_w \| \langle w\rangle^N \partial_s T^{osc}_{s,t}\|_{L^2_{x}} &\lesssim \epsilon (t-s)^{-2} , 
 \\ \sup_x \| \langle w\rangle^N \partial_s T^{osc}_{s,t}\|_{L^2_w} &\lesssim \epsilon \langle t\rangle^{-3/2}(t-s)^{-2} .
 \end{aligned}\end{equation}

We are now ready to bound the source integral \eqref{def-Sosc}, denoted by $S^{osc,H}(t,x) $, for the kernel $\cH_{s,t}(x,w)$ as in \eqref{def-newH}. Note that as stated, the function $H_{s,t}^{R,2}(x,w) $ satisfy the good remainder bounds \eqref{recall-keyHR}, which verify the assumption of Lemma \ref{lem-auxSosc}. The source term that corresponds to $H_{s,t}^{R,2}(x,w) $ thus follows. The remaining kernels in $\cH_{s,t}$ do not decay sufficiently fast to apply Lemma \ref{lem-auxSosc} directly, and we need to make use of oscillation. To this end, we work in Fourier spaces, using that $\FE^{osc}_\pm (s,\ell)= e^{\lambda_\pm(\ell) s} \FB_\pm(s,\ell) $. We have 
$$
\begin{aligned}
\FS^{osc,H}(t,k) 
&= \int_0^t  \iint   e^{-i\ell t \cdot \hw  } e^{\omega_\pm(\ell) s} \FB_\pm(s,\ell) \FcH_{s,t}(k-\ell,w)\; d\ell dw ds
\end{aligned}
$$ 
with $\omega_\pm(\ell) = \lambda_\pm(\ell)+i\hw  \cdot \ell$. 
Integrating by parts in $s$, we get 
$$
\begin{aligned}
\FS^{osc,H}(t,k) 
& = 
 \iint    \frac{1}{\omega_\pm(\ell)}e^{\lambda_\pm(\ell)t }  \FB_\pm(t,\ell) \FcH_{t,t}(k-\ell,w)\; d\ell dw 
\\
&\quad - \iint    \frac{1}{\omega_\pm(\ell)}e^{-i\ell t\cdot \hw  }  \FB_\pm(0,\ell) \FcH_{0,t}(k-\ell,w)\; d\ell dw 
\\&\quad  - \int_0^t  \iint  \frac{1}{\omega_\pm(\ell)}e^{-i\ell \cdot \hw (t-s)} e^{\lambda_\pm(\ell)s} \partial_s \FB_\pm(s,\ell)  \FcH_{s,t}(k-\ell,w)\; d\ell dw ds
\\&\quad - \int_0^t  \iint  \frac{1}{\omega_\pm(\ell)}e^{-i\ell \cdot \hw  t}  
e^{\omega_\pm(\ell) s} \FB_\pm(s,\ell) \partial_s\FcH_{s,t}(k-\ell,w)\; d\ell dw ds, 
\end{aligned}
$$
which we write 
\begin{equation}\label{decomp-SoscQ}
\FS^{osc,H}(t,k) 
= \FM^H_1 +  \FM^H_2 +  \FM^H_3 +\FM^H_4
\end{equation}
where $\FM^H_j(t,k)$ are defined accordingly. Recall from \eqref{defiFB}-\eqref{dtBF} that $\FB_\pm(0,k) = a_\pm(k) \FF_0(k)$, $e^{\lambda_\pm(\ell) s} \FB_\pm(s,\ell) =\FE^{osc}_\pm (s,\ell)$, and $ e^{\lambda_\pm(\ell)s} \partial_s \FB_\pm(s,\ell)  =  a_\pm(\ell)\FF(s,\ell)$. We may write the above integrals in the physical space, yielding 
$$
\begin{aligned}
M^H_1(t,x)& =  \int [\phi_{\pm,1}^v E^{osc}_\pm](t,x) \cH_{t,t}(x,w)\;  dw
\\
M^H_2(t,x)& =  \int[\phi_{\pm,1} ^v a_\pm F_0](x- t\hw ) \cH_{0,t}(x,w)\;  dw 
\\
M^H_3(t,x)& =  \int_0^t  \int [\phi_{\pm,1}^va_\pm F](s,x- (t-s)\hw ) \cH_{s,t}(x,w)\;  dw ds 
\\M^H_4(t,x)& =  \int_0^t  \int [\phi_{\pm,1}^vE^{osc}_\pm](s,x- (t-s)\hw ) \partial_s\cH_{s,t}(x,w)\;  dw ds .
\end{aligned}$$

Let us treat each integral term. Using \eqref{recall-bTHR} for $s=t$, with noting that $H^{tr}_{t,t}=0$, and $\|E^{osc}_\pm\|_{L^\infty_x}\les \epsilon \langle t\rangle^{-3/2}$, the integral $M_1^H(t)$ are clearly bounded by $\epsilon^2 \langle t\rangle^{-3}$ in $L^\infty_x$. The $L^1$ estimate on $M_1^H(t)$ follows from the boundedness of $E^{osc}_\pm$ and $\cH_{t,t}$ in $L^2_x$.

As for the integral $M^H_2(t,x)$ that involve initial data $F_0(x)$, we introduce the change of variables $y = x - t \hw $ 
as done in Section \ref{sec-S0} in order to take advantage of the dispersion to bound  
\begin{equation}\label{est-MH2}
\begin{aligned}
|M_2^H(t,x)| 
& \lesssim t^{-3} \int |[\phi_{\pm,1} ^va_\pm F_0](y)|| \cH_{0,t}(x,w_{0,t}(x,y))| \langle w_{0,t}(x,y) \rangle^{3}\;  dy
\\
& \lesssim \epsilon t^{-3} \|  a_\pm F_0 \|_{L^1_x}
\lesssim \epsilon^2 t^{-3} .
\end{aligned}\end{equation}

Next, the estimates on $M^H_3$ 
are similar to the estimates of $E^r$ and $H$ in Section \ref{sec-SEr}, since $F$ and $E^r$ 
satisfy the same decay estimates.

Finally, we treat the integral term $M^{tr}_4$. In view of \eqref{recall-keyHR}-\eqref{recall-bdTosc}, the functions $\partial_sT^{osc}_{s,t}$ and $\partial_sH^{R,1}_{s,t}$ have already satisfied the assumption of Lemma \ref{lem-auxSr}. The bounds on the corresponding integrals $M_4^H$ thus follow. It remains to treat the integral $M_4^H$ for $ \partial_s\cH_{s,t}(x,w) =  \partial_sH^{tr}_{s,t}(x,w)$; 
namely, the integral 
$$ S^{osc,tr}(t,x) =  \int_0^t  \int [\phi_{\pm,1}^vE^{osc}_\pm](s,x- (t-s)\hw ) \partial_sH^{tr}_{s,t}(x,w)\;  dw ds.
$$
Note that the decay of $ \partial_sH^{tr}_{s,t}(x,w)$ is not sufficient to apply Lemma \ref{lem-auxSr}. To proceed, we  
integrate by parts once more the integral in $s$, similarly as done above, and noting from the definition of $H^{tr}_{s,t}$ 
that  $\partial_sH^{tr}_{s,t}=0$ at $s=t$, we get 
$$
\begin{aligned}
\FS^{osc,tr}(t,k) 
&= -  \iint  \frac{1}{\omega_\pm(\ell)^2} e^{-i\ell \cdot \hw  t} \FB_\pm(0,\ell) \partial_s\FG^{tr}_{0,t}(k-\ell,w)\; d\ell dw 
\\&\quad - \int_0^t  \iint  \frac{1}{\omega_\pm(\ell)^2}e^{-i\ell \cdot \hw (t-s)} a_\pm(\ell)\FF(s,\ell) \partial_s\FG^{tr}_{s,t}(k-\ell,w)\; d\ell dw ds
\\&\quad - \int_0^t  \iint  \frac{1}{\omega_\pm(\ell)^2}e^{-i\ell \cdot \hw (t-s)}  \FE^{osc}_\pm(s,\ell) 
\partial^2_s\FG^{tr}_{s,t}(k-\ell,w)\; d\ell dw ds
.
\end{aligned}
$$ 
Writing the above integrals in the physical space, we have
$$ S^{osc,tr}(t,x)  = M_1^{tr}(t,x) + M_2^{tr} (t,x)+ M_3^{tr}(t,x)$$
where 
$$
\begin{aligned}
M^{tr}_1(t,x)& =  -\int[\phi_{\pm,2}^va_\pm F_0](x- t\hw ) \partial_sH^{tr}_{0,t}(x,w)\;  dw 
\\
M^{tr}_2(t,x)& = - \int_0^t  \int[\phi_{\pm,2}^va_\pm F](s,x- (t-s)\hw ) \partial_sH^{tr}_{s,t}(x,w)\;  dw ds 
\\M^{tr}_3(t,x)& =  -\int_0^t  \int [\phi_{\pm,2}^vE^{osc}_\pm](s,x- (t-s)\hw ) \partial_s^2H^{tr}_{s,t}(x,w)\;  dw ds .
\end{aligned}$$
The estimates on these integrals now follow identically to those done above for $M_1^H, M_2^H$, and $M_4^H$, upon noting that $\partial_s^2H^{tr}_{s,t}(x,w)$ now satisfy the good remainder estimates \eqref{recall-keyHR}. 
This completes the proof of the proposition. 
\end{proof}

\subsection{Derivative estimates}
 
In this section, we establish decay for high derivatives of the source density. Precisely, we prove the following proposition (namely, the second part in Proposition \ref{prop-bdS}).    
 
\begin{proposition} \label{prop-bdSD} Let $S$ be the nonlinear source terms defined as in \eqref{nonlinear-S} and fix $N_0\ge 8$. Under the bootstrap assumptions listed in Section \ref{sec-bootstrap}, there hold
\begin{equation}\label{Hsbounds-cSD0}
\begin{aligned}
\|\partial_x^\alpha S(t)\|_{L^p_x}   \le \epsilon^2 \langle t\rangle^{-3(1-1/p) + \delta_{\alpha,p}} 
, \qquad 
\|S(t)\|_{H^{\alpha_0}_x} &\lesssim \epsilon^2 \langle t\rangle^{\delta_1},
\end{aligned}
\end{equation}
for all $|\alpha|\le N_0-1$, with $\delta_{\alpha,p} = \max\{ \delta_\alpha,\frac32\delta_\alpha(1-1/p)\}$, where $\delta_\alpha = \frac{|\alpha|}{N_0-1}$. 
\end{proposition}

\begin{proof}
The bounds on $S^0(t,x)$ have been proved in Proposition \ref{prop-bdS0}. We focus on treating $S^\mu(t,x)$, which we recall 
$$
S^{\mu}(t,x) = \int_0^t \int_{\RR^3}  \Big[ E(s,x - (t-s)\hv )\cdot \nabla_v \mu(v) 
 -E(s,X_{s,t}(x,v))\cdot \nabla_v \mu(V_{s,t}(x,v)) \Big] \, dv  ds.
 $$
Let us study the second term in the above integration. Recalling the characteristic description from Proposition \ref{prop-charV} and Proposition \ref{prop-charPsi}, we may write  
\begin{equation}\label{def-rec}
X_{s,t}(x,v) = X^{mod}_{s,t}(x,v) +  Y_{s,t}(x,v)
, \qquad V_{s,t}(x,v) = V^{mod}_{t,t}(x,v) + W_{s,t}(x,v),
\end{equation}
in which 
$$
\begin{aligned}
X^{mod}_{s,t}(x,v) &= x - (t-s) \hV^{mod}_{t,t}(x,v)
\\V^{mod}_{t,t}(x,v) & =  v -  V^{osc}_{t,t}(x,v)
\\\hV^{mod}_{t,t}(x,v) & =  \hv - \hV^{osc}_{t,t}(x,v).
\end{aligned}
$$
Using the results from Proposition \ref{prop-charV} and Proposition \ref{prop-charPsi}, we have 
\begin{equation}\label{decomp-YW}
 \begin{aligned}
W_{s,t}(x,v) &= V^{osc}_{s,t}(x,v)  + V^{tr}_{s,t}(x,v)
\\
Y_{s,t}(x,v)&= (t-s)\hPsi^{osc}_{s,t}(x,v) + (t-s)\hPsi^{tr}_{s,t}(x,v) + (t-s)\Psi^{Q}_{s,t}(x,v).
 \end{aligned}
 \end{equation}
We first introduce the change of variable $v \mapsto w = V^{mod}_{t,t}(x,v) $, leading to  
$$
\begin{aligned}
\int_0^t \int_{\RR^3} E(s,X_{s,t}(x,v))\cdot \nabla_v \mu(V_{s,t}(x,v)) \, dv  ds 
 = \int_0^t \int_{\RR^3} E(s,\widetilde X_{s,t}(x,w))\cdot \nabla_v \mu(\widetilde V_{s,t}(x,w)) \, \frac{dw  ds}{\widetilde J_{s,t}(x,w)},
\end{aligned} 
$$
in which $\widetilde J_{s,t}(x,w) = \det (\nabla_v V^{mod}_{t,t}(x,v) )$, 
$$\widetilde X_{s,t}(x,w) = x - (t-s)\hw + \widetilde Y_{s,t}(x,w) , \qquad \widetilde V_{s,t}(x,w) = w + \widetilde W_{s,t}(x,w) ,$$
for $ \widetilde Y_{s,t}(x,w) =  Y_{s,t}(x,v)$ and $\widetilde W_{s,t}(x,w) = W_{s,t}(x,v)$, with   
$v = \widetilde V^{mod}_{t,t}(x,w)$, being the inverse map of $v\mapsto w = V^{mod}_{t,t}(x,v)$. Note that such an inverse map exists, upon recalling the bounds \eqref{bdVosct} and \eqref{daVosct} on $V^{osc}_{t,t}(x,v)$.

Next, using the Faa di Bruno's formula
for derivatives of a composite function, see \eqref{daVosci}, we compute 
$$
\begin{aligned}
\partial_x^\alpha & \int_0^t \int_{\RR^3} E(s,X_{s,t}(x,v))\cdot \nabla_v \mu(V_{s,t}(x,v)) \, dv  ds 
\\& = \sum C_{\beta, \gamma}\int_0^t \int_{\RR^3}  \partial_x^{\gamma_1} E(s,\TX_{s,t})\cdot \partial_v^{\gamma_2}\nabla_v \mu(\TV_{s,t} ) \prod_{1\le |\beta_j|\le |\alpha_j|} (\partial_x^{\beta_1} \TX_{s,t})^{k_{\beta_1}} (\partial_x^{\beta_2} \TV_{s,t})^{k_{\beta_2}} \partial_x^{\alpha_3}(\TJ_{s,t}^{-1}) \, dw  ds,
\end{aligned} 
$$
in which the summation is taken over the set where $\sum k_{\beta_j} = |\gamma_j|$ and $\sum |\beta_j|k_{\beta_j} = |\alpha_j|$ for $j=1,2$, with $|\alpha_1|+|\alpha_2|+|\alpha_3|=|\alpha|$. In the case when all the derivatives fall into $E(s,\TX_{s,t})$, namely when $|\gamma_1| = |\alpha|$ (and so $|\alpha_2| = |\alpha_3|=0$), using the fact that $\partial_x\TX_{s,t}= 1 + \partial_x\TY_{s,t}$, we can further write 
$$ 
\begin{aligned}
\int_0^t &\int_{\RR^3}  \partial_x^{\alpha} E(s,\TX_{s,t})\cdot \nabla_v \mu(\TV_{s,t} )  (\partial_x\TX_{s,t})^{|\alpha|} \,\frac{dw  ds}{\widetilde J_{s,t}(x,w)}
\\ & = \sum_{j=0}^{|\alpha|}c_j
\int_0^t \int_{\RR^3}  \partial_x^{\alpha} E(s,\TX_{s,t})\cdot \nabla_v \mu(\TV_{s,t} )  (\partial_x \TY_{s,t})^j \,\frac{dw  ds}{\widetilde J_{s,t}(x,w)}
\\ & = \sum_{j=0}^{|\alpha|}c_j
\int_0^t \int_{\RR^3}  \partial_x^{\alpha} E(s,X_{s,t}(x,v))\cdot \nabla_v \mu(V_{s,t} (x,v))  (\partial_x \TY_{s,t}(x,V^{mod}_{t,t}(x,v)))^j \,dv  ds
\end{aligned}
$$
upon going back the original velocity variables. 
Combining the above calculations, we may write 
\begin{equation}\label{decompDSmu} 
\partial_x^\alpha S^{\mu}(t,x)  = S^{\mu,\alpha}_1(t,x) + S^{\mu,\alpha}_2 (t,x)
+ S^{\mu,\alpha}_3 (t,x) 
\end{equation}
in which 
\begin{equation}\label{def-Smua123}
\begin{aligned}
S^{\mu,\alpha}_1(t,x) & =\int_0^t \int_{\RR^3}  \Big[ \partial_x^\alpha E(s,x - (t-s)\hv )\cdot \nabla_v \mu(v) 
 -\partial_x^\alpha E(s,X_{s,t})\cdot \nabla_v \mu(V_{s,t})\Big] \, dv  ds
\\
S^{\mu,\alpha}_2(t,x) & = -\sum_{j=1}^{|\alpha|}c_j
\int_0^t \int_{\RR^3}  \partial_x^{\alpha} E(s,X_{s,t}(x,v))\cdot \nabla_v \mu(V_{s,t} (x,v))  (\partial_x \TY_{s,t})^j \,dv  ds
\\
S^{\mu,\alpha}_3(t,x) & =- 
\sum_{|\gamma_1|< |\alpha|} C_{\beta, \gamma}\int_0^t \int_{\RR^3}  \partial_x^{\gamma_1} E(s,X_{s,t}(x,v))\cdot \partial_v^{\gamma_2}\nabla_v \mu(V_{s,t} (x,v)) \\& \qquad \times \prod_{1\le |\beta_j|\le |\alpha_j|} (\partial_x^{\beta_1} \TX_{s,t})^{k_{\beta_1}} (\partial_x^{\beta_2} \TV_{s,t})^{k_{\beta_2}} \partial_x^{\alpha_3}(\TJ_{s,t}^{-1}) \TJ_{s,t}\, dv  ds,
\end{aligned}
\end{equation}
in which the characteristic functions $\TX_{s,t}, \TY_{s,t}, \TV_{s,t}$ and $\TJ_{s,t}$ are evaluated at $(x,w) = (x,V^{mod}_{t,t}(x,v))$ with $V^{mod}_{t,t}(x,v) =  v -  V^{osc}_{t,t}(x,v)$.  Let us now treat each integral in \eqref{def-Smua123}. 

\subsection*{Bounds on $S^{\mu,\alpha}_1(t,x)$.}

By construction \eqref{decompDSmu}-\eqref{def-Smua123}, the integral $S^{\mu,\alpha}_1(t,x)$ is identical to $S^\mu(t,x)$ upon replacing $E$ by $\partial_x^\alpha E$. Therefore, following the proof of Proposition \ref{prop-SR}, we may write 
\begin{equation}\label{rewSmu1}
\begin{aligned}
S^{\mu,\alpha}_1(t,x)  =  \int_0^t \int_{\RR^3} \partial_x^\alpha E(s,x - (t-s)\hw ) \cdot H_{s,t}(x,w) \; dw ds
\end{aligned}
\end{equation}
in which the interaction kernel $H_{s,t}(x,w)$ is the same as that defined in Proposition \ref{prop-SR}. In the case when $|\alpha| = N_0$, using $\| \partial_x^{\alpha_0} E(s)\|_{L^2_x} \lesssim \epsilon$ and $\| \langle w\rangle^{4}H_{s,t}\|_{L^\infty_{x,w}} \lesssim \epsilon \langle s\rangle^{-1}$, 
we bound 
$$
\begin{aligned}
\|S^{\mu,\alpha_0}_1(t)\|_{L^2_x}  &\le  \int_0^t \int_{\RR^3} \| \partial_x^{\alpha_0} E(s)\|_{L^2_x} \| \langle w\rangle^{4}H_{s,t}\|_{L^\infty_{x,w}} \langle w\rangle^{-4}\; dw ds
\\
&\lesssim  \epsilon^2 \int_0^t  \langle s\rangle^{-1} ds \lesssim \epsilon^2 \log t,
\end{aligned}
$$
which is bounded by $\epsilon^2 \langle t\rangle^{\delta_1}$ as desired in \eqref{Hsbounds-cSD0}. On the other hand, in view of the bootstrap assumptions in \eqref{bootstrap-decaydaS}, $\partial_x^\alpha E(t,x)$ decays at a slower rate of order $\langle t\rangle^{\delta_{\alpha,p}}$ than those for $E(t,x)$, without disturbing the structure of oscillations $E^{osc}_\pm(t,x)$ (i.e. up to a multiplication of $k^\alpha$). Therefore, the decay estimates on $S^{\mu,\alpha,r}_1$ follow identically from the proof of Proposition \ref{prop-decayEHR}, propagating the same growth rate as reported in \eqref{Hsbounds-cSD0}, for $1\le |\alpha|\le N_0-1$.

\subsection*{Bounds on $S^{\mu,\alpha}_2(t,x)$.} 

Next, we shall bound $S^{\mu,\alpha}_2(t,x)$ defined as in \eqref{decompDSmu}-\eqref{def-Smua123}. Again, following the proof of Proposition \ref{prop-SR}, we introduce the change of variables $w = \Psi_{s,t}(x,v)$ in the above integral, leading to 
\begin{equation}\label{rewSmu2}
\begin{aligned}
S^{\mu,\alpha}_2(t,x)  =  - \int_0^t \int_{\RR^3} \partial_x^\alpha E(s,x - (t-s)\hw ) \cdot H_{s,t,2}(x,w) \; dw ds
\end{aligned}
\end{equation}
in which 
$$
H_{s,t,2}(x,w) =\frac{\nabla_v \mu(V_{s,t}(x,v))}{\det(\nabla_v\Psi_{s,t}(x,v))} \sum_{j=1}^{|\alpha|} c_j(\partial_x \TY_{s,t}(x,V^{mod}_{t,t}(x,v)))^j,
$$
for each $(x,w) \in \RR^3\times \RR^3$ and for $v = \widetilde \Psi_{s,t}(x,w)$, the inverse map of $v\mapsto w = \Psi_{s,t}(x,v)$. That is, the integral $S^{\mu,\alpha}_2(t,x)$ is of the exact same form as that of $S^{\mu,\alpha}_1(t,x)$ in \eqref{rewSmu1}
Therefore, it suffices to prove that the interaction kernel $H_{s,t,2}(x,w)$ behaves as does $H_{s,t}(x,w)$. Indeed, in view of \eqref{def-Hst}, we may write 
$$ H_{s,t,2}(x,w) = \nabla_w \mu(w)\sum_{j=1}^{|\alpha|} c_j(\partial_x \TY_{s,t}(x,V^{mod}_{t,t}(x,v)))^j - H_{s,t}(x,w)  \sum_{j=1}^{|\alpha|} c_j(\partial_x \TY_{s,t}(x,V^{mod}_{t,t}(x,v)))^j,$$
in which $v = \widetilde \Psi_{s,t}(x,w)$. 
By construction, we have $\TY_{s,t}(x,w) = Y_{s,t}(x, \TV^{mod}_{t,t}(x,w))$, and so 
$$\partial_x \TY_{s,t}(x,w) = \partial_xY_{s,t}(x, \TV^{mod}_{t,t}(x,w)) + \partial_x \TV^{mod}_{t,t}(x,w) \partial_v Y_{s,t}(x, \TV^{mod}_{t,t}(x,w)),$$
for each $(x,w) \in \RR^3\times \RR^3$. Therefore, since $\TV^{mod}_{t,t}(x,\TV^{mod}_{t,t}(x,v)) =v$, we have 
\begin{equation}\label{comp-dxTY}\partial_x \TY_{s,t}(x,V^{mod}_{t,t}(x,v)) = \partial_xY_{s,t}(x, v) + \partial_x \TV^{mod}_{t,t}(x,V^{mod}_{t,t}(x,v)) \partial_v Y_{s,t}(x, v),
\end{equation}
for each $(x,v)$. Recall that $V^{mod}_{t,t}(x,v) = v - V^{osc}_{t,t}(x,v)$, with $V^{osc}_{t,t}(x,v)$ satisfying the bounds \eqref{Dxv-VWend}. It thus follows that $\| \partial_x \TV^{mod}_{t,t}(x,w) \|_{L^\infty_{x,w}}\lesssim \|\partial_xV^{osc}_{t,t}\|_{L^\infty_{x,v}} \lesssim \epsilon \langle t\rangle^{-3/2}$. Therefore, recalling \eqref{decomp-YW} and following the proof of Proposition \ref{prop-Dchar} (precisely, see \eqref{def-zetadxV}), we obtain 
$$
\begin{aligned}
\|\partial_x \TY_{s,t}\|_{L^\infty_{x,w}} &\lesssim \|\partial_xY_{s,t}\|_{L^\infty_{x,v}}+ \| \partial_x \TV^{mod}_{t,t}\|_{L^\infty_{x,w}} \|\partial_v Y_{s,t}\|_{L^\infty_{x,v}}
\lesssim \epsilon \langle s\rangle^{-1}. 
\end{aligned}$$
In addition, by construction, we note that 
\begin{equation}\label{repYst}
Y_{s,t}(x,v)= (t-s)\hPsi^{osc}_{s,t}(x,v) + (t-s)\hPsi^{tr}_{s,t}(x,v) + (t-s)\Psi^{Q}_{s,t}(x,v)
\end{equation}
where the components $\hPsi^{osc}_{s,t}$, $\hPsi^{tr}_{s,t}$, and $\Psi^Q_{s,t}$ are defined as in Proposition \ref{prop-charPsi}. As a consequence, $\partial_x \TY_{s,t}(x,V^{mod}_{t,t}(x,v))$ and so $H_{s,t,2}(x,w)$ have a similar representation to that of $H_{s,t}(x,w)$. The estimates for $S^{\mu,\alpha}_2(t,x)$ thus follow similarly to those for $S^{\mu,\alpha}_1(t,x)$ obtained in the previous section. 

\subsection*{Bounds on $S^{\mu,\alpha}_3(t,x)$.} 

Finally, we shall bound $S^{\mu,\alpha}_3(t,x)$ defined as in \eqref{decompDSmu}-\eqref{def-Smua123}. In view of Section \eqref{sec-Hdxchar}, spatial derivatives of $V_{s,t}$ satisfy better estimates than those for $X_{s,t}$. As a consequence, it suffices to bound the above integral when all derivatives fall into $X_{s,t}$, namely 
$$
S^{\mu,\alpha}_{3,\gamma}(t,x) =\int_0^t \int_{\RR^3}  \partial_x^{\gamma} E(s,X_{s,t}(x,v))\cdot \nabla_v \mu(V_{s,t}(x,v) ) \prod_{1\le |\beta|\le |\alpha|}  (\partial_x^{\beta} \TX_{s,t}(x,V^{mod}_{t,t}(x,v)))^{k_{\beta}} \,dv ds
$$
in which $\sum k_{\beta} = |\gamma| <|\alpha|$ and $\sum |\beta|k_{\beta} = |\alpha|$. Note that the case when $|\gamma|=|\alpha|$ is already treated in $S^{\mu,\alpha}_{1}(t,x)$ and $S^{\mu,\alpha}_{2}(t,x)$. Again, as in the previous case, we write 
\begin{equation}\label{rewSmu3}
\begin{aligned}
S^{\mu,\alpha}_{3,\gamma}(t,x)  =   \int_0^t \int_{\RR^3} \partial_x^\gamma E(s,x - (t-s)\hw ) \cdot H_{s,t,3}(x,w) \; dw ds
\end{aligned}
\end{equation}
in which 
$$
\begin{aligned}
H_{s,t,3}(x,w) 
&=\frac{\nabla_v \mu(V_{s,t}(x,v))}{\det(\nabla_v\Psi_{s,t}(x,v))}  \prod_{1\le |\beta|\le |\alpha|}  (\partial_x^{\beta} \TX_{s,t}(x,V^{mod}_{t,t}(x,v)))^{k_{\beta}}
\\& = (\nabla_w\mu(w) - H_{s,t}(x,w))\prod_{1\le |\beta|\le |\alpha|}  (\partial_x^{\beta} \TX_{s,t}(x,V^{mod}_{t,t}(x,v)))^{k_{\beta}},
\end{aligned}$$
upon using \eqref{def-Hst}. It remains to prove that $H_{s,t,3}(x,w)$ satisfies a similar expression to that of $H_{s,t}(x,w)$. Recalling $\partial^\beta_x\TX_{s,t}= 1 + \partial^\beta_x\TY_{s,t}$ for $|\beta|=1$, and $\partial_x^\beta\TX_{s,t}= \partial^\beta_x\TY_{s,t}$ for $|\beta|\ge 2$, we write  
\begin{equation}\label{prodX3b} \prod_{1\le |\beta|\le |\alpha|}  (\partial_x^{\beta} \TX_{s,t})^{k_{\beta}}  =  \Big(1+\sum_{j=0}^{k_1} c_j (\partial_x\TY_{s,t})^j\Big)\prod_{2\le |\beta|\le |\alpha|}  (\partial_x^{\beta} \TY_{s,t})^{k_{\beta}}. 
\end{equation}
Recall that $\sum k_{\beta} = |\gamma|$ and $\sum |\beta|k_{\beta} = |\alpha|$. Since $|\gamma|<|\alpha|$, the set of $|\beta|\ge 2$ in the above product is non empty. 
We first prove that 
\begin{equation}\label{claimbTYst}
\begin{aligned}
\| \partial_x^\beta\TY_{s,t}\|_{L^\infty_{x,w}} 
&\lesssim \epsilon\langle s\rangle^{-1 +\frac12 \delta_\beta} \langle t\rangle^{ \delta_{\beta}},
\end{aligned}
\end{equation}
for $|\beta|\ge 2$. Indeed, from $\TY_{s,t}(x,w) = Y_{s,t}(x, \TV^{mod}_{t,t}(x,w))$, we have 
$$ \partial_x^\beta\TY_{s,t}(x,w) = \sum_{|\beta_1|+|\beta_2| = |\beta|} \partial_x^{\beta_1}\partial_v^{\sigma_2}Y_{s,t}(x, \TV^{mod}_{t,t}(x,w)) \prod_{1\le |\sigma|\le |\beta_2|} (\partial_x^\sigma \TV^{mod}_{t,t}(x,w))^{k_\sigma}$$
in which $\sum k_\sigma = |\sigma_2| \le |\beta_2|$ and $\sum |\sigma| k_\sigma = |\beta_2|$. Using \eqref{decayXV}, \eqref{daVosct}, 
and the fact that $V^{mod}_{t,t}(x,v) = v - V^{osc}_{t,t}(x,v)$, we obtain 
$$
\begin{aligned}
\| \partial_x^\beta\TY_{s,t}\|_{L^\infty_{x,w}} 
&\lesssim \sum_{|\beta_1|+|\beta_2| = |\beta|} \| \partial_x^{\beta_1}\partial_v^{\sigma_2}Y_{s,t}\|_{L^\infty_{x,v}} \prod_{1\le |\sigma|\le |\beta_2|} \|\partial_x^\sigma \TV^{mod}_{t,t}\|_{L^\infty_{x,w}}^{k_\sigma}
\\
&\lesssim \epsilon^2\sum_{|\beta_1|+|\beta_2| = |\beta|}  \langle s\rangle^{-1+\frac12\delta_{\beta_1}+\frac12 \delta_{\sigma_2}} \langle t\rangle^{\delta_{\beta_1}+\delta_{\sigma_2}+|\sigma_2|
}\prod_{1\le |\sigma|\le |\beta_2|} \langle t\rangle^{-\frac{3k_\sigma}{2} + \frac32\delta_\sigma k_\sigma}
\\
&\lesssim \epsilon^2\sum_{|\beta_1|+|\beta_2| = |\beta|}  \langle s\rangle^{-1+\frac12\delta_{\beta_1}+\frac12 \delta_{\sigma_2} } \langle t\rangle^{\delta_{\beta_1}+\delta_{\sigma_2}+|\sigma_2|
} \langle t\rangle^{-\frac{3}{2}\sum k_\sigma +\frac32\sum \delta_\sigma k_\sigma}
.\end{aligned}$$
Recalling $\sum k_\sigma = |\sigma_2|$ and $\sum |\sigma| k_\sigma = |\beta_2|$, we obtain 
$$
\begin{aligned}
\| \partial_x^\beta\TY_{s,t}\|_{L^\infty_{x,w}} 
&\lesssim \epsilon\sum_{|\beta_1|+|\beta_2| = |\beta|}  \langle s\rangle^{-1+\frac12\delta_{\beta_1}+\frac12 \delta_{\sigma_2} } \langle t\rangle^{\delta_{\beta_1}+\delta_{\sigma_2}+|\sigma_2|
} \langle t\rangle^{-\frac{3}{2}|\sigma_2| +\frac32\delta_{\beta_2}}
\\
&\lesssim \epsilon\sum_{|\beta_1|+|\beta_2| = |\beta|}  \langle s\rangle^{-1 +\frac12\delta_{\beta_1}+\frac12 \delta_{\sigma_2}} \langle t\rangle^{\delta_{\sigma_2} -\frac{1}{2}|\sigma_2|
+ \frac12 \delta_{\beta_2}} \langle t\rangle^{ \delta_{\beta}}.
\end{aligned}$$
We now prove the claim \eqref{claimbTYst}. We first assume that 
\begin{equation}\label{ineqsigma}\delta_{\sigma_2} -\frac{1}{2}|\sigma_2|
+ \frac12 \delta_{\beta_2} \le 0
\end{equation}
 in the above summation. Then, since $s\le t$, it follows that 
 $$
\begin{aligned}
\| \partial_x^\beta\TY_{s,t}\|_{L^\infty_{x,w}} 
&\lesssim \epsilon\sum_{|\beta_1|+|\beta_2| = |\beta|}  \langle s\rangle^{-1 +\frac12\delta_{\beta_1}+\frac12 \delta_{\sigma_2}} \langle s\rangle^{\delta_{\sigma_2} -\frac{1}{2}|\sigma_2|
+ \frac12 \delta_{\beta_2}} \langle t\rangle^{ \delta_{\beta}}
\\
&\lesssim \epsilon\sum_{|\beta_1|+|\beta_2| = |\beta|}  \langle s\rangle^{-1 +\frac12\delta_{\beta_1}+ \frac12 \delta_{\beta_2}+\frac32 \delta_{\sigma_2}-\frac{1}{2}|\sigma_2|
} \langle t\rangle^{ \delta_{\beta}}
\\
&\lesssim \epsilon
\langle s\rangle^{-1 +\frac12\delta_{\beta}
} \langle t\rangle^{ \delta_{\beta}}
,\end{aligned}$$
yielding \eqref{claimbTYst}, in which we used $\delta_{\beta_1}+\delta_{\beta_2} = \delta_\beta$ and $\frac32 \delta_{\sigma_2} = \frac32 \frac{|\sigma_2|}{N_0-1} \le \frac12 |\sigma_2|$, since $N_0 \ge 4$. It remains to check \eqref{ineqsigma}. Indeed, this is clear when $|\sigma_2|=0$ (i.e. no $v$-derivatives), since we must have $|\beta_2|=0$. 
In addition, the inequality is clear when $|\sigma_2|\ge 2$, since $\delta_{\beta_2} \le 1$ and $\delta_{\sigma_2} \le \frac14 |\sigma_2|$ (since $N_0\ge 5$). It remains to consider the case $|\sigma_2|=1$. If $|\beta_1|=0$, then then in view of Proposition \ref{prop-Dchar}, we have $\|\partial_vY_{s,t}\|_{L^\infty_{x,v}} \lesssim \epsilon \langle s\rangle^{-1}\langle t\rangle$, namely there is no growth factor $\langle t\rangle^{\delta_{\sigma_2}}$ in the above estimate, and 
\eqref{ineqsigma} holds, since $\delta_{\beta_2} \le 1$. On the other hand, if $|\beta_1|\ge 2$, we have $|\beta_2| = |\beta| -|\beta_1| \le |\beta|-2\le N_0-3$, and so $\delta_{\sigma_2} + \frac12 \delta_{\beta_2} \le \frac12 (\delta_{2}+\delta_{\beta_2})\le \frac12 \delta_{\beta_2+2} \le \frac12$ which verifies \eqref{ineqsigma}, upon recalling $|\sigma_2|=1$ and $\delta_\beta = \frac{|\beta|}{N_0-1}$. Finally, we consider the case when $|\sigma_2| = |\beta_1|=1$. If $|\beta_2| = |\beta|-1\le N_0-3$, then $\delta_{\sigma_2} + \frac12 \delta_{\beta_2} \le \frac12 \delta_{\beta_2+2} \le \frac12$, and \eqref{ineqsigma} is again valid. For the remaining case when $|\beta| = N_0-1$, we write 
$$\| \partial_x^{\beta_1}\partial_v^{\sigma_2}Y_{s,t}\|_{L^\infty_{x,v}} \prod_{1\le |\sigma|\le |\beta_2|} \|\partial_x^\sigma \TV^{mod}_{t,t}\|_{L^\infty_{x,w}}^{k_\sigma} 
= \| \partial_x\partial_vY_{s,t}\|_{L^\infty_{x,v}} \|\partial_x^{\beta_2} \TV^{mod}_{t,t}\|_{L^\infty_{x,w}} $$
for $|\beta_2|=|\beta|-1 = N_0-2$. Using Proposition \ref{prop-HDchar}, we bound $ \| \partial_x\partial_vY_{s,t}\|_{L^\infty_{x,v}} \lesssim \epsilon \langle s\rangle^{-1+\frac12\delta_2}\langle t\rangle^{1+\delta_2}$. On the other hand, we use the following stronger decay estimates, see Proposition \ref{prop-decayEosc},
\begin{equation}\label{strongdecayEosc}
 \| \partial_x^{\beta_0} E^{osc}_\pm(t)\|_{L^\infty_x} \lesssim \epsilon \langle t\rangle^{-3/2}, \qquad \|\partial_x^{\alpha_0-1} E^{osc}_\pm(t)\|_{L^\infty_x} \lesssim \epsilon, 
  \end{equation}
for $|\beta_0| = n_0$ and $|\alpha_0|=N_0$, where $n_0 = \frac14 N_0$. Therefore, by the interpolation inequality \eqref{dx-interpolate}, we bound 
$$\|\partial_x^{\beta_2} V^{mod}_{t,t}\|_{L^\infty_{x,v}} \lesssim \|\partial_x^{\beta_2} E^{osc}_\pm(t)\|_{L^\infty_{x,v}} \lesssim \epsilon \langle t\rangle^{-\frac32 + \frac32 \epsilon_{\beta_2}}$$
in which $\epsilon_{\beta_2} = \frac{|\beta_2| - n_0}{N_0 - 1 - n_0}$. Since $|\beta_2|= N_0-2$, we have $\epsilon_{\beta_2} = 1 - \frac{1}{N_0-n_0-1}= 1 - \frac{4}{3N_0-4}$, and so 
$$
\begin{aligned}
 \| \partial_x\partial_vY_{s,t}\|_{L^\infty_{x,v}} \|\partial_x^{\beta_2} \TV^{mod}_{t,t}\|_{L^\infty_{x,w}} 
 &\lesssim \epsilon \langle s\rangle^{-1+\frac12\delta_2}\langle t\rangle^{1+\delta_2} \langle t\rangle^{-\frac32 + \frac32 \epsilon_{\beta_2}}
\\
 &\lesssim \epsilon \langle s\rangle^{-1+\frac12\delta_\beta}\langle t\rangle^{1+\delta_2-\frac{6}{3N_0-4}}
,\end{aligned}$$  
in which $\delta_2 = \frac{2}{N_0-1} \le \frac{6}{3N_0-4}$. This proves \eqref{claimbTYst} for this term, upon recalling that we are in the case when $|\beta|=N_0-1$ and $\delta_\beta = 1$. The proof of \eqref{claimbTYst} is thus complete. 

Therefore, putting \eqref{claimbTYst} into \eqref{prodX3b}, we obtain 
\begin{equation}\label{prodX3b1}
\Big\| \prod_{1\le |\beta|\le |\alpha|}  (\partial_x^{\beta} \TX_{s,t})^{k_{\beta}} \Big\|_{L^\infty_{x,w}} \lesssim  \prod_{2\le |\beta|\le |\alpha|}  \|\partial_x^{\beta} \TY_{s,t}\|_{L^\infty_{x,w}}^{k_{\beta}} \lesssim \epsilon \langle s\rangle^{-\sum_* k_\beta } \langle t\rangle^{ \sum_*\delta_{\beta}k_\beta},
\end{equation}
in which the summation $\sum_*$ is taken over the set $|\beta|\ge 2$. By construction, we note that $k_1+\sum_* k_\beta = |\gamma|$ and $k_1+\sum_* |\beta|k_\beta = |\alpha|$, for $k_1$ being the number of first derivatives in the above product. We thus have 
\begin{equation}\label{prodX3b1}
\Big\| \langle s\rangle^{\delta_\gamma} \prod_{1\le |\beta|\le |\alpha|}  (\partial_x^{\beta} \TX_{s,t})^{k_{\beta}} \Big\|_{L^\infty_{x,w}} \lesssim \epsilon \langle s\rangle^{\delta_\gamma-|\gamma|+k_1} \langle t\rangle^{\delta_\alpha - \delta_{k_1}} \lesssim \epsilon \langle t\rangle^{\delta_\alpha},
\end{equation}
upon noting that $s\le t$ and $\delta_\gamma - \delta_{k_1} = \frac{|\gamma|-k_1}{N_0-1} \le |\gamma|-k_1$, since $N_0\ge 2$. This proves that the integrand $\partial_x^\gamma E(s,x - (t-s)\hw ) \cdot H_{s,t,3}(x,w)$ in the integral \eqref{rewSmu3} satisfies the same decay as that of $E(s,x - (t-s)\hw ) \cdot H_{s,t}(x,w)$, up to a growth factor of order $\langle t\rangle^{\delta_\alpha}$ as reported in \eqref{Hsbounds-cSD0}. The oscillatory structure of $H_{s,t,3}(x,w)$ follows from that of $\TY_{s,t}$ and $Y_{s,t}$ as seen in \eqref{comp-dxTY}-\eqref{repYst}. 

This completes the proof of Proposition \ref{prop-bdSD}. 
\end{proof}


\section{Decay of the oscillatory field}\label{sec-decayosc}


In this section, we bound the oscillatory electric field, proving Proposition \ref{prop-decayEosc-goal} which we recall below.  

\begin{proposition}\label{prop-decayEosc} 
Fix $N_0\ge 10$ and $n_0 = \frac14 N_0$. Let $S$ be the nonlinear source term defined as in \eqref{nonlinear-S}.
Then, there hold
\begin{equation}\label{decayEosc}
\begin{aligned}
\|\partial^\alpha_xG_\pm^{osc} \star_{t,x} \nabla_x S(t)\|_{L^\infty_x} 
&\lesssim (\epsilon_0 + \epsilon^2) \langle t\rangle^{-3/2} ,\qquad |\alpha|\le n_0,
\\
\|\partial^{\alpha_0-1}_xG_\pm^{osc} \star_{t,x} \nabla_x S(t)\|_{L^\infty_x} 
&\lesssim (\epsilon_0 + \epsilon^2) ,\qquad |\alpha_0| = N_0,
\end{aligned}\end{equation}
and 
\begin{equation}\label{bdEosc}
\begin{aligned}
\|G_\pm^{osc} \star_{t,x} \nabla_x S(t)\|_{H^{\alpha_0}_x} &\lesssim (\epsilon_0 + \epsilon^2)
\\
\|G_\pm^{osc} \star_{t,x} \nabla_x S(t)\|_{H^{\alpha_0+1}_x} &\lesssim (\epsilon_0 + \epsilon^2) \langle t\rangle^{\delta_1} .
\end{aligned}\end{equation}
\end{proposition}

The proof consists of several steps that heavily rely on the nonlinear structure of $S(t,x)$. 
Indeed, thanks to the transport property, the nonlinear source density $S(t,x)$ is at best of order $t^{-3}$, which is insufficient for the spacetime convolution $G^{osc}_{\pm} \star_{t,x} S(t,x)$ to disperse like a Klein-Gordon wave at order $t^{-3/2}$, see for instance Lemma \ref{lem-decayosc} below, roughly requiring $S(t)$ to be of order $t^{-4}$. For this reason, we need to further exploit the nonlinear structure hidden in the convolution.


%


\subsection{A useful lemma}\label{sec-lemmas}

In this section, we give some useful estimates on the oscillatory Green kernel. Precisely, we obtain the following.

\begin{lemma}\label{lem-decayosc} 
Suppose that 
$$ \int_0^t \| J(s)\|_{W_x^{1+\lfloor 3\left(1-\frac{2}{p}\right) \rfloor, p'}} \;ds \lesssim 1, \qquad \| J(t)\|_{H^{2}_x} \lesssim \langle t\rangle^{-5/2} .$$
Then, there holds 
$$ \|G^{osc}_\pm \star_{t,x} \nabla_x J(t)\|_{L^p_x} \lesssim \langle t\rangle^{-3(1/2-1/p)} $$
for $2\le p\le\infty$. 
\end{lemma}

\begin{proof}
The lemma is direct. Indeed, using the results from Proposition \ref{prop-Greenphysical} on the oscillatory Green function, we bound
$$
\begin{aligned} 
\| G^{osc}_\pm \star_{t,x} \nabla_x J\|_{L^p_x} 
&\le \int_0^t \| G^{osc}_\pm(t-s) \star_x \nabla_x J(s)\|_{L^p_x} \; ds
\\& \lesssim \int_0^{t/2} \langle t-s\rangle^{-3(1/2-1/p)} \| J(s)\|_{W_x^{1+\lfloor 3\left(1-\frac{2}{p}\right) \rfloor, p'}}\; ds + \int_{t/2}^t \| J(s)\|_{H^{2}_x} \; ds
\\& \lesssim \langle t\rangle^{-3(1/2-1/p)}\int_0^{t/2}\| J(s)\|_{W_x^{1+\lfloor 3\left(1-\frac{2}{p}\right) \rfloor, p'}}\; ds +\int_{t/2}^t \langle s\rangle^{-5/2}\; ds,
\end{aligned}$$
which ends the proof, using the assumptions made on $J(t)$. 
\end{proof}

\subsection{Collective oscillations} 

We obtain the following.  

\begin{proposition}\label{prop-convGosc}
Let $S^{\mu}(t,x)$ be the source term defined as in \eqref{nonlinear-Smu}.  
Then, there holds
\begin{equation}\label{convGosc}
G_\pm^{osc} \star_{t,x} S^{\mu}(t,x) = a_\pm(i\partial_x) S_{\pm,1}^{\mu}(t,x) + G_\pm^{osc} \star_{t,x} S_{\pm,2}^{\mu}(t,x)
\end{equation}
where $a_\pm(k)$ is the Fourier multiplier as in \eqref{def-Gosc}, and 
$$ 
\begin{aligned}
S_{\pm,1}^{\mu}(t,x)  &=  \int_0^t \int \phi^v_{\pm,1}(i\partial_x)\Big[ E(s,x - (t-s)\hv) \cdot \nabla_v \mu(v) 
 - E(s,X_{s,t}(x,v))\cdot \nabla_v \mu(V_{s,t}(x,v)) \Big]\,dv ds 
\\
S_{\pm,2}^{\mu}(t,x)  &= 
 \int_0^t \int \nabla_v\phi^v_{\pm,1}(i\partial_x)\Big[ E(t,x)E(s,X_{s,t}(x,v)) \cdot \nabla_v \mu(V_{s,t}(x,v)) \Big] \, dv ds 
 \end{aligned}$$ 
with a Fourier multiplier $\phi_{\pm,1}^v(k)$ defined as in  \eqref{def-phipmj}. 
\end{proposition}

We note that $S_{\pm,j}^{\mu}(t,x)$ are of the same order as that of $S^{\mu}(t,x) $. The advantage is however that $S_{\pm,2}^{\mu}(t,x) $ has a more apparent quadratic interaction between the electric field at $(t,x)$ and that along the characteristic $X_{s,t}(x,v)$ than that in the original source $S^{\mu}(t,x)$.

\begin{proof} Let $S_1(t,x)$ be the first integral in the definition of $S^{\mu}(t,x)$, namely
$$S_1(t,x) = \int_0^t \int_{\RR^3} E(s,x - (t-s)\hv )\cdot \nabla_v \mu(v)  \, dv  ds.
$$
By definition, we compute 
$$
\begin{aligned}
G^{osc}_\pm \star_{t,x} S_1(t,x) &= \int_0^t \int G^{osc}_\pm(t - \tau, x - y) S_1(\tau,y) \, dy \, d\tau
\\
&= \int_0^t \int G^{osc}_\pm(t  - \tau,x - y)  \Bigl( \int_0^\tau \int E(s, y - (\tau - s)  \hv) \cdot \nabla_v \mu(v)\, dv ds 
\Bigr) \, dy  d\tau
\\
&= \int_0^t \int \Bigl( \int_s^t \int G^{osc}_\pm(t  - \tau,x - y) E(s, y - (\tau - s)  \hv) \cdot \nabla_v \mu(v)  \, dy  d\tau \Bigr)  dv  ds.
\end{aligned}$$
Next, we set $Y = y - (\tau - s) \hv$ in the $y$-integration, which leads to 
$$
G^{osc}_\pm \star_{t,x} S_1(t,x) =
\int_0^t \iint \Bigl[ \int_s^t G^{osc}_\pm(t-\tau,x - Y - (\tau - s)  \hv) d\tau \Bigr] E(s,Y) \cdot \nabla_v \mu(v)  dY  dv  ds.
$$
Following the calculation done in \eqref{int-Eosc}, we compute the integral of $G^{osc}_\pm$ over a straight line, giving  
$$
\int_s^t G^{osc}_\pm(t-\tau,x - Y - (\tau - s) \hv) d\tau = G^{osc,1}_\pm(0,x - Y - (t - s) \hv,v) 
- G^{osc,1}_\pm(t -s, x - Y,v) ,
$$
where, as in \eqref{def-Eoscj}-\eqref{def-phipmj}, we denote $G^{osc,j}_\pm(t,x,v) = \phi^v_{\pm,j}(i\partial_x)G^{osc}_{\pm}(t,x)$. This leads to 
\begin{equation}\label{cal1}G^{osc}_\pm \star_{t,x} S_1(t,x) = J_{1,1} + J_{1,2}\end{equation}
where
$$
\begin{aligned}J_{1,1} &= \int_0^t \iint G^{osc,1}_\pm(0,x - Y - (t - s) \hv,v) E(s,Y) \cdot \nabla_v \mu(v)  \, dY \, dv \, d\tau 
\\
J_{1,2} &= - \int_0^t \iint G^{osc,1}_\pm(t-s,x - Y,v) E(s,Y) \cdot \nabla_v \mu(v) \, dY \, dv \, d\tau .
\end{aligned}$$

Similarly, letting $S_2(t,x)$ be the second integral in the definition of $S^{\mu}(t,x)$, we get  
$$
G^{osc}_\pm \star_{t,x} S_2(t,x) = \int_0^t \int \Bigl( \int_s^t \int G^{osc}_\pm(t  - \tau,x - y) 
E(s,X_{s,\tau}(y,v))\cdot \nabla_v \mu(V_{s,\tau}(y,v))\, dy d\tau \Bigr) \,  dvds.
$$
In the spirit of the case of straight characteristics, we fix the position and speed of the characteristics at time $s$,
and introduce the change of variables  
$$
(Y,V) = ( X_{s,\tau}(y,v), V_{s,\tau}(y,v)), 
$$ 
whose Jacobian determinant is equal to one. By the time reversibility, we note that
$$
(y,v) = (X_{\tau,s}(Y,V), V_{\tau,s}(Y,V)) .
$$
This leads to 
$$
G^{osc}_\pm \star_{t,x} S_2(t,x) = \int_0^t \iint \Bigl[ \int_s^t G^{osc}_\pm(t  - \tau,x - X_{\tau,s}(Y,V)) 
\, \, d\tau \Bigr] E(s,Y)\cdot \nabla_v \mu(V) \, dY dV ds.
$$
We next evaluate $G^{osc}_\pm$ along the nonlinear characteristics. Indeed, following the calculation done in \eqref{int-Eosc}, we obtain 
$$
\begin{aligned}
&\int_s^t G^{osc}_\pm(t - \tau, x - X_{\tau,s}) d \tau  
= G^{osc,1}_\pm(0, x - X_{t,s},V_{t,s}) 
- G^{osc,1}_\pm(t - s, x - Y,V) + R_{s,t}
\end{aligned}$$
upon using $X_{s,s} = Y$ and $V_{s,s} = V$, where
\begin{equation}\label{def-RstYV}
\begin{aligned}
R_{s,t} (Y,V)&= -  \int_s^t \partial_\tau V_{\tau,s}\cdot \nabla_v(G^{osc,1}_{\pm}(t-\tau,x-X_{\tau,s},V_{\tau,s}) )  
\;d \tau .
\end{aligned}\end{equation}
We stress that $\nabla_v$ in the above expression denotes the partial derivative in $v$ of the function $G^{osc,1}_{\pm}(t,x,v)  $. This yields 
\begin{equation}\label{cal2}G^{osc}_\pm \star_{t,x} S_2(t,x) = J_{2,1} + J_{2,2} - \cE^{m,R} \end{equation}
where
$$
\begin{aligned}J_{2,1} &=  \int_0^t \iint G^{osc,1}_\pm(0, x - X_{t,s}(Y,V),V_{t,s}(Y,V)) E(s,Y)\cdot \nabla_v \mu(V)  \, dY dV \, ds
\\
J_{2,2} &= -  \int_0^t \iint G^{osc,1}_\pm(t - s, x - Y,V) E(s,Y)\cdot \nabla_v \mu(V)\, dY dV \, ds,
\\
\cE^{m,R} &= \int_0^t \iint R_{s,t} (Y,V)\, E(s,Y) \cdot \nabla_v \mu(V)  \, dV \, dY \, ds.
\end{aligned}$$ 
Observe that $J_{1,2}$ and $J_{2,2}$ are identical and thus cancelled out in $G^{osc}_\pm \star_{t,x} S^{\mu}(t,x)$. Putting \eqref{cal1} and \eqref{cal2} together, we thus obtain 
$$
\begin{aligned}
 G^{osc}_\pm \star_{t,x} S^{\mu}(t,x) &=  G^{osc}_\pm \star_{t,x} S_1(t,x) -  G^{osc}_\pm \star_{t,x} S_2(t,x) 
 \\&= J_{1,1} - J_{2,1} + \cE^{m,R} .
 \end{aligned}$$
Reversing the change of variables $y= Y + (t- s) \hv$ in the integral $J_{1,1}$ and $(y,v) = (X_{t,s}(Y,V),V_{t,s}(Y,V))$ in the integral $J_{2,1}$, we have 
$$
\begin{aligned}
J_{1,1} - J_{2,1} 
&=  \int_0^t \iint G^{osc,1}_\pm(0,x - y,v) \Big[ E(s,y - (t-s)\hv) \cdot \nabla_v \mu(v) 
\\&\quad - E(s,X_{s,t}(y,v))\cdot \nabla_v \mu(V_{s,t}(y,v)) \Big]  \, dydv ds .
\end{aligned}$$
Recall that $G^{osc,1}_\pm(0,x,v) = \phi^v_{\pm,1}(i\partial_x)G^{osc}_{\pm}(0,x) = \phi^v_{\pm,1}(i\partial_x)a_\pm(i\partial_x)$, where we used from \eqref{Gsreal} that
$G^{osc}_\pm(0,x)$ is a Fourier multiplier $a_\pm(k)$. This defines $ S_{\pm,1}^{\mu}(t,x) $ as stated in the lemma.  

Finally, we study the remainder integral $\cE^{m,R} $. Similarly as done above, we reverse the change of variables $(y,v) = (X_{\tau,s}(Y,V),V_{\tau,s}(Y,V))$  in the integration and recall that $\partial_\tau V_{\tau,s} = E(\tau, X_{\tau,s})$ to get 
$$ 
\begin{aligned}
\cE^{m,R} (t,x) 
&= \int_0^t \int \int_0^\tau \int \nabla_v [G^{osc,1}_{\pm}(t-\tau,x-y,v)]  \cdot E(\tau,y)
\\&\qquad \times E(s,X_{s,\tau}(y,v)) \cdot \nabla_v \mu(V_{s,\tau}(y,v)) \, dv ds \, dyd \tau .
\end{aligned}$$ 
Recalling that $G^{osc,1}_\pm(t,x,v) =  \phi^v_{\pm,1}(i\partial_x)G^{osc}_{\pm}(t,x)$, we obtain the formulation for the remainder $\cE^{m,R} (t,x)$ as stated. 
\end{proof}

Similarly, we also obtain the following proposition concerning the convolution with initial data. 

\begin{proposition}\label{prop-convGoscS0}
Let $S^0(t,x)$ be the source density defined as in \eqref{nonlinear-S}. Then, there holds
\begin{equation}\label{convGoscS0}
G^{osc}_\pm\star_{t,x}  S^0(t,x)  = - G^{osc}_\pm(t,x)\star_x S^0_{\pm,0}(x) + a_\pm(i\partial_x) S^0_{\pm,1}(t,x) + G_\pm^{osc} \star_{t,x} S^0_{\pm,2}(t,x)
\end{equation}
where $a_\pm(k)$ is the Fourier multiplier as in \eqref{def-Gosc}, and 
\begin{equation}\label{def-Spm12}
\begin{aligned}
S^0_{\pm,0}(x)&=  \int  \phi^v_{\pm,1}(i\partial_x)f_0(x,v) \varphi(v)\,dv 
\\
S^0_{\pm,1}(t,x)  &=  \int  \phi^v_{\pm,1}(i\partial_x) f_0(X_{0,t}(x,v) , V_{0,t}(x,v)) \varphi(v)\,dv 
\\
S^0_{\pm,2}(t,x)  &= -
 \int \nabla_v \phi^v_{\pm,1}(i\partial_x)\Big[ E(t,x)f_0(X_{0,t}(x,v) , V_{0,t}(x,v)) \Big]  \varphi(v)\, dv 
 \end{aligned}
 \end{equation}
 recalling $ \phi^v_{\pm,1}(i\partial_x)$ defined as in \eqref{def-phipmj}.
\end{proposition}

\begin{proof} By definition, we write
$$
\begin{aligned}
G^{osc}_\pm\star_{t,x}  S(t,x)
&= \int_0^t \iint G^{osc}_\pm(t  - \tau,x - y)  f_0(X_{0,\tau}(y,v) , V_{0,\tau}(y,v)) dvdy ds. 
\end{aligned}$$
We introduce the change of variables  
$$
(y,v) \quad \longmapsto \quad (Y,V) = ( X_{0,\tau}(y,v), V_{0,\tau}(y,v)), 
$$ 
whose Jacobian determinant is equal to one, recalling that $(X_{s,t}, V_{s,t})$ are the nonlinear characteristic curves of the divergence-free vector field $(\hv, E(t,x))$ in the phase space $\RR^3_x\times \RR^3_v$. By the time reversibility, we can write 
$$
(y,v) = (X_{\tau,0}(Y,V), V_{\tau,0}(Y,V)) .
$$
This leads to 
$$
\begin{aligned}
G^{osc}_\pm\star_{t,x}  S(t,x)
&= \int_0^t \iint G^{osc}_\pm(t  - \tau,x - y)  f_0(X_{0,\tau}(y,v) , V_{0,\tau}(y,v)) dvdy ds
\\
&= \iint \Bigl[ \int_0^t G^{osc}_\pm(t  - \tau,x - X_{\tau,0}(Y,V)) 
\, d\tau \Bigr] f_0(Y,V) \, dY dV.
\end{aligned}$$
Recall that the Fourier transform  $G^{osc}_\pm(t,x)$ is of the form $\FG_\pm^{osc}(t,k) = e^{\lambda_\pm(k)t} a_\pm(k)$. Therefore, similar to the calculation done in \eqref{int-Eosc}, we compute the integral of $G^{osc}_\pm(t,x)$ over the nonlinear characteristic curves, yielding 
$$
\int_0^t G^{osc}_\pm(t-\tau,x - Y - (\tau - s) \hv) d\tau = G^{osc,1}_\pm(0,x - Y - t  \hv,v) 
- G^{osc,1}_\pm(t , x - Y,v) ,
$$
where, as in \eqref{def-Eoscj}-\eqref{def-phipmj}, we denote $G^{osc,j}_\pm(t,x,v) =  \phi^v_{\pm,j}(i\partial_x)G^{osc}_{\pm} (t,x)$. This leads to 
$$
G^{osc}_\pm\star_{t,x}  S(t,x) = J_1(t,x)  + J_2(t,x)   + \cE^{R}(t,x) $$
where
$$
\begin{aligned}
J_1(t,x)  &=  \iint G^{osc,1}_\pm(0, x - X_{t,0}(Y,V),V_{t,0}(Y,V)) f_0(Y,V)  \, dY dV 
\\
J_2(t,x)  &= -  \iint G^{osc,1}_\pm(t,  x - Y,V)  f_0(Y,V)  \, dY dV 
\\
\cE^R(t,x)  &= - \int_0^t \iint E(\tau, X_{\tau,0})\cdot \nabla_v(G^{osc,1}_{\pm}(t-\tau,x-X_{\tau,0},V_{\tau,0})  f_0(Y,V) 
\, dV dYd \tau.
\end{aligned}$$ 
Reversing the change of variables $(y,v) = (X_{\tau,0}(Y,V),V_{\tau,0}(Y,V))$, we obtain the proposition.  
\end{proof}

\subsection{Convolution with initial data}

We are now ready to prove Proposition \ref{prop-decayEosc} by estimating each term in the representation \eqref{convGoscS0} of $G^{osc}_\pm \star_{t,x}  S^0(t,x)$. Indeed, we recall that 
\begin{equation}\label{GSosc1}
\begin{aligned}
G^{osc}_\pm\star_{t,x}  \nabla_x S^0(t,x)  &= - G^{osc}_\pm(t,\cdot)\star_x  \nabla_x S^0_{\pm,0}(x) + a_\pm(i\partial_x)  \nabla_x S^0_{\pm,1}(t,x) 
\\&\quad + G_\pm^{osc} \star_{t,x}  \nabla_x S^0_{\pm,2}(t,x),
\end{aligned}\end{equation}
in which the second term is put into $E^r(t,x)$, and the other terms remain in $E^{osc}_\pm(t,x)$. 
Let us estimate term by term in \eqref{GSosc1}. 

\subsubsection*{Source term $S^0_{\pm,0}$.}
First, using the results from Proposition \ref{prop-Greenphysical}, for any $|\alpha|\ge 0$, we bound 
$$
\begin{aligned}
\|\partial_x^\alpha G^{osc}_\pm(t,\cdot)\star_x  \nabla_x S^0_{\pm,0}(x)\|_{L^\infty_x} &\lesssim t^{-3/2}\| S^0_{\pm,0}\|_{W^{4+|\alpha|,1}_x}
\\
\|G^{osc}_\pm(t,\cdot)\star_x  \nabla_x S^0_{\pm,0}(x)\|_{H^{\alpha_0+1}_x} &\lesssim \| S^0_{\pm,0}\|_{H^{\alpha_0+1}_x}.
\end{aligned}$$
Recall that by definition, 
$$S^0_{\pm,0}(x)=  \int \phi^v_{\pm,1} (i\partial_x) f_0(x,v) \varphi(v)\,dv .$$
Therefore, using \eqref{Lp-convphi}, we bound 
$$\| S^0_{\pm,0}\|_{W^{4+|\alpha|,1}_x} \lesssim \int \|  f_0 (\cdot,v) \|_{W^{4+|\alpha|,1}_x} \varphi(v) \; dv \le C_0 \| f_0\|_{L^1_v W^{4+|\alpha|,1}_x} .$$
Similarly, we have $\| S^0_{\pm,0}\|_{H^{\alpha_0+1}_x} \lesssim  \| f_0\|_{L^1_v H^{\alpha_0+1}_x} .$ This proves the desired estimates stated in Proposition \ref{prop-decayEosc} for the first term in \eqref{GSosc1}. 

\subsubsection*{Source term $S^0_{\pm,1}$.}

Next we treat the second term in \eqref{GSosc1}, namely the source term $a_\pm(i\partial_x)  \nabla_x S^0_{\pm,1}(t,x)$. Recall that 
$$S^0_{\pm,1}(t,x)  =  \int \phi^v_{\pm,1} (i\partial_x)f_0(X_{0,t}(x,v) , V_{0,t}(x,v)) \varphi(v)\,dv .
$$
Next, using the expansion \eqref{expand-phi1}, we may write 
\begin{equation}\label{seriesn}
\begin{aligned}
S^0_{\pm,1}(t,x)  &= \sum_{n\ge 0} a_{\pm,n}(i\partial_x) :: S^0_{1,n}(t,x)
 \end{aligned}\end{equation}
where $a_{\pm,n}(k) =  \mp i \nu_*(|k|)^{-1} (\pm 1)^n\nu_*(|k|)^{-n} k^{\otimes n}$, which are smooth Fourier multipliers and satisfy $|a_{\pm,n}(k)|\le \nu_*(|k|)^{-1}$, uniformly for all $n\ge 0$, since $\nu_*(|k|) >|k|$, and  
$$ S^0_{1,n}(t,x) =  
 \int f_0(X_{0,t}(x,v) , V_{0,t}(x,v)) \varphi_{1,n}(v)\,dv$$
with $\varphi_{1,n}(v) =  \hv^{\otimes n} \hv \varphi(v)$. Recall that the notation $k^{\otimes n}::\hv^{\otimes n} = (k\cdot \hv)^n$ is simply for sake of presentation. Note that since $f_0(x,v)$ is compactly supported in $v$ and $\| V_{0,t}(x,v) -v\|_{L^\infty_{x,v}} \lesssim \epsilon$, we have $f_0(X_{0,t}(x,v) , V_{0,t}(x,v))$ vanishes for $|v|\ge 2R_0$ for some $R_0>0$. Therefore, for $A_0  = 2|R_0|/\langle 2R_0\rangle <1$, there exists some universal constant $C_0$, independent of $n$, so that \begin{equation}\label{bdvarphin}
|\varphi_{1,n}(v)| \le C_0 A_0^n , \qquad |\nabla_v \varphi_{1,n}(v)| \le C_0 nA_0^n, 
\end{equation} 
uniformly for $|v|\le 2R_0$, for all $n\ge 1$. Therefore, we may use \eqref{eq:fouriermult} with $\delta =2$ and the results in Proposition \ref{prop-bdS0}, yielding 
$$ 
\begin{aligned}
\|\partial_x^\alpha S^0_{\pm,1}(t) \|_{L^\infty_{x} } &\le \sum_{n\ge 0} \| \partial_x^\alpha a_{\pm,n}(i\partial_x)  :: a_\pm(i\partial_x) \nabla_x S^0_{1,n}(t)\|_{L^\infty_x}
\\&\lesssim \sum_{n\ge 0}  \| \partial_x^\alpha S^0_{1,n}(t)\|_{L^\infty_x}
\lesssim \epsilon_0 \langle t\rangle^{-3+\delta_\alpha} 
\sum_{n\ge 0}  A_0^n 
\\&\lesssim \epsilon_0 \langle t\rangle^{-3+\delta_\alpha} .
\end{aligned}$$
Similarly, we bound 
$$ 
\begin{aligned}
\|S^0_{\pm,1}(t) \|_{H^{\alpha_0+1}_x} \le \sum_{n\ge 0}
\| a_{\pm,n}(i\partial_x) ::a_\pm(i\partial_x)  \nabla_x S^0_{1,n}(t)\|_{H^{\alpha_0+1}_x}
&\lesssim \sum_{n\ge 0}\| S^0_{1,n}(t)\|_{H^{\alpha_0}_x} \lesssim \epsilon_0 \langle t\rangle^{\delta_1},
\end{aligned}$$
which prove the desired bounds for the second term in \eqref{GSosc1}. 

\subsubsection*{Source term $S^0_{\pm,2}$.}

Finally, we bound the last convolution $G_\pm^{osc} \star_{t,x} \nabla_xS^0_{\pm,2}(t,x)$  in the representation \eqref{GSosc1}. Recalling the definition of $S^0_{\pm,2}$ in \eqref{def-Spm12} and the expansion \eqref{expand-phi1}, we may write 
 $$ 
\begin{aligned}
S^0_{\pm,2}(t,x)  &= \sum_{n\ge 0} a_{\pm,n}(i\partial_x) [E(t,x)S^0_{2,n}(t,x)]
 \end{aligned}$$ 
for the same coefficients $a_{\pm,n}(k)$ as in \eqref{seriesn}, recalling $a_{\pm,n}(k) \le \langle k\rangle^{-1}$, where 
$$ S^0_{2,n}(t,x) =  
 \int f_0(X_{0,t}(x,v) , V_{0,t}(x,v)) \varphi_{2,n}(v)\,dv$$
for $\varphi_{2,n}(v) = \nabla_v ( \hv^{\otimes n} \hv )\varphi(v).$  
Observe that $S^0_{2,n}(t,x)$ are again of the same form as that of the source density $S^0(t,x)$ defined as in \eqref{nonlinear-S}. Therefore, we may write 
$$
  G_\pm^{osc} \star_{t,x} \nabla_xS^0_{\pm,2}(t,x) = 
   \sum_{n\ge 0} a_{\pm,n}(i\partial_x) :: G_\pm^{osc} \star_{t,x} \nabla_x[E(t,x)S^0_{2,n}(t,x)]
.$$
Using Lemma \ref{lem-decayosc}, it suffices to prove that 
\begin{equation}\label{checkES} \int_0^t \| \partial_x^{\alpha} [ES^0_{2,n}](s)\|_{W_x^{3,1}} \;ds \lesssim 1, \qquad \| \partial_x^\alpha [ ES^0_{2,n}](t)\|_{H^{1}_x} \lesssim \langle t\rangle^{-5/2}, \end{equation}
for any $|\alpha|\le n_0 = \frac14N_0$. 
Indeed, using the decay estimates in Proposition \ref{prop-bdS}, we bound 
$$
\begin{aligned}
\| \partial_x^\alpha [ES^0_{2,n}](t)\|_{H^1_x} 
& \lesssim \| E(t)\|_{L^\infty}\| S^0_{2,n}(t)\|_{H^{|\alpha|+1}_x} + \| E(t)\|_{H^{|\alpha|+1}_x}\| S^0_{2,n}(t)\|_{L^\infty}
\\& \lesssim \epsilon^2 \langle t\rangle^{-3+\delta_{\alpha+1}},
\end{aligned}
$$
which verifies the last estimate in \eqref{checkES}, since $\delta_{\alpha+1} = \frac{|\alpha|+1}{N_0-1} \le \frac14\frac{N_0+4}{N_0-1}\le \frac12$ (by taking $N_0\ge 6$). 
Similarly, we bound 
$$
\begin{aligned}
\| \partial_x^\alpha [ES^0_{2,n}](t)\|_{W_x^{3,1}} 
& \lesssim \sum_{|\alpha_1|+|\alpha_2| = |\alpha|}\| \partial_x^{\alpha_1}E(t)\|_{H_x^{3}} \|\partial_x^{\alpha_2}S^0_{2,n}(t)\|_{H_x^{3}}  
\\& \lesssim \epsilon^2 \langle t\rangle^{-\frac32 + \delta_{\alpha+3}}\end{aligned}
$$
which again verifies the first estimate in \eqref{checkES}, since $\delta_{\alpha+3} \le \delta_{n_0+3} = \frac{\frac14N_0 + 3}{N_0-1} < \frac12$ (by taking $N_0\ge 14$). Therefore, we obtain 
$$ \| \partial_x^\alpha G_\pm^{osc} \star_{t,x} \nabla_xS^0_{\pm,2}(t)\|_{L^\infty_x} \lesssim \epsilon^2 \langle t\rangle^{-3/2} ,$$
for any $|\alpha|\le n_0 = \frac14N_0$. 
It remains to prove the estimates in $H^{\alpha_0+1}_x$. Indeed, using the boundedness of $G^{osc}_\pm$ from $L^2$ to $L^2$, we bound 
$$ 
\begin{aligned}
\| a_{\pm,n}(i\partial_x) & G_\pm^{osc} \star_{t,x} \nabla_x[E(t,x)S^0_{2,n}(t,x)](t)\|_{H^{\alpha_0+1}_x} 
\\&\lesssim
  \int_0^t \| a_{\pm,n}(i\partial_x)[ES^0_{2,n}](s)\|_{H^{\alpha_0+1}_x} \; ds
\\
&\lesssim
 \int_0^t \Big[ \| E(s)\|_{L^\infty_x}\|S^0_{2,n}(s)\|_{H^{\alpha_0}_x} +  \| E(s)\|_{H^{\alpha_0}_x}\|S^0_{2,n}(s)\|_{L^\infty_x}\Big]\; ds
  \\&\lesssim \epsilon^2\int_0^t \langle s\rangle^{-3/2+\delta_1} \; ds 
    \lesssim \epsilon^2,
  \end{aligned}$$
 noting that $\delta_1 < 1/2$, provided that $N_0\ge 4$. The infinite series in $n$ converges absolutely as in the previous case, using \eqref{bdvarphin}. This proves the desired estimates for the last term in  \eqref{GSosc1}, and therefore completes the proof of Proposition \ref{prop-decayEosc}.

\subsection{Convolution with nonlinear sources}

In this section, we prove the convolution with the nonlinear source $G_\pm^{osc} \star_{t,x} S^{\mu}(t,x)$. We recall the results from Proposition \ref{prop-convGosc} that 
\begin{equation}\label{convGosc-re}
G_\pm^{osc} \star_{t,x} \nabla_x S^{\mu}(t,x) = a_\pm(i\partial_x) \nabla_x S_{\pm,1}^{\mu}(t,x) + G_\pm^{osc} \star_{t,x} \nabla_x S_{\pm,2}^{\mu}(t,x)
\end{equation}
in which the first term is put into $E^r(t,x)$, and the other terms remain in $E^{osc}_\pm(t,x)$. The following proposition treats the first term in the expression. 

\begin{proposition} Let $S_{\pm,1}^{\mu}(t,x) $ be defined as in Proposition \ref{prop-convGosc}. Then, there holds
\begin{equation}\label{bd-Smu1}
\| a_\pm(i\partial_x) \nabla_x S_{\pm,1}^{\mu}(t)\|_{L^\infty_x} + \| a_\pm(i\partial_x) \nabla^2_x S_{\pm,1}^{\mu}(t)\|_{L^\infty_x} \lesssim \epsilon^2 \langle t\rangle^{-3}.
\end{equation}
In addition, 
\begin{equation}\label{daSmu1}
\begin{aligned}
\|\partial_x^\alpha a_\pm(i\partial_x) \nabla_x S_{\pm,1}^{\mu}(t)\|_{L^p_x}  & \lesssim \epsilon \langle t\rangle^{-3(1-1/p) + \delta_{\alpha,p}} 
,\\
\|a_\pm(i\partial_x) \nabla_x S_{\pm,1}^{\mu}(t)\|_{H^{\alpha_0}_x} &\lesssim \epsilon \langle t\rangle^{\delta_1},
\end{aligned}
\end{equation}
for all $|\alpha|\le N_0-1$, with $\delta_{\alpha,p} = \max\{ \delta_\alpha, \frac32\delta_\alpha (1-1/p)\}$, where $\delta_\alpha = \frac{|\alpha|}{N_0-1}$.

\end{proposition}

\begin{proof} As done in the previous section, we may again write 
$$ 
\begin{aligned}
S_{\pm,1}^{\mu}(t,x)  &= \sum_{n\ge 0} a_{\pm,n}(i\partial_x)  S_{n,1}^{\mu}(t,x)
 \end{aligned}$$ 
in which $$ S_{n,1}^{\mu}(t,x) =  \int_0^t \int  \Big[ E(s,x - (t-s)\hv) \cdot \nabla_v \mu(v) 
 - E(s,X_{s,t}(x,v))\cdot \nabla_v \mu(V_{s,t}(x,v)) \Big] \hv^{\otimes (n+2)}\,dv ds.$$
Observe that $S_{n,1}^{\mu}(t,x) $ are exactly the source term $S^{\mu}(t,x)$ studied in the previous section, up to an addition of the factor $ \hv^{\otimes (n+2)}$. Hence, following the proof of Proposition \ref{prop-bdS}, we obtain \eqref{bd-Smu1} for each source integral $S_{n,1}^{\mu}(t,x)$.  The proposition thus follows.  
\end{proof}

Next, we study the last term in the expression \eqref{convGosc-re}. We obtain the following proposition.

\begin{proposition} 
Let $S_2^{\mu}(t,x) $ be defined as in Proposition \ref{prop-convGosc}.  Then, there is a sequence of smooth Fourier multipliers $\{b_{\pm,n}(k)\}_{n\ge 0}$, whose summation is absolutely convergent, so that   
$$ 
\begin{aligned}
S_{\pm,2}^{\mu}(t,x)  &= \sum_{n\ge 0} b_{\pm,n}(i\partial_x) Q_n^{\mu}(t,x) + R^{\mu}(t,x)
 \end{aligned}$$ 
where 
\begin{equation}\label{def-Qmu}
\begin{aligned}
Q_n^{\mu}(t,x)  &= E(t,x)
 \int_0^t \int E(s,x-(t-s)\hv) \cdot \nabla_v \mu(v) \varphi_n(v) \, dv ds 
 \end{aligned}\end{equation}
for $\varphi_n(v) = \hv^{\otimes (n+1)}$. In addition, the remainder $R^{\mu}(t,x)$ satisfies 
$$\|G_\pm^{osc} \star_{t,x} R^{\mu}\|_{L^\infty_x}  \lesssim \epsilon^3 \langle t\rangle^{-3/2} .$$  
\end{proposition}

\begin{proof} As in the previous proposition, defining $b_{\pm,n} = (n+2)a_{\pm,n}$, we first write 
$$ 
\begin{aligned}
S_{\pm,2}^{\mu}(t,x)  &= \sum_{n\ge 0} b_{\pm,n}(i\partial_x) S_n^{\mu}(t,x)
 \end{aligned}$$ 
where 
$$ S_n^{\mu}(t,x) =
 \int_0^t \int  E(t,x)E(s,X_{s,t}(x,v)) \cdot \nabla_v \mu(V_{s,t}(x,v))  \varphi_n(v) \, dv ds $$
for $\varphi_n(v) =\hv^{\otimes (n+1)}$. Replacing the nonlinear characteristic by the straight ones in the above integral defines the quadratic term $Q_n^{\mu}(t,x)$ as stated in \eqref{def-Qmu}. This leaves a remainder of the form 
$$ 
\begin{aligned}
 R_n^{\mu}(t,x)
&= E(t,x) S^{m,\mu,R}_n(t,x) 
\end{aligned}$$ 
where we have set 
$$  S^{m,\mu,R}_n(t,x) =  \int_0^t \int  \Big[ E(s,X_{s,t}(x,v)) \cdot \nabla_v \mu(V_{s,t}(x,v)) - E(s,x - (t-s)\hv) \cdot \nabla_v \mu(v) \Big] \varphi_n(v)\, dv ds .$$ 
 Observe that $S_n^{m,\mu,R}(t,x) $ are again exactly the source term $S^{\mu}(t,x)$ studied in the previous section, up to an addition of the factor $\varphi_n(v)$. 
 Hence, following the proof of Proposition \ref{prop-bdS}, we obtain 
 $$ \| S^{m,\mu,R}_n(t)\|_{L^p_x} \lesssim \epsilon^2 \langle t\rangle^{-3(1-1/p)} \log t,$$
 for any $p\ge 1$. This, together with $\| E(t)\|_{L^\infty_x} \lesssim \epsilon \langle t\rangle^{-3/2}$, yields 
$$ \|R_n^{\mu}(t)\|_{L^1_x} \lesssim \epsilon^3 \langle t\rangle^{-3/2} \log t, \qquad 
 \|R_n^{\mu}(t)\|_{L^2_x} \lesssim \epsilon^3 \langle t\rangle^{-3} \log t. 
 $$
In particular, Lemma \ref{lem-decayosc} can be applied, yielding the claimed bounds on $G_\pm^{osc} \star_{t,x} R^{\mu}$, upon recalling that the summation of $\{b_{\pm,n}(k)\}_{n\ge 0}$ is absolutely convergent.  
 \end{proof}

For sake of presentation, we summarize the results in this section into the following corollary. 

\begin{corollary}\label{cor-mainred}
Let $S^{\mu}(t,x)$ be the source term defined as in \eqref{nonlinear-S}, and for $n\ge 0$, let $Q_n^{\mu}(t,x)$ be defined as in \eqref{def-Qmu}. Then, there holds
\begin{equation}\label{convGosc-more}
G_\pm^{osc} \star_{t,x} S^{\mu}(t,x) = \sum_{n\ge 0} b_{\pm,n}(i\partial_x) \Big[ G_\pm^{osc} \star_{t,x}  Q_n^{\mu}\Big](t,x) + \cE^{m,\mu,R}(t,x)
 \end{equation}
for some sequence of functions $\{b_{\pm,n}(k)\}_{n\ge 0}$, whose infinite series is absolutely convergent. In addition, the remainder $\cE^{m,\mu,R}(t,x)$ satisfies $\| \cE^{m,\mu,R}(t)\|_{L^\infty_x} \lesssim \epsilon^2 \langle t\rangle^{-3/2}$. 
\end{corollary}

\subsection{Quadratic interaction}\label{sec-quadinteraction}

In view of Corollary \ref{cor-mainred}, it remains to prove the dispersive decay of the convolution $G_\pm^{osc} \star_{t,x} \cQ(t)$, where 
\begin{equation}\label{def-cQmu}
\cQ(t,x) = E(t,x)
 \int_0^t \int E(s,x-(t-s)\hv) \cdot \nabla_v \mu(v) \varphi(v) \, dv ds 
\end{equation}
for some sufficiently smooth and bounded function $\varphi(v)$. 

\begin{proposition}\label{prop-cQmu} 
Let $\cQ(t,x)$ be defined as in \eqref{def-cQmu}. Then, there holds
$$\|G_\pm^{osc} \star_{t,x} \cQ(t)\|_{L^\infty_x}  \lesssim (\epsilon_0 + \epsilon^2) \langle t\rangle^{-3/2} .$$  

\end{proposition}

We are aimed to prove Proposition \ref{prop-cQmu} via Lemma \ref{lem-decayosc}. Recalling the bootstrap assumption
\eqref{boots-repE} on the electric field 
\begin{equation}\label{recall-deE} E = E^{osc}_\pm (t,x) + E^r(t,x),\end{equation}
we shall study the quadratic interaction $\cQ(t,x) $ respectively for each pair of the electric field components.

\subsubsection{Interaction with $E^r$}

In this section, we study the interaction with $E^r$: namely, 
\begin{equation}\label{def-cQr}
\cQ^r(t,x) = E(t,x)
 \int_0^t \int E^r(s,x-(t-s)\hv) \cdot \nabla_v \mu(v) \varphi(v) \, dv ds .
\end{equation}
We shall verify the assumption of Lemma \ref{lem-decayosc} for $J = \cQ^r(t,x)$.
Indeed, recalling from \eqref{grad-F} and \eqref{goal-EoscEr} that $E^r = \nabla_x \rho^r$, we write 
$$ 
E^r(s,x-(t-s)\hv) = \nabla_x\rho^r(s,x - \hv (t-s)) = - \frac{1}{t-s}\nabla_{\hv} v\nabla_v [ \rho^r(s,x - \hv (t-s))]
$$
in which we used that $\nabla_{\hv} v$ is the inverse of the matrix $\nabla_v\hv$. Using this, we may now integrate by parts in $v$ to get 
$$
\begin{aligned}
\cQ^r(t,x) &= 
E(t,x)
 \int_0^t \int (t-s)^{-1}\rho^r(s,x-(t-s)\hv) \nabla_v \cdot [ \nabla_{\hv} v\nabla_v \mu(v)\varphi(v)] \, dv ds .
\end{aligned}$$
{ 
For any pair $q,q'\in (1,\infty)$, with $1/q+1/q' = 1$, we use the bootstrap bounds on $E$ and $\rho^r$ to bound 
$$
\begin{aligned}
\|\cQ^r(t)\|_{L^1_x} 
&\lesssim \| E(t)\|_{L^q_x}\Big\| 
\int_0^t \int (t-s)^{-1}\rho^r(s,x-(t-s)\hv) \nabla_v \cdot [ \nabla_{\hv} v\nabla_v \mu(v)\varphi(v)] \, dv ds \Big\|_{L^{q'}_x} 
\\&\lesssim \epsilon \langle t\rangle^{-3(1/2-1/q)}
\int_0^t \int (t-s)^{-1} \|\rho^r(s)\|_{L^{q'}_x} \langle v\rangle^{-4} \, dv ds 
\\&\lesssim \epsilon^2
\langle t\rangle^{-3(1/2-1/q)} \int_0^t (t-s)^{-1} \langle s\rangle^{-3(1-1/q')} \log ^2 s
\; ds 
\\&\lesssim \epsilon^2
\langle t\rangle^{-3(1/2-1/q)} \log t. 
\end{aligned}$$
Taking $q>6$, the above proves that $\|\cQ^r(t)\|_{L^1_x}$ decays at an integrable rate, and so the assumption of Lemma \ref{lem-decayosc} is verified for the $L^1$ norm of $J = \cQ^r(t,x)$. 
On the other hand, for any $q\in (1,\infty)$, using the rapid decay of $\mu(v)$ and introducing the change of the variable $y = x-(t-s)\hv$,
we bound
$$
\begin{aligned}
\Big| \int \rho^r(s,x-(t-s)\hv) \nabla_v \cdot [ \nabla_{\hv} v\nabla_v \mu(v)\varphi(v)] \, dv\Big| 
& \lesssim \Big( \int |\rho^r(s,x-(t-s)\hv)|^q  \langle v\rangle^{-5} \, dv\Big)^{1/q} 
\\
& \lesssim \Big(\int |\rho^r(s,y)|^q \; \frac{dy}{(t-s)^3}\Big)^{1/q} 
\\& \lesssim \epsilon (t-s)^{-3/q} \langle s\rangle^{-3(1-1/q)} \log^2 s. 
\end{aligned}
$$
Therefore, for any $q_1,q_2\in (1,\infty)$, using the above estimate with $q = q_1$ and $q= q_2$, we bound
$$
\begin{aligned}
\|\cQ^r(t)\|_{L^2_x} 
&\lesssim \| E(t)\|_{L^2_x}\Big\| 
\int_0^t \int (t-s)^{-1}\rho^r(s,x-(t-s)\hv) \nabla_v \cdot [ \nabla_{\hv} v\nabla_v \mu(v)\varphi(v)] \, dv ds \Big\|_{L^{\infty}_x} 
\\&\lesssim 
\epsilon^2
 \int_0^{t/2} (t-s)^{-1-3/q_1} \langle s\rangle^{-3(1-1/q_1)} \log^2 s\; ds + \epsilon^2
 \int_{t/2}^t (t-s)^{-1-3/q_2} \langle s\rangle^{-3(1-1/q_2)} \log^2 s\; ds 
\\&\lesssim 
\epsilon^2   \langle t\rangle^{-1-3/q_1}
 \int_0^{t/2}  \langle s\rangle^{-3(1-1/q_1)} \log^2 s\; ds + \epsilon^2
 \langle t\rangle^{-3(1-1/q_2)}  \log^2 t \int_{t/2}^t (t-s)^{-1-3/q_2} \; ds 
\\&\lesssim \epsilon^2  \langle t\rangle^{-3/q_1} + \epsilon^2  \langle t\rangle^{-3(1-1/q_2)}  \log^2 t.
\end{aligned}$$
The assumption of Lemma \ref{lem-decayosc} is thus verified for the $L^2$ norm of $J = \cQ^r(t,x)$, upon taking $q_1 < 6/5$ and $q_2  >6$.  
This completes the proof of Proposition \ref{prop-cQmu} for the convolution $G_\pm^{osc} \star_{t,x} \cQ^r(t,x)$. 

}

\subsubsection{Interaction with $E^{osc}_\pm$}

In this section, we study the interaction with $E^{osc}_\pm$: namely, 
\begin{equation}\label{def-cQosc}
\cQ^{osc}(t,x) = E(t,x)
 \int_0^t \int E_\pm^{osc}(s,x-(t-s)\hv) \cdot \nabla_v \mu(v) \varphi(v) \, dv ds .
\end{equation}
We first study the integration in $s$. Indeed, following the calculation in \eqref{int-Eosc}, now on a straight line, we obtain 
$$
\begin{aligned}
 \int_0^t E_\pm^{osc}(s,x-(t-s)\hv)\; d\tau
& = \phi^v_{\pm,1}(i\partial_x)E^{osc}_{\pm}(t,x)  - \phi^v_{\pm,1}(i\partial_x)E^{osc}_{\pm}(0,x - t \hv) 
\\&\quad -  \int_0^t [a_\pm(i\partial_x) \phi^v_{\pm,1}(i\partial_x)F](s,x - (t-s)\hv) \; ds
\end{aligned}
$$
where $\phi_{\pm,1}$ is defined as in \eqref{def-phipmj}. This yields 
$$\begin{aligned}
\cQ^{osc}(t,x) 
&= E(t,x)\int  \phi^v_{\pm,1}(i\partial_x)E^{osc}_{\pm}(t,x) \cdot\nabla_v \mu(v) \varphi(v) \, dv
\\&\quad - E(t,x)  \int \phi^v_{\pm,1}(i\partial_x)E^{osc}_{\pm}(0,x - t \hv)  \cdot\nabla_v \mu(v) \varphi(v) \, dv
\\& \quad - 
E(t,x)  \int_0^t \int  [a_\pm(i\partial_x)\phi^v_{\pm,1}(i\partial_x) F](s,x - (t-s)\hv) \cdot \nabla_v \mu(v) \varphi(v) \, dv ds.
\end{aligned}$$
The last term is treated exactly as done for $\cQ^r(t,x)$ in \eqref{def-cQr}, { since $F(t,x)$ and $E^r(t,x)$ satisfy the same bootstrap bounds.} Similarly, recalling $E^{osc}_{\pm}(0) =a_\pm(i\partial_x) F_0(x)$ by \eqref{defiFB} and using the assumption on the initial data $F_0(x)$ and the fact that $\phi^v_{\pm,1}(i\partial_x)$ is a bounded operator from $L^1$ to $L^1$, we bound 
$$ \Big | \int \phi^v_{\pm,1}(i\partial_x)E^{osc,1}_{\pm}(0,x - t \hv)  \cdot\nabla_v \mu(v) \varphi(v) \, dv \Big| \lesssim  \langle t\rangle^{-3} \|E^{osc,1}_{\pm}(0)\|_{L^1_x} \lesssim \epsilon_0  \langle t\rangle^{-3},$$ 
while the integral is clearly bounded by $C_0 \epsilon_0$ in $L^1_x$. This yields 
$$ \Big\| E(t,x)  \int \phi^v_{\pm,1}(i\partial_x)E^{osc}_{\pm}(0,x - t \hv)  \cdot\nabla_v \mu(v) \varphi(v) \, dv\Big\|_{L^p_x} \lesssim \epsilon_0 \epsilon \langle t\rangle^{-3/2 - 3(1-1/p)}$$ 
for $p\ge 1$, which verifies the assumption of Lemma \ref{lem-decayosc}, giving the desired bounds for this term. 

Finally, we treat the first term in $\cQ^{osc}(t,x) $, which we decompose into 
$$\cQ^{r,osc}(t,x) + \cQ^{osc,osc}(t,x)$$
that corresponds to the decomposition \eqref{recall-deE} on the electric field. It follows that 
$$
\begin{aligned} \| \cQ^{r,osc}(t)\|_{L^p_x} 
&\lesssim \|E^r(t)\|_{L^p_x} \Big\| \int \phi^v_{\pm,1}(i\partial_x) E^{osc}_{\pm}(t,x) \cdot\nabla_v \mu(v) \varphi(v) \, dv\Big\|_{L^\infty_x}
\\
&\lesssim \|E^r(t)\|_{L^p_x}  \| E^{osc}_{\pm}(t) \|_{L^\infty_x} \lesssim \epsilon^2 \langle t\rangle^{-3(1-1/p) - 3/2}
\end{aligned}$$
for $p\ge 1$, which again verifies the assumption of Lemma \ref{lem-decayosc}. It remains to study the quadratic oscillation $ \cQ^{osc,osc}(t,x)$, which is handled in the next section.

\subsubsection{Quadratic oscillations}

It remains to prove Proposition \ref{prop-cQmu} for the quadratic oscillation, which is defined by
\begin{equation}\label{def-cQosc2}
\cQ^{osc,osc}(t,x) = E_\pm^{osc}(t,x)\int  \phi^v_{\pm,1}(i\partial_x)E^{osc}_{\pm}(t,x) \cdot\nabla_v \mu(v) \varphi(v) \, dv,
\end{equation}
for any combination of $\pm$ in each $E^{osc}_\pm(t,x)$ (that is, a summation of four terms). 
Since $E_\pm^{osc}(t,x)$ does not decay in $L^2_x$, the quadratic term $\cQ^{osc,osc}(t,x)$ does not decay in $L^1_x$ and so the assumption of  Lemma \ref{lem-decayosc} is not valid for $J = \cQ^{osc,osc}(t,x)$. We need to further exploit oscillations in the convolution 
$$\cE^{osc,osc}(t,x): = G_\pm^{osc} \star_{t,x} \cQ^{osc,osc}(t,x)$$ 

In Fourier, recalling $\FG^{osc}_\pm (t,k)= e^{\lambda_\pm(k)t}a_\pm(k)$ from \eqref{Gsreal} and $\FE^{osc}_\pm(t,k) = e^{\lambda_{\pm}(k) t} \FB_\pm(t,k)$ from \eqref{defiFB}, we compute 
$$
\begin{aligned}
\FcE^{osc,osc}(t,k) & = e^{\lambda_\pm(k) t} a_\pm(k) \star_t \FcQ^{osc,osc}(t,k)
\\&=  a_\pm(k) \iint \frac{1}{\omega_\pm(\ell)} e^{\lambda_\pm(k)t }\star_t \FE^{osc}_{\pm} (t,\ell)\FE^{osc}_{\pm} (t,k-\ell)\cdot\nabla_v \mu(v) \varphi(v)\; d\ell dv
\\&= a_\pm(k) e^{\lambda_\pm(k)t }\int_0^t  \iint \frac{1}{\omega_\pm(\ell)} e^{s\Phi_\pm(k,\ell) } 
\FB_\pm(s,\ell)\FB_\pm(s,k-\ell) \cdot\nabla_v \mu(v) \varphi(v)\; d\ell dv
\end{aligned}
$$
in which $\omega_\pm(\ell) = \lambda_\pm(\ell) + i\ell \cdot \hv$, and the phase function $\Phi_\pm(k,\ell)$ is defined by
$$
\Phi_\pm(k,\ell): = -  \lambda_\pm(k) + \lambda_\pm(\ell) + \lambda_\pm(k-\ell) .
$$
We emphasize again that above any combination of $\pm$ is allowed and for any such combination 
the corresponding phase is bounded from below in absolute value by a constant independent of 
$(k,\ell)$. Precisely, on the support of both $\FB_\pm(s,\ell)$ and $\FB_\pm(s,k-\ell) $, since $\lambda_\pm(0) = \pm i $, we have 
\begin{equation}\label{nonres-condition} 
| \Phi_\pm(k,\ell)| \ge \frac12 (1 + |\ell|^2+ |k-\ell|^2)
\end{equation}
uniformly for $|k|\le 1/2$. Therefore, we may integrate by parts in $s$ to improve decay. Precisely, set
\begin{equation}\label{def-Jkl}
\cI(t,k,\ell): = \int_0^t \frac{1}{\omega_\pm(\ell)} e^{s\Phi_\pm(k,\ell) } 
\FB_\pm(s,\ell)\FB_\pm(s,k-\ell) \; ds
\end{equation}
and so 
\begin{equation}\label{cE-Jkl}
\begin{aligned}
\FcE^{osc,osc}(t,k) 
&= a_\pm(k) e^{\lambda_\pm(k)t } \iint 
\cI(t,k,\ell)\cdot\nabla_v \mu(v) \varphi(v)\; d\ell dv.
\end{aligned}
\end{equation}
Integrating in $s$, we obtain 
$$
\begin{aligned}
\cI(t,k,\ell)&= 
\int_0^t \frac{1}{\omega_\pm(\ell)}  
\frac{1}{\Phi_\pm(k,\ell)} \partial_s  e^{s\Phi_\pm(k,\ell) } \FB_\pm(s,\ell)\FB_\pm(s,k-\ell)ds
\\
& =\frac{1}{\omega_\pm(\ell)}  
\frac{1}{\Phi_\pm(k,\ell)} \Big[ e^{t\Phi_\pm(k,\ell) } \FB_\pm(t,\ell)\FB_\pm(t,k-\ell) 
- \FB_\pm(0,\ell)\FB_\pm(0,k-\ell) \Big]
\\
& \quad -  \int_0^t \frac{1}{\omega_\pm(\ell)}  
\frac{1}{\Phi_\pm(k,\ell)} e^{s\Phi_\pm(k,\ell) } \partial_s \FB_\pm(s,\ell)\FB_\pm(s,k-\ell) \; ds\\
&\quad - \int_0^t \frac{1}{\omega_\pm(\ell)}  
\frac{1}{\Phi_\pm(k,\ell)} e^{s\Phi_\pm(k,\ell) } \FB_\pm(s,\ell) \partial_s\FB_\pm(s,k-\ell)\; ds\end{aligned}
$$
Next, writing back $ e^{\lambda_\pm(k)t} \FB_\pm(t,k) = \FE_{\pm}^{osc}(t,k) $ and using 
$$  e^{\lambda_\pm(k)t} \partial_t \FB_\pm(t,k) = a_\pm(k) \FF(t,k),$$
with $\FB_\pm(0,k) = a_\pm(k)\FF_0(k)$, 
we obtain 
$$
\begin{aligned}
\cI(t,k,\ell)&=  e^{-\lambda_\pm(k)t }\frac{1}{\omega_\pm(\ell)}  \frac{1}{\Phi_\pm(k,\ell)} 
\FE_{\pm}^{osc} (t,\ell) \FE_{\pm}^{osc}(t,k-\ell)
\\
& \quad - \frac{1}{\omega_\pm(\ell)}  
\frac{1}{\Phi_\pm(k,\ell)} a_\pm(\ell) a_\pm(k-\ell) \FF_0(\ell) \FF_0(k-\ell) 
\\
& \quad -   \int_0^t e^{-\lambda_\pm(k)s}  \frac{1}{\omega_\pm(\ell)}  
\frac{1}{\Phi_\pm(k,\ell)} a_\pm(\ell) \FF(s,\ell)\FE_{\pm}^{osc}(s,k-\ell) \; ds\\
&\quad -  \int_0^t e^{-\lambda_\pm(k)s} \frac{1}{\omega_\pm(\ell)}  
\frac{1}{\Phi_\pm(k,\ell)} \FE_{\pm}^{osc}(s,\ell) a_\pm(k - \ell) \FF (s,k-\ell)\;ds.
\end{aligned}
$$
Putting this into $\FcE^{osc,osc}(t,k) $ in \eqref{cE-Jkl} and setting, for convenience,  
\begin{equation}\label{def-symbolm} 
m(k,\ell) =  \frac{a_\pm(k)}{\Phi_\pm(k,\ell)}\int \frac{1}{\omega_\pm(\ell)} \nabla_v \mu(v) \varphi(v) \;  dv ,
\end{equation}
we write 
$$\FcE^{osc,osc}(t,k)  =  \FcE_1(t,k) +  \FcE_2(t,k) +  \FcE_3(t,k) +  \FcE_4(t,k)$$
where 
\begin{equation}\label{def-J123}
\begin{aligned}
 \FcE_1(t,k) &= \int m(k,\ell)\FE_{\pm}^{osc} (t,\ell) \FE_{\pm}^{osc}(t,k-\ell)\; d\ell
 \\
  \FcE_2(t,k) &= e^{\lambda_\pm(k)t} \int m(k,\ell) a_\pm(\ell) a_\pm(k-\ell) \FF_0(\ell) \FF_0(k-\ell)\; d\ell
   \\
  \FcE_3(t,k) &= e^{\lambda_\pm(k)t} \star_t \int m(k,\ell) a_\pm(\ell) \FF(t,\ell) \FE_{\pm}^{osc}(t,k-\ell)\; d\ell
  \\
   \FcE_4(t,k) &= e^{\lambda_\pm(k)t} \star_t \int m(k,\ell)\FE_{\pm}^{osc} (t,\ell) a_\pm(k- \ell) \FF(t,k-\ell)\; d\ell.
\end{aligned}\end{equation}
Let us now bound each term in the physical space. Precisely, we shall prove that 
\begin{equation}\label{recall-goalcE}
\| \cE_j(t)\|_{L^\infty_x} \lesssim (\epsilon_0 + \epsilon^2)\langle t\rangle^{-3/2} 
\end{equation}
for $j=1,2,3,4$, which would complete the proof of Proposition \ref{prop-cQmu}.  

First, thanks to \eqref{nonres-condition} and the nonvanishing of $\omega_\pm(k) = \lambda_\pm(k) + ik\cdot \hv $, the symbol $m(k,\ell) $ defined in \eqref{def-symbolm} is smooth and bounded by $\langle k,\ell\rangle^{-1}$, 
and thus is a Coifman-Meyer multiplier. Therefore, Lemma \ref{lem-bilinear} yields  
$$
\begin{aligned}
 \| \cE_1(t)\|_{L^\infty_x} &\lesssim \|E^{osc}_\pm(t)\|_{L^\infty_x} ^2\lesssim \epsilon^2 \langle t\rangle^{-3} ,
 \end{aligned}$$
which verifies \eqref{recall-goalcE}. In addition, as for the second integral term, we write $ \FcE_2(t,k) = e^{\lambda_\pm(k)t} \FJ_2(k)$, where 
$$
\FJ_2(k) =  \int m(k,\ell) a_\pm(\ell) a_\pm(k-\ell) \FF_0(\ell) \FF_0(k-\ell)\; d\ell.
$$
By the dispersive estimate \eqref{est-HoscLp}, it suffices to bound $\| J_2\|_{W^{4,1}_x}$. Indeed, similarly as done above, Lemma \ref{lem-bilinear} yields
$$
\| J_2\|_{W^{4,1}_x} \lesssim \| F_0\|_{H^4_x}^2 \lesssim \epsilon_0^2,
$$
yielding the desired bound \eqref{recall-goalcE} on $\cE_2(t,x)$. 

Finally, we bound $\cE_3(t,x)$ and $\cE_4(t,x)$, which are of the same form. Setting 
$$
 \FJ_3(t,k) = \int m(k,\ell) a_\pm(\ell)\FF(t,\ell) \FE_{\pm}^{osc}(t,k-\ell)\; d\ell ,
 $$
 it suffices to verify the assumption of Lemma \ref{lem-decayosc} for $J_3(t,x)$, since by definition, $ 
\cE_3(t,x) = G^{osc}_\pm\star_{t,x} J_3(t,x) .$ Indeed, this follows directly from Lemma \ref{lem-bilinear}, giving   
$$ 
\| J_3(t) \|_{L^p_x} \lesssim \|F(t)\|_{L^p_x}\|E^{osc}_\pm(t)\|_{L^\infty_x} \lesssim \epsilon^2 \langle t\rangle^{-3(1-1/p)-3/2}
$$
for $p\ge 1$. The desired bound on the convolution $ 
\cE_3(t,x) = G^{osc}_\pm\star_{t,x} J_3(t,x) $ thus follows. This completes the proof of Proposition \ref{prop-cQmu}.

\appendix 

\section{Fourier multipliers}
\label{sec-FM}
We recall the homogeneous Littlewood-Paley decomposition on $\mathbb{R}^{d}$, $d \in \mathbb{N}$, which reads
\begin{equation}\label{def-LP}
h= \sum_{k \in \mathbb{Z}} P_k h, 
\end{equation}
where $P_k$ denotes the standard Littlewood-Paley projection on the dyadic interval $[2^{k-1}, 2^{k+1}]$. We note the following classical Bernstein inequalities (see, e.g.,
\cite{BCD})
\begin{equation}\label{Berstein}
\| P_k \partial_x h \|_{L^p} \lesssim 2^{k} \| h\|_{L^p}, \qquad  2^{k} \| P_kh\|_{L^p} \lesssim \| \partial_x h\|_{L^p_x}
\end{equation}
for all $p \in [1,\infty]$ and $k \in \mathbb{Z}$. In addition, we recall the following definition of Besov spaces.

\begin{defi}
Let  $s \in \mathbb{R}$, $p,r \in [1,+\infty]$. The Besov space $B^{s}_{p,r}$ is defined so that the following norm is finite
$$
\| h\|_{B^{s}_{p,r}} =  \left\| \sum_{k\leq 0} P_k h \right\|_{L^p} +\left ( \sum_{k\ge 0} 2^{s kr}\| P_k h\|^r_{L^p} \right)^{1/r}< +\infty  .
$$
\end{defi}

\begin{lem}
\label{lem:fouriermult}
 Fix $\delta>0$. 
Let $\sigma(\xi)$ be sufficiently smooth and satisfy 
$$|\partial^\alpha \sigma(\xi)| \leq C_\alpha \langle \xi \rangle^{-\delta-|\alpha|} $$ 
for $|\alpha|\ge 0$.
Then, there hold 

\begin{enumerate}

\item The Fourier multiplier $\sigma(\xi)$ is a bounded operator from $L^p$ to $L^p$ for $p \in [1,\infty]$. 

\item We have for all $p \in [1,\infty]$,
\begin{equation}
\label{eq:fouriermult}
\|  \sigma(i\partial) \nabla  f \|_{L^p} \lesssim  \| f \|_{L^p}^{\frac{\delta}{1+\delta}}  \| \nabla f \|_{L^p}^{\frac{1}{1+\delta}} \lesssim \| f\|_{W^{1,p}}.
\end{equation}

\item We have for all $p \in [1,+\infty]$, $r \in [1,\infty)$, and $s\ge 0$ so that $\delta > s - \lfloor s \rfloor$, 
\begin{equation}
\label{eq:fouriermultbesov}
\|  \sigma(i\partial) \nabla  f \|_{B^{s}_{p,r}} \lesssim   \| f \|_{L^p}^{\frac{\delta}{1+\delta}}  \| \nabla f \|_{L^p}^{\frac{1}{1+\delta}}+ \| \nabla^{1+\lfloor s \rfloor} f\|_{L^p} \lesssim \| f\|_{W^{1+\lfloor s \rfloor,p}},
\end{equation}
where $\lfloor s \rfloor$ denotes the largest integer that is smaller than $s$. 

\end{enumerate}
\end{lem}

\begin{proof} Using the Littlewood-Paley decomposition, we write
$$
\sigma(i\partial)f = P_{k\le 0}(  \sigma(i\partial) f)  +  \sum_{k \geq 0 } P_k( \sigma(i\partial)  f)
$$
in which $P_{k\le 0}(  \sigma(i\partial) f)$ is a convolution in space of a rapidly decaying function against $f$. The boundedness from $L^p$ to $L^p$ thus follows. On the other hand, using Berstein's inequalities, we bound   
$$ \| \sum_{k \geq 0 } P_k( \sigma(i\partial)  f)\|_{L^p}\le  \sum_{k \geq 0 } \|P_k( \sigma(i\partial)  f)\|_{L^p} \lesssim \sum_{k\ge 0} 2^{-k\delta} \| f\|_{L^p} \lesssim \| f\|_{L^p},$$
in which the summation is finite, since $\delta>0$. As for (2), again using the Littlewood-Paley decomposition, we write
$$
\sigma(i\partial)  \nabla  f = \sum_{k \le A } P_k(  \sigma(i\partial) \nabla  f)  +  \sum_{k \geq A } P_k( \sigma(i\partial)  \nabla  f) =: f_1 + f_2, 
$$
for $A\in \ZZ$ to be determined. 
Using \eqref{Berstein}, we first bound 
$$
\| f_1\|_{L^p} \lesssim \sum_{k <A } \| P_k( \sigma(i\partial) \nabla  f)\|_{L^p} \lesssim  \sum_{k <A } 2^k    \| f\|_{L^p}  \lesssim 2^A \| f\|_{L^p}. 
$$
Similarly, recalling $| \sigma(\xi)| \lesssim\langle \xi \rangle^{-\delta} $, we bound 
$$
\| f_2\|_{L^p} \lesssim \sum_{k \geq A } \| P_k( \sigma(i\partial) \nabla  f)\|_{L^p} 
\lesssim  \sum_{k \geq A } 2^{-\delta k}    \| \nabla f\|_{L^p}  \lesssim 2^{-\delta A}    \|\nabla f\|_{L^p}. 
$$
We can now fix $A$ such that $2^A \| f\|_{L^p}=2^{-\delta A}    \|\nabla f\|_{L^p}$, leading to \eqref{eq:fouriermult}. Finally, by definition and \eqref{Berstein}, we bound 
$$
\begin{aligned}
\|  \sigma(i\partial) \nabla  f \|_{B^{s}_{p,r}} &\lesssim  \|  \sigma(i\partial) \nabla  f \|_{L^p}  
+ \left( \sum_{k \geq 0} 2^{k sr} \| P_k (\sigma(i\partial) \nabla  f)\|_{L^p}^r\right)^{1/r} \\
&\lesssim   \| f \|_{L^p}^{\frac{\delta}{1+\delta}}  \| \nabla f \|_{L^p}^{\frac{1}{1+\delta}}+  \left( \sum_{k \geq 0} 2^{-k(\delta-s+\lfloor s \rfloor) r}\right)^{1/r} \|  \nabla^{1+\lfloor s \rfloor}  f\|_{L^p},
\end{aligned}
$$
yielding \eqref{eq:fouriermultbesov}. The lemma follows. 
\end{proof}

\begin{remark}
Following the proof of Lemma \ref{lem:fouriermult}, we obtain the following standard interpolation inequalities
  \begin{equation}\label{dx-interpolate}
\| \partial^\beta_x f\|_{L^\infty}  \lesssim \| f\|_{L^\infty}^{1-|\beta|/|\alpha|} \| \partial_x^\alpha f\|_{L^\infty}^{|\beta|/|\alpha|} \lesssim \| f\|_{L^\infty}^{1-|\beta|/|\alpha|} \| f\|_{H^{2+\alpha}}^{|\beta|/|\alpha|} 
\end{equation}
for $0 \le |\beta| \le |\alpha|$. 
\end{remark}

\section{Bilinear estimates}

We repeatedly use the following result, which can be seen as a particular case of Coifman-Meyer result:
\begin{lemma}\label{lem-bilinear}
Let 
$$ Q[f,g] (x)= \iint e^{ik_1 \cdot x + i k_2 \cdot x} m(k_1,k_2) \Ff_{k_1} \Fg_{k_2} \; dk_1 d k_2 $$
for a sufficiently smooth Fourier symbol $m(k_1,k_2)$ that satisfies $|\partial_{k_1}^\alpha\partial_{k_2}^\beta m(k_1,k_2)| \lesssim \langle k_1,k_2\rangle^{-1-|\alpha|-|\beta|}$ for $|\alpha|, |\beta|\ge0$.
Then, there holds 
\begin{equation}\label{sup-bilinear} 
 \| Q[f,g]\|_{L^r_x} \lesssim \| f\|_{L^p_x} \| g\|_{L^q_x} \end{equation}
for any $1/r = 1/p+1/q$, with $1\le r,p,q \le \infty$. In case of $r=\infty$, we take $p=q=\infty$. 
\end{lemma}
\begin{proof} 
The case when $r<\infty$ is the standard bilinear estimate which in fact holds for a much larger class of multipliers 
(e.g., of a Coifman-Meyer type). As for the $L^\infty_x$ estimate, we set 
$$ 
F(x,y) = f(x) g(y).
$$
It follows that $Q[f,g](x) = T_m(F)(x,x)$, where 
$$ 
T_m(F)(x,y) = \iint e^{i (k_1,k_2)\cdot (x,y)} m(k_1,k_2) \FF(k_1,k_2) \; dk_1 dk_2, 
$$  
noting $ \FF(k_1,k_2) = \Ff_{k_1} \Fg_{k_2}$. That is, $T_m(F)$ is a Fourier multiplier with sufficiently smooth  
symbol $m(k_1,k_2) $ that is bounded by $\langle k_1,k_2\rangle^{-1}$. Therefore, using Item 1 in Lemma \ref{lem:fouriermult}, $T_m(F)$ is a bounded operator from $L^p$ to $L^p$ for $p\in [1,\infty]$. The $L^\infty$ 
estimate in \eqref{sup-bilinear} thus follows.

\end{proof}

\bibliographystyle{abbrv}

\end{document}